\documentclass[11pt, reqno]{amsart}
\usepackage{txfonts}
\usepackage{amsmath,amssymb,amsthm,mathrsfs,enumerate,bm,xcolor,multirow,pbox}
\usepackage{graphicx,color,framed,tikz,caption,subcaption}
\usepackage{enumitem}
\setlist{leftmargin=9mm}
\usepackage[colorlinks,linkcolor=black,citecolor=black,urlcolor=black]{hyperref}
\allowdisplaybreaks[4]
\numberwithin{equation}{section}
\newcommand{\N}{\mathbb{N}}
\newcommand{\R}{\mathbb{R}}

\newcommand{\pnorm}[2]{\lVert #1\rVert_{#2}}
\newcommand{\bigpnorm}[2]{\big\lVert#1\big\rVert_{#2}}
\newcommand{\biggpnorm}[2]{\bigg\lVert#1\bigg\rVert_{#2}}
\newcommand{\abs}[1]{\lvert#1\rvert}
\newcommand{\bigabs}[1]{\big\lvert#1\big\rvert}
\newcommand{\biggabs}[1]{\bigg\lvert#1\bigg\rvert}
\newcommand{\iprod}[2]{\langle#1,#2\rangle}
\newcommand{\bigiprod}[2]{\big\langle#1,#2\big\rangle}
\newcommand{\biggiprod}[2]{\bigg\langle#1,#2\bigg\rangle}

\renewcommand{\epsilon}{\varepsilon}

\renewcommand{\d}[1]{\mathrm{d}#1}

\newcommand{\smallop}{\mathfrak{o}_{\mathbf{P}}}

\newcommand{\smallo}{\mathfrak{o}}
\newcommand{\bigo}{\mathcal{O}}

\renewcommand{\hat}{\widehat}
\renewcommand{\tilde}{\widetilde}

\DeclareMathOperator{\E}{\mathbb{E}}
\DeclareMathOperator{\Prob}{\mathbb{P}}

\DeclareMathOperator{\tr}{tr}
\DeclareMathOperator{\var}{Var}
\DeclareMathOperator{\cov}{Cov}

\DeclareMathOperator{\op}{op}

\DeclareMathOperator{\dof}{\mathsf{dof}}
\DeclareMathOperator{\err}{\mathsf{err}}
\DeclareMathOperator{\seq}{\mathsf{seq}}

\DeclareMathOperator{\prox}{\mathsf{prox}}

\DeclareMathOperator{\res}{\mathsf{res}}
\DeclareMathOperator{\pred}{\mathsf{pred}}
\DeclareMathOperator{\est}{\mathsf{est}}
\DeclareMathOperator{\ins}{\mathsf{in}}
\DeclareMathOperator{\SNR}{\mathsf{SNR}}

\DeclareMathOperator{\CV}{\mathsf{CV}}
\DeclareMathOperator{\GCV}{\mathsf{GCV}}
\DeclareMathOperator{\dR}{\mathsf{dR}}

\DeclareMathOperator*{\argmax}{arg\,max\,}
\DeclareMathOperator*{\argmin}{arg\,min\,}

\newcommand{\beq}{\begin{equation}}
\newcommand{\eeq}{\end{equation}}
\newcommand{\beqa}{\begin{equation} \begin{aligned}}
\newcommand{\eeqa}{\end{aligned} \end{equation}}
\newcommand{\beqas}{\begin{equation*} \begin{aligned}}
\newcommand{\eeqas}{\end{aligned} \end{equation*}}

\newcommand{\bit}{\begin{itemize}}
	\newcommand{\eit}{\end{itemize}}
\newcommand{\bmat}{\begin{bmatrix}}
	\newcommand{\emat}{\end{bmatrix}}

\theoremstyle{definition}\newtheorem{problem}{Problem}[section]
\theoremstyle{definition}
\theoremstyle{remark}\newtheorem{assumption}{Assumption}

\theoremstyle{remark}\newtheorem{remark}{Remark}
\theoremstyle{definition}
\theoremstyle{plain}\newtheorem{theorem}[problem]{Theorem}
\theoremstyle{plain}
\theoremstyle{plain}\newtheorem{lemma}[problem]{Lemma}
\theoremstyle{plain}\newtheorem{proposition}[problem]{Proposition}
\theoremstyle{plain}
\theoremstyle{plain}

\AtBeginDocument{%
	\def\MR#1{}
}

\begin{document}

\title[The Distribution of Ridgeless Interpolators]{The Distribution of Ridgeless Least Squares Interpolators}

\thanks{}

\author[Q. Han]{Qiyang Han}

\address[Q. Han]{
Department of Statistics, Rutgers University, Piscataway, NJ 08854, USA.
}
\email{qh85@stat.rutgers.edu}

\author[X. Xu]{Xiaocong Xu}

\address[X. Xu]{
	Data Sciences and Operations Department, University of Southern California, Los Angeles, CA 90089, USA
}
\email{xuxiaoco@marshall.usc.edu}

\date{\today}

\keywords{comparison inequality, cross validation, minimum norm interpolator, random matrix theory, ridge regression, universality}
\subjclass[2000]{60E15, 60G15}

\begin{abstract}
The Ridgeless minimum $\ell_2$-norm interpolator in overparametrized linear regression has attracted considerable attention in recent years in both machine learning and statistics communities. While it seems to defy conventional wisdom that overfitting leads to poor prediction, recent theoretical research on its $\ell_2$-type risks reveals that its norm minimizing property induces an `implicit regularization' that helps prediction in spite of interpolation.

This paper takes a further step that aims at understanding its precise stochastic behavior as a statistical estimator. Specifically, we characterize the distribution of the Ridgeless interpolator in high dimensions, in terms of a Ridge estimator in an associated Gaussian sequence model with positive regularization, which provides a precise quantification of the prescribed implicit regularization in the most general distributional sense. Our distributional characterizations hold for general non-Gaussian random designs and extend uniformly to positively regularized Ridge estimators. 

As a direct application, we obtain a complete characterization for a general class of weighted $\ell_q$ risks of the Ridge(less) estimators that are previously only known for $q=2$ by random matrix methods. These weighted $\ell_q$ risks not only include the standard prediction and estimation errors, but also include the non-standard covariate shift settings. Our uniform characterizations further reveal a surprising feature of the commonly used generalized and $k$-fold cross-validation schemes: tuning the estimated $\ell_2$ prediction risk by these methods alone lead to simultaneous optimal $\ell_2$ in-sample, prediction and estimation risks, as well as the optimal length of debiased confidence intervals. 
\end{abstract}

\maketitle

\setcounter{tocdepth}{1}
\tableofcontents

\sloppy

\section{Introduction}

\subsection{Overview}

Consider the standard linear regression model
\begin{align}\label{eqn:model}
Y_i=X_i^\top \mu_0 + \xi_i,\quad 1\leq i\leq m,
\end{align}
where we observe i.i.d. feature vectors $X_i \in \R^{n}$ and responses $Y_i\in \R$, and $\xi_i$'s are unobservable errors. For notational simplicity, we write $X=[X_1\cdots X_m]^\top \in \R^{m\times n}$ as the design matrix that collects all the feature vectors, and $Y=(Y_1,\ldots,Y_m)^\top \in \R^m$ as the response vector. The feature vectors $X_i$'s are assumed to satisfy $\E X_1 =0$ and $\cov(X_1)=\Sigma$, and the errors satisfy $\E \xi_1=0$ and $\var(\xi_1)=\sigma_\xi^2$.  

Throughout this paper, we reserve $m$ for the sample size, and $n$ for the signal dimension. The aspect ratio $\phi$, i.e., the number of samples per dimension, is then defined as $\phi\equiv m/n$. Accordingly, we refer to $\phi^{-1}>1$ as the \emph{overparametrized regime}, and $\phi^{-1}<1$ as the \emph{underparametrized regime}.

Within the linear model (\ref{eqn:model}), the main object of interest is to recover/estimate the unknown signal $\mu_0 \in \R^n$. While a large class of regression techniques can be used for the purpose of signal recovery under various structural assumptions on $\mu_0$, here we will focus our attention on one widely used class of regression estimators, namely, the \emph{Ridge estimator} (cf. \cite{hoerl1970ridge}) with regularization $\eta>0$,
\begin{align}\label{def:ridge_est}
\hat{\mu}_{\eta} =\argmin_{\mu \in \R^n} \bigg\{\frac{1}{2n}\pnorm{Y-X\mu}{}^2+\frac{\eta}{2}\pnorm{\mu}{}^2\bigg\}=\frac{1}{n}\bigg(\frac{1}{n}X^\top X+\eta I_n\bigg)^{-1}X^\top Y,
\end{align}
and the \emph{Ridgeless estimator} (also known as the \emph{minimum-norm interpolator}), 
\begin{align}\label{def:minimum_norm_est}
\hat{\mu}_{0}=\argmin_{\mu \in \R^n}\big\{ \pnorm{\mu}{}^2:  Y=X\mu\big\}=(X^\top X)^{-}X^\top Y,
\end{align}
which is almost surely (a.s.) well-defined in the overparametrized regime $\phi^{-1}>1$. Here $A^{-}$ is the Moore-Penrose pseudo-inverse of $A$. The notation $\hat{\mu}_0$ is justified since for $\phi^{-1}>1$, $\hat{\mu}_\eta\to \hat{\mu}_0$ a.s. as $\eta\downarrow 0$.
 
From a conventional statistical point of view, the Ridgeless estimator seems far from an obviously good choice: As $\hat{\mu}_0$ perfectly interpolates the data, it is susceptible to high variability due to the widely recognized bias-variance tradeoff inherent in `optimal' statistical estimators \cite{james2021introduction,derumigny2023lower}. On the other hand, as the Ridgeless estimator $\hat{\mu}_0$ is the limit point of the gradient descent algorithm run on the squared loss in the overparametrized regime $\phi^{-1}>1$, it provides a simple yet informative test case for understanding one major enigma of modern machine learning methods: these methods typically interpolate training data perfectly; still, they enjoy good generalization properties \cite{jacot2018neural,du2019gradient,allen2019convergence,belkin2019reconciling,chizat2019lazy,zhang2021understanding}. 

Inspired by this connection, recent years have witnessed a surge of interest in understanding the behavior of the Ridgeless estimator $\hat{\mu}_0$ and its closely related Ridge estimator $\hat{\mu}_\eta$, with an exclusive focus on their prediction risks, cf. \cite{tulino2004random,elkaroui2013asymptotic,hsu2014random,dicker2016ridge,dobriban2018high,elkaroui2018impact,advani2020high,belkin2020two,muthukumar2020harmless,wu2020optimal,bartlett2020benign,bartlett2021deep,chang2021provable,koehler2021uniform,richards2021asymptotics,hastie2022surprises,tsigler2023benign,cheng2024dimension,zhou2024optimistic}. The readers are referred to \cite[Sections 1.2 \& 9]{tsigler2023benign} for a thorough review on the relation between various $\ell_2$ risk results for the Ridge(less) estimator. A unique insight from these works is the existence of `implicit regularization' within the Ridgeless interpolator $\hat{\mu}_0$, so that for certain scenarios of $(\Sigma,\mu_0)$, the prediction risk of $\hat{\mu}_0$ could be small (i.e., benign overfitting) or even optimal \cite{kobak2020optimal,hastie2022surprises,tsigler2023benign}.

Despite substantial progress in understanding the $\ell_2$ risk behavior of the Ridgeless estimator $\hat{\mu}_0$, our understanding of its stochastic behavior as a statistical estimator remains limited. This gap is particularly important if we aim to consider $\hat{\mu}_0$ also as a `good' estimator that can be applied in a broader context of statistical inference tasks, rather than merely viewing it as a theoretical proxy for modern interpolating learning algorithms.

The main goal of this paper is to advance our understanding of the precise stochastic behavior of the Ridge(less) estimator $\hat{\mu}_\eta$. We achieve this by developing a high-dimensional \emph{distributional} characterization in the so-called proportional regime where $m$ and $n$ is of the same order. This approach allows us to move beyond the exclusive focus in the existing literature on $\ell_2$-type risks of $\hat{\mu}_\eta$. As will be clear, the distributional characterization of the Ridge(less) estimator $\hat{\mu}_\eta$ not only provides a precise quantitative understanding of the `implicit regularization' phenomenon for the Ridgeless interpolator $\hat{\mu}_0$ in the most general distributional sense, but also unveils major new insights on the utility of the widely used cross-validation schemes in machine learning/statistics practice.

\subsection{Distribution of Ridge(less) estimators}

Before formally describing our high dimensional distributional characterization, it is insightful to consider the low dimensional regime $\phi^{-1}\ll 1$ where the sample size $m$ far exceeds the signal dimension $n$. In this regime, with $(\bar{\eta},\bar{\sigma}^2_\xi)\equiv (\eta,\sigma^2_\xi)/\phi$, using the closed form of (\ref{def:ridge_est}) and the fact that $m^{-1}X^\top X\approx \Sigma$, we may safely regard $
\hat{\mu}_\eta \approx \big(\Sigma+\bar{\eta} I_n\big)^{-1} \big(\Sigma \mu_0+ m^{-1}X^\top \xi\big)$. Using central limit theorem for $m^{-1}X^\top \xi\stackrel{d}{\approx} \bar{\sigma}_\xi\cdot n^{-1/2}\Sigma^{1/2}g$ where $g\sim \mathcal{N}(0,I_n)$, we have
\begin{align}\label{eqn:ridge_dist_lowd}
\hat{\mu}_\eta \stackrel{d}{\approx} (\Sigma+\bar{\eta} I_n)^{-1}\Sigma^{1/2} \big(\Sigma^{1/2}\mu_0+n^{-1/2}\cdot\bar{\sigma}_\xi g\big),\quad \phi^{-1}\ll 1.
\end{align}
A principled way to understand the above formula (\ref{eqn:ridge_dist_lowd}) is to consider an `effective regression problem' in the \emph{Gaussian sequence model}. Suppose for a given pair of $(\Sigma,\mu_0)$ and a noise level $\gamma>0$, we observe
\begin{align}\label{eqn:gaussian_seq}
y^{\seq}_{(\Sigma,\mu_0)}(\gamma) \equiv  \Sigma^{1/2}\mu_0+n^{-1/2}\cdot \gamma g,\quad g\sim \mathcal{N}(0,I_n).
\end{align}
The Ridge estimator $\hat{\mu}_{(\Sigma,\mu_0)}^{\seq}(\gamma;\tau)$ with regularization $\tau\geq 0$ in the Gaussian sequence model (\ref{eqn:gaussian_seq}) is defined as 
\begin{align}\label{def:ridge_seq}
\hat{\mu}_{(\Sigma,\mu_0)}^{\seq}(\gamma;\tau)&\equiv \argmin_{\mu \in \R^n} \bigg\{\frac{1}{2}\pnorm{\Sigma^{1/2}\mu-y_{(\Sigma,\mu_0)}^{\seq}(\gamma)}{}^2+ \frac{\tau}{2}\pnorm{\mu}{}^2\bigg\}\nonumber\\
& = (\Sigma+\tau I_n)^{-1}\Sigma^{1/2} \big(\Sigma^{1/2}\mu_0+n^{-1/2}\cdot \gamma g\big).
\end{align}
Here, the subscript $(\Sigma,\mu_0)$ emphasizes the dependence on the underlying Gaussian sequence model with covariance $\Sigma$ and signal $\mu_0$.
Comparing (\ref{eqn:ridge_dist_lowd}) and (\ref{def:ridge_seq}), it is clear that we may interpret (\ref{eqn:ridge_dist_lowd}) as $\hat{\mu}_\eta \stackrel{d}{\approx} \hat{\mu}_{(\Sigma,\mu_0)}^{\seq}(\bar{\sigma}_\xi;\bar{\eta})$. In the proportional regime $\phi^{-1}\asymp 1$, the aforementioned interpretation still applies, but a crucial modification will be needed: the pair of the (scaled) original noise and regularization $(\bar{\sigma}_\xi,\bar{\eta})$ must be replaced by a pair of `\emph{effective noise and regularization}' 
\begin{align}\label{def:gamma-tau}
    (\gamma_{\eta,\ast},\tau_{\eta,\ast}) \equiv \text{unique solution of the fixed point equation (\ref{eqn:fpe})}
\end{align}
when $\eta>0$ and when $\eta=0$ in the overparametrized regime (cf. Proposition \ref{prop:fpe_est_simplify}).

More precisely, in the overparametrized regime $\phi^{-1}>1$, under standard assumptions on (i) the design matrix $X=\Sigma^{1/2}Z$, where $Z$ consists of independent mean $0$, unit-variance and light-tailed entries, and (ii) the error vector $\xi$ with light-tailed components, we show in Theorems \ref{thm:min_norm_dist} and \ref{thm:universality_min_norm} that the distribution $\hat{\mu}_\eta$ can be characterized via $\hat{\mu}_{(\Sigma,\mu_0)}^{\seq}(\gamma_{\eta,\ast};\tau_{\eta,\ast})$ in the following sense: for any $1$-Lipschitz function $\mathsf{g}: \R^n \to \R$ and any $K>0$, with high probability,
\begin{align}\label{eqn:ridge_dist_uniform}
\sup_{\eta \in [0,K]} \bigabs{\mathsf{g}(\hat{\mu}_\eta)- \E\mathsf{g}\big( \hat{\mu}_{(\Sigma,\mu_0)}^{\seq}(\gamma_{\eta,\ast};\tau_{\eta,\ast})\big)}\approx 0.
\end{align}
A particularly important technical aspect of (\ref{eqn:ridge_dist_uniform}) is that the distributional approximation (\ref{eqn:ridge_dist_uniform}) holds uniformly down to the interpolation regime $\eta=0$ for $\phi^{-1}>1$. This uniform guarantee will prove essential in the results ahead.

Interestingly, the distributional characterization (\ref{eqn:ridge_dist_uniform}) offers a principled approach to understand the `implicit regularization' phenomenon for the Ridgeless interpolator $\hat{\mu}_0$, through the lens of its distributionally equivalent Ridge estimator $\hat{\mu}_{(\Sigma,\mu_0)}^{\seq}(\gamma_{0,\ast},\tau_{0,\ast})$ in the Gaussian sequence model (\ref{eqn:gaussian_seq}). Specifically, the prescribed implicit regularization can be directly attributed to the quantity $\tau_{0,\ast}>0$ that can be solved as the unique positive solution to the equation
\begin{align}\label{eqn:implicit_reg_intro}
\phi = \frac{1}{n} \tr\big((\Sigma+\tau_{0,\ast} I_n)^{-1}\Sigma\big).
\end{align}
While this interpretation has been suggested in the context of $\ell_2$ risks \cite{hastie2022surprises,cheng2024dimension} via a-posterior calculations, our theory (\ref{eqn:ridge_dist_uniform}) provides a formal justification for this equivalent understanding of the implicit regularization phenomenon for $\hat{\mu}_0$ via $\hat{\mu}_{(\Sigma,\mu_0)}^{\seq}$, in the most precise and general distributional sense. The readers are referred to Section \ref{subsection:intro_further_literature} for a more detailed comparison on the relation of our characterization of the implicit regularization via $\tau_{0,\ast}$ and a different line of interpretation in \cite{bartlett2020benign,bartlett2021deep,tsigler2023benign}.

\subsection{General $\ell_q$-type risk formulae}\label{subsection:intro_phase_transition}
As mentioned above, most prior works on the risk properties of Ridge(less) estimators $\hat{\mu}_\eta$ have focused exclusively on $\ell_2$-type risks, leveraging random matrix theory (RMT)  \cite{tulino2004random,elkaroui2013asymptotic,dicker2016ridge,dobriban2018high,elkaroui2018impact,advani2020high,wu2020optimal,bartlett2021deep,richards2021asymptotics,hastie2022surprises,cheng2024dimension}. This RMT approach is viable due to a direct reduction of $\ell_2$-type risks of $\hat{\mu}_\eta$ to the spectrum of $X$. In contrast, the more general $\ell_q$ risks depend not only on the spectrum but also on the structure of $X$'s singular vectors in a highly nontrivial manner; therefore, the feasibility of a similar RMT-based analysis is in question.

Our uniform distributional theory in (\ref{eqn:ridge_dist_uniform}) is strong enough to characterize all $\ell_q$ risks of the Ridge(less) estimator. Specifically, for any $q \in [1,\infty)$ and a p.s.d. matrix $\mathsf{A} \in \R^{n\times n}$, with high probability,
\begin{align}\label{eqn:intro_lq_risk}
	\frac{\pnorm{\mathsf{A}\hat{\mu}_\eta - \mu_0}{q}}{n^{-1/2}\pnorm{\mathrm{diag}\big(\Gamma_{\eta;(\Sigma,\pnorm{\mu_0}{})}^{\mathsf{A}}\big) }{q/2}^{1/2}M_q  } \approx 1,
\end{align}
where $M_q=  \E^{1/q}\abs{\mathcal{N}(0,1)}^q$ and $\Gamma_{\eta;(\Sigma,\pnorm{\mu_0}{})}^{\mathsf{A}} =  \mathsf{A}(\Sigma+\tau_{\eta,\ast} I_n)^{-1}( \tilde{\gamma}_{\eta,\ast}^2(\pnorm{\mu_0}{}) \Sigma+\tau_{\eta,\ast} ^2 \pnorm{\mu_0}{}^2 I_n)(\Sigma+\tau_{\eta,\ast} I_n)^{-1}\mathsf{A}$; see Theorem \ref{thm:lq_risk} for the precise definition of $\tilde{\gamma}_{\eta,\ast}^2(\pnorm{\mu_0}{})$ and the formal statement of the above result (\ref{eqn:intro_lq_risk}).

Beyond providing a precise characterization of all $\ell_q$ risks, the uniform nature of (\ref{eqn:ridge_dist_uniform}) also illuminates novel insights into certain global, qualitative behavior of the most commonly studied $\ell_2$ risks for finite samples. To fix notation, we define
\begin{itemize}
	\item (\emph{prediction risk}) $R^{\pred}_{(\Sigma,\mu_0)}(\eta)\equiv \pnorm{\Sigma^{1/2}(\hat{\mu}_{\eta}-\mu_0)}{}^2$,
	\item (\emph{estimation risk}) $R^{\est}_{(\Sigma,\mu_0)}(\eta)\equiv \pnorm{\hat{\mu}_{\eta}-\mu_0}{}^2$,
	\item (\emph{in-sample risk}) $R^{\ins}_{(\Sigma,\mu_0)}(\eta)\equiv n^{-1}\pnorm{X(\hat{\mu}_{\eta}-\mu_0)}{}^2$.
\end{itemize}
Using our uniform distributional characterization in (\ref{eqn:ridge_dist_uniform}), we show that for `most' $\mu_0$'s, the global optimum of $\eta \mapsto R^{\#}_{(\Sigma,\mu_0)}(\eta)$ for all $\# \in \{\pred,\est,\ins\}$ will be achieved approximately \emph{at the same point $\eta_\ast = \SNR_{\mu_0}^{-1}$} with high probability\footnote{Here $
\SNR_{\mu_0} = \pnorm{\mu_0}{}^2/\sigma_\xi^2$ is the usual notion of signal-to-noise ratio; when $\mu_0\neq 0$ and $\sigma_\xi^2=0$, we shall interpret  $\SNR_{\mu_0}^{-1}=0$.} (cf. Theorem \ref{thm:small_intep}).

It must be stressed that, for different $\# \in \{\pred,\est,\ins\}$, the empirical risk curves $\eta \mapsto R^{\#}_{(\Sigma,\mu_0)}(\eta)$ concentrate on genuinely different deterministic counterparts $\eta\mapsto \bar{R}^{\#}_{(\Sigma,\mu_0)}(\eta)$ with different mathematical expressions (cf. Theorem \ref{thm:errors_explicit}). As such, there are no apriori reasons to expect that these risk curves share approximately the same global minimum.  Remarkably, as a consequence of the approximate formulae for the deterministic risk curves $\eta\mapsto \bar{R}^{\#}_{(\Sigma,\mu_0)}(\eta)$ (cf. Theorem \ref{thm:error_rmt}), we show that the curves $\eta\mapsto \bar{R}^{\#}_{(\Sigma,\mu_0)}(\eta)$ are qualitatively similar, in that they approximately behave locally like a quadratic function centered around $\eta_\ast=\SNR_{\mu_0}^{-1}$ (cf. Proposition \ref{prop:derivative_R_simplify}), at least for `most' signal $\mu_0$'s.

\subsection{Cross-validation: optimality beyond prediction}

\begin{figure}[t]
	\begin{minipage}[t]{0.3\textwidth}
		\includegraphics[width=\textwidth]{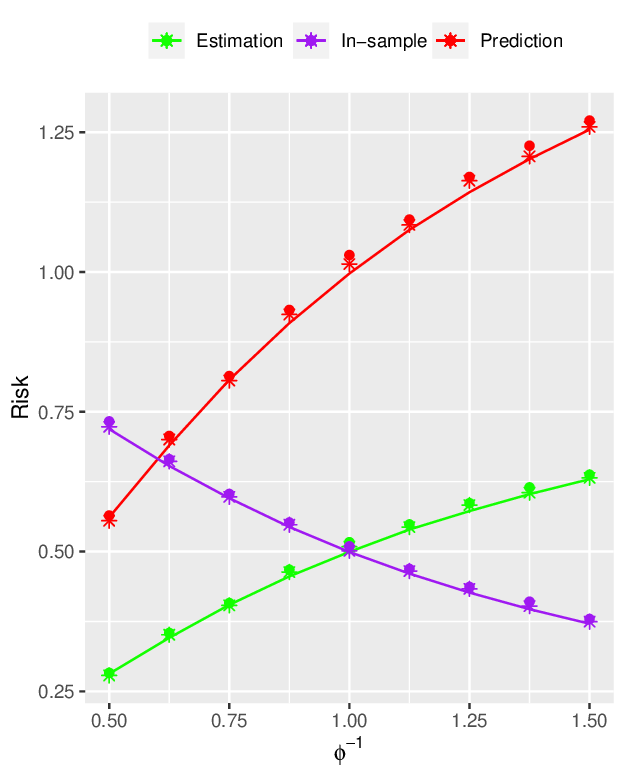}
	\end{minipage}
	\begin{minipage}[t]{0.3\textwidth}
		\includegraphics[width=\textwidth]{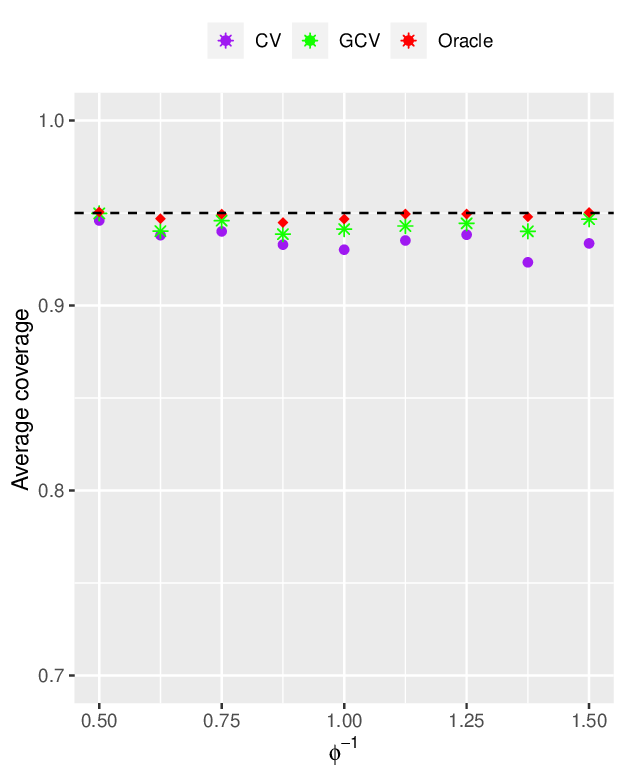}
	\end{minipage}
	\begin{minipage}[t]{0.3\textwidth}
		\includegraphics[width=\textwidth]{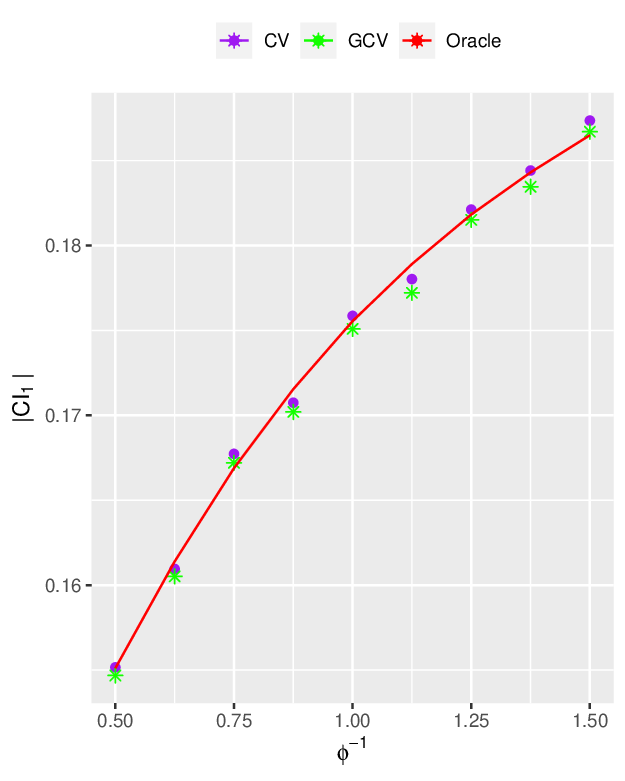}
	\end{minipage}
	\caption{\emph{Left panel}: Comparison between empirical risks and theoretical risks for $\ast=\GCV$ and $\bullet = \CV$ with $k=5$.  \emph{Middle panel}: Averaged CI coverage $\mathscr{C}^{\dR}(\hat{\eta}^{\#})$ for $\# \in \{\GCV,\CV\}$ and the oracle $\mathscr{C}^{\dR}(\eta_\ast)$. \emph{Right panel}: CI length of  $\mathrm{CI}_1(\hat{\eta}^{\#})$ for $\# \in \{\GCV,\CV\}$ and the oracle CI length. See Section \ref{section:application} for the precise definitions.}
	\label{fig:2}
\end{figure}

The discussion in Section \ref{subsection:intro_phase_transition} naturally raises the question of how one can choose the optimal regularization in a data-driven manner. Here we study two widely used adaptive tuning methods, namely,
\begin{enumerate}
	\item the generalized cross-validation scheme $\hat{\eta}^{\GCV}$, and
	\item the $k$-fold cross-validation scheme $\hat{\eta}^{\CV}$. 
\end{enumerate}
The readers are referred to (\ref{def:res_tuning}) and (\ref{def:CV_tuning}) for precise definitions and literature review of $\hat{\eta}^{\GCV},\hat{\eta}^{\CV}$ in the context of Ridge regression.

By design, both methods $\hat{\eta}^{\GCV},\hat{\eta}^{\CV}$ are intended to estimate the prediction risk, so it is natural to expect that they perform well for the task of prediction. Interestingly, the insight from Section \ref{subsection:intro_phase_transition} suggests a far broader utility of these adaptive tuning methods. Indeed, as all the empirical risk curves $\eta \mapsto R^{\#}_{(\Sigma,\mu_0)}(\eta)$ are approximately minimized at the same point $\eta_\ast = \SNR_{\mu_0}^{-1}$, it is reasonable to conjecture that  $\hat{\eta}^{\GCV},\hat{\eta}^{\CV}$ could also yield optimal performance for estimation and in-sample risks. We show in Theorems \ref{thm:tuning_res} and \ref{thm:tuning_cv} that this is indeed the case: for `most' signal $\mu_0$'s and all $\# \in \{\pred,\est,\ins\}$, with high probability,
\begin{align}\label{eqn:intro_cv}
R^{\#}_{(\Sigma,\mu_0)}(\hat{\eta}^{\GCV}), R^{\#}_{(\Sigma,\mu_0)}(\hat{\eta}^{\CV})\approx\min_{\eta \in [0,K]} R^{\#}_{(\Sigma,\mu_0)}(\eta).
\end{align}
A typical simulation for this phenomenon is reported in the left panel of Figure \ref{fig:2} above, where empirical risks tuned by $\hat{\eta}^{\GCV}$ (in $\ast$) and $\hat{\eta}^{\CV}$ (in $\bullet$) achieve optimal theoretical risks (in solid lines) for estimation and in-sample risks as well.

Even more surprisingly, the optimality of $\hat{\eta}^{\GCV},\hat{\eta}^{\CV}$ extends to the much more challenging task of statistical inference. In fact, we show in Theorem \ref{thm:CI_cv} that within the so-called debiased Ridge scheme, these two adaptive tuning methods $\hat{\eta}^{\GCV},\hat{\eta}^{\CV}$ yield an asymptotically valid construction of confidence intervals for the coordinates of $\mu_0$ with the shortest possible length. This is numerically validated in the middle and right panels of Figure \ref{fig:2}.

To the best of our knowledge, theoretical optimality properties for the  cross-validation schemes beyond the realm of prediction accuracy has not been established in the literature, either for Ridge regression or for other regularized regression estimators. 

On the other hand, in the related Lasso setting, some numerical evidence for the broader utility of cross-validation and other adaptive tuning methods is reported in \cite[Figure 1]{miolane2021distribution}. There it is shown that  the SURE method, which is designed to tune in-sample risk, nearly matches the performance of 
$k$-fold cross-validation in prediction tasks, despite not being expected to perform well in prediction apriori. Our findings here in the context of Ridge regression can therefore be viewed as a first step toward understanding the broader potential of cross validation and other adaptive tuning schemes for a wider range of statistical inference problems.

\subsection{Further literature}\label{subsection:intro_further_literature}

\subsubsection{Relation to mean-field asymptotics}

Our distributional theory (\ref{eqn:ridge_dist_uniform}) for the Ridge(less) estimator is closely related to a recent line of research that examines the mean-field behavior of statistical estimators in the proportional regime $m\asymp n$, see, e.g. \cite{bayati2012lasso,donoho2016high,thrampoulidis2018precise,sur2019modern,li2021minimum,miolane2021distribution,liang2022precise,celentano2023lasso,han2023universality} for an incomplete list and many more references can be found therein. 

A key feature of this line of works is the use of a simplified `effective' regression problem to understand the complicated behavior of the original statistical estimator. For instance, in the closely related Lasso setting, the `equivalence' between the Lasso estimator $\hat{\mu}_\eta^{\mathsf{L}}$ in the linear model and a corresponding Lasso estimator in the sequence model $\hat{\mu}^{\seq,\mathsf{L}}_{\Sigma,\mu_0}(\gamma_{\eta,\ast}^{\mathsf{L}}, \tau_{\eta,\ast}^{\mathsf{L}})$ has been established under Gaussian designs with positive regularization. This equivalence was first shown for $\ell_2$-type risks in \cite{bayati2012lasso}, and later in the distributional sense akin to (\ref{eqn:ridge_dist_uniform}) in \cite{miolane2021distribution,celentano2023lasso}. Such equivalence for Lasso is further extended to the interpolating regime in \cite{li2021minimum} for the $\ell_2$ risk under a standard Gaussian isotropic design. Our theory (\ref{eqn:ridge_dist_uniform}) here can thus be placed into a similar position as the progress made in  \cite{miolane2021distribution,celentano2023lasso} over the Lasso risk characterization in \cite{bayati2012lasso}, but now in the context of Ridge(less) estimator beyond a purely $\ell_2$ risk as obtained in the references cited above.

While we have developed our distributional theory (\ref{eqn:ridge_dist_uniform}) primarily in the proportional regime $m\asymp n$, we conjecture that our theory (\ref{eqn:ridge_dist_uniform}) remains valid in the full nonparametric regime in which the $\ell_2$ risk of the Ridge(less) estimator exceeds $\bigo(m^{-1/2})$. Some progress in this direction is made in \cite{han2023noisy} in a related context of convex-constrained least squares estimator under a Gaussian design. 

\subsubsection{Relation to existing interpretation of `implicit regularization'}
A separate line of research  \cite{bartlett2020benign,bartlett2021deep,tsigler2023benign} offers a different perspective on the implicit regularization phenomenon within the Ridgeless interpolator $\hat{\mu}_0$. Specifically, by writing $X = [X_{\leq k}, X_{>k}]$ with `effective dimension' $k$ and expressing the Ridgeless interpolator as $\hat{\mu}_0 = X^\top (X_{\leq k}X_{\leq k}^\top + X_{> k}X_{> k}^\top )^{-1}Y$, this line of research identifies covariance structures $\Sigma$ for which $X_{> k}X_{> k}^\top$ scales proportionally to the identity matrix (in a suitable sense). This implies that  $X_{> k}X_{> k}^\top$ qualitatively plays the same role as if positive Ridge regularization were applied to the effective data $X_{\leq k}$. In particular, this line of theory suggests that the prediction risk of $\hat{\mu}_0$ can indeed vanish (i.e., benign overfitting), provided that the eigen-decay of $\Sigma$ is neither too fast nor too slow. 

While this approach is insightful, it falls short of providing an \emph{exact} understanding for the emergence of the implicit regularization phenomenon. This is so, as this approach seeks sufficient conditions for $X_{>k}X_{>k}^\top\asymp I$, and produces risk bounds for $\hat{\mu}_0$ modulo unspecified multiplicative constants. In contrast, our characterization of the implicit regularization via  (\ref{eqn:implicit_reg_intro}) is exact up to the leading constant order, and is susceptible to be also exact in other regimes as well; see \cite{cheng2024dimension} for some recent partial progress along this line. 

Furthermore, both the approaches of \cite{bartlett2020benign,bartlett2021deep,tsigler2023benign} and \cite{cheng2024dimension} rely heavily on the closed form of the Ridgeless interpolator and thus do not generalize to more general interpolators. In contrast, a significant technical advantage of our characterization of implicit regularization via (\ref{eqn:implicit_reg_intro}) lies in its natural connection to the mean-field theory for general regression estimators. This suggests a general paradigm to quantify the implicit regularization for a large class of interpolators via mean-field asymptotics. For instance, the minimum $\ell_1$-norm interpolator studied in \cite{li2021minimum} demonstrates implicit regularization in the prediction risk that can be characterized via $\tau_{0,\ast}^{\mathsf{L}}$ as the `Lasso' version of (\ref{eqn:implicit_reg_intro}). Our approach developed here for Ridgeless interpolator is expected to be useful for quantifying the implicit regularization phenomenon for a more general class of interpolators.

\subsection{Organization}
The rest of the paper is organized as follows. In Section \ref{section:main_results}, we present our main results on the distributional characterizations (\ref{eqn:ridge_dist_uniform}) of the Ridge(less) estimator $\hat{\mu}_\eta$. In Section \ref{sec:risk_asymptotic}, we provide a number of approximate $\ell_q$ risk formulae, and derive the optimal regularization for $\ell_2$ risks. In Section \ref{section:application}, we give a formal validation for the two cross validation schemes mentioned above, both in terms of (\ref{eqn:intro_cv}) and statistical inference via the debiased Ridge estimator. Due to the high technicalities involved in the proof of (\ref{eqn:ridge_dist_uniform}), a proof outline will be given in Section \ref{section:proof_outline}. All the proof details are then presented in Appendices \ref{section:preliminaries} to \ref{section:proof_application}.

\subsection{Notation}
For any positive integer $n$, let $[n]=[1:n]$ denote the set $\{1,\ldots,n\}$. For $a,b \in \R$, $a\vee b\equiv \max\{a,b\}$ and $a\wedge b\equiv\min\{a,b\}$. For $a \in \R$, let $a_\pm \equiv (\pm a)\vee 0$. For $x \in \R^n$, let $\pnorm{x}{p}$ denote its $p$-norm $(0\leq p\leq \infty)$, and $B_{n;p}(R)\equiv \{x \in \R^n: \pnorm{x}{p}\leq R\}$. We simply write $\pnorm{x}{}\equiv\pnorm{x}{2}$ and $B_n(R)\equiv B_{n;2}(R)$. For a matrix $M \in \R^{m\times n}$, let $\pnorm{M}{\op},\pnorm{M}{F}$ denote the spectral and Frobenius norm of $M$, respectively. $I_n$ is reserved for an $n\times n$ identity matrix, written simply as $I$ (in the proofs) if no confusion arises. For a square matrix $M\in \R^{n\times n}$, we let $\mathrm{diag}(M)\equiv (M_{ii})_{i=1}^n \in \R^n$.

We use $C_{x}$ to denote a generic constant that depends only on $x$, whose numeric value may change from line to line unless otherwise specified. $a\lesssim_{x} b$ and $a\gtrsim_x b$ mean $a\leq C_x b$ and $a\geq C_x b$, abbreviated as $a=\bigo_x(b), a=\Omega_x(b)$ respectively;  $a\asymp_x b$ means $a\lesssim_{x} b$ and $a\gtrsim_x b$, abbreviated as $a=\Theta_x(b)$. $\bigo$ and $\smallo$ (resp. $\mathcal{O}_{\mathbf{P}}$ and $\mathfrak{o}_{\mathbf{P}}$) denote the usual big and small O notation (resp. in probability). For a random variable $X$, we use $\Prob_X,\E_X$ (resp. $\Prob^X,\E^X$) to indicate that the probability and expectation are taken with respect to $X$ (resp. conditional on $X$).

For a measurable map $f:\R^n \to \R$, let $\pnorm{f}{\mathrm{Lip}}\equiv \sup_{x\neq y} \abs{f(x)-f(y)}/\pnorm{x-y}{}$. $f$ is called \emph{$L$-Lipschitz} iff $\pnorm{f}{\mathrm{Lip}}\leq L$. For a proper, closed convex function $f$ defined on $\R^n$, its \emph{Moreau envelope} $\mathsf{e}_f(\cdot;\tau)$ and \emph{proximal operator} $\prox_f(\cdot;\tau)$ for any $\tau>0$ are defined by $
\mathsf{e}_f(x;\tau)\equiv \min_{z\in \R^n}\big\{\frac{1}{2\tau}\pnorm{x-z}{}^2+f(z)\big\}$ and $\prox_f(x;\tau)\equiv \argmin_{z \in \R^n} \big\{\frac{1}{2\tau}\pnorm{x-z}{}^2+f(z)\big\}$. 

Throughout this paper, for an invertible covariance matrix $\Sigma \in \R^{n\times n}$, we write $
\mathcal{H}_\Sigma\equiv \tr(\Sigma^{-1})/n$
as the harmonic mean of the eigenvalues of $\Sigma$.

\section{Distribution of Ridge(less) estimators}\label{section:main_results}

\subsection{Some definitions}

For $K>1$, let
\begin{align}\label{def:Xik}
\Xi_K\equiv [\bm{1}_{\phi^{-1}< 1+1/K}K^{-1},K].
\end{align}
This notation will be used throughout the paper for uniform-in-$\eta$ statements. In particular, in the overparametrized regime $\phi^{-1}\geq 1+1/K$, we have $\Xi_K=[0,K]$. 

Next, for $\gamma,\tau\geq 0$, recall $\hat{\mu}_{(\Sigma,\mu_0)}^{\seq}(\gamma;\tau)$ in (\ref{def:ridge_seq}), and we define its associated estimation error $\err_{(\Sigma,\mu_0)}(\gamma;\tau)$ and the degrees-of-freedom $\dof_{(\Sigma,\mu_0)}(\gamma;\tau)$ as
\begin{align}
\begin{cases}
\err_{(\Sigma,\mu_0)}(\gamma;\tau)\equiv \pnorm{\Sigma^{1/2}\big(\hat{\mu}_{(\Sigma,\mu_0)}^{\seq}(\gamma;\tau)-\mu_0\big)}{}^2,\nonumber\\
\dof_{(\Sigma,\mu_0)}(\gamma;\tau)\equiv \bigiprod{\frac{\gamma g}{\sqrt{n}}}{ \Sigma^{1/2}\big(\hat{\mu}_{(\Sigma,\mu_0)}^{\seq}(\gamma;\tau)-\mu_0\big) }.
\end{cases}
\end{align}
The $\dof_{(\Sigma,\mu_0)}(\gamma;\tau)$ defined above is naturally related to the usual notion of degrees-of-freedom (cf. \cite{stein1981estimation,efron2004estimation}) for $\hat{\mu}_{(\Sigma,\mu_0)}^{\seq}(\gamma;\tau)$, in the sense that $
\mathrm{df}\big(\hat{\mu}_{(\Sigma,\mu_0)}^{\seq}(\gamma;\tau)\big) \equiv \sum_{j=1}^n \frac{1}{\gamma^2/n}\cov\big( (\Sigma^{1/2}\hat{\mu}_{(\Sigma,\mu_0)}^{\seq})_j, y_{(\Sigma,\mu_0),j}^{\seq} \big) = \frac{n}{\gamma^2}\E \dof_{(\Sigma,\mu_0)}(\gamma;\tau)$.

\subsection{Working assumptions}

\begin{assumption}\label{assump:design}
$X=Z\Sigma^{1/2}$, where (i) $Z\in \R^{m\times n}$ has independent, mean-zero, unit variance, uniformly sub-gaussian entries, and (ii) $\Sigma \in \R^{n\times n}$ is an invertible covariance matrix with eigenvalues $\lambda_1\geq \cdots\geq \lambda_n>0$.
\end{assumption}
Here `uniform sub-gaussianity'  means $\sup_{i \in [m],j\in [n]} \pnorm{Z_{ij}}{\psi_2}\leq C$ for some universal $C>0$, where $\psi_2$ is the Orlicz 2-norm (cf. \cite[Section 2.2, pp. 95]{van1996weak}).

We shall often write the Gaussian design as $Z=G$, where $G \in \R^{m\times n}$ consists of i.i.d. $\mathcal{N}(0,1)$ entries.

\begin{assumption}\label{assump:noise}
$\xi=\sigma_\xi\cdot \xi_0$ for some $\xi_0$ with i.i.d. mean zero, unit variance and uniform sub-gaussian entries.
\end{assumption}

\begin{remark}
The requirement on the noise level $\sigma_\xi^2$ will be specified in concrete 
results below. We assume sub-gaussian noise for simplicity, but our proofs use it only through the high probability events in Section \ref{sec:hpb-events}. With easy modifications, all results extend to more general noise distributions, including certain heavy-tailed or weakly dependent cases for which these events still hold.
\end{remark}

\subsection{The fixed point equation}

Fix $\eta\geq 0$. Consider the following fixed point equation in $(\gamma,\tau)$:
\begin{align}\label{eqn:fpe}
\begin{cases}
\phi \gamma^2  = \sigma_\xi^2+ \E\err_{(\Sigma,\mu_0)}(\gamma;\tau),\\
\phi-\frac{\eta}{\tau}  = \frac{1}{n}\tr\big((\Sigma+\tau I_n)^{-1}\Sigma\big) =\frac{1}{\gamma^2}\E\dof_{(\Sigma,\mu_0)}(\gamma;\tau).
\end{cases}
\end{align}

Fixed point equations of the type described above have appeared in the general mean-field theory for high dimensional regularized least squares estimators (LSEs), see, e.g. \cite{bayati2012lasso,bu2021algorithmic,li2021minimum,han2023noisy,celentano2023lasso} for a sample of this type of equations in the i.i.d. sampling setting, and \cite{bao2025leave}  in the i.n.i.d. sampling setting. A common theme of these works characterizes the behavior of the regularized LSE in the linear model---at various levels of generality---via a regularized LSE in the equivalent sequence model, whose `effective noise' and `effective regularization' are determined by the solution pair to the fixed point equation.
	
In the context of Ridge regression, the form of the fixed point equation (\ref{eqn:fpe}) appeared in \cite{bartlett2021deep,cheng2024dimension} for the purpose of characterizing $\ell_2$ risks for the Ridge(less) estimator $\hat{\mu}_\eta$. It is now well understood that for the purpose of distributional characterizations of $\hat{\mu}_\eta$, further stability properties for the solution pair to the fixed point equation will be needed \cite{miolane2021distribution,celentano2023lasso,han2023noisy}. We establish these properties for the solution to (\ref{eqn:fpe}) in the following proposition. Recall $\Xi_K$ from (\ref{def:Xik}).

\begin{proposition}\label{prop:fpe_est_simplify}
	Recall $\mathcal{H}_\Sigma=\tr(\Sigma^{-1})/n$. The following hold.
	\begin{enumerate}
		\item The fixed point equation (\ref{eqn:fpe}) admits a unique solution $(\gamma_{\eta,\ast},\tau_{\eta,\ast}) \in (0,\infty)^2$, for all $(m,n)\in \N^2$ when $\eta>0$ and $m<n$ when $\eta=0$. 
		\item Suppose $1/K\leq \phi^{-1}\leq K$ and $\pnorm{\Sigma}{\op} \vee \mathcal{H}_\Sigma\leq K$ for some $K>1$. Then there exists some $C=C(K)>1$ such that uniformly in $\eta \in \Xi_K$, 
		\begin{align*}
		1/C\leq \tau_{\eta,\ast} \leq C,\quad 1/C\leq (-1)^{q+1}\partial_\eta^q \tau_{\eta,\ast}\leq C, \quad q\in \{1,2\}.
		\end{align*}
		If furthermore $1/K\leq \sigma_\xi^2\leq K$ and $\pnorm{\mu_0}{}\leq K$, then uniformly in $\eta \in \Xi_K$,
		\begin{align*}
		1/C\leq \gamma_{\eta,\ast}\leq C,\quad \abs{ \partial_\eta \gamma_{\eta,\ast} }\leq C.
		\end{align*}
		\item Suppose $1/K\leq \phi^{-1}\leq K$ and $\pnorm{\Sigma}{\op} \vee \mathcal{H}_\Sigma\leq K$ for some $K>1$. Then there exists some $C=C(K)>1$ such that the following hold. For any $\epsilon \in (0,1/2]$, we may find some $\mathcal{U}_\epsilon\subset B_n(1)$ with $\mathrm{vol}(\mathcal{U}_\epsilon)/\mathrm{vol}(B_n(1))\geq 1- C \epsilon^{-1}e^{-n\epsilon^2/C}$, 
		\begin{align*}
		\sup_{\mu_0 \in \mathcal{U}_\epsilon} \sup_{\eta \in \Xi_K} \bigabs{\gamma_{\eta,\ast}^2- \tilde{\gamma}_{\eta,\ast}^2(\pnorm{\mu_0}{}) } \leq \epsilon,
		\end{align*}
		where $
		\tilde{\gamma}_{\eta,\ast}^2(\pnorm{\mu_0}{})\equiv \sigma_\xi^2 \partial_\eta \tau_{\eta,\ast}+\pnorm{\mu_0}{}^2(\tau_{\eta,\ast}-\eta\partial_\eta \tau_{\eta,\ast})>0$. When $\Sigma=I_n$, we may take $\mathcal{U}_\epsilon=B_n(1)$ and the above inequality holds with $\epsilon=0$.
	\end{enumerate}
	
\end{proposition}

The above proposition combines parts of Propositions \ref{prop:fpe_est} and \ref{prop:tau_gamma_m}. 

As an important qualitative consequence of (2), under the condition $\pnorm{\Sigma}{\op}\vee \mathcal{H}_\Sigma\leq K$, the effective regularization $\eta\mapsto \tau_{\eta,\ast}$ is a strictly increasing and concave function of $\eta$. Moreover, in the overparametrized regime $\phi^{-1}>1$, the quantity $\tau_{0,\ast}$---also known as `\emph{implicit regularization}' in the literature \cite{bartlett2020benign,bartlett2021deep,hastie2022surprises,tsigler2023benign,cheng2024dimension}---is strictly bounded away from zero. 

The claim in (3)  offers a useful approximate representation of the effective noise $\gamma_{\eta,\ast}^2$ in terms of the original noise $\sigma_\xi^2$, the effective regularization $\tau_{\eta,\ast}$ and the signal energy $\pnorm{\mu_0}{}$ without explicitly dependence of $\Sigma$. This representation will prove useful in understanding some qualitative aspects of the risk curves in Section \ref{section:err_phase_trans} ahead.

\subsection{Distribution of Ridge(less) estimators}

In addition to $\hat{\mu}_{\eta}$, we will also consider the distribution of the (scaled) residual $\hat{r}_{\eta}$, defined by 
\begin{align}
\hat{r}_{\eta}\equiv \frac{1}{\sqrt{n}}\big(Y-X\hat{\mu}_{\eta}\big).
\end{align}
We define the `population' version of $\hat{r}_\eta$ as
\begin{align}
r_{\eta,\ast} \equiv \frac{\eta}{\phi\tau_{\eta,\ast}}\bigg(-\sqrt{\phi\gamma_{\eta,\ast}^2 - \sigma_\xi^2 }\cdot \frac{h}{\sqrt{n}} + \frac{\xi}{\sqrt{n}}  \bigg).
\end{align}
Here $h\sim \mathcal{N}(0,I_m)$ is independent of $\xi$. 

We are now in a position to state our main results on the distributional results for the Ridge(less) estimator $\hat{\mu}_{\eta}$ and the residual $\hat{r}_{\eta}$.

First we work under the Gaussian design $Z=G$, and we write $\hat{\mu}_\eta=\hat{\mu}_{\eta;G},\hat{r}_\eta=\hat{r}_{\eta;G}$. Recall $\mathcal{H}_\Sigma=\tr(\Sigma^{-1})/n$ and $\Xi_K$ from (\ref{def:Xik}). 

\begin{theorem}\label{thm:min_norm_dist}
Suppose Assumption \ref{assump:design} holds with $Z=G$ and the following hold for some $K>0$.
\begin{itemize}
	\item $1/K\leq \phi^{-1}\leq K$, $\pnorm{\Sigma}{\op}\vee \mathcal{H}_\Sigma \leq K$.
	\item Assumption \ref{assump:noise} holds with $\sigma_\xi^2 \in [1/K,K]$. 
\end{itemize}
Then there exists some constant $C=C(K)>0$ such that the following hold.
\begin{enumerate}
	\item For any $1$-Lipschitz function $\mathsf{g}:\R^n\to \R$ and $\epsilon \in (0,1/2]$,
	\begin{align*}
	\sup_{\mu_0 \in B_n(1)}\Prob\Big(\sup_{\eta \in \Xi_K}\bigabs{\mathsf{g}(\hat{\mu}_{\eta;G})-\E \mathsf{g}\big(\hat{\mu}_{(\Sigma,\mu_0)}^{\seq}(\gamma_{\eta,\ast};\tau_{\eta,\ast})\big) }\geq \epsilon\Big)\leq C n e^{-n\epsilon^4/C}. 
	\end{align*}
    Here we recall that $\hat{\mu}_{\eta;G}$ is defined in \eqref{def:ridge_est} with $X = G\Sigma^{1/2}$, and that $\hat{\mu}_{(\Sigma,\mu_0)}^{\seq}$ with effective noise and regularization pair $(\gamma_{\eta,\ast},\tau_{\eta,\ast})$ is defined in \eqref{def:ridge_seq}-\eqref{def:gamma-tau}.
	\item For any $\epsilon \in (0,1/2]$, $\xi \in \R^m$ satisfying  $\abs{\,\pnorm{\xi}{}^2/m-\sigma_\xi^2}\leq \epsilon^2/C$, and $1$-Lipschitz function $\mathsf{h}:\R^m\to \R$ (which may depend on $\xi$),
	\begin{align*}
	\sup_{\mu_0 \in B_n(1)}\Prob^\xi \Big(\sup_{\eta \in [1/K,K]}\abs{\mathsf{h}(\hat{r}_{\eta;G}) - \E^\xi \mathsf{h}(r_{\eta,\ast}) }\geq \epsilon \Big)\leq C n e^{-n\epsilon^4/C}.
	\end{align*}
\end{enumerate}
\end{theorem}

The choice $\mu_0 \in B_n(1)$  is made merely for simplicity of presentation; it can be replaced by $\mu_0\in B_n(R)$ with another constant $C$ that depends further on $R$. The assumption $\mathcal{H}_\Sigma\lesssim 1$ is quite common in the literature of Ridge(less) regression; see, e.g., \cite[Assumption 4.12]{bartlett2021deep} or a slight variant in \cite[Assumption 1]{montanari2023generalization}. The major assumption in the above theorem is the Gaussianity on the design $X$. This may be lifted at the cost of a set of slightly stronger conditions.

\begin{theorem}\label{thm:universality_min_norm}
	Suppose Assumption \ref{assump:design} holds and the following hold for some $K>0$. 
	\begin{itemize}
		\item $1/K\leq \phi^{-1}\leq K$, $\pnorm{\Sigma}{\op}\vee \pnorm{\Sigma^{-1}}{\op} \leq K$.
		\item Assumption \ref{assump:noise} holds with $\sigma_\xi^2 \in [1/K,K]$.
	\end{itemize}
	Fix $\vartheta \in (0,1/18)$. There exist some $C=C(K,\vartheta)>0$ and two measurable sets $\mathcal{U}_{\vartheta}\subset B_n(1),\mathcal{E}_{\vartheta}\subset \R^m$ with $\min\{\mathrm{vol}(\mathcal{U}_\vartheta)/\mathrm{vol}(B_n(1)),\Prob(\xi \in \mathcal{E}_\vartheta)\}\geq 1-Ce^{-n^{2\vartheta}/C}$, such that the following hold.
	\begin{enumerate}
		\item For any $1$-Lipschitz function $\mathsf{g}:\R^n\to \R$, and $\epsilon \in (0,1/2]$,
		\begin{align*}
		\sup_{\mu_0 \in \mathcal{U}_\vartheta}\Prob\Big(\sup_{\eta \in \Xi_K }\bigabs{\mathsf{g}(\hat{\mu}_{\eta})-\E \mathsf{g}\big(\hat{\mu}_{(\Sigma,\mu_0)}^{\seq}(\gamma_{\eta,\ast};\tau_{\eta,\ast})\big) }\geq \epsilon \Big)\leq C \epsilon^{-13}n^{-1/6+3\vartheta}. 
		\end{align*}
        Here we recall that $\hat{\mu}_{\eta}$ is defined in \eqref{def:ridge_est}, and that $\hat{\mu}_{(\Sigma,\mu_0)}^{\seq}$ with effective noise and regularization pair $(\gamma_{\eta,\ast},\tau_{\eta,\ast})$ is defined in \eqref{def:ridge_seq}-\eqref{def:gamma-tau}.
		\item For any $\epsilon \in (0,1/2]$, $\xi \in \mathcal{E}_\vartheta$ and $1$-Lipschitz function $\mathsf{h}:\R^m\to \R$ (which may depend on $\xi$),
		\begin{align*}
		\sup_{\mu_0 \in \mathcal{U}_\vartheta}\Prob^\xi \Big(\sup_{\eta \in [1/K,K]}\abs{\mathsf{h}(\hat{r}_{\eta}) - \E^\xi \mathsf{h}(r_{\eta,\ast}) }\geq \epsilon \Big)\leq C \epsilon^{-7} n^{-1/6+3\vartheta}.
		\end{align*}
	\end{enumerate}
    Concrete forms of $\mathcal{U}_\vartheta,\mathcal{E}_\vartheta$ are specified in Proposition \ref{prop:delocal_u_v}.
\end{theorem}

\begin{remark}
Compared with the Gaussian case (Theorem \ref{thm:min_norm_dist}), which admits exponential tails via Gaussian concentration and a direct CGMT argument, the sub-Gaussian universality in Theorem \ref{thm:universality_min_norm} yields only polynomial rates. This stems from the quantitative
comparison inequalities \cite{han2023universality} employed in the universality step. Extending these
bounds to exponential decay would likely require methods beyond the comparison framework, which is beyond
the scope of the present paper.
\end{remark}

Theorems \ref{thm:min_norm_dist} and \ref{thm:universality_min_norm} are proved in Section \ref{section:proof_gaussian_design} and Section \ref{section:proof_general_design}, respectively. Due to the high technicalities in the proof, a sketch is outlined in Section \ref{section:proof_outline}. These distributional results are the main input for all the applications developed in the subsequent sections. In particular, the flexibility in the choice of the test functions $\mathsf{g}$ and $\mathsf{h}$ allows us to obtain a variety of functionals of interest. By choosing  the test functions appropriately, we derive (i) in Section~\ref{sec:risk_asymptotic} the $\ell_q$-risk asymptotics for general $q\in [1,\infty)$, extending the classical $\ell_2$-risk formulas that are typically accessible via random matrix theory, and (ii) in Section \ref{section:application} the optimality of cross-validation tuning rules.

We mention two particular important features on the theorems above:
\begin{enumerate}
	\item The distributional characterizations for $\hat{\mu}_{\eta}$ in both theorems above are uniformly valid down to the interpolation regime $\eta=0$ for $\phi^{-1}>1$. This uniform control will play a crucial role in our non-asymptotic analysis of cross-validation methods to be studied in Section \ref{section:application} ahead.
	\item The distribution of the residual $\hat{r}_\eta$ in (2) is formulated \emph{conditional on the noise $\xi$}. A fundamental reason for adopting this formulation is that the distribution of $\hat{r}_\eta$ is \emph{not} universal with respect to the law of $\xi$. In other words, one cannot simply assume Gaussianity of $\xi$ in Theorem \ref{thm:min_norm_dist} in hope of proving universality of $\hat{r}_\eta$ in Theorem \ref{thm:universality_min_norm}.
\end{enumerate}

In the context of distributional characterizations for regularized regression estimators in the proportional regime, results in similar vein to Theorem \ref{thm:min_norm_dist} have been obtained in the closely related Lasso setting for isotropic $\Sigma=I_n$ in \cite{miolane2021distribution}, and for general $\Sigma$ in \cite{celentano2023lasso}, both under Gaussian designs and with strictly non-vanishing regularization. A substantially simpler, isotropic ($\Sigma=I_n$) version of Theorem \ref{thm:universality_min_norm} is obtained in \cite{han2023universality} that holds pointwise in non-vanishing regularization level $\eta>0$. As will be clear from the proof sketch in Section \ref{section:proof_outline}, in addition to the complications due to the implicit nature of the solution to the fixed point equation (\ref{eqn:fpe})  for general $\Sigma$, the major difficulty in proving Theorems \ref{thm:min_norm_dist} and \ref{thm:universality_min_norm} rests in handling the singularity of the optimization problem (\ref{def:ridge_est}) as $\eta \downarrow 0$.

\section{General $\ell_q$-type risk formulae}\label{sec:risk_asymptotic}
As a demonstration of the analytic power of Theorems \ref{thm:min_norm_dist} and \ref{thm:universality_min_norm}, this section will be devoted to a detailed study for the $\ell_q$-type risks for Ridgeless interpolators. We then conduct a more in-depth study of $\ell_2$ risks, where techniques from RMT lead to a detailed characterization of the optimal regularization strategy for risk minimization.

\subsection{Weighted $\ell_q$ risks and delocalization}\label{subsection:lq_risk}
We compute below the weighted $\ell_q$ risk $\pnorm{\mathsf{A}(\hat{\mu}_\eta-\mu_0)}{q}$ for a well-behaved matrix $\mathsf{A}\in \R^{n\times n}$ and $q \in [1,\infty)$. Recall $\Xi_K$ from (\ref{def:Xik}).

\begin{theorem}\label{thm:lq_risk}
	Suppose the same conditions in Theorem \ref{thm:universality_min_norm} hold for some $K>0$. Fix $q \in [1,\infty)$ and a p.s.d. matrix $\mathsf{A} \in \R^{n\times n}$ with $\pnorm{\mathsf{A}}{\op}\vee \pnorm{\mathsf{A}^{-1}}{\op}\leq K$. Then there exist constants $C>1, \vartheta \in (0,1/50)$ depending on $K,q$, and a measurable set $\mathcal{U}_{\vartheta}\subset B_n(1)$ with $\mathrm{vol}(\mathcal{U}_{\vartheta})/\mathrm{vol}(B_n(1))\geq 1-C e^{-n^{\vartheta}/C}$, such that 
	\begin{align*}
	\sup_{\mu_0 \in \mathcal{U}_{\vartheta} }\Prob\bigg(\sup_{\eta \in \Xi_K} \biggabs{  \frac{\pnorm{ \mathsf{A}(\hat{\mu}_\eta-\mu_0)}{q}}{ \bar{R}_{(\Sigma,\mu_0);q}^{\mathsf{A}}(\eta) }- 1 }\geq n^{-\vartheta}\bigg)\leq C n^{-1/7}.
	\end{align*}
	Here $\bar{R}_{(\Sigma,\mu_0);q}^{\mathsf{A}}(\eta) \in \big\{\E \pnorm{ \mathsf{A}\big(\hat{\mu}_{(\Sigma,\mu_0)}^{\seq}(\gamma_{\eta,\ast};\tau_{\eta,\ast})-\mu_0\big) }{q}, n^{-1/2}\pnorm{\mathrm{diag}\big(\Gamma_{\eta;(\Sigma,\pnorm{\mu_0}{})}^{\mathsf{A}}\big) }{q/2}^{1/2}M_q\big\}$, where $M_q\equiv \E^{1/q}\abs{\mathcal{N}(0,1)}^q=2^{1/2}\big\{\Gamma\big((q+1)/2\big)/\sqrt{\pi}\big\}^{1/q}$, 
	\begin{align}\label{def:Gamma_matrix}
	\Gamma_{\eta;(\Sigma,\pnorm{\mu_0}{})}^{\mathsf{A}}\equiv \mathsf{A}(\Sigma+\tau_{\eta,\ast} I_n)^{-1}\Big(\tilde{\gamma}_{\eta,\ast}^2(\pnorm{\mu_0}{})\Sigma+\tau_{\eta,\ast} ^2 \pnorm{\mu_0}{}^2 I_n\Big)(\Sigma+\tau_{\eta,\ast} I_n)^{-1}\mathsf{A},
	\end{align}
	and $\tilde{\gamma}_{\eta,\ast}^2(\pnorm{\mu_0}{})$ is defined in Proposition \ref{prop:fpe_est_simplify}-(3).
\end{theorem}

The proof of the above theorem can be found in Section \ref{subsection:proof_lq_risk}. To the best of our knowledge, general weighted $\ell_q$ risks for the Ridge(less) estimator $\hat{\mu}_\eta$ have not be available in the literature except for the special case $q=2$, for which $\pnorm{ \mathsf{A}(\hat{\mu}_\eta-\mu_0)}{2}$ admits a closed-form expression in terms of the spectral statistics of $X$ that facilitates direct applications of RMT techniques, cf. \cite{tulino2004random,elkaroui2013asymptotic,dicker2016ridge,dobriban2018high,elkaroui2018impact,advani2020high,wu2020optimal,bartlett2021deep,richards2021asymptotics,hastie2022surprises,cheng2024dimension}.

For $\mathsf{A}\neq \Sigma$, Theorem \ref{thm:lq_risk} above characterizes the out-of-distribution $\ell_q$ risk for the Ridge(less) estimators. This setting is naturally related to the covariate shift setting, where $\ell_2$-type risks are studied in \cite{patil2024optimal,tang2024benign} using random matrix methods in slightly different specific settings.

Let us remark that obtaining $\ell_q$ risks for $q \in [1,2]$ via our Theorems \ref{thm:min_norm_dist} and \ref{thm:universality_min_norm} is relatively easy, as $x\mapsto \pnorm{x}{q}/n^{1/q-1/2}$ is $1$-Lipschitz with respect to $\pnorm{\cdot}{}$ for $q\in [1,2]$. The stronger norm case $q\in (2,\infty)$ is significantly harder. In fact, we need additionally the following delocalization result for $\hat{\mu}_\eta$.

\begin{proposition}\label{prop:delocal_u_v_simplify}
	Suppose the same conditions as in Theorem \ref{thm:lq_risk} hold for some $K>0$. Fix $\vartheta \in (0,1/2]$. Then there exist some constant $C=C(K,\vartheta)>0$ and a measurable set $\mathcal{U}_\vartheta\subset B_n(1)$ with $\mathrm{vol}(\mathcal{U}_\vartheta)/\mathrm{vol}(B_n(1))\geq 1-Ce^{-n^{2\vartheta}/C}$, such that
	\begin{align*}
	\sup_{\mu_0 \in \mathcal{U}_\vartheta} \Prob\Big(\sup_{\eta \in \Xi_K}\pnorm{\mathsf{A}(\hat{\mu}_\eta-\mu_0)}{\infty}\geq C n^{-1/2+\vartheta}\Big)\leq C n^{-100}. 
	\end{align*}
\end{proposition}
The above proposition is a simplified version of Proposition \ref{prop:delocal_u_v}, proved via the anisotropic local laws developed in \cite{knowles2017anisotropic}. In essence, delocalization allows us to apply Theorems \ref{thm:min_norm_dist} and \ref{thm:universality_min_norm} with a truncated version of the $\ell_q$ norm ($q>2$) with a well-controlled Lipschitz constant with respect to $\ell_2$. Moreover, delocalization of $\hat{\mu}_\eta$ also serves as a key technical ingredient in proving the universality Theorem \ref{thm:universality_min_norm}; the readers are referred to Section \ref{section:proof_outline} for a detailed account on the technical connection between delocalization and universality.

\vspace{0.5em}

\noindent \textbf{Convention on probability estimates}:

\begin{enumerate}
	\item When $Z=G$, $n^{-1/7}$ in Theorem \ref{thm:lq_risk} can be replaced by $n^{-D}$ for any $D>0$.
	\item $n^{-100}$ in Proposition \ref{prop:delocal_u_v_simplify} can be replaced by $n^{-D}$ for any $D>0$.
\end{enumerate}
The cost will be a possibly enlarged constant $C>0$ that depends further on $D$. This convention applies to other statements in the following sections in which the probability estimates $n^{-1/7}, n^{-100}$ appear. 

\subsection{$\ell_2$ risk formulae and optimal regularization}\label{section:err_phase_trans}

In this subsection, we will study in some detail the behavior of various $\ell_2$ risks associated with $\hat{\mu}_\eta$. As will be clear below, a major analytic advantage of studying  $\ell_2$ risks is their close connection to techniques from RMT.

Recall the notation $R^{\#}_{(\Sigma,\mu_0)}(\eta)$ defined in Section \ref{subsection:intro_phase_transition}. Let their `theoretical' versions be defined as follows:
\begin{itemize}
	\item 
	 $\bar{R}^{\pred}_{(\Sigma,\mu_0)}(\eta)\equiv \tau_{\eta,\ast}^2 \pnorm{(\Sigma+\tau_{\eta,\ast} I_n)^{-1}\Sigma^{1/2}\mu_0}{}^2+ \frac{\gamma_{\eta,\ast}^2}{n} \tr\big(\Sigma^2 (\Sigma+\tau_{\eta,\ast} I_n)^{-2}\big)$.
	\item $\bar{R}^{\est}_{(\Sigma,\mu_0)}(\eta)\equiv \tau_{\eta,\ast}^2 \pnorm{(\Sigma+\tau_{\eta,\ast} I_n)^{-1}\mu_0}{}^2+ \frac{\gamma_{\eta,\ast}^2}{n} \tr\big(\Sigma (\Sigma+\tau_{\eta,\ast} I_n)^{-2}\big)$.
	\item  $\bar{R}^{\ins}_{(\Sigma,\mu_0)}(\eta)\equiv \big(\frac{\eta \gamma_{\eta,\ast}}{\tau_{\eta,\ast}}\big)^2+\phi \sigma_\xi^2\cdot \big(1- \frac{2\eta}{\phi \tau_{\eta,\ast}}\big)$.
\end{itemize}
We also define the residual and its theoretical version as
\begin{itemize}
	\item $R^{\res}_{(\Sigma,\mu_0)}(\eta)\equiv n^{-1}\pnorm{Y-X\hat{\mu}_{\eta}}{}^2$, 	$\bar{R}^{\res}_{(\Sigma,\mu_0)}(\eta)\equiv \big(\frac{\eta \gamma_{\eta,\ast}}{\tau_{\eta,\ast}}\big)^2$.
\end{itemize}
From Theorems \ref{thm:min_norm_dist} and \ref{thm:universality_min_norm}, it is natural to expect that for $\# \in \{\pred,\est,\ins,\res\}$,
\begin{align}\label{eqn:errors_explicit}
\sup_{\eta \in \Xi^{\#}}\abs{R^{\#}_{(\Sigma,\mu_0)}(\eta)-\bar{R}^{\#}_{(\Sigma,\mu_0)}(\eta)}\approx 0\hbox{ with high probability}.
\end{align}
A rigorous statement of (\ref{eqn:errors_explicit}) is deferred to Theorem \ref{thm:errors_explicit}; its proof and the proofs for all other results in this section can be found in Section \ref{section:proof_error}.

Using the so-called Stieltjes transformation $\mathfrak{m}(\cdot)$ in the RMT literature (defined formally via (\ref{eq:StieljesTransform}) in Section \ref{subsection:fpe_rmt} ahead), the following theorem provides an efficient RMT representation of $\bar{R}^{\#}_{(\Sigma,\mu_0)}(\eta)$ that holds for `most' $\mu_0$'s. Recall $\Xi_K$ from (\ref{def:Xik}).

\begin{theorem}\label{thm:error_rmt}
Suppose $1/K\leq \phi^{-1}\leq K$, $\sigma_\xi^2 \in [0,K]$ and $\pnorm{\Sigma}{\op}\vee \mathcal{H}_\Sigma\leq K$ for some $K>0$. There exists some constant $C=C(K)>0$ such that for any $\epsilon \in (0,1/2]$, we may find a measurable set $\mathcal{U}_\epsilon\subset B_n(1)$ with $\mathrm{vol}(\mathcal{U}_\epsilon)/\mathrm{vol}(B_n(1))\geq 1- C \epsilon^{-1}e^{-n\epsilon^2/C}$, 
\begin{align}\label{ineq:error_rmt}
\sup_{\mu_0 \in \mathcal{U}_\epsilon} \sup_{\eta \in \Xi_K}\abs{\bar{R}^{\#}_{(\Sigma,\mu_0)}(\eta)- \mathscr{R}^{\#}_{(\Sigma,\mu_0)}(\eta)} \leq \epsilon,\quad \# \in \{\pred,\est,\ins\}.
\end{align}
Here with $\SNR_{\mu_0}=\pnorm{\mu_0}{}^2/\sigma_\xi^2$, $ \mathfrak{m}_\eta\equiv  \mathfrak{m}(-\eta/\phi)$ and $ \mathfrak{m}_\eta'\equiv  \mathfrak{m}'(-\eta/\phi)$, 
\begin{itemize}
	\item $\mathscr{R}^{\pred}_{(\Sigma,\mu_0)}(\eta)\equiv  \sigma_\xi^2\cdot\Big\{\frac{1}{\mathfrak{m}_\eta^2}\Big(\phi\cdot \SNR_{\mu_0} \mathfrak{m}_\eta-\big(\eta\cdot \SNR_{\mu_0}-1\big) \mathfrak{m}_\eta'\Big)-1\Big\}$,
	\item $\mathscr{R}^{\est}_{(\Sigma,\mu_0)}(\eta)\equiv \sigma_\xi^2 \cdot \Big\{\SNR_{\mu_0}(1-\phi)+  \mathfrak{m}_\eta + \frac{\eta}{\phi}\big(\eta\cdot \SNR_{\mu_0}-1\big) \mathfrak{m}_\eta'\Big\}$,
	\item $\mathscr{R}^{\ins}_{(\Sigma,\mu_0)}(\eta)\equiv\sigma_\xi^2 \cdot \frac{\eta^2 }{\phi } \Big(\phi\cdot \SNR_{\mu_0} \mathfrak{m}_\eta-\big(\eta\cdot \SNR_{\mu_0}-1\big) \mathfrak{m}_\eta'\Big)+  \sigma_\xi^2\cdot (\phi- 2\eta\mathfrak{m}_\eta)$,
	\item $\mathscr{R}^{\res}_{(\Sigma,\mu_0)}(\eta)\equiv  \sigma_\xi^2 \cdot \frac{\eta^2 }{\phi } \Big(\phi\cdot \SNR_{\mu_0} \mathfrak{m}_\eta-\big(\eta\cdot \SNR_{\mu_0}-1\big) \mathfrak{m}_\eta'\Big)$. 
\end{itemize}
When $\Sigma=I_n$, we may take $\mathcal{U}_\epsilon =  B_n(1)$ and (\ref{ineq:error_rmt}) holds with $\epsilon=0$.
\end{theorem}

The RMT representation above yields the following crucial insight into the extremal behavior of the risk maps $\eta\mapsto \bar{R}^{\#}_{(\Sigma,\mu_0)}(\eta)$.

\begin{proposition}\label{prop:derivative_R_simplify}
		Suppose $1/K\leq \phi^{-1}\leq K$ and $\pnorm{\Sigma}{\op} \vee \mathcal{H}_\Sigma\leq K$ for some $K>0$. Then there exists some $C=C(K)>0$  such that for all $\# \in \{\pred,\est,\ins\} $, the derivative formulae
		\begin{align*}
		\partial_\eta \mathscr{R}^{\#}_{(\Sigma,\mu_0)}(\eta) = \sigma_\xi^2 \cdot \mathfrak{M}^{\#}(\eta)\cdot  \big(\eta\cdot \SNR_{\mu_0}-1\big),\quad \eta\geq 0
		\end{align*}
		hold for some measurable functions $\big\{\mathfrak{M}^{\#}:\R_{\geq 0}\to [1/C,C] \big\}$.
		
		Consequently, for all $\# \in \{\pred,\est,\ins\} $, $\mathscr{R}^{\#}_{(\Sigma,\mu_0)}(\cdot)$ attains its global minimum at the same $\eta_\ast\equiv \SNR_{\mu_0}^{-1} \in \Xi_K$, and $
		1/C\leq {\abs{\mathscr{R}^{\#}_{(\Sigma,\mu_0)}(\eta)- \mathscr{R}^{\#}_{(\Sigma,\mu_0)}(\eta_\ast)}}\big/\{ \pnorm{\mu_0}{}^2(\eta-\eta_\ast)^2\}\leq C$.
\end{proposition}

A more general version of the above proposition with precise formulae for $\mathfrak{M}^{\#}$ can be found in Proposition \ref{prop:derivative_R}. As the maps $\eta \mapsto \mathscr{R}^{\#}_{(\Sigma,\mu_0)}(\eta)$ are almost quadratic with the same global minimizer $\eta_\ast=\SNR_{\mu_0}^{-1}$ for all $\# \in \{\pred,\est,\ins\}$, in view of (\ref{eqn:errors_explicit}) and Theorem \ref{thm:error_rmt}, it is natural to expect that for `most' signal $\mu_0$'s and all $\# \in \{\pred,\est,\ins\}$,
\begin{align}\label{eqn:opt_reg_l2}
R^{\#}_{(\Sigma,\mu_0)}(\eta_\ast)\approx\min_{\eta \in \Xi_L} R^{\#}_{(\Sigma,\mu_0)}(\eta) \hbox{ with high probability}.
\end{align}
A rigorous formulation of (\ref{eqn:opt_reg_l2}) is given in Theorem \ref{thm:small_intep}, which, along with its proof, is provided in Section \ref{sec:thm_statement_opt_reg_l2}.

\section{Cross-validation: optimality beyond prediction }\label{section:application}

This section is devoted to the validation of the broad optimality of two widely used cross-validation schemes beyond the prediction risk. Some consequences to statistical inference via debiased Ridge(less) estimators will also be discussed.

\subsection{Estimation of effective noise and regularization}

We shall first take a slight detour, by considering estimation of the effective regularization $\tau_{\eta,\ast}$ and the effective noise $\gamma_{\eta,\ast}$. We propose the following estimators:
\begin{align}\label{def:estimator_tau_gamma}
\begin{cases}
\hat{\tau}_\eta\equiv \Big\{\frac{1}{m}\tr\big(\frac{1}{m}XX^{\top}+ \frac{\eta}{\phi} I_m\big)^{-1}\Big\}^{-1} = \big\{\tr (XX^\top+\eta\cdot nI_m)^{-1}\big\}^{-1},\\
\hat{\gamma}_\eta \equiv  \frac{\hat{\tau}_\eta}{\sqrt{n}}\Big(\eta^{-1}\pnorm{Y-X\hat{\mu}_\eta}{}\bm{1}_{\phi^{-1}<1}+ \pnorm{(XX^\top/n)^{-1}X\hat{\mu}_\eta}{}\bm{1}_{\phi^{-1}\geq 1}\Big).
\end{cases}
\end{align}
It can be easily shown that
\begin{align}\label{eqn:noi_reg_est}
    \sup_{\eta \in \Xi_K}  \abs{\hat{\tau}_\eta-\tau_{\eta,\ast}}, \; \sup_{\eta \in \Xi_K}  \abs{\hat{\gamma}_\eta-\gamma_{\eta,\ast}} \approx 0  \hbox{ with high probability}.
\end{align}
A rigorous statement of (\ref{eqn:noi_reg_est}) is deferred to Theorem \ref{thm:trace_app}; its proof and the proofs for all other results in this section can be found in Section \ref{section:proof_application}. These estimators will not only be useful in their own rights, they will also play an important rule in understanding the generalized cross-validation scheme in the next subsection.

\subsection{Validation of generalized cross-validation}

Consider choosing $\eta$ by minimizing the estimated effective noise $\hat{\gamma}_\eta$ given in (\ref{def:estimator_tau_gamma}): for any $L>0$,
\begin{align}\label{def:res_tuning}
\hat{\eta}^{\GCV}_L\in \argmin_{\eta \in \Xi_L} \hat{\gamma}_\eta.
\end{align}d
Here we recall $\Xi_K$ from (\ref{def:Xik}). The subscript on $L$ in $\hat{\eta}^{\GCV}_L$ will usually be suppressed for notational simplicity. 

The proposal (\ref{def:res_tuning}) is known in the literature as the \emph{generalized cross validation} \cite{craven1978smoothing,golub1979generalized}, and is strongly tied to the so-called shortcut formula for leave-one-out cross validation that exists uniquely for Ridge regression, cf. \cite[Eqn. (46)]{hastie2022surprises}. Here we take a different perspective on (\ref{def:res_tuning}). From our developed theory, this tuning scheme is easily believed to ``work'' since
\begin{align}\label{ineq:res_tuning_heuristic}
\hat{\gamma}_\eta^2 \stackrel{\Prob}{\approx} \gamma_{\eta,\ast}^2 = \phi^{-1}\big(\sigma_\xi^2+\bar{R}^{\pred}_{(\Sigma,\mu_0)}(\eta)\big)\stackrel{\Prob}{\approx} \phi^{-1}\big(\sigma_\xi^2+R^{\pred}_{(\Sigma,\mu_0)}(\eta)\big).
\end{align}
So minimization of $\eta \mapsto \hat{\gamma}_\eta$ is approximately the same as that of $\eta \mapsto R^{\pred}_{(\Sigma,\mu_0)}(\eta)$, and therefore simultaneously of $\eta \mapsto R^{\#}_{(\Sigma,\mu_0)}(\eta)$ for $\# \in \{\est,\ins\}$ as per (\ref{eqn:opt_reg_l2}). We make precise the foregoing heuristics in the following theorem. 

\begin{theorem}\label{thm:tuning_res}
	Suppose Assumption \ref{assump:design} holds, and the following hold some $K>0$.
	\begin{itemize}
		\item $1/K\leq \phi^{-1} \leq K$, $\pnorm{\Sigma^{-1}}{\op}\vee \pnorm{\Sigma}{\op}\leq K$.
		\item Assumption \ref{assump:noise} holds with either (i)
		$\sigma_\xi^2 \in [1/K,K]$ or (ii) $\sigma_\xi^2 \in [0,K]$ with $\phi^{-1}\geq 1+1/K$. 
	\end{itemize}
	Fix $\delta \in (0,1/2]$, $L\geq K/\delta^2$ and a small enough $\vartheta \in (0,1/50)$. There exist a constant $C=C(K,L,\delta,\vartheta)>0$ and a measurable set $\mathcal{U}_{\delta,\vartheta}\subset B_n(1)\setminus B_n(\delta)$ with $\mathrm{vol}(\mathcal{U}_{\delta,\vartheta})/\mathrm{vol}(B_n(1)\setminus B_n(\delta))\geq 1- C e^{-n^{\vartheta}/C}$, such that for $\# \in \{\pred,\est,\ins\}$,
	\begin{align*}
	&\sup_{\mu_0 \in \mathcal{U}_{\delta,\vartheta}}\Prob\Big(R^{\#}_{(\Sigma,\mu_0)}(\hat{\eta}^{\GCV}_L)\geq \min_{\eta \in \Xi_{L}} R^{\#}_{(\Sigma,\mu_0)}(\eta)+n^{-\vartheta}\Big)\leq C n^{-1/7}.
	\end{align*}
\end{theorem}

\begin{remark}
Formally, the set $\mathcal{U}_{\delta,\vartheta}$ is defined as $\mathcal{U}_{\delta,\vartheta} \equiv \mathcal{U}_\vartheta \setminus B_n(\delta)$, where $\mathcal{U}_\vartheta$ is defined in Proposition \ref{prop:delocal_u_v}. The cutoff $ \pnorm{\mu_0}{} \geq \delta$ excludes vanishing signals and ensures that $\eta_\ast = \SNR_{\mu_0}^{-1}$ is uniformly bounded, so that $\eta_\ast \in \Xi_L$ whenever $L \ge K/\delta^2$.
\end{remark}

Earlier low-dimension results for generalized cross validation in Ridge regression include \cite{stone1974cross,stone1977asymptotics,craven1978smoothing,li1985from,li1986asymptotic,li1987asymptotic,dudoit2005asymptotics}. In the proportional high-dimensional regime, \cite{hastie2022surprises,patil2021uniform} validate the optimality of $\hat{\eta}^{\GCV}$ with respect to the prediction risk $R^{\pred}_{(\Sigma,\mu_0)}$ with increasing generality. In Theorem \ref{thm:tuning_res} above, we prove that the optimality of $\hat{\eta}^{\GCV}$ holds \emph{simultaneously} for all the three indicated risks. To the best of our knowledge, such optimality of $\hat{\eta}^{\GCV}$ beyond the prediction risk has not been previously observed in the literature.

\subsection{Validation of $k$-fold cross-validation}
Next we consider the widely used $k$-fold cross-validation. We need some further notation:
\begin{itemize}
	\item Let $m_\ell$ be the sample size of batch $\ell \in [k]$, so $\sum_{\ell \in [k]} m_\ell = m$. In the standard $k$-fold cross validation, we choose equal sized batch with $m_\ell = m/k$ (assumed to be integer without loss of generality). 
	\item Let $X^{(\ell)}\in \R^{m_\ell\times n}$ (resp. $Y^{(\ell)} \in \R^{m_\ell}$) be the submatrix of $X$ (resp. subvector of $Y$) that contains all rows corresponding to the training data in batch $\ell$.
	\item  In a similar fashion, let $X^{(-\ell)}\in \R^{(m-m_\ell)\times n}$ (resp. $Y^{(-\ell)} \in \R^{m-m_\ell}$) be the submatrix of $X$ (resp. subvector of $Y$) that removes all rows corresponding to $X^{(\ell)}$ (resp. $Y^{(\ell)}$). 
\end{itemize}

The $k$-fold cross-validation works as follows. For $\ell \in [k]$, let $\hat{\mu}^{(\ell)}_\eta\equiv \argmin_{\mu \in \R^n} \big\{\frac{1}{2n}\pnorm{Y^{(-\ell)}-X^{(-\ell)}\mu}{}^2+\frac{\eta}{2}\pnorm{\mu}{}^2\big\}$ be the Ridge estimator over $(X^{(-\ell)}, Y^{(-\ell)})$ with regularization $\eta\geq 0$. We then pick the tuning parameter that minimizes the averaged test errors of $\hat{\mu}^{(\ell)}_\eta$ over $(X^{(\ell)},Y^{(\ell)})$: for any $L>0$,
\begin{align}\label{def:CV_tuning}
\hat{\eta}^{\CV}_L \in \argmin_{\eta \in \Xi_L} \bigg\{\frac{1}{k}\sum_{\ell \in [k]} \frac{1}{m_\ell}\pnorm{Y^{(\ell)}-X^{(\ell)} \hat{\mu}^{(\ell)}_{\eta} }{}^2\bigg\}\equiv \argmin_{\eta \in \Xi_L} R^{\CV,k}_{(\Sigma,\mu_0)}(\eta).
\end{align}
We shall often omit the subscript $L$ in $\hat{\eta}^{\CV}_L$. 

Intuitively, due to the independence between $\hat{\mu}^{(\ell)}_{\eta}$ and $(X^{(\ell)},Y^{(\ell)})$, $R^{\CV,k}_{(\Sigma,\mu_0)}(\eta)$ can be viewed as an estimator of the generalization error $R^{\pred}_{(\Sigma,\mu_0)}(\eta)+\sigma_\xi^2$. So it is natural to expect that $\hat{\eta}^{\CV}$ approximately minimizes $\eta \mapsto R^{\pred}_{(\Sigma,\mu_0)}(\eta)$. Based on the same heuristics as for $\hat{\eta}^{\GCV}$ in (\ref{def:res_tuning}), we may therefore expect that $\hat{\eta}^{\CV}$ in (\ref{def:CV_tuning}) simultaneously provides optimal prediction, estimation and in-sample risks for `most' signal $\mu_0$'s. This is the content of the following theorem. 
\begin{theorem}\label{thm:tuning_cv}
    Suppose the same conditions as in Theorem \ref{thm:tuning_res} and $\max_{\ell \in [k]} m_\ell/n\leq 1/(2K)$ hold for some $K>0$. Fix $\delta \in (0,1/2]$, $L\geq K/\delta^2$ and a small enough $\vartheta \in (0,1/50)$. Further assume $\min_{\ell \in [k]} m_\ell \geq \log^{2/\delta} m$. There exist a constant $C=C(K,L,\delta,\vartheta)>0$ and a measurable set $\mathcal{U}_{\delta,\vartheta}\subset B_n(1)\setminus B_n(\delta)$ with $\mathrm{vol}(\mathcal{U}_{\delta,\vartheta})/\mathrm{vol}(B_n(1)\setminus B_n(\delta))\geq 1- C e^{-n^{\vartheta}/C}$, such that for $\# \in \{\pred,\est,\ins\}$, 
	\begin{align*}
	&\sup_{\mu_0 \in \mathcal{U}_{\delta,\vartheta}}\Prob\bigg(R^{\#}_{(\Sigma,\mu_0)}(\hat{\eta}^{\CV}_L)\geq \min_{\eta \in \Xi_L} R^{\#}_{(\Sigma,\mu_0)}(\eta) +C\cdot\bigg\{\frac{1}{k}\sum_{\ell\in [k]}\frac{1}{m_\ell^{(1-\delta)/2}}+\frac{1}{k}+n^{-\vartheta}\bigg\}\bigg)\nonumber\\
	&\qquad \leq C (1+\mathcal{L}_{\{m_\ell\}})\cdot n^{-1/7}.
	\end{align*}
	Here  $\mathcal{L}_{\{m_\ell\}}\equiv \sum_{\ell \in [k]} (m_\ell/m)^{-1}$. 
\end{theorem}
Non-asymptotic results of this type for $k$-fold cross validation are previously obtained for $R^{\pred}_{(\Sigma,\mu_0)}(\hat{\eta}^{\CV})$ in the Lasso setting \cite[Proposition 4.3]{miolane2021distribution} under isotropic $\Sigma=I_n$, where the range of the regularization must be strictly away from the interpolation regime. In contrast, our results above are valid down to $\eta=0$ when $\phi^{-1}>1$, and allow for general anisotropic $\Sigma$.   

Interestingly, the error bound in the above theorem reflects the folklore tension between the bias and variance in the selection of $k$ in the cross validation scheme (cf. \cite[Chapter 5]{james2021introduction}):
\begin{itemize}
	\item For a small number of $k$, $R^{\CV,k}_{(\Sigma,\mu_0)}(\eta)$ is biased for estimating $R^{\pred}_{(\Sigma,\mu_0)}(\eta)$; this corresponds to the term $\bigo(1/k)$ in the error bound, which is known to be of the optimal order in Ridge regression (cf. \cite{liu2019ridge}). 
	\item For a large number of $k$,  $R^{\CV,k}_{(\Sigma,\mu_0)}(\eta)$ has large fluctuations; this corresponds to the term $\bigo\big(k^{-1}\sum_{\ell \in [k]} m_\ell^{-(1-\delta)/2}\big)=\bigo\big((k/m)^{(1-\delta)/2})$ in the equal-sized case. By a central limit heuristic (cf. \cite{kissel2022high,austern2025asymptotics}), we also expect this term to be of a near optimal order. 
\end{itemize} 

For the choice of $k$, it is instructive to consider the common equal-sized–folds case $m_\ell = m/k$. In this case, the error bound in Theorem~\ref{thm:tuning_cv} suggests that the optimal theoretical value of $k$ is $k \sim m^{1/3}$ (when $\delta$ is small). In our numerical experiments, choosing $k=5$ already yields cross-validation performance that is close to the theoretically optimal behavior (cf.\ Figure~\ref{fig:2}), whereas the theoretically prescribed optimal value of $k$ offers limited practical gain.

\subsection{Implications to statistical inference via $\hat{\mu}_\eta$}

As Ridge(less) estimators $\hat{\mu}_\eta$ are in general biased, debiasing is necessary for statistical inference of $\mu_0$, cf. \cite{bellec2023debias}. Here the debiasing scheme for $\hat{\mu}_\eta$ can be readily read off from the distributional characterizations in Theorems \ref{thm:min_norm_dist} and \ref{thm:universality_min_norm}. Assuming known covariance $\Sigma$, let the debiased Ridge(less) estimator be defined as 
\begin{align}\label{def:debiased_ridge}
\hat{\mu}_\eta^{\dR}\equiv (\Sigma+\tau_{\eta,\ast} I)\Sigma^{-1} \hat{\mu}_\eta. 
\end{align}
Note that $\tau_{\eta,\ast}$ and $\hat{\tau}_\eta$ is interchangeable in the above display due to known $\Sigma$. Using Theorems \ref{thm:min_norm_dist} and \ref{thm:universality_min_norm}, we expect that $\hat{\mu}_\eta^{\dR}\stackrel{d}{\approx} \mu_0+ \gamma_{\eta,\ast} \Sigma^{-1/2} g/\sqrt{n}$. This motivates the following confidence intervals for $\{\mu_{0,j}\}$:
\begin{align}\label{def:CI_debiased_ridge}
\mathrm{CI}_j(\eta) \equiv \Big[\hat{\mu}_{\eta,j}^{\dR} \pm \hat{\gamma}_\eta\cdot (\Sigma^{-1})_{jj}^{1/2} \cdot\frac{z_{\alpha/2}}{\sqrt{n}}\Big],\quad j \in [n].
\end{align}
Here $z_\alpha$ is the normal upper-$\alpha$ quantile defined via $\Prob(\mathcal{N}(0,1)>z_\alpha)=\alpha$. It is easy to see from the above definition that minimization of $\eta \mapsto \hat{\gamma}_\eta$ is equivalent to that of the CI length. As the former minimization procedure corresponds exactly to the proposal $\hat{\eta}^{\GCV}$ in (\ref{def:res_tuning}), we expect that $\{\mathrm{CI}_j(\hat{\eta}^{\GCV})\}$ provide the shortest (asymptotic) $(1-\alpha)$-CIs along the regularization path, and so do $\{\mathrm{CI}_j(\hat{\eta}^{\CV})\}$. 

Below we give a rigorous statement on the above informal discussion. Let $\mathscr{C}^{\dR}(\eta)\equiv n^{-1}\sum_{j=1}^n \bm{1}(\mu_{0,j} \in \mathrm{CI}_j(\eta))$ denote the averaged coverage of $\{\mathrm{CI}_j(\eta)\}$ for $\{\mu_{0,j}\}$. We have the following.

\begin{theorem}\label{thm:CI_cv}
	Suppose the same conditions as in Theorem \ref{thm:tuning_res} (resp. Theorem \ref{thm:tuning_cv}) for $\hat{\eta}^{\GCV}$ (resp. $\hat{\eta}^{\CV}$) hold for some $K>0$. Fix $\alpha \in (0,1/4],\delta \in (0,1/2]$, $L\geq K/\delta^2$ and a small enough $\vartheta \in (0,1/50)$. There exist a constant $C=C(K,L,\delta,\vartheta)>0$ and a measurable set $\mathcal{U}_{\delta,\vartheta}\subset B_n(1)\setminus B_n(\delta)$ with $\mathrm{vol}(\mathcal{U}_{\delta,\vartheta})/\mathrm{vol}(B_n(1)\setminus B_n(\delta))\geq 1- C e^{-n^{\vartheta}/C}$, such that  the CI length and the averaged coverage satisfy
	\begin{align*}
	&\sup_{\mu_0 \in \mathcal{U}_{\delta,\vartheta}}\Big\{\Prob\Big(\sqrt{n} z_{\alpha/2}^{-1}\cdot \max_{j \in [n]} \bigabs{ |\mathrm{CI}_j(\hat{\eta}_L^{\#})|-\min_{\eta \in \Xi_L} |\mathrm{CI}_j(\eta)| }\geq C\mathcal{E}^{\#}_n\Big)\\
	&\qquad\qquad  \vee \Prob\Big( \abs{ \mathscr{C}^{\dR}(\hat{\eta}^{\#}_L)-(1-\alpha)}\geq C(\mathcal{E}^{\#}_n)^{1/4}\Big) \Big\}\leq C\mathfrak{p}_n^{\#}.
	\end{align*}
	Here for $\# \in \{\GCV,\CV\}$, the quantities $\mathcal{E}^{\#}_n,\mathfrak{p}_n^{\#}$ are defined via
	\vspace{0.5ex}
	\renewcommand{\arraystretch}{1.2} 
	\begin{center}
		\begin{tabular}{|c||c|c|}
			\hline 
			& $\mathcal{E}^{\#}_n$ & $\mathfrak{p}_n^{\#}$\\
			\hline\hline
			$\#=\GCV$ & $n^{-\vartheta}$ & $n^{-1/7}$\\
			\hline
			$\#=\CV$ & $k^{-1}\sum_{\ell\in [k]}{m_\ell^{-(1-\delta)/2}}+k^{-1}+n^{-\vartheta}$ & $(1+\mathcal{L}_{\{m_\ell\}})\cdot n^{-1/7}$\\
			\hline
		\end{tabular}
	\end{center}
\end{theorem}

A somewhat non-standard special case of the above theorem is the noiseless setting $\sigma_\xi^2=0$ in the overparametrized regime $\phi^{-1}>1$. In this case, exact recovery of $\mu_0$ is impossible and our CI's above provide a precise scheme for partial recovery of $\mu_0$. As the effective noise $\phi \gamma_{\eta,\ast}^2(0) = \bar{R}^{\pred}_{(\Sigma,\mu_0)}(\eta)$, Theorem \ref{thm:error_rmt} and Proposition \ref{prop:derivative_R_simplify} suggest that $\eta\mapsto \gamma^2_{\eta,\ast}(0)$ is approximately minimized at $\eta=0$ for `most' $\mu_0$'s. This means that, in this noiseless case, the length of the adaptively tuned CIs is also approximately minimized at the interpolation regime for `most' $\mu_0$'s.

\section{Proof outlines}\label{section:proof_outline}

\subsection{Technical tools}\label{subsection:CGMT}

The main technical tool we use for the proof of Theorem \ref{thm:min_norm_dist} is the following version of convex Gaussian min-max theorem, taken from \cite[Corollary G.1]{miolane2021distribution}.

\begin{theorem}[Convex Gaussian Min-Max Theorem]\label{thm:CGMT}
	Suppose $D_u \in \R^{n_1+n_2}, D_v \in \R^{m_1+m_2}$ are compact sets, and $Q: D_u\times D_v \to \R$ is continuous. Let $G=(G_{ij})_{i \in [n_1],j\in[m_1]}$ with $G_{ij}$'s i.i.d. $\mathcal{N}(0,1)$, and $g \sim \mathcal{N}(0,I_{n_1})$, $h \sim \mathcal{N}(0,I_{m_1})$ be independent Gaussian vectors. For $u \in \R^{n_1+n_2}, v \in \R^{m_1+m_2}$, write $u_1\equiv u_{[n_1]}\in \R^{n_1}, v_1\equiv v_{[m_1]} \in \R^{m_1}$. Define
	\begin{align*}
	\Phi^{\textrm{p}} (G)& = \min_{u \in D_u}\max_{v \in D_v} \Big( u_1^\top G v_1 + Q(u,v)\Big), \nonumber\\
	\Phi^{\textrm{a}}(g,h)& = \min_{u \in D_u}\max_{v \in D_v} \Big(\pnorm{v_1}{} g^\top u_1 + \pnorm{u_1}{} h^\top v_1+ Q(u,v)\Big).
	\end{align*}
	Then the following hold.
	\begin{enumerate}
		\item For all $t \in \R$, $
		\Prob\big(\Phi^{\textrm{p}} (G)\leq t\big)\leq 2 \Prob\big(\Phi^{\textrm{a}}(g,h)\leq t\big)$. 
		\item If $(u,v)\mapsto u_1^\top G v_1+ Q(u,v)$ satisfies the conditions of Sion's min-max theorem for  the pair $(D_u,D_v)$ a.s. (for instance, $D_u,D_v$ are convex, and $Q$ is convex-concave), then for any $t \in \R$, $
		\Prob\big(\Phi^{\textrm{p}} (G)\geq t\big)\leq 2 \Prob\big(\Phi^{\textrm{a}}(g,h)\geq t\big)$. 
	\end{enumerate}
\end{theorem}

Clearly, $\geq$ (resp. $\leq$) in (1) (resp. (2)) can be replaced with $>$ (resp $<$). In the proofs below, we shall assume without loss of generality that $G,g,h$ are independent Gaussian matrix/vectors defined on the same probability space. 

As mentioned above, the CGMT above has been utilized for deriving precise risk/distributional asymptotics for a number of canonical statistical estimators across various important models; we only refer the readers to \cite{thrampoulidis2015regularized,thrampoulidis2018precise,salehi2019impact,loureiro2021learning,deng2022model,celentano2023lasso,han2023noisy,liang2022precise,wang2022does,zhang2022modern,montanari2023generalization} for some selected references.

\subsection{Reparametrization and further notation}

Consider the reparametrization
\begin{align*}
w=\Sigma^{1/2}(\mu-\mu_0),\quad \hat{w}_{\eta;Z}\equiv \Sigma^{1/2}(\hat{\mu}_{\eta;Z}-\mu_0).
\end{align*}
Then with 
\begin{align}\label{def:F_fcn}
F(w)\equiv F_{(\Sigma,\mu_0)}(w)=\frac{1}{2}\pnorm{\mu_0+\Sigma^{-1/2}w}{}^2,
\end{align}
we have the following reparametrized version of $\hat{\mu}_{\eta;Z}$:
\begin{align*}
\hat{w}_{\eta;Z}=
\begin{cases}
\argmin_{w \in \R^n} \big\{F(w): Zw=\xi\big\}, & \eta =0;\\
\argmin_{w \in \R^n} \big\{F(w)+\frac{1}{\eta}\cdot \frac{1}{2 n}\pnorm{Zw-\xi}{}^2\big\}, & \eta>0.
\end{cases}
\end{align*}
Next we give some further notation for cost functions. Let for $\eta\geq 0$,
\begin{align}\label{def:h_l}
h_{\eta;Z}(w,v)&\equiv \frac{1}{\sqrt{n}}\iprod{v}{Zw-\xi}+F(w)-\frac{\eta\pnorm{v}{}^2}{2},\nonumber\\
\ell_\eta(w,v)&\equiv \frac{1}{\sqrt{n}}\Big( -\pnorm{v}{}\iprod{g}{w}+\pnorm{w}{}\iprod{h}{v}-\iprod{v}{\xi}\Big)+F(w)-\frac{\eta\pnorm{v}{}^2}{2},
\end{align}
and for $L_v \in [0,\infty]$,
\begin{align}\label{def:H_L}
H_{\eta;Z}(w;L_v)&\equiv  \max_{v \in B_n(L_v)} h_{\eta;Z}(w,v)\equiv  \max_{v \in B_n(L_v)}\bigg\{ \frac{\iprod{v}{Zw-\xi}}{\sqrt{n}}+F(w)-\frac{\eta \pnorm{v}{}^2}{2}\bigg\},\\ 
L_\eta(w;L_v)&\equiv \max_{v \in B_n(L_v)} \ell_\eta(w,v)=\max_{\beta\in [0,L_v]}\bigg\{\frac{\beta}{\sqrt{n}} \Big(\bigpnorm{\pnorm{w}{}h-\xi}{}-\iprod{g}{w}\Big)+F(w)- \frac{\eta \beta^2}{2}\bigg\}. \nonumber
\end{align}
We shall simply write $H_{\eta;Z}(\cdot)=H_{\eta;Z}(\cdot;\infty)$ and $L_\eta(\cdot)=L_\eta(\cdot;\infty)$. When $Z=G$, we sometimes write $h_{\eta;G}=h_\eta$ and $H_{\eta;G}=H_\eta$ for simplicity of notation.

Let the empirical noise $\sigma_m^2$ and its modified version be
\begin{align}\label{def:sigma_pm}
\sigma_m^2\equiv \frac{\pnorm{\xi}{}^2}{\pnorm{h}{}^2},\quad \sigma_{\pm}^2(L_w)\equiv \bigg(\sigma_m^2\pm 2L_w \frac{\abs{\iprod{h}{\xi}}}{\pnorm{h}{}^2}\bigg)_+.
\end{align}

Finally we define $\mathsf{D}_{\eta,\pm}$ and its deterministic version $\overline{\mathsf{D}}_\eta$ as follows:
\begin{align}\label{def:D_fcn}
\mathsf{D}_{\eta,\pm}(\beta,\gamma)&\equiv \frac{\beta}{2}\bigg(\gamma \big(\phi e_h^2-e_g^2\big)+\frac{\sigma_\pm^2}{\gamma}\bigg)-\frac{\eta\beta^2}{2}+\mathsf{e}_F\bigg(\frac{\gamma}{\sqrt{n}} g;\frac{\gamma}{\beta}\bigg),\nonumber\\
\overline{\mathsf{D}}_\eta(\beta,\gamma)& \equiv \frac{\beta}{2}\bigg(\gamma \big(\phi-1\big)+\frac{\sigma_\xi^2}{\gamma}\bigg)-\frac{\eta\beta^2}{2}+\E\mathsf{e}_F\bigg(\frac{\gamma}{\sqrt{n}} g;\frac{\gamma}{\beta}\bigg).
\end{align}
Here recall $\mathsf{e}_F$ is the Moreau envelope of $F$ in (\ref{def:F_fcn}). Note that $\mathsf{D}_{\eta,\pm}$ depends on the choice of $L_w$, but for notational convenience we drop this dependence here.

\subsection{Proof outline for Theorem \ref{thm:min_norm_dist} for $\eta=0$}\label{subsection:proof_outline_gaussian}

We shall outline below the main steps for the proof of Theorem \ref{thm:min_norm_dist} for $\eta=0$ in the regime $\phi^{-1}>1$ under a stronger condition $\pnorm{\Sigma^{-1}}{\op}\lesssim 1$. 
The high level strategy of the proof shares conceptual similarities to \cite{miolane2021distribution,celentano2023lasso}, but the details differ significantly. 

\vspace{1.5mm}

\noindent (\emph{\textbf{Step 1}: Localization of the primal optimization}). In this step, we show that for $L_w,L_v>0$ such that $L_w\wedge L_v \gtrsim 1$, with high probability (w.h.p.),
\begin{align}\label{ineq:proof_outline_1}
\min_{w \in B_n(L_w)} H_0(w;L_v)= \min_{w \in \R^n} H_0(w). 
\end{align}
A formal statement of the above localization can be found in Proposition \ref{prop:H_local}. The key point here is that despite $\min_w H_0(w)$ optimizes a deterministic function with a random constraint, it can be efficiently rewritten (in a probabilistic sense) in a minimax form indexed by \emph{compact sets} that facilitate the application of the convex Gaussian min-max Theorem \ref{thm:CGMT}.

\vspace{1.5mm}

\noindent (\emph{\textbf{Step 2}: Characterization of the Gordon cost optimum}). In this step, we show that a suitably localized version of $\min_w L_0(w)$ concentrates around some \emph{deterministic} quantity involving the function $\overline{\mathsf{D}}_0$ in (\ref{def:D_fcn}). In particular, we show in Theorem \ref{thm:char_gordon_cost_opt} that for $L_w,L_v\asymp 1$ chosen large enough, w.h.p.,
\begin{align}\label{ineq:proof_outline_2}
\min_{w \in B_n(L_w)}L_0(w;L_v) \approx \max_{\beta>0}\min_{\gamma >0} \overline{\mathsf{D}}_0(\beta,\gamma). 
\end{align}
The proof of (\ref{ineq:proof_outline_2}) is fairly involved, as the minimax problem $\min_w L_0(w)=\min_w \max_v \ell_0(w,v)$ (and its suitably localized versions) cannot be computed exactly. We get around this technical issue by the following bracketing strategy:
\begin{itemize}
	\item (\emph{\textbf{Step 2.1}}). We show in Proposition \ref{prop:L_local} that for the prescribed choice of $L_w,L_v$, w.h.p., both
	\begin{align*}
	\max_{\beta>0}\min_{\gamma>0} \mathsf{D}_{0,-}(\beta,\gamma)\leq \min_{w \in B_n(L_w)}L_0(w;L_v)\leq \max_{\beta>0}\min_{\gamma>0} \mathsf{D}_{0,+}(\beta,\gamma),
	\end{align*}
	and the localization 
	\begin{align*}
	\max_{\beta>0}\min_{\gamma>0} \mathsf{D}_{0,\pm}(\beta,\gamma) = \max_{1/C \leq \beta \leq C } \min_{1/C\leq \gamma \leq C} \mathsf{D}_{0,\pm}(\beta,\gamma)
	\end{align*}
	hold for some large $C>0$.
	\item (\emph{\textbf{Step 2.2}}). We show in Proposition \ref{prop:D_pm_barD} that for localized minimax problems, we may replace $\mathsf{D}_{0,\pm}$ by $\overline{\mathsf{D}}_{0}$: w.h.p.,
	\begin{align*}
	\max_{1/C \leq \beta \leq C} \min_{1/C\leq \gamma \leq C} \mathsf{D}_{0,\pm}(\beta,\gamma) \approx \max_{1/C \leq \beta \leq C} \min_{1/C\leq \gamma \leq C} \overline{\mathsf{D}}_0(\beta,\gamma).
	\end{align*}
	\item (\emph{\textbf{Step 2.3}}). We show in Proposition \ref{prop:local_barD} that (de)localization holds for the (deterministic) max-min optimization problem with $\overline{\mathsf{D}}_0$:
	\begin{align*}
	\max_{\beta>0}\min_{\gamma >0} \overline{\mathsf{D}}_0(\beta,\gamma) = \max_{1/C \leq \beta \leq C} \min_{1/C\leq \gamma \leq C} \overline{\mathsf{D}}_0(\beta,\gamma). 
	\end{align*}
\end{itemize}
Combining the above Steps 2.1-2.3 yields (\ref{ineq:proof_outline_2}). An important step to prove the (de)localization claims above is to derive apriori estimates for the solutions of the fixed point equation (\ref{eqn:fpe}) and its sample version, to be defined in (\ref{eqn:fpe_sample}). These estimates will be detailed in Section \ref{section:proof_fpe}. 

\vspace{1.5mm}

\noindent (\emph{\textbf{Step 3}: Locating the global minimizer of the Gordon objective}). In this step, we show that a suitably localized version of the Gordon objective $w\mapsto L_0(w)$ attains its global minimum approximately at $w_{0,\ast} \equiv \Sigma^{1/2}\big(\hat{\mu}_{(\Sigma,\mu_0)}^{\seq}(\gamma_{0,\ast};\tau_{0,\ast})-\mu_0\big)$ in the following sense. For any $\epsilon>0$ and any $g:\R^n\to \R$ that is $1$-Lipschitz with respect to $\pnorm{\cdot}{\Sigma^{-1}}$, let $
D_{0;\epsilon}(\mathsf{g})\equiv \big\{w\in \R^n: \abs{\mathsf{g}(w)-\E \mathsf{g}(w_{0,\ast})}\geq \epsilon\big\}$ be the `exceptional set'. We show in Theorem \ref{thm:gordon_gap} that again for $L_w,L_v\asymp 1$ chosen large enough, w.h.p.,
\begin{align}\label{ineq:proof_outline_4}
\min_{w \in D_{0;\epsilon}(\mathsf{g})\cap B_n(L_w)}L_0(w;L_v) \geq  \max_{\beta>0}\min_{\gamma >0} \overline{\mathsf{D}}_0(\beta,\gamma)  + \Omega_\epsilon(1).
\end{align}
The main challenge in proving (\ref{ineq:proof_outline_4}) is partly attributed to the possible violation of strong convexity of the map $w\mapsto L_0(w;L_v)$, due to the necessity of working with non-Gaussian $\xi$'s. We will get around this technical issue in similar spirit to Step 2 by another bracketing strategy. In particular:
\begin{itemize}
	\item (\emph{\textbf{Step 3.1}}). In Lemma \ref{lem:L_L_pm_property}, we will use surrogate, strongly convex functions $L_{0,\pm}(\cdot;L_v)$, formally defined in (\ref{def:L_eta_pm}), to provide a sufficiently tight bracket for $L_0(\cdot;L_v)$ over large enough compact sets. 
	\item (\emph{\textbf{Step 3.2}}). In Proposition \ref{prop:L_pm_minimizer}, we show that the minimizers of $w\mapsto L_{0,\pm}(\cdot;L_v)$ can be computed exactly and are close enough to $w_{0,\ast}$. 
	\item (\emph{\textbf{Step 3.3}}). In Proposition \ref{prop:L_minimizer}, combined with the tight bracketing and certain apriori estimates, we then conclude that all minimizers of $w \mapsto L_0(\cdot;L_v)$ must be close to $w_{0,\ast}$.
\end{itemize}
With all the above steps, finally we prove (\ref{ineq:proof_outline_4}) by (i) using the proximity of $L_0$ and its surrogate $L_{0,\pm}$ and (ii) exploiting the strong convexity of $L_{0,\pm}$.

\vspace{1.5mm}

\noindent (\emph{\textbf{Step 4}: Putting pieces together and establishing uniform guarantees}). In this final step, we shall use the convex Gaussian min-max theorem to translate the estimates (\ref{ineq:proof_outline_2}) in Step 2 and (\ref{ineq:proof_outline_4}) in Step 3 to their counterparts with primal cost function $H_0$. For the global cost optimum, with the help of the localization in (\ref{ineq:proof_outline_1}), by choosing $L_w, L_v\asymp 1$, we have w.h.p.,
\begin{align*}
\min_{w \in \R^n} H_0(w) \stackrel{(\ref{ineq:proof_outline_1})}{=} \min_{w \in B_n(L_w)} H_0(w;L_v)\stackrel{\Prob}{\approx}  \min_{w \in B_n(L_w)} L_0(w;L_v) \stackrel{(\ref{ineq:proof_outline_2})}{\approx}\max_{\beta>0}\min_{\gamma >0} \overline{\mathsf{D}}_0(\beta,\gamma).
\end{align*}
For the cost over the exceptional set, we have w.h.p.,
\begin{align*}
\min_{w \in D_{0;\epsilon}(\mathsf{g})\cap B_n(L_w)}H_0(w)&\geq \min_{w \in D_{0;\epsilon}(\mathsf{g})\cap B_n(L_w)}H_0(w;L_v) \\
&\stackrel{\Prob}{\geq}\min_{w \in D_{0;\epsilon}(\mathsf{g})\cap B_n(L_w)}L_0(w;L_v) &\stackrel{(\ref{ineq:proof_outline_4})}{\geq} \max_{\beta>0}\min_{\gamma >0} \overline{\mathsf{D}}_0(\beta,\gamma)  + \Omega_\epsilon(1). 
\end{align*}
Combining the above two displays, we then conclude that w.h.p., $\hat{w}_0\notin D_{0;\epsilon}(\mathsf{g})\cap B_n(L_w)$. Finally using apriori estimate on $\pnorm{\hat{w}_0}{}$ we may conclude that w.h.p., $\hat{w}_0\notin D_{0;\epsilon}(\mathsf{g})$, i.e., $\abs{\mathsf{g}(\hat{w}_0)-\E \mathsf{g}(w_{0,\ast})}\leq \epsilon$. 

The uniform guarantee in $\eta$ is then proved by (i) extending the above arguments to include any positive $\eta>0$, and (ii) establishing (high probability) Lipschitz continuity (w.r.t. $\pnorm{\cdot}{\Sigma^{-1}}$) of the maps $\eta \mapsto \hat{w}_\eta$ and $\eta \mapsto w_{\eta,\ast}$. 

Details of the above outline are implemented in Section \ref{section:proof_gaussian_design}. 

\subsection{Proof outline for Theorem \ref{thm:universality_min_norm} for $\eta=0$}\label{subsection:proof_outline_general}

The main tool we will use to prove the universality Theorem \ref{thm:universality_min_norm} is the following set of comparison inequalities developed in \cite{han2023universality}: Suppose $Z$ matches the first two moments of $G$, and possesses enough high moments. Then for any measurable sets $\mathcal{S}_w\subset [-L_n/\sqrt{n},L_n/\sqrt{n}]^n, \mathcal{S}_v \subset [-L_n/\sqrt{n},L_n/\sqrt{n}]^m$, and any smooth test function $\mathsf{T}:\R \to \R$ (standardized with derivatives of order $1$ in $\pnorm{\cdot}{\infty}$), 
\begin{align}\label{ineq:proof_outline_5}
\Big|\E\mathsf{T}\Big(\min_{w \in \mathcal{S}_w}\max_{v \in \mathcal{S}_v} h_{\eta;Z}(w,v)\Big)- \E\mathsf{T}\Big(\min_{w \in \mathcal{S}_w}\max_{v \in \mathcal{S}_v} h_{\eta;G}(w,v)\Big)\Big|&\leq \mathsf{r}_n(L_n),\nonumber\\
\Big|\E\mathsf{T}\Big(\min_{w \in \mathcal{S}_w} H_{\eta;Z}(w)\Big)- \E\mathsf{T}\Big(\min_{w \in \mathcal{S}_w} H_{\eta;G}(w)\Big)\Big|&\leq \mathsf{r}_n(L_n). 
\end{align}
Here $\mathsf{r}_n(L_n)\to 0$ for $L_n=n^{\vartheta}$ with sufficiently small $\vartheta>0$. The readers are referred to Theorems \ref{thm:universality_smooth} and \ref{thm:min_max_universality} for a precise statement of (\ref{ineq:proof_outline_5}). 

An important technical subtlety here is that while the first inequality in (\ref{ineq:proof_outline_5}) holds down to $\eta=0$, the second inequality does not. This is so because $\min_w H_{0;Z}(w)$, which minimizes a deterministic function under a random constraint due to the unbounded constraint in the maximization of $v$, is qualitatively different from $\min_w H_{\eta;Z}(w)$ for any $\eta>0$.

Now we shall sketch how the comparison inequalities (\ref{ineq:proof_outline_5}) lead to universality. 

\vspace{1.5mm}

\noindent (\emph{\textbf{Step 1}: Universality of the global cost optimum}). In this step, we shall use the first inequality in (\ref{ineq:proof_outline_5}) to establish the universality of the global Gordon cost:
\begin{align}\label{ineq:proof_outline_6}
\min_{w \in \R^n} H_{0;Z}(w)=\min_{w \in \R^n} \max_{v \in \R^m} h_{0;Z}(w,v)\stackrel{\Prob}{\approx} \min_{w \in \R^n} \max_{v \in \R^m} h_{0;G}(w,v).
\end{align}
See Theorem \ref{thm:universality_global_cost} for a formal statement of (\ref{ineq:proof_outline_6}). 

The crux to establish (\ref{ineq:proof_outline_6}) via the first inequality of (\ref{ineq:proof_outline_5}) is to show that, the ranges of the minimum and the maximum of $\min_w \max_v h_{0;Z}(w,v)$ can be localized into an $L_\infty$ ball of order close to $\bigo(1/\sqrt{n})$. This amounts to showing that the stationary points $(\hat{w}_{0;Z},\hat{v}_{0;Z})$, where $\hat{w}_{\eta;Z}=\Sigma^{1/2}(\hat{\mu}_{\eta;Z}-\mu_0)$ and  $\hat{v}_{\eta;Z}=-n^{-1/2}(XX^\top/n+\eta I_m)^{-1}Y$ (cf. Eqn. (\ref{eqn:u_v_alter_form})), are delocalized. We prove such delocalization properties in Proposition \ref{prop:delocal_u_v}  for `most' $\mu_0 \in B_n(1)$.

\vspace{1.5mm}

\noindent (\emph{\textbf{Step 2}: Universality of the cost over exceptional sets}). In this step, we shall use the second inequality in (\ref{ineq:proof_outline_5}) to establish the universality of the Gordon cost over exceptional sets $D_{0;\epsilon}(\mathsf{g})$. In particular, we show in Theorem \ref{thm:universality_exceptional_set} that with $L_n=Cn^{\vartheta}$ for sufficiently small $\vartheta>0$ and a large enough $C_0>0$, w.h.p.,
\begin{align}\label{ineq:proof_outline_8}
\min_{w \in D_{0;\epsilon}(\mathsf{g})\cap B_{(2,\infty)}(C_0, {L_n}/{\sqrt{n}}) }H_{0;Z}(w)\geq \max_{\beta>0}\min_{\gamma >0}\overline{\mathsf{D}}_0(\beta,\gamma)+\Omega_\epsilon(1).
\end{align}
Here $B_{(2,\infty)}(C_0,L_n/\sqrt{n})=B_n(C_0)\cap L_\infty(L_n/\sqrt{n})$. A technical difficulty to apply the second inequality of (\ref{ineq:proof_outline_5}) rests in its singular behavior near the interpolation regime $\eta=0$. Also, we note that for a general exceptional set $D_{0;\epsilon}(\mathsf{g})$, the maximum over $v$ in $\min_{w \in D_{0;\epsilon}(\mathsf{g})} H_{0;Z}(w)=\min_{w \in D_{0;\epsilon}(\mathsf{g})} \max_v h_{0;Z}(w,v)$ need not be delocalized, so the first inequality of (\ref{ineq:proof_outline_5}) cannot be applied. This singularity issue will be resolved in two steps:
\begin{itemize}
	\item (\emph{\textbf{Step 2.1}}). First, we use the second inequality of (\ref{ineq:proof_outline_5}) to show that,  (\ref{ineq:proof_outline_8}) is valid for a version with small enough $\eta>0$:
	\begin{align*}
	\Prob\bigg(\min_{w \in D_{\eta;\epsilon}(\mathsf{g})\cap B_{(2,\infty)}(C_0, {L_n}/{\sqrt{n}}) }H_{\eta;Z}(w)\geq \max_{\beta>0}\min_{\gamma >0}\overline{\mathsf{D}}_\eta(\beta,\gamma)+\Omega_\epsilon(1)\bigg)\geq 1-c_\eta\cdot \mathfrak{o}(1).
	\end{align*}
	See (\ref{ineq:gordon_cost_except_universality_2}) for a precise statement. As expected, $c_\eta$ blows up as $\eta \downarrow 0$. 
	\item (\emph{\textbf{Step 2.2}}). Next, by using the `stability' of the set $D_{\eta;\epsilon}(\mathsf{g})$ (cf. Lemma \ref{lem:D_0_eta}) and $\max_{\beta>0}\min_{\gamma  >0}\overline{\mathsf{D}}_\eta(\beta,\gamma)$ (cf. Eqn. (\ref{ineq:cont_D_eta})) with respect to $\eta$, for a small enough $\eta>0$, we have the following series of inequalities:
	\begin{align*}
	&\min_{w \in D_{0;\epsilon}(\mathsf{g})\cap B_{(2,\infty)}(C_0, \frac{L_n}{\sqrt{n}}) }H_{0;Z}(w)\\
	&\geq \min_{w \in D_{0;\epsilon}(\mathsf{g})\cap B_{(2,\infty)}(C_0, \frac{L_n}{\sqrt{n}}) }H_{\eta;Z}(w)\quad \hbox{(by definition of $H_{\eta;Z}$)}\\
	&\geq \min_{w \in D_{\eta;\epsilon_\eta}(\mathsf{g})\cap B_{(2,\infty)}(C_0, \frac{L_n}{\sqrt{n}}) }H_{\eta;Z}(w)\quad (\hbox{$\epsilon_\eta\approx \epsilon$ by Lemma \ref{lem:D_0_eta}})\\
	&\stackrel{\Prob}{\geq}\max_{\beta>0}\min_{\gamma >0}\overline{\mathsf{D}}_\eta(\beta,\gamma)+\Omega_\epsilon(1)\quad \hbox{(by Step 2.1 above)}\\
	&\geq  \max_{\beta>0}\min_{\gamma >0}\overline{\mathsf{D}}_0(\beta,\gamma)-\bigo(\eta)+\Omega_\epsilon(1)\quad \hbox{(by Eqn. (\ref{ineq:cont_D_eta}))}. 
	\end{align*}
	Now for a given $\epsilon>0$, we may choose $\eta>0$ small enough so that the term $-\bigo(\eta)$ is absorbed into $\Omega_\epsilon(1)$, and therefore concluding (\ref{ineq:proof_outline_8}). 
\end{itemize}
A complete proof of the above outline is detailed in Section \ref{section:proof_general_design}.

\section{Proof preliminaries}\label{section:preliminaries}

\subsection{Some properties of $\mathsf{e}_F$ and $\prox_F$}

We write $g_n\equiv g/\sqrt{n}$ in this subsection. First we give an explicit expression for $\E \err_{(\Sigma,\mu_0)}(\gamma;\tau)$ and $\E \dof_{(\Sigma,\mu_0)}(\gamma;\tau)$.

\begin{lemma}\label{lem:est_dof}
	For any $(\gamma,\tau) \in (0,\infty)^2$,
	\begin{align*}
	\E \err_{(\Sigma,\mu_0)}(\gamma;\tau)& = \tau^2 \pnorm{(\Sigma+\tau I)^{-1}\Sigma^{1/2}\mu_0}{}^2+ \gamma^2\cdot n^{-1} \tr\big(\Sigma^2 (\Sigma+\tau I)^{-2}\big),\\
	\E \dof_{(\Sigma,\mu_0)}(\gamma;\tau)
	& = \gamma^2\cdot n^{-1} \tr\big(\Sigma (\Sigma+\tau I)^{-1}\big).
	\end{align*}
\end{lemma}
\begin{proof}
	Using the closed-form of $\hat{\mu}_{(\Sigma,\mu_0)}^{\seq}$, we may compute
	\begin{align}\label{ineq:est_dof_1}
	\Sigma^{1/2}\big(\hat{\mu}_{(\Sigma,\mu_0)}^{\seq}(\gamma;\tau)-\mu_0\big)& = (\Sigma+\tau I)^{-1}\Sigma^{1/2} \big(-\tau \mu_0+ \gamma\Sigma^{1/2}g_n\big).
	\end{align}
	The claims follow from direct calculations.
\end{proof}

Next we give explicit expression for $\prox_F(\gamma g_n;\tau)$ and $\mathsf{e}_F(\gamma g_n;\tau)$.

\begin{lemma}\label{lem:prox_env_F}
	It holds that
	\begin{align*}
	\prox_F(\gamma g_n;\tau)&=\Sigma^{1/2}\big(\hat{\mu}_{(\Sigma,\mu_0)}^{\seq}(\gamma;\tau)-\mu_0\big),\\
	\mathsf{e}_F(\gamma g_n;\tau)&=\frac{1}{2\tau}\pnorm{\Sigma^{1/2} \hat{\mu}_{(\Sigma,\mu_0)}^{\seq}(\gamma;\tau)-y_{(\Sigma,\mu_0)}^{\seq}(\gamma) }{}^2 +\frac{1}{2}\pnorm{\hat{\mu}_{(\Sigma,\mu_0)}^{\seq}(\gamma;\tau)}{}^2.
	\end{align*}
	Furthermore,
	\begin{align*}
	\E \mathsf{e}_F(\gamma g_n;\tau)& = \frac{1}{2\tau}\big(\E \err_{(\Sigma,\mu_0)}(\gamma;\tau)-2 \E \dof_{(\Sigma,\mu_0)}(\gamma;\tau)+\gamma^2\big)\\
	&\qquad + \frac{1}{2}\Big(\pnorm{(\Sigma+\tau I)^{-1}\Sigma \mu_0}{}^2+ \gamma^2\cdot \frac{1}{n}\tr(\Sigma (\Sigma+\tau I)^{-2})\Big).
	\end{align*}
\end{lemma}
\begin{proof}
	The two identities in the first display follows from the definition of $F$. For the second display, note that $\E \mathsf{e}_F(\gamma g_n;\tau)$ is equal to
	\begin{align*}
	 \frac{1}{2\tau}\big(\E \err_{(\Sigma,\mu_0)}(\gamma;\tau)-2 \E \dof_{(\Sigma,\mu_0)}(\gamma;\tau)+\gamma^2\big)+\frac{1}{2} \E \pnorm{\hat{\mu}_{(\Sigma,\mu_0)}^{\seq}(\gamma;\tau)}{}^2 .
	\end{align*}
	Using $
	\E \pnorm{\hat{\mu}_{(\Sigma,\mu_0)}^{\seq}(\gamma;\tau)}{}^2 = \pnorm{(\Sigma+\tau I)^{-1}\Sigma \mu_0}{}^2+ \gamma^2\cdot n^{-1}\tr\big(\Sigma (\Sigma+\tau I)^{-2}\big)$
	to conclude. 
\end{proof}

The derivative formula below for $\mathsf{e}_{F}$ will be useful.

\begin{lemma}\label{lem:derivative_eF}
	It holds that
	\begin{align*}
	\nabla_x \mathsf{e}_{F}(x;\tau) = \frac{1}{\tau}\big(x-\prox_{F}(x;\tau)\big), \quad	\partial_\tau \mathsf{e}_{F}(x;\tau) = -\frac{1}{2\tau^2}\pnorm{x-\prox_{F}(x;\tau)}{}^2.
	\end{align*}
\end{lemma}
\begin{proof}
	See e.g., \cite[Lemmas B.5 and D.1]{thrampoulidis2018precise}.
\end{proof}

Finally we provide a concentration inequality for $\mathsf{e}_F(\gamma g_n;\tau)$.

\begin{proposition}\label{prop:conc_e_F}
	There exists some universal constant $C>0$ such that
	\begin{align*}
	&\Prob\Big(\bigabs{\mathsf{e}_F(\gamma g_n;\tau)-\E \mathsf{e}_F(\gamma g_n;\tau)} \geq C\Big\{ v \E^{1/2}  \mathsf{e}_F(\gamma g_n;\tau) \sqrt{\frac{t}{n} }+ v^2\cdot \frac{t}{n}\Big\}\Big)\leq Ce^{-t/C}
	\end{align*}
	holds for any $t\geq 0$. Here $v^2\equiv v^2(\gamma,\tau) \equiv \gamma^2\big(\tau \pnorm{(\Sigma+\tau I)^{-1}}{\op}^2+\pnorm{(\Sigma+\tau I)^{-1}\Sigma^{1/2}}{\op}^2\big)$.
\end{proposition}
\begin{proof}
	Using that $
	\nabla_g \hat{\mu}_{(\Sigma,\mu_0)}^{\seq}(\gamma;\tau) = \frac{\gamma}{\sqrt{n}} (\Sigma+\tau I)^{-1}\Sigma^{1/2}$ and $\nabla_g y^{\seq}(\gamma) = \frac{\gamma}{\sqrt{n}} I$, 
	\begin{align*}
	\nabla_g \mathsf{e}_F(\gamma g_n;\tau) &= \frac{1}{\tau}\cdot \frac{\gamma}{\sqrt{n}}\big((\Sigma+\tau I)^{-1}\Sigma-I\big)\big(\Sigma^{1/2} \hat{\mu}_{(\Sigma,\mu_0)}^{\seq}(\gamma;\tau)-y_{(\Sigma,\mu_0)}^{\seq}(\gamma) \big)\\
	&\qquad +\frac{\gamma}{\sqrt{n}}(\Sigma+\tau I)^{-1}\Sigma^{1/2} \nabla_g \hat{\mu}_{(\Sigma,\mu_0)}^{\seq}(\gamma;\tau).
	\end{align*} 
	This means
	\begin{align}\label{ineq:conc_eF_1}
	\pnorm{\nabla_g \mathsf{e}_F(\gamma g_n;\tau)}{}^2&\leq 2\gamma^2\cdot n^{-1}\Big\{\pnorm{(\Sigma+\tau I)^{-1}}{\op}^2\pnorm{\Sigma^{1/2} \hat{\mu}_{(\Sigma,\mu_0)}^{\seq}(\gamma;\tau)-y_{(\Sigma,\mu_0)}^{\seq}(\gamma) }{}^2\nonumber\\
	&\qquad\qquad\qquad\qquad  + \pnorm{(\Sigma+\tau I)^{-1}\Sigma^{1/2}}{\op}^2  \pnorm{ \nabla_g \hat{\mu}_{(\Sigma,\mu_0)}^{\seq}(\gamma;\tau) }{}^2 \Big\} \nonumber\\
	&\leq 4\gamma^2\cdot  n^{-1}\Big(\tau \pnorm{(\Sigma+\tau I)^{-1}}{\op}^2+\pnorm{(\Sigma+\tau I)^{-1}\Sigma^{1/2}}{\op}^2 \Big)\cdot  \mathsf{e}_F(\gamma g_n;\tau).
	\end{align}
	From here we may conclude by setting $
	H(g)\equiv \mathsf{e}_F\big(\gamma g_n;\tau\big)$ and $\Gamma^2 \equiv 4\gamma^2 n^{-1}\big(\tau \pnorm{(\Sigma+\tau I)^{-1}}{\op}^2+\pnorm{(\Sigma+\tau I)^{-1}\Sigma^{1/2}}{\op}^2 \big)$
	in Proposition \ref{prop:conc_H_generic}.
\end{proof}

\subsection{Some high probability events}\label{sec:hpb-events}
Let
\begin{align}\label{def:eh_eg}
e_h^2={\pnorm{h}{}^2}/{m}, \quad e_g^2\equiv {\pnorm{g}{}^2}/{n}.
\end{align}
For $M,\delta>0$, consider the event 
\begin{align*}
\mathscr{E}_0(M)&\equiv \big\{ ({\pnorm{G}{\op}}/{\sqrt{n}})\vee \big[\pnorm{({GG^\top}/{n})^{-1}}{\op}\bm{1}_{\eta=0}\big]\leq M\big\},\\
\mathscr{E}_{1,0}(\delta)&\equiv \big\{\abs{e_g^2-1}\vee \abs{e_h^2-1}\vee  \abs{ n^{-1/2}\iprod{\Sigma^{1/2}g}{\mu_0}}\vee \abs{n^{-1}\iprod{h}{\xi}} \leq \delta \big\},\\
\mathscr{E}_{1,\xi}(\delta)&\equiv \big\{\abs{ (\pnorm{\xi}{}^2/m)-\sigma_\xi^2} \leq \delta \big\},\\
\mathscr{E}_1(\delta)&\equiv \mathscr{E}_{1,0}(\delta)\cap \mathscr{E}_{1,\xi}(\delta).
\end{align*}
Here in the definition of $\mathscr{E}_0(M)$, we interpret $\infty\cdot 0 =0$.  Typically we think of $M\asymp 1$ and $\delta \asymp 1/\sqrt{n}$.

\begin{lemma}\label{lem:sigma_pm_E1}
Fix $\delta \in (0,1/2)$ and $L_w>0$. Then $\mathscr{E}_1(\delta)\subset \mathscr{E}_2\big(4(\sigma_\xi^2+1+\phi^{-1}L_w)\delta,L_w\big)$, where $\mathscr{E}_2(\delta,L_w)\equiv \big\{ \abs{ \sigma_\pm^2(L_w)-\sigma_\xi^2 }\leq \delta \big\}$.
\end{lemma}
\begin{proof}
Using the definition of $\sigma_\pm^2(L_w)$ in (\ref{def:sigma_pm}), on $\mathscr{E}_1(\delta)$, we have
\begin{align*}
\bigabs{\sigma_\pm^2(L_w)-\sigma_\xi^2}\leq \frac{\pnorm{\xi}{}^2}{\pnorm{h}{}^2}\abs{e_h^2-1}+\biggabs{ \frac{\pnorm{\xi}{}^2}{m}-\sigma_\xi^2}+ \frac{2L_w \abs{\iprod{h}{\xi}} }{\pnorm{h}{}^2}\leq 4(\sigma_\xi^2+1+\phi^{-1}L_w)\delta.
\end{align*}
The claim follows.
\end{proof}

\begin{lemma}\label{lem:conc_E_0}
	Suppose $1/K\leq \phi^{-1}-\bm{1}_{\eta=0}\leq K$. Then there exists some $C=C(K)>0$ such that $\Prob(\mathscr{E}_0(C))\geq 1-Ce^{-n/C}$.
\end{lemma}
\begin{proof}
	The claim for $\pnorm{G}{\op}/\sqrt{n}$ follows from standard concentration estimates. The claim for $\pnorm{(GG^\top/n)^{-1}}{\op}$ follows from, e.g., \cite[Theorem 1.1]{rudelson2009smallest}.
\end{proof}

\begin{lemma}\label{lem:conc_E_12}
	Suppose $1/K\leq \phi^{-1}\leq K$, and $\pnorm{\mu_0}{}\vee \pnorm{\Sigma}{\op} \leq K$ for some $K>0$, and Assumption \ref{assump:noise} hold with $\sigma_\xi^2>0$. There exists some constant $C=C(K,\sigma_\xi)>0$ such that for all $t\geq 0$, with $\delta(t,n)\equiv C(\sqrt{t/n}+t/n)$, for $\xi \in \mathscr{E}_{1,\xi}(\delta(t,n))$, we have $\Prob^\xi(\mathscr{E}_1(\delta(t,n)))\geq 1-e^{-t}$.
\end{lemma}
\begin{proof}
	The claim follows by standard concentration inequalities.
\end{proof}

\subsection{Some connections of the fixed point equation (\ref{eqn:fpe}) to RMT}\label{subsection:fpe_rmt}

The second equation of (\ref{eqn:fpe}) has a natural connection to RMT. To detail this connection, let $\hat{\Sigma}\equiv \Sigma^{1/2}G^\top G \Sigma^{1/2}/m \in \R^{n\times n}$ and $\check{\Sigma}\equiv G\Sigma G^\top/m \in \R^{m\times m}$ be the sample covariance matrix and its dimension flipped, companion matrix. For $z\in\mathbb{C}^+\equiv\{z \in \mathbb{C}: \Im z> 0\}$, let $\mathfrak{m}_n(z) \equiv m^{-1}\tr \big( \check{\Sigma}  - zI_m\big)^{-1}$ and $\mathfrak{m}(z)$ be the Stieltjes transforms of the empirical spectral distribution and the asymptotic eigenvalue density (cf. \cite[Definition 2.3]{knowles2017anisotropic}) of $\check{\Sigma}$, respectively. It is well-known that $\mathfrak{m}(z)$ can be determined uniquely via the fixed point equation
\begin{align}\label{eq:StieljesTransform}
z = -\frac{1}{\mathfrak{m}(z)} + \frac{1}{\phi}\cdot \frac{1}{n}\tr \Big(\big( I_n + \Sigma \mathfrak{m}(z) \big)^{-1}\Sigma\Big). 
\end{align}
See, e.g., \cite[Lemma 2.2]{knowles2017anisotropic} for more technical details and historical references. We also note that while the above equation is initially defined for $z \in \mathbb{C}^{+}$, it can be straightforwardly extended to the real axis provided that $z$ lies outside the support of the asymptotic spectrum of $\check{\Sigma}$. 

The following proposition provides a precise connection between the effective regularization $\tau_{\eta,\ast}$ defined via the second equation of (\ref{eqn:fpe}), and the Stieltjes transform $\mathfrak{m}$. This connection will prove important in some of the results ahead.
\begin{proposition}\label{prop:rmt_effp}
	For any $\eta>0$ and $\eta=0$ with $\phi^{-1}>1$, 
	\begin{align}\label{eq:traceSigma_to_m}
	n^{-1}\mathrm{tr} \big(( \Sigma + \tau_{\eta,\ast}I_n )^{-1}\Sigma\big) =\phi - \eta \cdot \mathfrak{m}(-\eta/\phi).
	\end{align}
\end{proposition}
\begin{proof}
	By comparing (\ref{eq:StieljesTransform}) and the second equation of (\ref{eqn:fpe}), we may identify the two equations by setting $
	\tau_{\eta,\ast} \equiv {1}/{\mathfrak{m}(-z_\eta)}$ with $z_\eta\equiv  \eta/\phi$, as claimed.
\end{proof}
While (\ref{eq:traceSigma_to_m}) appears somewhat purely algebraic, it actually admits a natural statistical interpretation. Suppose $\xi$ is also Gaussian. We may then compute 
\begin{align}\label{eqn:dof_calculation}
\mathrm{df}(\hat{\mu}_\eta)&=\sum_{j=1}^n \frac{\cov^X \big( (X\hat{\mu}_\eta)_j, Y_j\big)}{\sigma_\xi^2}= \tr\big((\hat{\Sigma}+z_\eta I_n)^{-1}\hat{\Sigma}\big)=n\big(\phi-\eta\cdot \mathfrak{m}_n(-z_\eta)\big).
\end{align}
Now comparing the above display with (\ref{eq:traceSigma_to_m}), we arrive at the following intriguing equivalence between the averaged law in RMT, and the proximity of $\hat{\mu}_\eta$ and $\hat{\mu}_{(\Sigma,\mu_0)}^{\seq}(\gamma_{\eta,\ast};\tau_{\eta,\ast})$ in terms of ``degrees-of-freedom'':
\begin{align*}
\mathfrak{m}_n(-z_\eta)\stackrel{\Prob}{\approx} \mathfrak{m}(-z_\eta)\,\Leftrightarrow\, \mathrm{df}(\hat{\mu}_\eta)/n \stackrel{\Prob}{\approx}  \mathrm{df}\big(\hat{\mu}_{(\Sigma,\mu_0)}^{\seq}(\gamma_{\eta,\ast};\tau_{\eta,\ast})\big)/n.
\end{align*}

\section{Properties of the fixed point equations}\label{section:proof_fpe}

\subsection{The fixed point equation (\ref{eqn:fpe})}

\begin{proposition}\label{prop:fpe_est}
	The following hold.
	\begin{enumerate}
		\item The fixed point equation (\ref{eqn:fpe}) admits a unique solution $(\gamma_{\eta,\ast},\tau_{\eta,\ast}) \in (0,\infty)^2$, for all $(m,n)\in \N^2$ when $\eta>0$ and $m<n$ when $\eta=0$. 
		\item The following apriori bounds hold:
		\begin{align*}
		\frac{1-\phi+\sqrt{\big(1-\phi\big)^2+4\mathcal{H}_\Sigma \eta} }{2\mathcal{H}_\Sigma} &\leq \tau_{\eta,\ast} \leq \inf_{k \in [0:\min\{m-1,n\}]} \bigg\{\frac{\sum_{j>k}\lambda_j}{m-k} + \frac{n}{m-k}\cdot \eta\bigg\},\\
		\frac{\sigma_\xi^2}{\phi}&\leq\gamma_{\eta,\ast}^2\leq \frac{\sigma_\xi^2+\pnorm{\Sigma}{\op}\pnorm{\mu_0}{}^2}{\phi}\bigg(1+\frac{\pnorm{\Sigma}{\op}}{\tau_{\eta,\ast}}\bigg).
		\end{align*}
		\item If $1/K\leq \phi^{-1}\leq K$ and $\pnorm{\Sigma}{\op} \vee \mathcal{H}_\Sigma\leq K$ for some $K>1$, then there exists some $C=C(K)>1$ such that uniformly in $\eta \in \Xi_K$, 
		\begin{align*}
		1/C\leq \tau_{\eta,\ast} \leq C,\quad 1/C\leq (-1)^{q+1}\partial_\eta^q \tau_{\eta,\ast}\leq C, \quad q\in \{1,2\}.
		\end{align*}
		If furthermore $1/K\leq \sigma_\xi^2\leq K$ and $\pnorm{\mu_0}{}\leq K$, then uniformly in $\eta \in \Xi_K$,
		\begin{align*}
		1/C\leq \gamma_{\eta,\ast}\leq C,\quad \abs{ \partial_\eta \gamma_{\eta,\ast} }\leq C.
		\end{align*}
	\end{enumerate}
	
\end{proposition}

\begin{proof}
	We shall write $(\gamma_{\eta,\ast},\tau_{\eta,\ast})=(\gamma_\ast,\tau_\ast)$ for notational simplicity. All the constants in $\lesssim,\gtrsim,\asymp$ below may depend on $K$.
	
	\noindent (1). First we prove the existence and uniqueness of $\tau_\ast$. We rewrite the second equation of (\ref{eqn:fpe}) as 
	\begin{align}\label{ineq:fpe_property_1}
	\phi = \frac{1}{n}\tr\big((\Sigma+\tau_\ast I)^{-1}\Sigma\big) + \frac{\eta}{\tau_\ast} = \frac{1}{n}\sum_{j=1}^n \frac{\lambda_j}{\lambda_j+\tau_\ast}+\frac{\eta}{\tau_\ast}\equiv \mathsf{f}(\tau_\ast). 
	\end{align}
	Clearly $\mathsf{f}(\tau)$ is smooth, non-increasing, $\mathsf{f}(0)=1>\phi$ for $\eta=0$ and $\mathsf{f}(0)=\infty$ for $\eta>0$, and $\mathsf{f}(\infty)=0$, so $\tau\mapsto  \mathsf{f}(\tau)-\phi$ must admit a unique zero $\tau_\ast \in (0,\infty)$. 
	
	Next we prove the existence and uniqueness of $\gamma_\ast$. Using Lemma \ref{lem:est_dof}, the equation $\phi\gamma_\ast^2  = \sigma_\xi^2+ \E\err_{(\Sigma,\mu_0)}(\gamma_\ast;\tau_\ast)$ reads
	\begin{align}\label{ineq:fpe_property_3}
	\phi=\frac{1}{\gamma_\ast^{2}}\big(\sigma_\xi^2+\tau_\ast^2 \pnorm{(\Sigma+\tau_\ast I)^{-1}\Sigma^{1/2}\mu_0}{}^2\big)+ \frac{1}{n} \tr\big( (\Sigma+\tau_\ast I)^{-2}\Sigma^2\big).
	\end{align}
	As $n^{-1}\tr\big((\Sigma+\tau_\ast I)^{-2}\Sigma^2 \big)<n^{-1}\tr\big( (\Sigma+\tau_\ast I)^{-1}\Sigma\big)\leq \phi$ by (\ref{ineq:fpe_property_1}) and the fact $\tau_\ast>0$, the above equation admits a unique solution $\gamma_\ast \in (0,\infty)$, analytically given by
	\begin{align}\label{ineq:fpe_property_2}
	\gamma_\ast^2 = \frac{\sigma_\xi^2+\tau_\ast^2 \pnorm{(\Sigma+\tau_\ast I)^{-1}\Sigma^{1/2}\mu_0}{}^2}{\phi- \frac{1}{n} \tr\big( (\Sigma+\tau_\ast I)^{-2}\Sigma^2\big)}= \frac{\sigma_\xi^2+\tau_\ast^2 \pnorm{(\Sigma+\tau_\ast I)^{-1}\Sigma^{1/2}\mu_0}{}^2}{\frac{\eta}{\tau_{\ast}}+\frac{\tau_{\ast}}{n}\tr\big((\Sigma+\tau_{\ast} I)^{-2}\Sigma\big)}.
	\end{align}
	\noindent (2). 	For the upper bound for $\tau_\ast$, using the equation (\ref{ineq:fpe_property_1}), we have 
	\begin{align*}
	m=n\phi \leq k+ \frac{1}{\tau_\ast}\sum_{j>k} \lambda_j +\frac{n\eta}{\tau_\ast},\quad \forall k \in [0:n],\, k\leq m-1.
	\end{align*}
	Solving for $\tau_\ast$ yields the desired upper bound. For the lower bound for $\tau_\ast$, note that (\ref{ineq:fpe_property_1}) leads to
	\begin{align*}
	\phi = 1-\tau_\ast \cdot \frac{1}{n} \sum_{j=1}^n \frac{1}{\lambda_j+\tau_\ast}+\frac{\eta}{\tau_\ast}\geq 1-\tau_\ast \mathcal{H}_\Sigma +\frac{\eta}{\tau_\ast},
	\end{align*}
	or equivalently $
	\mathcal{H}_\Sigma \tau_\ast^2+\big(\phi-1\big)\tau_\ast - \eta\geq 0$. 
	Solving this quadratic inequality yields the lower bound for $\tau_\ast$.

	On the other hand, the lower bound $\gamma_\ast^2\geq \sigma_\xi^2/\phi$ is trivial by (\ref{ineq:fpe_property_2}). For the upper bound for $\gamma_\ast$, using that 
	\begin{align}\label{ineq:fpe_property_4}
	\phi - \frac{1}{n}\tr\big((\Sigma+\tau_\ast I)^{-2}\Sigma^2 \big)&\geq \phi- \frac{1}{n}\tr\big((\Sigma+\tau_\ast I)^{-1}\Sigma \big)\cdot \max_{j \in [n]} \frac{\lambda_j}{\lambda_j+\tau_\ast}\geq  \phi\cdot \frac{\tau_\ast}{\pnorm{\Sigma}{\op}+\tau_\ast},
	\end{align} 
	and the first identity in (\ref{ineq:fpe_property_2}), we have
	\begin{align*}
	\gamma_\ast^2\leq \phi^{-1}\big(\sigma_\xi^2+\pnorm{\Sigma}{\op}\pnorm{\mu_0}{}^2\big)\big(1+{\pnorm{\Sigma}{\op}}/{\tau_\ast}\big).
	\end{align*}
	Collecting the bounds proves the claim.
	
	\noindent (3).  The claim on $\gamma_\ast,\tau_\ast$ is a simple consequence of (2). We shall prove the other claim on their derivatives. Viewing $\tau_\ast=\tau_\ast(\eta)$ and taking derivative with respect to $\eta$ on both sides of (\ref{ineq:fpe_property_1}) yield that, with $T_{-p,q}(\eta)\equiv n^{-1}\tr\big((\Sigma+\tau_\ast(\eta)I)^{-p}\Sigma^q\big)$ for $p,q \in \N$,
	\begin{align*}
	0 = -T_{-2,1}(\eta)\cdot \tau_\ast'(\eta)+\frac{1}{\tau_\ast(\eta)}-\frac{\eta}{\tau_\ast^2(\eta)}\cdot \tau_\ast'(\eta). 
	\end{align*}
	Solving for $\tau_\ast'(\eta)$ yields that
	\begin{align}\label{ineq:fpe_property_5_0}
	\tau_\ast'(\eta)= \frac{ \tau_\ast(\eta)}{\eta+\tau_\ast^2(\eta)\cdot T_{-2,1}(\eta)}\equiv\frac{\tau_\ast(\eta)}{G_0(\eta)} .
	\end{align}
	Further taking derivative with respect to $\eta$ on both sides of the above display (\ref{ineq:fpe_property_5_0}), we have
	\begin{align}\label{ineq:fpe_property_5_1}
	\tau_{\ast}''(\eta)&= \frac{1}{G_0^2(\eta)}\big(\tau_\ast'(\eta) G_0(\eta)-\tau_\ast(\eta)G_0'(\eta) \big)\nonumber\\
	&=\frac{1}{G_0^2(\eta)}\Big\{\tau_\ast(\eta)-\tau_\ast(\eta)\Big(1+2\tau_\ast(\eta)\tau_\ast'(\eta)T_{-2,1}(\eta)- 2\tau_\ast^2(\eta)\tau_\ast'(\eta) T_{-3,1}(\eta) \Big)\Big\}\nonumber\\
	& = \frac{2\tau_\ast^2(\eta)\tau_\ast'(\eta)}{G_0^2(\eta)}\Big(\tau_\ast(\eta) T_{-3,1}(\eta)-T_{-2,1}(\eta)\Big)= -\frac{2\tau_\ast^2(\eta)\tau_\ast'(\eta)}{G_0^2(\eta)} T_{-3,2}(\eta).
	\end{align}
	Using the apriori estimate for $\tau_\ast(\eta)$ proved in (2), it follows that for $q \in \{1,2\}$,
	\begin{align}\label{ineq:fpe_property_5}
	1\lesssim \inf_{\eta \in \Xi_K} (-1)^{q+1}\tau_\ast^{(q)}(\eta)\leq \sup_{\eta \in \Xi_K} (-1)^{q+1}\tau_\ast^{(q)}(\eta)\lesssim 1.
	\end{align}
	For $\gamma_\ast'(\eta)$, let us define 
	\begin{align*}
	G_1(\eta)\equiv \sigma_\xi^2+\tau_\ast^2(\eta) \pnorm{(\Sigma+\tau_\ast(\eta) I)^{-1}\Sigma^{1/2}\mu_0}{}^2,\quad
	 G_2(\eta)\equiv \phi- n^{-1}\tr\big( (\Sigma+\tau_\ast(\eta) I)^{-2}\Sigma^2\big). 
	\end{align*}
	Then 
	\begin{align}\label{ineq:fpe_property_6}
	\gamma_\ast'(\eta)= \frac{G_1'(\eta)G_2(\eta)-G_1(\eta)G_2'(\eta)}{2\gamma_\ast(\eta) G_2^2(\eta)}.
	\end{align}
	We shall now prove bounds for $G_1,G_1',G_2,G_2'$. First, using (\ref{ineq:fpe_property_4}), we have
	\begin{align*}
	\sigma_\xi^2\leq G_1(\eta)\leq \sigma_\xi^2+\frac{\tau_\ast(\eta)}{2}\pnorm{\mu_0}{}^2, \quad  \phi\cdot \frac{\tau_\ast(\eta)}{\pnorm{\Sigma}{\op}+\tau_\ast(\eta)}\leq G_2(\eta)\leq \phi.
	\end{align*}
	In particular, uniformly in $\eta \in \Xi_K$,
	\begin{align}\label{ineq:fpe_property_7}
	G_1(\eta),G_2(\eta)\asymp 1. 
	\end{align}
	The derivatives $G_1', G_2'$ are 
	\begin{align*}
	G_1'(\eta)& = 2\tau_\ast(\eta)\tau_\ast'(\eta) \pnorm{(\Sigma+\tau_\ast(\eta) I)^{-1}\Sigma^{1/2}\mu_0}{}^2\\
	&\qquad -2 \tau_\ast^2(\eta)\pnorm{(\Sigma+\tau_\ast(\eta) I)^{-3/2}\Sigma^{1/2}\mu_0}{}^2\cdot \tau_\ast'(\eta),\\
	G_2'(\eta)& = 2\cdot n^{-1}\tr\big((\Sigma+\tau_\ast(\eta) I)^{-3}\Sigma^2\big)\cdot \tau_\ast'(\eta). 
	\end{align*}
	Using the apriori estimates on $\tau_\ast(\eta)$ and (\ref{ineq:fpe_property_5}), it now follows that 
	\begin{align}\label{ineq:fpe_property_8}
	\sup_{\eta \in \Xi_K} \big\{ \abs{G_1'(\eta)}\vee \abs{G_2'(\eta)}\big\}\lesssim 1.
	\end{align}
	Combining (\ref{ineq:fpe_property_6})-(\ref{ineq:fpe_property_8}) and using apriori estimates on $\gamma_\ast(\eta)$, we arrive at 
	\begin{align}\label{ineq:fpe_property_9}
	\sup_{\eta \in \Xi_K} \abs{\gamma_\ast'(\eta)} \lesssim 1.
	\end{align}
	The claim follows by collecting (\ref{ineq:fpe_property_5}) and (\ref{ineq:fpe_property_9}).
\end{proof}

\subsection{Sample version of (\ref{eqn:fpe})}

Let the sample version of (\ref{eqn:fpe}) be defined by
\begin{align}\label{eqn:fpe_sample}
\begin{cases}
\phi e_h^2 \gamma^2  = \sigma_\pm^2(L_w)+ \err_{(\Sigma,\mu_0)}(\gamma;\tau),\\
\big(\phi e_h^2-\frac{\eta}{\tau}\big)\cdot \gamma^2 =\dof_{(\Sigma,\mu_0)}(\gamma;\tau).
\end{cases}
\end{align}
Here recall that $e_h^2$ is defined in (\ref{def:eh_eg}), and $\sigma_\pm^2(L_w)$ is defined in (\ref{def:sigma_pm}). 
\begin{proposition}\label{prop:fpe_sample_est}
$1/K\leq \phi^{-1},\sigma_\xi^2\leq K$ and $\pnorm{\mu_0}{}\vee \pnorm{\Sigma}{\op}\vee \mathcal{H}_\Sigma \leq K$ for some $K>0$.  There exist some $C, C_0>1$ depending on $K$, such that with $\delta \in (0,1/C^{100})$, $1\leq M\leq \sqrt{n}/C$ and $L_w\leq C$, on the event $\mathscr{E}_1(\delta)\cap \mathscr{E}_{\Delta,\Xi}(M)$, where
\begin{align*}
\mathscr{E}_{\Delta,\Xi}(M)\equiv \Big\{\max_{\ell=1,2}\sup_{\tau\geq 0}\abs{\Delta_\ell(\tau)}\vee \max_{\ell =1,2}\sup_{\tau\geq 0} n^{-1/2}\abs{\Xi_\ell(\tau)-\E \Xi_\ell(\tau)}\leq M\Big\}
\end{align*}
with $\Delta_\ell,\Xi_\ell$ defined in Lemmas \ref{lem:conc_Delta} and \ref{lem:conc_Xi} ahead, the following hold.
\begin{enumerate}
	\item All solutions $(\gamma_{n,\eta,\pm},\tau_{n,\eta,\pm})$ to the system of equations in (\ref{eqn:fpe_sample}) satisfy
	\begin{align*}
	1/C_0\leq \tau_{n,\eta,\pm} \leq C_0,\quad 1/C_0\leq \gamma_{n,\eta,\pm}\leq C_0
	\end{align*}
	uniformly in $\eta \in \Xi_K$. 
	\item Moreover, 
	\begin{align*}
	\sup_{\eta \in \Xi_K}\big\{\abs{\tau_{n,\eta,\pm}-\tau_{\eta,\ast}}\vee  \abs{\gamma_{n,\eta,\pm}-\gamma_{\eta,\ast}}\big\} \leq C_0\cdot \big(M/\sqrt{n}+\delta\big).
	\end{align*}
\end{enumerate}
\end{proposition}

We need two concentration lemmas before the proof of Proposition \ref{prop:fpe_sample_est}.

\begin{lemma}\label{lem:conc_Delta}
Let $
\Delta_\ell(\tau)\equiv -\ell\cdot \tau \bigiprod{(\Sigma+\tau I)^{-\ell}\Sigma^{\ell-1/2}\mu_0}{ g}$ for $\ell=1,2$. Suppose that $\pnorm{\mu_0}{}\vee \pnorm{\Sigma}{\op}\vee \mathcal{H}_\Sigma \leq K$ for some $K>0$. Then there exists some constant $C=C(K)>1$ such that for $t\geq C\log (en)$, 
\begin{align*}
\Prob\Big(\max_{\ell =1,2}\sup_{\tau\geq 0} \abs{\Delta_\ell(\tau)}\geq C\sqrt{t}\Big)\leq e^{-t}. 
\end{align*}
\end{lemma}

\begin{lemma}\label{lem:conc_Xi}
Let $
\Xi_\ell(\tau)\equiv \pnorm{(\Sigma+\tau I)^{-\ell/2}\Sigma^{\ell/2}g}{}^2$ for $\ell=1,2$. 
Suppose that $\pnorm{\Sigma}{\op}\vee \mathcal{H}_\Sigma\leq K$ for some $K>0$. Then there exists some constant $C=C(K)>1$ such that for $t\geq C\log (en)$, 
\begin{align*}
\Prob\Big(\max_{\ell =1,2}\sup_{\tau\geq 0} \bigabs{\Xi_\ell(\tau)-\E \Xi_\ell(\tau)}\geq C(\sqrt{nt}+t)\Big)\leq e^{-t}. 
\end{align*}
\end{lemma}

The proofs of these lemmas are deferred to the next subsection.

\begin{proof}[Proof of Proposition \ref{prop:fpe_sample_est}]
	All the constants in $\lesssim$, $\gtrsim$, $\asymp$ and $\bigo$ below may possibly depend on $K$. We often suppress the dependence of $\sigma_\pm^2(L_w)$ on $L_w$ for simplicity.

	\noindent (1). We shall write $(\gamma_{n,\eta,\pm},\tau_{n,\eta,\pm})$ as $(\gamma_{n},\tau_{n})$ and $(\gamma_{\eta,\ast},\tau_{\eta,\ast})=(\gamma_\ast,\tau_\ast)$ for notational simplicity. Using (\ref{ineq:est_dof_1}), any solution $(\gamma_n,\tau_n)$ to the equations in (\ref{eqn:fpe_sample}) satisfies
	\begin{align}\label{ineq:fpe_sample_est_1}
	\begin{cases}
	\phi e_h^2-\frac{\eta}{\tau_n} +\frac{\Delta_1(\tau_n)}{\sqrt{n}\gamma_n} =  \frac{1}{n} \tr\big((\Sigma+\tau_n I)^{-1}\Sigma\big)+\frac{1}{n}(\mathrm{id}-\E)\Xi_1(\tau_n),\\
	\phi e_h^2+\frac{\Delta_2(\tau_n)}{\sqrt{n}\gamma_n} = \frac{1}{\gamma_n^2}\big(\sigma_\pm^2 +\tau_n^2 \pnorm{(\Sigma+\tau_n I)^{-1}\Sigma^{1/2}\mu_0}{}^2\big)\\
	\qquad \qquad\qquad\qquad +\frac{1}{n}\tr\big((\Sigma+\tau_n I)^{-2}\Sigma^2\big)+\frac{1}{n}(\mathrm{id}-\E)\Xi_2(\tau_n).
	\end{cases}
	\end{align}
	On the event $\mathscr{E}_1(\delta)\cap \mathscr{E}_{\Delta,\Xi}(M)$ with $\delta \in (0,1/C^{100})$, $1\leq M\leq \sqrt{n}/C$ and $L_w\leq C$, using Lemma \ref{lem:sigma_pm_E1}, the second equation in (\ref{ineq:fpe_sample_est_1}) becomes
	\begin{align*}
	&\phi +\bigo\big(M\big(1\vee \gamma_n^{-1}\big)/\sqrt{n}+\delta\big) \\
	&= \frac{1}{\gamma_n^2}\Big(\sigma_\pm^2 +\tau_n^2 \pnorm{(\Sigma+\tau_n I)^{-1}\Sigma^{1/2}\mu_0}{}^2\Big)+\frac{1}{n}\tr\big((\Sigma+\tau_n I)^{-2}\Sigma^2\big)\gtrsim \frac{1}{\gamma_n^2}. 
	\end{align*}
	Rearranging terms we obtain the inequality
	\begin{align*}
	\frac{1}{\gamma_n^2}\lesssim 1+ \frac{M}{\sqrt{n}}+ \frac{M}{\sqrt{n}\gamma_n}\,\Rightarrow\, \gamma_n \gtrsim \frac{1}{1+M/\sqrt{n}}\gtrsim 1. 
	\end{align*}
	So with $\epsilon_n\equiv \epsilon_n(M,\delta)\equiv M/\sqrt{n}+\delta$, the equations in (\ref{ineq:fpe_sample_est_1}) reduce to
	\begin{align}\label{ineq:fpe_sample_est_2}
	\begin{cases}
	\phi-\frac{\eta}{\tau_n} +\bigo(\epsilon_n) =  \frac{1}{n} \tr\big((\Sigma+\tau_n I)^{-1}\Sigma\big),\\
	\phi+\bigo(\epsilon_n) = \frac{1}{\gamma_n^2}\big(\sigma_\xi^2 +\tau_n^2 \pnorm{(\Sigma+\tau_n I)^{-1}\Sigma^{1/2}\mu_0}{}^2\big) +\frac{1}{n}\tr\big((\Sigma+\tau_n I)^{-2}\Sigma^2\big).
	\end{cases}
	\end{align}
	The above equations match (\ref{eqn:fpe}) up to the small perturbation $\bigo(\epsilon_n)=\bigo(\epsilon_n(M,\delta))$ that can be assimilated into the leading term $\phi$ with small enough $c_0>0$ such that $M\leq c_0\sqrt{n}$. From here the existence (but not uniqueness) and apriori bounds for $\gamma_n,\tau_n$ can be established similarly to the proof of Proposition \ref{prop:fpe_est}. 
	
	\noindent (2). Now we shall prove the claimed error bounds. By using (\ref{ineq:fpe_property_1}) and the first equation of (\ref{ineq:fpe_sample_est_2}), we have
	\begin{align*}
	\frac{1}{n} \tr\big((\Sigma+\tau_n I)^{-1}\Sigma\big)+\frac{\eta}{\tau_n} = \frac{1}{n} \tr\big((\Sigma+\tau_\ast I)^{-1}\Sigma\big)+\frac{\eta}{\tau_\ast}+\bigo(\epsilon_n).
	\end{align*}
	Let $
	\mathsf{f}(\tau)\equiv \frac{1}{n} \tr\big((\Sigma+\tau I)^{-1}\Sigma\big)+\frac{\eta}{\tau}$. 
	Then it is easy to calculate $
	\mathsf{f}'(\tau) = -\frac{1}{n} \tr\big((\Sigma+\tau I)^{-2}\Sigma\big)-\frac{\eta}{\tau^2}\leq 0$, 
	and for any $C_0>1$,
	\begin{align*}
	\inf_{1/C_0\leq \tau\leq C_0}\abs{\mathsf{f}'(\tau)}&\geq \inf_{1/C_0\leq \tau\leq C_0} \frac{1}{n}\tr\big((\Sigma+\tau I)^{-2}\Sigma\big)
	\geq (\pnorm{\Sigma}{\op}+C_0)^{-2}\mathcal{H}_\Sigma^{-1}. 
	\end{align*}
	Now using the apriori estimates on $\tau_\ast,\tau_n$, we may conclude
	\begin{align}\label{ineq:fpe_sample_est_5}
	\sup_{\eta \in \Xi_K}\abs{\tau_n-\tau_\ast}\lesssim \epsilon_n. 
	\end{align}
	On the other hand, using (\ref{ineq:fpe_property_3}) and the second equation of (\ref{ineq:fpe_sample_est_2}), we have
	\begin{align}\label{ineq:fpe_sample_est_6}
	&\frac{1}{\gamma_n^2}\Big(\sigma_\xi^2 +\tau_n^2 \pnorm{(\Sigma+\tau_n I)^{-1}\Sigma^{1/2}\mu_0}{}^2\Big) +\frac{1}{n}\tr\big((\Sigma+\tau_n I)^{-2}\Sigma^2\big)\nonumber\\
	& = \frac{1}{\gamma_\ast^2}\Big(\sigma_\xi^2 +\tau_\ast^2 \pnorm{(\Sigma+\tau_\ast I)^{-1}\Sigma^{1/2}\mu_0}{}^2\Big) +\frac{1}{n}\tr\big((\Sigma+\tau_\ast I)^{-2}\Sigma^2\big)+\bigo(\epsilon_n).
	\end{align}
	Using the error bound in (\ref{ineq:fpe_sample_est_5}) and apriori estimates for $\tau_n,\tau_\ast$, and the fact that $\mathcal{H}_\Sigma\lesssim 1$, by an easy derivative estimate we have
	\begin{itemize}
		\item $\bigabs{\frac{1}{n}\tr\big((\Sigma+\tau_n I)^{-2}\Sigma^2\big)- \frac{1}{n}\tr\big((\Sigma+\tau_\ast I)^{-2}\Sigma^2\big)}\lesssim \epsilon_n$, and
		\item $\bigabs{\tau_n^2 \pnorm{(\Sigma+\tau_n I)^{-1}\Sigma^{1/2}\mu_0}{}^2-\tau_\ast^2 \pnorm{(\Sigma+\tau_\ast I)^{-1}\Sigma^{1/2}\mu_0}{}^2 }\lesssim \epsilon_n$.
	\end{itemize}
	Now plugging these estimates into (\ref{ineq:fpe_sample_est_6}), with $\mathscr{C}_0\equiv \sigma_\xi^2 +\tau_\ast^2 \pnorm{(\Sigma+\tau_\ast I)^{-1}\Sigma^{1/2}\mu_0}{}^2$ satisfying $\mathscr{C}_0\asymp 1$, we arrive at
	\begin{align*}
	\frac{\mathscr{C}_0+\bigo(\epsilon_n)}{\gamma_n^2} = \frac{\mathscr{C}_0}{\gamma_\ast^2}+\bigo(\epsilon_n).
	\end{align*}
	Using apriori estimates on $\gamma_n,\gamma_\ast$, we may then invert the above estimate into
	\begin{align}\label{ineq:fpe_sample_est_9}
	\sup_{\eta \in \Xi_K}\abs{\gamma_n-\gamma_\ast}\lesssim \epsilon_n. 
	\end{align}
	The claimed error bounds follow by combining (\ref{ineq:fpe_sample_est_5}) and (\ref{ineq:fpe_sample_est_9}).
\end{proof}

\subsection{Proofs of Lemmas \ref{lem:conc_Delta} and \ref{lem:conc_Xi}}

\begin{proof}[Proof of Lemma \ref{lem:conc_Delta}]
	We only handle the case $\ell=1$. The case $\ell=2$ is similar. Note that the assumption on $\mu_0$ invariant over orthogonal transforms, so for notational simplicity we assume without loss of generality that $\Sigma$ is diagonal. As $
	\sup_{\tau \geq Kn}\abs{\Delta_1(\tau)}\leq \bigabs{\sum_{j=1}^n \lambda_j^{1/2}\mu_{0,j}g_j}+ C e_g\cdot n^{-1/2}$, a standard concentration for the first term shows for $t\geq 1$, with probability $1-e^{-t}$,
	\begin{align}\label{ineq:conc_delta_1}
	\sup_{\tau \geq Kn}\abs{\Delta_1(\tau)}&\leq C_0 \sqrt{t}.
	\end{align}
	On the other hand, for $\epsilon>0$ to be chosen later, by taking an $\epsilon$-net $\mathcal{S}_\epsilon$ of $[0,Kn]$, a union bound shows that with probability at least $1-(Kn/\epsilon+1)e^{-t}$, 
	\begin{align*}
	\sup_{\tau \in [0,K n]}\abs{\Delta_1(\tau)}&\leq \max_{\tau \in \mathcal{S}_\epsilon}\abs{\Delta_1(\tau)}+ \sup_{\tau,\tau' \in [0,Kn]: \abs{\tau-\tau'}\leq \epsilon} \abs{\Delta_1(\tau)-\Delta_1(\tau')}\\
	&\leq C_1\cdot\Big(\sqrt{t}+ \sqrt{n}(\sqrt{\log n}+\sqrt{t})\epsilon\Big). 
	\end{align*}
	Here in the last inequality we used the simple estimate $
	\sup_{\tau \in [0,K n]}\abs{\partial_\tau\Delta_1(\tau)} \leq C \sqrt{n} \pnorm{\mu_0}{}\pnorm{g}{\infty}$. 
	Finally by choosing $\epsilon\equiv \sqrt{t}/\big\{\sqrt{n}(\sqrt{\log n}+\sqrt{t})\big\}$, we conclude that for $t\geq C_2\log (en)$, with probability $1-e^{-t}$, 
	\begin{align}\label{ineq:conc_delta_2}
	\sup_{\tau \in [0,Kn]}\abs{\Delta_1(\tau)}&\leq C_2 \sqrt{t}.
	\end{align}
	The claim follows by combining (\ref{ineq:conc_delta_1}) and (\ref{ineq:conc_delta_2}).
\end{proof}

\begin{proof}[Proof of Lemma \ref{lem:conc_Xi}]
	We focus on the case $\ell=1$ and will follow a similar idea used in the proof of Lemma \ref{lem:conc_Delta} above. Similarly we assume $\Sigma$ is diagonal without loss of generality. All the constants in $\lesssim,\gtrsim,\asymp$ below may depend on $K$.
	
	First note by a standard concentration, for any $t\geq 1$, with probability at least $1-e^{-t}$, $
	\sup_{\tau >Kn} \abs{\Xi_1(\tau)}\lesssim e_g^2 \lesssim 1+t/n$. 
	Similarly we have $
	\sup_{\tau >Kn} \E\abs{\Xi_1(\tau)}\lesssim 1$. 
	This means for any $t\geq 1$, with probability at least $1-e^{-t}$,
	\begin{align}
	\sup_{\tau >Kn} \Big(\abs{\Xi_1(\tau)}\vee \E \abs{\Xi_1(\tau)}\Big)\lesssim 1+t/n. 
	\end{align}
	Next we handle the suprema over $[0,Kn]$ by discretization over an $\epsilon$-net $\mathcal{S}_\epsilon$. To this end, we shall establish a pointwise concentration. Note that $
	\pnorm{\nabla \Xi_1(\tau)}{}^2 = 4 \pnorm{(\Sigma+\tau I)^{-1}\Sigma  g}{}^2\leq 4 \Xi_1(\tau)$. 
	An application of Proposition \ref{prop:conc_H_generic} then yields that, for each $\tau \geq 0$ and $t\geq 1$, with probability at least $1-e^{-t}$,
	\begin{align*}
	\abs{\Xi_1(\tau)-\E \Xi_1(\tau)}\leq C\big( \E^{1/2} \Xi_1(\tau)\cdot \sqrt{t} + t\big)\lesssim (\sqrt{n t}+t).
	\end{align*}
	On the other hand, as $
	\sup_{\tau \in [0,K n]}\abs{\partial_\tau \Xi_1(\tau)} \lesssim n \pnorm{g}{\infty}^2$ and  $
	\sup_{\tau \in [0,K n]}\abs{\partial_\tau\E\Xi_1(\tau)} \lesssim n\log n$, we deduce that with probability at least $1-(Kn/\epsilon+1)e^{-t}$, 
	\begin{align*}
	\sup_{\tau \in [0,K n]}\abs{\Xi_1(\tau)-\E\Xi_1(\tau)}&\leq \max_{\tau \in \mathcal{S}_\epsilon}\abs{\Xi_1(\tau)-\E \Xi_1(\tau)}+ \sup_{\tau,\tau' \in [0,Kn]: \abs{\tau-\tau'}\leq \epsilon} \abs{\Xi_1(\tau)-\Xi_1(\tau')}\\
	&\qquad\qquad + \sup_{\tau,\tau' \in [0,Kn]: \abs{\tau-\tau'}\leq \epsilon} \abs{\E\Xi_1(\tau)-\E \Xi_1(\tau')}\\
	&\lesssim \sqrt{nt}+ t+n(\log n+t)\epsilon.
	\end{align*}
	From here the claim follows by the same arguments used in the proof of Lemma \ref{lem:conc_Delta} above. 
\end{proof}

\section{Gaussian designs: Proof of Theorem \ref{thm:min_norm_dist}}\label{section:proof_gaussian_design}

We assume without loss of generality that $\Sigma=\mathrm{diag}(\lambda_1,\ldots,\lambda_n)$, so $\mathsf{V}=I$ unless otherwise specified. Recall $\mathcal{H}_\Sigma=\tr(\Sigma^{-1})/n$. 

\subsection{Localization of the primal problem}

\begin{proposition}\label{prop:H_local}
Suppose $1/K\leq \phi^{-1}-\bm{1}_{\eta=0}, \sigma_\xi^2\leq K$, and $\pnorm{\mu_0}{}\vee \pnorm{\Sigma}{\op}\leq K$ for some $K>0$. Fix $M>1,\delta \in (0,1/2)$ and $\eta\geq 0$. On the event $\mathscr{E}_0(M)\cap \mathscr{E}_1(\delta)$, there exists some $C=C(K)>0$ such that for any deterministic choice of $(L_w,L_v)$ with
\begin{align*}
L_w\wedge L_v \geq C\big\{1+ \big(\pnorm{\Sigma^{-1}}{\op}M\bm{1}_{\phi^{-1}\geq 1+1/K}^{-1}\wedge \eta^{-1}\big)\cdot M^2\big\},
\end{align*}
we have $
\min_{w \in B_n(L_w)} H_\eta(w;L_v)= \min_{w \in \R^n} H_\eta(w)$. 
\end{proposition}

\begin{proof}
Using the first-order optimality condition for the minimax problem
\begin{align}\label{ineq:localization_H_1}
\min_{w \in \R^n} H_\eta(w) = \min_{w \in \R^n} \max_{v \in \R^m}\bigg\{ \frac{1}{\sqrt{n}}\iprod{v}{Gw-\xi}+F(w)-\frac{\eta\pnorm{v}{}^2}{2}\bigg\},
\end{align}
any saddle point $(w_\ast,v_\ast)$ of (\ref{ineq:localization_H_1}) must satisfy $\nabla F(w_\ast) =- \frac{1}{\sqrt{n}} G^\top v_\ast$ and $\frac{1}{\sqrt{n}}(Gw_\ast-\xi) = \eta v_\ast$, or equivalently,
\begin{align*}
\begin{cases}
w_\ast = -\Sigma^{1/2}\mu_0+\frac{1}{n}\Sigma G ^\top \big(\phi\check{\Sigma}+\eta I\big)^{-1}(G\Sigma^{1/2}\mu_0+\xi),\\
v_\ast = -\frac{1}{\sqrt{n}} \big(\phi\check{\Sigma}+\eta I\big)^{-1}\big(G\Sigma^{1/2}\mu_0+\xi\big).
\end{cases}
\end{align*}
Here recall $\check{\Sigma}= m^{-1}G\Sigma G^\top$. On the event $\mathscr{E}_0(M)$, 
\begin{align*}
\pnorm{(\phi\check{\Sigma}+\eta I)^{-1}}{\op}\lesssim_K \pnorm{\Sigma^{-1}}{\op}M\bm{1}_{\phi^{-1}\geq 1+1/K}^{-1}\wedge \eta^{-1}.
\end{align*}
So on $\mathscr{E}_0(M)\cap \mathscr{E}_1(\delta)$, 
\begin{align*}
\pnorm{w_\ast}{}\vee \pnorm{v_\ast}{}\lesssim_{K} 1+ (\pnorm{\Sigma^{-1}}{\op}M\bm{1}_{\phi^{-1}\geq 1+1/K}^{-1}\wedge \eta^{-1})M^2.
\end{align*}
This means that on the event $\mathscr{E}_0(M)\cap \mathscr{E}_1(\delta)$, for any $L_w,L_v$ chosen as in the statement of the lemma, 
\begin{align*}
\min_{w \in \R^n} H_\eta(w) = \min_{w \in B_n(L_w)} \max_{v \in B_m(L_v)}\bigg\{ \frac{1}{\sqrt{n}}\iprod{v}{Gw-\xi}+F(w)-\frac{\eta\pnorm{v}{}^2}{2}\bigg\}. 
\end{align*}
The proof is complete by recalling the definition of $H_\eta(\cdot;L_v)$. 
\end{proof}

\subsection{Characterization of the Gordon cost optimum}

\begin{theorem}\label{thm:char_gordon_cost_opt}
	Suppose the following hold for some $K>0$.
	\begin{itemize}
		\item $1/K\leq \phi^{-1}\leq K$, $\pnorm{\mu_0}{}\vee \pnorm{\Sigma}{\op}\vee \mathcal{H}_\Sigma \leq K$.
		\item Assumption \ref{assump:noise} with $\sigma_\xi^2 \in [1/K,K]$.
	\end{itemize}
	There exist some $C,C'>1$ depending on $K$ such that for any deterministic choice of $L_w,L_v \in [C,C^2]$, it holds 
	for any $C' \log (en)\leq t \leq n/C'$, $\eta \in \Xi_K$ and $\xi \in \mathscr{E}_{1,\xi}(\sqrt{t/n})$,
	\begin{align*}
	&\Prob^\xi\Big(\bigabs{\min_{w \in B_n(L_w)}L_\eta(w;L_v) -  \max_{\beta>0}\min_{\gamma >0} \overline{\mathsf{D}}_\eta(\beta,\gamma)}  \geq \sqrt{t/n}\Big)\leq C e^{-t/C}.
	\end{align*}
\end{theorem}

In the next subsection we will show that for large $L_v>0$, the map $w\mapsto L_\eta(w;L_v)$ attains its global minimum in an $\ell_2$ ball of constant order radius (under $\mathcal{H}_\Sigma\lesssim 1$) with high probability. This means that although the initial localization radius for the primal optimization may be highly suboptimal (which involves $\pnorm{\Sigma^{-1}}{\op}$), the Gordon objective can be further localized  into an $\ell_2$ ball with constant order radius. 

To prove Theorem \ref{thm:char_gordon_cost_opt}, we shall first relate $\min_{w \in B_n(L_w)}L_\eta(w;L_v)$ to $\max_{\beta>0}\min_{\gamma>0} \mathsf{D}_{\eta,\pm}(\beta,\gamma)$ and its localized versions.

\begin{proposition}\label{prop:L_local}
Suppose $1/K\leq \phi^{-1},\sigma_\xi^2\leq K$, and $\pnorm{\mu_0}{}\vee \pnorm{\Sigma}{\op}\vee \mathcal{H}_\Sigma \leq K$ for some $K>0$. There exists constant $C=C(K)>1$ such that for any deterministic choice of $L_w,L_v \in [C,C^2]$, on the event $\mathscr{E}_1(\delta)\cap\mathscr{E}_{\Delta,\Xi}(M)$ (defined in Proposition \ref{prop:fpe_sample_est}) with $\delta\in (0,1/C^{100})$ and $M\leq \sqrt{n}/C$, we have for any $\eta \in \Xi_K$, 
\begin{align*}
\max_{\beta>0}\min_{\gamma>0} \mathsf{D}_{\eta,-}(\beta,\gamma)\leq \min_{w \in B_n(L_w)}L_\eta(w;L_v)\leq \max_{\beta>0}\min_{\gamma>0} \mathsf{D}_{\eta,+}(\beta,\gamma),
\end{align*}
and the following localization holds:
\begin{align*}
\max_{\beta>0}\min_{\gamma>0} \mathsf{D}_{\eta,\pm}(\beta,\gamma) = \max_{1/C \leq \beta \leq C } \min_{1/C\leq \gamma \leq C} \mathsf{D}_{\eta,\pm}(\beta,\gamma). 
\end{align*}
\end{proposition}
\begin{proof}
We write $g_n\equiv g/\sqrt{n}$ in the proof.
	
\noindent (\textbf{Step 1}). Fix any $L_w,L_v>0$. We may compute 
\begin{align}\label{ineq:L_local_1}
&\min_{w \in B_n(L_w)}L_\eta(w;L_v)  \\
&= \min_{w \in B_n(L_w)}\max_{\beta\in [0, L_v]}\bigg\{\frac{\beta}{\sqrt{n}} \Big(\bigpnorm{\pnorm{w}{}h-\xi}{}-\iprod{g}{w}\Big)+F(w)-\frac{\eta\beta^2}{2}\bigg\}\nonumber\\
& = \max_{\beta\in [0,L_v]}\min_{\gamma>0}\bigg\{\frac{\beta\gamma \pnorm{h}{}^2 }{2 n}-\frac{\eta\beta^2}{2}+\min_{w \in B_n(L_w)}\bigg(\frac{\beta }{2\gamma}\frac{\pnorm{\pnorm{w}{}h-\xi}{}^2}{\pnorm{h}{}^2} - \iprod{w}{\beta g_n}+F(w)  \bigg)\bigg\}\nonumber.
\end{align}
Here in the last line we used Sion's min-max theorem to flip the order of minimum and maximum in $\min_{w \in B_n(L_w)}\max_{\beta\in [0, L_v]}$. The minimum over $\gamma$ is achieved exactly at $\frac{\pnorm{ \pnorm{w}{}h-\xi}{}/\pnorm{h}{}}{\pnorm{h}{}/\sqrt{n}}$, so when $\sigma_-^2\neq 0$, using the simple inequality 
\begin{align}\label{ineq:L_local_2}
\pnorm{w}{}^2+\sigma_{-}^2\leq 
\pnorm{\pnorm{w}{}h-\xi}{}^2/\pnorm{h}{}^2\leq \pnorm{w}{}^2+\sigma_{+}^2,
\end{align}
on the event $\mathscr{E}_1(\delta)$, we may further bound (\ref{ineq:L_local_1}) as follows:
\begin{align}\label{ineq:L_local_3}
&\pm \min_{w \in B_n(L_w)}L_\eta(w;L_v)\leq \pm \max_{\beta\in [0,L_v]}\min_{\gamma >0}\nonumber\\
&\qquad \bigg\{\frac{\beta\gamma \pnorm{h}{}^2 }{2 n}-\frac{\eta\beta^2}{2}+\min_{w \in B_n(L_w)}\bigg(\frac{\beta }{2\gamma}\big(\pnorm{w}{}^2+\sigma_\pm^2(L_w)\big)- \iprod{w}{\beta g_n}+F(w)  \bigg)\bigg\}.
\end{align}
We note that $\sigma_\pm^2$ depends on $L_w$, but this notational dependence will be dropped from now on for convenience.

\noindent (\textbf{Step 2}). Consider the minimax optimization problem in (\ref{ineq:L_local_3}):
\begin{align}\label{ineq:L_local_6}
&\max_{\beta>0}\min_{\gamma>0, w \in \R^n} \bigg\{\frac{\beta\gamma \pnorm{h}{}^2 }{2 n}-\frac{\eta\beta^2}{2} +\bigg(\frac{\beta }{2\gamma}\big(\pnorm{w}{}^2+\sigma_\pm^2\big)- \iprod{w}{\beta g_n}+F(w)  \bigg)\bigg\}\nonumber\\
& = \max_{\beta>0}\min_{\gamma >0} \bigg\{\frac{\beta}{2}\bigg(\gamma \big(\phi e_h^2-e_g^2\big)+\frac{\sigma_\pm^2}{\gamma}\bigg)-\frac{\eta\beta^2}{2}+\mathsf{e}_F(\gamma g_n;{\gamma}/{\beta})  \bigg\}.
\end{align}
Any saddle point $(\beta_{n,\eta,\pm},\gamma_{n,\eta,\pm},w_{n,\eta,\pm})=(\beta_{n,\pm},\gamma_{n,\pm},w_{n,\pm})$ of the above program must satisfy the first-order optimality condition
\begin{align}\label{ineq:L_local_7}
\begin{cases}
0=\frac{1}{2}\big(\gamma_{n,\pm} (\phi e_h^2-e_g^2)+\frac{\sigma_\pm^2}{\gamma_{n,\pm}}\big) -\eta \beta_{n,\pm}+ \partial_\beta \mathsf{e}_F\big(\gamma_{n,\pm} g_n;{\gamma_{n,\pm}}/{\beta_{n,\pm}}\big),\\
0=\frac{\beta_{n,\pm}}{2}\big( (\phi e_h^2-e_g^2)-\frac{\sigma_\pm^2}{\gamma_{n,\pm}^2}\big)+\partial_\gamma \mathsf{e}_F\big(\gamma_{n,\pm} g_n;{\gamma_{n,\pm}}/{\beta_{n,\pm}}\big),\\
w_{n,\pm} = \prox_F \big( \gamma_{n,\pm}g_n;{\gamma_{n,\pm}}/{\beta_{n,\pm}}\big).
\end{cases}
\end{align}
Using the derivative formula in Lemma \ref{lem:derivative_eF} and the form of $\prox_F$ in Lemma \ref{lem:prox_env_F}, we may compute
\begin{align}\label{ineq:L_local_8}
\begin{cases}
\partial_\beta \mathsf{e}_F(\gamma g_n; {\gamma}/{\beta}) = \frac{1}{2\gamma}\Big(\err_{(\Sigma,\mu_0)}(\gamma;\gamma/\beta)-2\dof_{(\Sigma,\mu_0)}(\gamma;\gamma/\beta)+ \gamma^2 e_g^2 \Big),\\
\partial_\gamma  \mathsf{e}_F(\gamma g_n; {\gamma}/{\beta}) 
= \frac{\beta}{2\gamma^2}\Big(\gamma^2 e_g^2 - \err_{(\Sigma,\mu_0)}(\gamma;\gamma/\beta)\Big).
\end{cases}
\end{align}
Plugging (\ref{ineq:L_local_8}) into (\ref{ineq:L_local_7}), the first-order optimality condition for $(\beta_{n,\pm},\gamma_{n,\pm})$ in the minimax program (\ref{ineq:L_local_6}) is given by
\begin{align*}
\begin{cases}
\big(\phi e_h^2-e_g^2\big)\gamma_{n,\pm}^2+\sigma_\pm^2 =2\eta\cdot \gamma_{n,\pm}\beta_{n,\pm} -\err_{(\Sigma,\mu_0)}(\gamma_{n,\pm};\gamma_{n,\pm}/\beta_{n,\pm})\\
\qquad\qquad \qquad\qquad \qquad+2 \dof_{(\Sigma,\mu_0)}(\gamma_{n,\pm};\gamma_{n,\pm}/\beta_{n,\pm})- e_g^2\gamma_{n,\pm}^2,\\
\big(\phi e_h^2-e_g^2\big)\gamma_{n,\pm}^2-\sigma_\pm^2 = -e_g^2\gamma_{n,\pm}^2+ \err_{(\Sigma,\mu_0)}(\gamma_{n,\pm};\gamma_{n,\pm}/\beta_{n,\pm}).
\end{cases}
\end{align*}
Equivalently, 
\begin{align}\label{ineq:L_local_10}
\begin{cases}
\phi e_h^2\gamma_{n,\pm}^2  = \sigma_\pm^2+  \err_{(\Sigma,\mu_0)}(\gamma_{n,\pm};\gamma_{n,\pm}/\beta_{n,\pm}),\\
\big(\phi e_h^2-\frac{\eta}{\gamma_{n,\pm}/\beta_{n,\pm}}\big)\gamma_{n,\pm}^2  =   \dof_{(\Sigma,\mu_0)}(\gamma_{n,\pm};\gamma_{n,\pm}/\beta_{n,\pm}).
\end{cases}
\end{align}
Using the apriori estimates in Proposition \ref{prop:fpe_sample_est}, on the event $\mathscr{E}_1(\delta)\cap\mathscr{E}_{\Delta,\Xi}(M)$ we have $
 {\gamma_{n,\pm}}/{\beta_{n,\pm}} \asymp_{K} 1$ and $\gamma_{n,\pm}^2\asymp_{K} 1$. This implies on the same event,
\begin{align}\label{ineq:L_local_11}
\gamma_{n,\pm}\asymp_{K} 1,\quad  \beta_{n,\pm} \asymp_{K} 1.  
\end{align}
Using the last equation of (\ref{ineq:L_local_7}), we have
\begin{align}\label{ineq:L_local_12}
\pnorm{w_{n,\pm}}{}
& = \biggpnorm{ \Sigma^{1/2}\Big(\Sigma+\frac{\gamma_{n,\pm}}{\beta_{n,\pm}} I\Big)^{-1}\Big(-\frac{\gamma_{n,\pm}}{\beta_{n,\pm}} \mu_0+\gamma_{n,\pm}\Sigma^{1/2}g_n\Big)}{}\lesssim_K 1.
\end{align}
In view of (\ref{ineq:L_local_11})-(\ref{ineq:L_local_12}), by choosing $L_w, L_v \in [C,C^2]$ for large enough $C>0$, the constraints in the optimization in (\ref{ineq:L_local_3}) can be dropped for free. 
\end{proof}

Next we replace the random function $\mathsf{D}_{\eta,\pm}$ in the above proposition by its deterministic counterpart $\overline{\mathsf{D}}_\eta$ in their localized versions. 

\begin{proposition}\label{prop:D_pm_barD}
Suppose $1/K\leq \phi^{-1},\sigma_\xi^2\leq K$, and $\pnorm{\mu_0}{}\vee \pnorm{\Sigma}{\op}\vee \mathcal{H}_\Sigma \leq K$ for some $K>0$.  There exist some $C,C'>1$ depending on $K$ such that for $L_w \in [C,C^2]$, $\delta \in (0,1/C^{100})$, $\xi \in \mathscr{E}_{1,\xi}(\delta)$ and $t\geq C'\log (en)$, 
\begin{align*}
&\Prob^\xi\Big[\sup_{\eta \in \Xi_K}\big|\max_{1/C \leq \beta \leq C} \min_{1/C\leq \gamma \leq C} \mathsf{D}_{\eta,\pm}(\beta,\gamma) -  \max_{1/C \leq \beta \leq C} \min_{1/C\leq \gamma \leq C} \overline{\mathsf{D}}_\eta(\beta,\gamma)\big|\\
&\qquad   \geq C \big(\sqrt{t/n}+t/n+\delta\big)\Big]\leq C e^{-t/C}+\Prob^\xi\big(\mathscr{E}_{1,0}(\delta)^c\big). 
\end{align*}
\end{proposition}
\begin{proof}
In the proof, we write $g_n\equiv g/\sqrt{n}$. All the constants in $\lesssim,\gtrsim,\asymp$ and $\bigo$ below may  depend on $K$. 

\noindent (\textbf{Step 1}). We first prove the following: On the event $\mathscr{E}_1(\delta)$, 
for any $C_0>1$, 
\begin{align}\label{ineq:uniform_conc_1}
\sup_{\gamma,\tau \in [1/C_0,C_0]^2} \abs{\partial_\# \mathsf{e}_F\big(\gamma g_n;\tau \big)}\vee \abs{\partial_\# \E \mathsf{e}_F\big(\gamma g_n;\tau \big)}&\lesssim 1.
\end{align}	
To this end, with $\hat{\mu}_{(\Sigma,\mu_0)}^{\seq},y_{(\Sigma,\mu_0)}^{\seq}$ written as $\hat{\mu},y$, and using $
\partial_\gamma \hat{\mu}= (\Sigma+\tau I)^{-1} \Sigma^{1/2} g_n$, $\partial_\gamma y= g_n$, $
 \partial_\tau \hat{\mu}= -\big(\Sigma+\tau I\big)^{-2}\Sigma^{1/2} \big(\Sigma^{1/2}\mu_0+\gamma g_n\big)$, 
\begin{align*}
\partial_\gamma\mathsf{e}_F(\gamma g_n;\tau)&= \tau^{-1}\bigiprod{\Sigma^{1/2}\hat{\mu}-y}{\Sigma^{1/2} \partial_\gamma \hat{\mu}-\partial_\gamma y}+\bigiprod{\hat{\mu}}{ \partial_\gamma \hat{\mu} },\\
\partial_\tau \mathsf{e}_F(\gamma g_n;\tau)& =-\frac{1}{2\tau^2} \pnorm{\Sigma^{1/2}\hat{\mu}-y}{}^2+\frac{1}{\tau}\bigiprod{\Sigma^{1/2}\hat{\mu}-y}{\Sigma^{1/2} \partial_\tau \hat{\mu} }+\iprod{\hat{\mu}}{\partial_\tau \hat{\mu} },
\end{align*} 
on the event $\mathscr{E}_1(\delta)$, we may estimate $
\abs{\partial_\gamma \mathsf{e}_F(\gamma g_n;\tau)}\vee  \abs{\partial_\tau \mathsf{e}_F(\gamma g_n;\tau)}\lesssim 1$. A similar estimate applies to the expectation versions, proving (\ref{ineq:uniform_conc_1}).

\noindent (\textbf{Step 2}). Next we show that for any $C_0>1$, there exists $C_1>0$ such that for $t \geq C_1 \log (en)$, 
\begin{align}\label{ineq:uniform_conc_2}
&\Prob\Big(\sup_{(\gamma,\tau) \in [C_0^{-1},C_0]^2} \bigabs{ (\mathrm{id}-\E)\mathsf{e}_F(\gamma g_n;\tau)} \geq C_1\big(\sqrt{t/n}+t/n\big), \mathscr{E}_1(\delta) \Big)\leq C_1e^{-t/C_1}.
\end{align}
To prove the claim, we fix $\epsilon>0$ to be chosen later, and take an $\epsilon$-net $\mathcal{S}(\epsilon)$ for $[1/C_0,C_0]$. Then $\abs{\mathcal{S}(\epsilon)}\leq C_0/\epsilon+1$. So on the event $\mathscr{E}_1(\delta)$, using the estimate in (\ref{ineq:uniform_conc_1}) and a union bound via the pointwise concentration inequality in Proposition \ref{prop:conc_e_F}, for $t\geq 1$, with probability at least $1-C\epsilon^{-2} e^{-t/C}$,
\begin{align*}
\sup_{(\gamma,\tau)\in  [C_0^{-1},C_0]^2 } \bigabs{ (\mathrm{id}-\E)\mathsf{e}_F(\gamma g_n;\tau)}&\lesssim \sup_{\gamma,\tau \in \mathcal{S}(\epsilon)} \bigabs{(\mathrm{id}-\E) \mathsf{e}_F(\gamma g_n;\tau)}+ C\epsilon \lesssim \sqrt{\frac{t}{n}}+\frac{t}{n}+\epsilon.
\end{align*}
Here in the last inequality we used Lemma \ref{lem:prox_env_F} to estimate $\sup_{(\gamma,\tau) }v^2(\gamma,\tau) \vee \sup_{(\gamma,\tau) } v^2(\gamma,\tau)  \E  \mathsf{e}_F(\gamma g_n;\tau)\lesssim 1$, where $v^2(\gamma,\tau)$ is defined in Proposition \ref{prop:conc_e_F}. The claim (\ref{ineq:uniform_conc_2}) follows by choosing $\epsilon\equiv\sqrt{t/n}+t/n$ and some calculations.

\noindent (\textbf{Step 3}). By (\ref{ineq:uniform_conc_2}), for $t\geq C \log (en)$, on the event $\mathscr{E}_1(\delta)$, it holds with probability at least $1-C_2 e^{-t/C_2}$ that
\begin{align*}
&\max_{1/C \leq \beta \leq C} \min_{1/C\leq \gamma \leq C} \mathsf{D}_{\eta,\pm}(\beta,\gamma)\\
& = \max_{1/C \leq \beta \leq C} \min_{1/C\leq \gamma \leq C}\bigg\{\frac{\beta}{2}\bigg(\gamma \big(\phi e_h^2-e_g^2\big)+\frac{\sigma_\pm^2}{\gamma}\bigg)-\frac{\eta \beta^2}{2}+\mathsf{e}_F(\gamma g_n;\gamma/\beta)\bigg\}\\
& = \max_{1/C \leq \beta \leq C } \min_{1/C\leq \gamma \leq C} \overline{\mathsf{D}}_\eta(\beta,\gamma)  + \bigo\big(\sqrt{t/n}+t/n+\delta\big). 
\end{align*}
The estimate in $\bigo$ is uniform in $\eta \in \Xi_K$, so the claim follows.
\end{proof}

Finally we delocalize the range constraints for $\beta,\gamma$ in the deterministic minimax problem with $\overline{\mathsf{D}}_\eta$ in the above proposition.

\begin{proposition}\label{prop:local_barD}
Suppose $1/K\leq \phi^{-1},\sigma_\xi^2\leq K$, and $\pnorm{\mu_0}{}\vee \pnorm{\Sigma}{\op}\vee \mathcal{H}_\Sigma\leq K$ for some $K>0$. There exists some $C=C(K)>1$ such that for any $\eta \in \Xi_K$, 
\begin{align*}
\max_{\beta>0}\min_{\gamma >0} \overline{\mathsf{D}}_\eta(\beta,\gamma) = \max_{1/C \leq \beta \leq C} \min_{1/C\leq \gamma \leq C} \overline{\mathsf{D}}_\eta(\beta,\gamma).
\end{align*}
Consequently,
\begin{align}\label{ineq:cont_D_eta}
\bigabs{\max_{\beta>0}\min_{\gamma >0} \overline{\mathsf{D}}_\eta(\beta,\gamma)-\max_{\beta>0}\min_{\gamma >0} \overline{\mathsf{D}}_0(\beta,\gamma)}\leq C\eta. 
\end{align}
\end{proposition}
\begin{proof}
The proof is essentially a deterministic version of Step 2 in the proof of Proposition \ref{prop:L_local}. We give some details below. We write $g_n\equiv g/\sqrt{n}$. First, using similar calculations as that of (\ref{ineq:L_local_8}), 
\begin{align*}
\begin{cases}
\partial_\beta \E \mathsf{e}_F(\gamma g_n; {\gamma}/{\beta}) = \frac{1}{2\gamma}\Big(\E\err_{(\Sigma,\mu_0)}(\gamma;\gamma/\beta)-2\E \dof_{(\Sigma,\mu_0)}(\gamma;\gamma/\beta)+ \gamma^2 \Big),\\
\partial_\gamma \E \mathsf{e}_F(\gamma g_n; {\gamma}/{\beta}) 
 = \frac{\beta}{2\gamma^2}\Big(\gamma^2-\E \err_{(\Sigma,\mu_0)}(\gamma;\gamma/\beta)\Big).
\end{cases}
\end{align*}
Then the first-order optimality condition for $(\beta_\ast,\gamma_\ast)$ to be the saddle point of $\max_{\beta>0}\min_{\gamma >0} \overline{\mathsf{D}}_\eta(\beta,\gamma)$, i.e., a deterministic version of (\ref{ineq:L_local_10}),  is given by
\begin{align*}
\begin{cases}
\phi \gamma_\ast^2  = \sigma_\xi^2+ \E \err_{(\Sigma,\mu_0)}(\gamma_\ast;\gamma_\ast/\beta_\ast),\\
\big(\phi-\frac{\eta}{\gamma_\ast/\beta_\ast}\big)\gamma_\ast^2  = \E \dof_{(\Sigma,\mu_0)}(\gamma_\ast;\gamma_\ast/\beta_\ast).
\end{cases}
\end{align*}
Finally using the apriori estimates in Proposition \ref{prop:fpe_est}, we obtain a deterministic analogue of (\ref{ineq:L_local_11}) in that $
\gamma_\ast \asymp_{K}  1,\quad  \beta_\ast \asymp_{K} 1$. The claimed localization follows. The continuity follows by the definition of $\overline{\mathsf{D}}_\eta$ and the proven localization. 
\end{proof}

\begin{proof}[Proof of Theorem \ref{thm:char_gordon_cost_opt}]
By Propositions \ref{prop:L_local}, \ref{prop:D_pm_barD} and \ref{prop:local_barD}, there exist $C,C'>0$ such that for any $\delta\in (0,1/C^{100})$, $M\leq \sqrt{n}/C$, $t\geq C'\log (en)$, $\xi \in \mathscr{E}_{1,\xi}(\delta)$ and $\eta \in \Xi_K$,
\begin{align*}
&\Prob^\xi\Big[\big|\min_{w \in B_n(L_w)}L_\eta(w;L_v) -  \max_{\beta>0}\min_{\gamma >0} \overline{\mathsf{D}}_\eta(\beta,\gamma)\big| \geq C \big(\sqrt{t/n}+t/n+\delta\big)\Big]\\
&\leq C e^{-t/C}+\Prob^\xi\big(\mathscr{E}_{1,0}(\delta)^c\big)+\Prob(\mathscr{E}_{\Delta,\Xi}(M)^c). 
\end{align*}
The claim now follows from the concentration estimates in Lemmas \ref{lem:conc_E_12}, \ref{lem:conc_Delta} and \ref{lem:conc_Xi}, by choosing $M\equiv \sqrt{n}/C$ and $\delta\equiv C(\sqrt{t/n}+t/n)$, which is valid in the regime $t\leq n/C_0$ for large $C_0$. 
\end{proof}

\subsection{Locating the global minimizer of the Gordon objective}

With $(\gamma_{\eta,\ast},\tau_{\eta,\ast})$ denoting the unique solution to the system of equations (\ref{eqn:fpe}), let 
\begin{align}\label{def:w_ast}
w_{\eta,\ast} \equiv \prox_F \big(\gamma_{\eta,\ast} g/\sqrt{n} ;\tau_{\eta,\ast}\big)=\Sigma^{1/2}\big(\hat{\mu}_{(\Sigma,\mu_0)}^{\seq}(\gamma_{\eta,\ast};\tau_{\eta,\ast})-\mu_0\big).
\end{align}
For any $\epsilon>0$, let the exceptional set be defined as 
\begin{align}\label{def:D_eps}
D_{\eta;\epsilon}(\mathsf{g})\equiv \big\{w\in \R^n: \abs{\mathsf{g}(w)-\E \mathsf{g}(w_{\eta,\ast})}\geq \epsilon\big\}.
\end{align}

\begin{theorem}\label{thm:gordon_gap}
	Suppose the following hold for some $K>0$.
	\begin{itemize}
		\item $1/K\leq \phi^{-1}\leq K$, $\pnorm{\mu_0}{}\vee \pnorm{\Sigma}{\op}\vee \mathcal{H}_\Sigma \leq K$.
		\item Assumption \ref{assump:noise} holds with $\sigma_\xi^2 \in [1/K,K]$. 
	\end{itemize}
    Fix any $\mathsf{g}:\R^n\to \R$ that is $1$-Lipschitz with respect to $\pnorm{\cdot}{\Sigma^{-1}}$. There exist constants $C,C'>10$ depending on $K$ such that for $L_w, L_v\in [C,C^2]$, $C'\log (en)\leq t\leq n/C'$, $\xi \in \mathscr{E}_{1,\xi}(\sqrt{t/n})$ and $\eta \in \Xi_K$, 
	\begin{align*}
	\Prob^\xi\bigg(\min_{w \in D_{\eta;C(t/n)^{1/4}}(\mathsf{g})\cap B_n(L_w)}L_\eta(w;L_v) \leq  \max_{\beta>0}\min_{\gamma >0} \overline{\mathsf{D}}_\eta(\beta,\gamma)  + \sqrt{t/n}\bigg)\leq C e^{-t/C}. 
	\end{align*}
\end{theorem}

Roughly speaking, the above theorem will be proved by approximating $L_\eta$ both from above and below by nicer strongly convex, surrogate functions whose minimizers can be directly located. Then we may relate the minimizer of $L_\eta$ and those of the surrogate functions. 

We first formally define these surrogate functions. For $L_w>0,L_v>0$, let
\begin{align}\label{def:L_eta_pm}
L_{\eta,\pm}(w;L_v)&\equiv \max_{\beta\in [0,L_v]}\bigg\{\frac{\beta}{\sqrt{n}} \Big(\pnorm{h}{}\sqrt{\pnorm{w}{}^2+\sigma_\pm^2(L_w)}-\iprod{g}{w}\Big)-\frac{\eta\beta^2}{2}+F(w)\bigg\}.
\end{align}
Again we omit notational dependence of $L_{\eta,\pm}$ on $L_w$ for simplicity.

The following lemma provides uniform (bracketing) approximation of $L_\eta$ via $L_{\eta,\pm}$ on compact sets. 

\begin{lemma}\label{lem:L_L_pm_property}
Fix $L_v>0$. The following hold when $\sigma_-^2(L_w)\neq 0$.
\begin{enumerate}
	\item For any $w \in B_n(L_w)$, $
	L_{\eta,-}(w;L_v)\leq L_\eta(w;L_v)\leq L_{\eta,+}(w;L_v)$. 
	\item For any $L_w>0$,
	\begin{align*}
	\sup_{w \in \R^n}\abs{L_{\eta,+}(w;L_v)-L_{\eta,-}(w;L_v)}\leq \frac{4e_h}{\sigma_m}\cdot L_v L_w \frac{ \abs{\iprod{h}{\xi}}}{\pnorm{h}{}^2}.
	\end{align*}
\end{enumerate}
\end{lemma}
\begin{proof}
The first claim (1) follows by the definition of $\sigma_\pm^2(L_w)$ in (\ref{def:sigma_pm}) and the simple inequality (\ref{ineq:L_local_2}). For (2), note that
\begin{align*}
\bigabs{L_{\eta,+}(w;L_v)-L_{\eta,-}(w;L_v)}&\leq L_v e_h \cdot  \bigabs{ \sqrt{\pnorm{w}{}^2+\sigma_+^2(L_w)}-\sqrt{\pnorm{w}{}^2+\sigma_-^2(L_w)}  }\\
&\leq L_ve_h\cdot  \frac{\abs{\sigma_+^2(L_w)-\sigma_-^2(L_w)}}{\sigma_+(L_w)+\sigma_-(L_w)} \leq \frac{4e_h}{\sigma_m}\cdot  L_v L_w \frac{ \abs{\iprod{h}{\xi}}}{\pnorm{h}{}^2},
\end{align*}
as desired. 
\end{proof}

Next, we will study the properties of the global minimizers for $L_{\eta,\pm}$. 

\begin{proposition}\label{prop:L_pm_minimizer}
Suppose $1/K\leq \phi^{-1},\sigma_\xi^2\leq K$, and $\pnorm{\mu_0}{}\vee \pnorm{\Sigma}{\op}\vee \mathcal{H}_\Sigma \leq K$ for some $K>0$. There exists some constant $C=C(K)>1$ such that for any deterministic choice of $L_w,L_v \in [C,C^2]$, 
on the event $\mathscr{E}_1(\delta)\cap \mathscr{E}_{\Delta,\Xi}(M)$ (defined in Proposition \ref{prop:fpe_sample_est}) with $\delta\in (0,1/C^{100})$ and $M\leq \sqrt{n}/C$, for any $\eta \in \Xi_K$, the maps $w\mapsto L_{\eta,\pm}(w;L_v)$ attain its global minimum at $w_{n,\eta,\pm}$ with $\pnorm{w_{n,\eta,\pm}}{\Sigma^{-1}}\leq C$. Moreover, $\pnorm{w_{n,\eta,\pm}-w_{\eta,\ast} }{\Sigma^{-1}}\leq C(M/\sqrt{n}+\delta)^{1/2}$.
\end{proposition}
\begin{proof}
Note that the optimization problem 
\begin{align*}
&\min_{w \in \R^n} L_{\eta,\pm}(w;L_v)\\
&=\min_{w \in \R^n} \max_{\beta\in [0,L_v]}\bigg\{\frac{\beta}{\sqrt{n}}\Big(\pnorm{h}{}\sqrt{\pnorm{w}{}^2+\sigma_\pm^2(L_w)}-\iprod{g}{w}\Big)-\frac{\eta \beta^2}{2}+F(w)\bigg\}\\
& \stackrel{(\ast)}{=} \max_{\beta\in [0,L_v]} \min_{w \in \R^n} \bigg\{\frac{\beta}{\sqrt{n}}\Big(\pnorm{h}{}\sqrt{\pnorm{w}{}^2+\sigma_\pm^2(L_w)}-\iprod{g}{w}\Big)-\frac{\eta \beta^2}{2}+F(w)\bigg\} \\
& = \max_{\beta\in [0,L_v]}  \min_{\gamma>0,w\in \R^n} \bigg\{\frac{\beta\gamma \pnorm{h}{}^2 }{2 n}-\frac{\eta \beta^2}{2}+\bigg(\frac{\beta }{2\gamma}\big(\pnorm{w}{}^2+\sigma_\pm^2(L_w)\big)- \biggiprod{w}{\frac{\beta}{\sqrt{n}} g}+F(w)  \bigg)\bigg\}.
\end{align*}
Here in $(\ast)$ we used Sion's min-max theorem to exchange minimum and maximum, as the maximum is taken over a compact set. The difference of the above minimax problem compared to (\ref{ineq:L_local_6}) rests in its range constraint on $\beta$. As proven in (\ref{ineq:L_local_11}), all solutions $\beta_{n,\pm}$ to the unconstrained minimax problem  (\ref{ineq:L_local_6}) must satisfy $\beta_{n,\eta,\pm}\leq C$ on the event $\mathscr{E}_1(\delta)\cap\mathscr{E}_{\Delta,\Xi}(M)$. So on this event, for the choice $L_w, L_v\in [C,C^2]$ for some large $C>0$, $\min_{w \in \R^n} L_{\eta,\pm}(w;L_v)$ exactly corresponds to (\ref{ineq:L_local_6}), whose minimizers $w_{n,\pm}$ admit the apriori estimate (\ref{ineq:L_local_12}) (with minor modifications that change $\pnorm{\cdot}{}$ to the stronger estimate in $\pnorm{\cdot}{\Sigma^{-1}}$). 

Next, for the error bound, using the last equation in (\ref{ineq:L_local_7}) and the definition of $w_{\eta,\ast}$ in (\ref{def:w_ast}), along with the estimates in Proposition \ref{prop:fpe_sample_est}, we have
\begin{align*}
\pnorm{w_{n,\eta,\pm}-w_{\eta,\ast}}{\Sigma^{-1}}^2 &= \bigpnorm{\hat{\mu}_{(\Sigma,\mu_0)}^{\seq}(\gamma_{n,\eta,\pm};\tau_{n,\eta,\pm})-\hat{\mu}_{(\Sigma,\mu_0)}^{\seq}(\gamma_{\eta,\ast};\tau_{\eta,\ast}) }{}^2\\
&\lesssim_K \abs{  \gamma_{n,\eta,\pm}- \gamma_{\eta,\ast}}\vee \abs{ \tau_{n,\eta,\pm}-  \tau_{\eta,\ast} }\leq C(M/\sqrt{n}+\delta),
\end{align*}
as desired.
\end{proof}

Finally we shall relate back to the global minimizer of $L_\eta$. We note that the proposition below by itself is not formally used in the proof of Theorem \ref{thm:gordon_gap}, but will turn out to be useful in the proof of Theorem \ref{thm:min_norm_dist} ahead. 

\begin{proposition}\label{prop:L_minimizer}
Suppose the conditions in Theorem \ref{thm:gordon_gap} hold for some $K>0$. There exist constants $C,C'>1$ depending on $K$ such that for $L_w, L_v\in [C,C^2]$, $C'\log (en)\leq t\leq n/C'$, $\xi \in \mathscr{E}_{1,\xi}(\sqrt{t/n})$ and $\eta \in \Xi_K$, 
\begin{align*}
&\Prob^\xi\Big(\hbox{The map $w\mapsto L_\eta(w;L_v)$ attains its global minimum at $w_{n,\eta}$ with $\pnorm{w_{n,\eta}}{\Sigma^{-1}}\leq C$,}\\
&\qquad \hbox{and $\pnorm{w_{n,\eta}-w_{\eta,\ast}}{\Sigma^{-1}}\leq C(t/n)^{1/4}$}\Big)\geq 1-C e^{-t/C}. 
\end{align*}
\end{proposition}

\begin{proof}
Let us fix $\xi \in \mathscr{E}_{1,\xi}(\sqrt{t/n})$. 
	
\noindent (\textbf{Step 1}). We first prove the apriori estimate for $\pnorm{w_{n,\eta}}{\Sigma^{-1}}$. To this end, for large enough $C_0, C_0'>0$ depending on $K$, we choose $L_w\equiv C_0, \delta\equiv 1/C_0^{100}$ and $M\equiv \delta\sqrt{n}$ in Proposition \ref{prop:L_pm_minimizer}, it follows that
\begin{align}\label{ineq:L_minimizer_1}
&\Prob^\xi\Big(E_1\equiv \Big\{\pnorm{w_{n,\eta,\pm}}{\Sigma^{-1}}\vee \pnorm{w_{n,\eta,\pm}}{}\leq C_0/2, \nonumber\\
&\qquad\qquad\qquad L_{\eta,\pm}(w_{n,\eta,\pm};L_v)= \hbox{(\ref{ineq:L_local_6})}\Big\}\Big)\geq 1-C_0e^{-n/C_0}.
\end{align}
On the other hand, choosing $\delta\equiv \sqrt{t/n}$ with $C_0'\log (en)\leq t\leq n/C_0'$ leads to
\begin{align}\label{ineq:L_minimizer_2}
&\Prob^\xi\Big(E_2(t)\equiv\Big\{ \pnorm{w_{n,\eta,\pm}-w_{\eta,\ast} }{\Sigma^{-1}}\leq C_0 (t/n)^{1/4}\Big\}\Big)\geq 1-C_0e^{-t/C_0}. 
\end{align}
On $E_1$, we may characterize the value of $L_{\eta,\pm}(w_{n,\eta,\pm};L_v)$ by applying Propositions \ref{prop:L_local}-\ref{prop:local_barD}:  for $C_0'\log (en)\leq t\leq n/C_0'$,
\begin{align}\label{ineq:L_minimizer_3}
&\Prob^\xi\Big(E_3(t)\equiv \Big\{\bigabs{L_{\eta,\pm}(w_{n,\eta,\pm};L_v)- \max_{\beta>0}\min_{\gamma >0} \overline{\mathsf{D}}_\eta(\beta,\gamma)  }\leq C_0\sqrt{t/n}\Big\}\Big)\geq 1- C_0e^{-t/C_0}.
\end{align}
Note by the strong convexity of $L_{\eta,\pm}(\cdot;L_v)$ with respect to $\pnorm{\cdot}{\Sigma^{-1}}$, we have
\begin{align*}
\inf_{ w \in\R^n: \pnorm{w-w_{n,\eta,\pm}}{\Sigma^{-1}}\geq \sqrt{6C_0}(t/n)^{1/4} } L_{\eta,\pm}(w;L_v)-L_{\eta,\pm}(w_{n,\eta,\pm};L_v)\geq 3C_0\sqrt{t/n}.
\end{align*}
This means on $E_3(t)$, 
\begin{align*}
\inf_{ w \in\R^n: \pnorm{w-w_{n,\eta,\pm}}{\Sigma^{-1}}\geq \sqrt{6C_0}(t/n)^{1/4} }L_{\eta,\pm}(w;L_v)&\geq \max_{\beta>0}\min_{\gamma >0} \overline{\mathsf{D}}_\eta(\beta,\gamma) + 2C_0\sqrt{t/n},\\
L_{\eta,\pm}(w_{n,\eta,\pm};L_v)&\leq \max_{\beta>0}\min_{\gamma >0}\overline{\mathsf{D}}_\eta(\beta,\gamma) + C_0\sqrt{t/n}. 
\end{align*}
This in particular means on $E_1\cap E_3(t)$, 
\begin{align*}
w_{n,\eta,\pm}&\in \Big\{w\in \R^n: L_{\eta,\pm}(w;L_v)\leq \max_{\beta>0}\min_{\gamma >0}\overline{\mathsf{D}}_\eta(\beta,\gamma) + C_0\sqrt{t/n}  \Big\}\\
&\subset \big\{w\in \R^n: \pnorm{w}{\Sigma^{-1}}\leq \sqrt{6C_0}(t/n)^{1/4}+C_0/2\big\}.
\end{align*}
Consequently, by enlarging $C_0>0$ if necessary, using Lemma \ref{lem:L_L_pm_property}-(1), on $E_1\cap E_3(t)$
\begin{align*}
& \Big\{w\in \R^n: L(w;L_v)\leq \max_{\beta>0}\min_{\gamma >0}\overline{\mathsf{D}}_\eta(\beta,\gamma) + C_0\sqrt{t/n}  \Big\}\\
&\subset \big\{w\in \R^n: \pnorm{w}{\Sigma^{-1}}\leq 3C_0/5\big\}\subset B_n(3C_0/4)\subsetneq B_n(C_0)=B_n(L_w). 
\end{align*}
This implies, on $E_1\cap E_3(t)$, we have $
\pnorm{w_{n,\eta}}{\Sigma^{-1}}\vee\pnorm{w_{n,\eta}}{}\leq 3C_0/4$, proving the apriori bound.

\noindent (\textbf{Step 2}). Next we establish the announced error bound. On the event $\mathscr{E}_{1,0}(\sqrt{t/n})$, by Lemma \ref{lem:L_L_pm_property}-(2), 
\begin{align}\label{ineq:L_minimizer_4}
\sup_{w\in B_n(C_0)}\bigabs{L_\eta(w;L_v)-L_{\eta,\pm}(w;L_v)}\leq C_1\sqrt{t/n}.
\end{align}
Consequently, on $E_1\cap E_3(t)\cap \mathscr{E}_{1,0}(\sqrt{t/n})$, 
\begin{align}\label{ineq:L_minimizer_5}
\bigabs{\min_{w \in \R^n} L_\eta(w;L_v) -\max_{\beta>0}\min_{\gamma >0} \overline{\mathsf{D}}_\eta(\beta,\gamma)}\leq C_2 \sqrt{t/n}.
\end{align}
On this event, combining (\ref{ineq:L_minimizer_4})-(\ref{ineq:L_minimizer_5}) with (\ref{ineq:L_minimizer_3}), and using again the strong convexity of $L_{\eta,+}(\cdot;L_v)$ respect to $\pnorm{\cdot}{\Sigma^{-1}}$, we have for $C_3=2\sqrt{(C_0+C_1+C_2)}$,
\begin{align*}
&\inf_{w \in B_n(C_0): \pnorm{w-w_{n,\eta,+}}{\Sigma^{-1}}\geq C_3(t/n)^{1/4} }L_\eta(w;L_v)-\min_{w \in \R^n} L_\eta(w;L_v)\\
&\geq \inf_{w \in B_n(C_0): \pnorm{w-w_{n,\eta,+}}{\Sigma^{-1}}\geq C_3(t/n)^{1/4} }L_{\eta,+}(w;L_v)- \max_{\beta>0}\min_{\gamma >0} \overline{\mathsf{D}}_\eta(\beta,\gamma) - (C_1+C_2)\sqrt{t/n}\\
&\geq \inf_{w \in B_n(C_0): \pnorm{w-w_{n,\eta,+}}{\Sigma^{-1}}\geq C_3(t/n)^{1/4} }L_{\eta,+}(w;L_v)-L_{\eta,+}(w_{n,+};L_v) - (C_0+C_1+C_2)\sqrt{t/n}\\
&\geq (C_3^2/2)\sqrt{t/n}- (C_0+C_1+C_2)\sqrt{t/n} = (C_0+C_1+C_2)\sqrt{t/n}. 
\end{align*}
This means that $\pnorm{w_{n,\eta}-w_{n,\eta,+}}{\Sigma^{-1}}\leq C_3(t/n)^{1/4}$ on $E_1\cap E_3(t)\cap \mathscr{E}_{1,0}(\sqrt{t/n})$. The claim follows by intersecting the prescribed event with $E_2(t)$ in (\ref{ineq:L_minimizer_2}) that controls the $\Prob^\xi$-probability of $\pnorm{w_{n,\eta,+}-w_{\eta,\ast}}{\Sigma^{-1}}\leq C_0(t/n)^{1/4}$. 
\end{proof}

\begin{proof}[Proof of Theorem \ref{thm:gordon_gap}]
Fix $\xi \in \mathscr{E}_{1,\xi}(\sqrt{t/n})$, and $\epsilon>0$ to be chosen later on. First, as $\mathsf{g}$ is Lipschitz with respect to $\pnorm{\cdot}{\Sigma^{-1}}$, by the Gaussian concentration inequality, there exists $C_0=C_0(K)>0$ such that for $t\geq 1$, on an event $E_0(t)$ with $\Prob^\xi$-probability at least $1-e^{-t}$, 
\begin{align*}
\abs{\mathsf{g}(w_{\eta,\ast})-\E \mathsf{g}(w_{\eta,\ast})}\leq C_0 \sqrt{t/n}.
\end{align*}
Moreover, by Proposition \ref{prop:L_pm_minimizer} and Propositions \ref{prop:L_local}-\ref{prop:local_barD}, there exist some $C_1,C_1'>0$ depending on $K$ such that for $C_1' \log (en)\leq t\leq n/C_1'$, on an event $E_1(t)$ with  $\Prob^\xi$-probability $1-C_1e^{-t/C_1}$, we have 
\begin{enumerate}
	\item $\pnorm{w_{n,\eta,-}}{\Sigma^{-1}}\vee \pnorm{w_{n,\eta,-}}{}\leq C_1$, $\pnorm{w_{n,\eta,-}-w_{\eta,\ast} }{\Sigma^{-1}}\leq C_1  (t/n)^{1/4}$, and
	\item $\abs{ L_{\eta,-}(w_{n,\eta,-};L_v) -\max_{\beta>0}\min_{\gamma >0} \overline{\mathsf{D}}_\eta(\beta,\gamma)}\leq C_1 \sqrt{t/n}.$
\end{enumerate}
Consequently, for $C_1' \log (en)\leq t\leq n/C_1'$, on the event $E_0(t)\cap E_1(t)$, uniformly in $w \in D_{\eta;\epsilon}(\mathsf{g})\cap B_n(L_w)$,
\begin{align*}
\epsilon&\leq \abs{\mathsf{g}(w)-\E \mathsf{g}(w_{\eta,\ast})}\leq \abs{ \mathsf{g}(w)-\mathsf{g}(w_{\eta,\ast}) }+ \abs{\mathsf{g}(w_{\eta,\ast})-\E \mathsf{g}(w_{\eta,\ast})}\\
&\leq \pnorm{w-w_{\eta,n,-}}{\Sigma^{-1}}+ \pnorm{w_{\eta,n,-}-w_{\eta,\ast} }{\Sigma^{-1}}+ C_0\sqrt{t/n}\\
&\leq \pnorm{w-w_{\eta,n,-}}{\Sigma^{-1}}+(C_0+C_1)(t/n)^{1/4}.
\end{align*}
This implies that, for the prescribed range of $t$ and on the event $E_0(t)\cap E_1(t)$, 
\begin{align*}
\min_{w \in D_{\eta;\epsilon}(\mathsf{g})\cap B_n(L_w)}\pnorm{w-w_{\eta,n,-}}{\Sigma^{-1}}\geq \big(\epsilon-(C_0+C_1)(t/n)^{1/4}\big)_+. 
\end{align*}
Using the strong convexity of $L_{\eta,-}(\cdot;L_v)$ with respect to $\pnorm{\cdot}{\Sigma^{-1}}$, we have for $C_1' \log (en)\leq t\leq n/C_1'$, on the event $E_0(t)\cap E_1(t)$,
\begin{align*}
&\min_{w \in D_{\eta;\epsilon}(\mathsf{g})\cap B_n(L_w) }L_\eta(w;L_v) \geq \min_{w \in D_{\eta;\epsilon}(\mathsf{g})\cap B_n(L_w) }L_{\eta,-}(w;L_v)\\
&\geq L_{\eta,-}(w_{\eta,n,-};L_v)+ \frac{1}{2}\big(\epsilon-(C_0+C_1)(t/n)^{1/4}\big)_+^2\\
&\geq \max_{\beta>0}\min_{\gamma >0} \overline{\mathsf{D}}_\eta(\beta,\gamma)+ \frac{1}{2}\big(\epsilon-(C_0+C_1)(t/n)^{1/4}\big)_+^2-C_1 \sqrt{t/n}.
\end{align*}
Now we may choose $
\epsilon\equiv \epsilon(t,n)\equiv (C_0+C_1+2\sqrt{C_1})(t/n)^{1/4}$ to conclude by adjusting constants.
\end{proof}

\subsection{Proof of Theorem \ref{thm:min_norm_dist} for $\hat{\mu}_{\eta;G}$}

Fix $\xi \in \mathscr{E}_{1,\xi}(\sqrt{t/n})$. All the constants in $\lesssim,\gtrsim,\asymp$ below may depend on $K$. 

\noindent (\textbf{Step 1}). In this step, we will obtain an upper bound $\min_{w \in \R^n} H_\eta(w)$. By Proposition \ref{prop:H_local} and the concentration estimate in Lemma \ref{lem:conc_E_0}, there exists some $C_0=C_0(K)>0$ such that on an event $E_0$ with $\Prob^\xi(E_0)\geq 1-C_0e^{-n/C_0}$, 
\begin{align}\label{ineq:dist_min_norm_0}
\min_{w \in \R^n} H_\eta(w)=\min_{w \in \R^n} H_\eta(w;L_0)=\min_{w \in B_n(L_0)} H_\eta(w)=\min_{w \in B_n(L_0)} H_\eta(w;L_0). 
\end{align}
where
\begin{align}\label{ineq:dist_min_norm_00}
L_0\equiv C_0\Big\{1+ \Big(\pnorm{\Sigma^{-1}}{\op}\bm{1}_{\phi^{-1}\geq 1+1/K}^{-1}\wedge \eta^{-1}\Big)\Big\}.
\end{align}
Now we shall apply the convex(-side) Gaussian min-max theorem to obtain an upper bound for the right hand side of (\ref{ineq:dist_min_norm_0}). Recall the definition of $h_\eta=h_{\eta;G}$ and $\ell_\eta$ in (\ref{def:h_l}). Using Theorem \ref{thm:CGMT}-(2), for any $z \in \R$,
\begin{align}\label{ineq:dist_min_norm_1}
\Prob^\xi\Big(\min_{w \in \R^n} H_\eta(w)\geq z\Big)&\leq \Prob^\xi\Big( \min_{w \in B_n(L_0)} H_\eta(w;L_0)\geq z \Big)+\Prob^\xi(E_0^c)\nonumber\\
&= \Prob^\xi\Big( \min_{w \in B_n(L_0)}\max_{v \in B_m(L_0)} h_\eta(w,v)\geq z \Big)+\Prob^\xi(E_0^c)\nonumber\\
&\leq 2 \Prob^\xi\Big( \min_{w \in B_n(L_0)}\max_{v \in B_m(L_0)} \ell_\eta(w,v)\geq z \Big)+\Prob^\xi(E_0^c)\nonumber\\
& = 2\Prob^\xi\Big( \min_{w \in B_n(L_0)} L_\eta(w;L_0)\geq z\Big)+\Prob^\xi(E_0^c).
\end{align} 
By Proposition \ref{prop:L_minimizer}, there exist some $C_1,C_1'>0$ depending on $K$ (which we assume without loss of generality $L_0>C_1$ and $C_1$ exceeds the constants in Theorems \ref{thm:char_gordon_cost_opt} and \ref{thm:gordon_gap}), such that on an event $E_1$ with $\Prob^\xi$-probability at least $1-C_1e^{-n/C_1}$, the map $w\mapsto L_\eta(w;L_0)$ attains its global minimum in $B_n(C_1)$. We may now apply Theorem \ref{thm:char_gordon_cost_opt}: with $
z\equiv \bar{z}(t)=\max_{\beta>0}\min_{\gamma >0}\overline{\mathsf{D}}_\eta(\beta,\gamma) +\sqrt{t/n}$, for $C_1'\log (en)\leq t\leq n/C_1'$, 
\begin{align}\label{ineq:dist_min_norm_3}
&\Prob^\xi\Big( \min_{w \in B_n(L_0)} L_\eta(w;L_0)\geq \bar{z}(t) \Big)\nonumber\\
&\leq \Prob^\xi\Big( \min_{w \in B_n(C_1)} L_\eta(w;L_0)\geq \bar{z}(t) \Big)+\Prob^\xi(E_1^c)\leq C_1 e^{-t/C_1}+\Prob^\xi(E_1^c).
\end{align}
Combining (\ref{ineq:dist_min_norm_1})-(\ref{ineq:dist_min_norm_3}), by enlarging $C_1$ if necessary, for $C_1'\log (en)\leq t\leq n/C_1'$, and $\eta \in \Xi_K$,
\begin{align}\label{ineq:dist_min_norm_4}
\Prob^\xi\Big(\min_{w \in \R^n} H_\eta(w)\geq \max_{\beta>0}\min_{\gamma >0}\overline{\mathsf{D}}_\eta(\beta,\gamma) +\sqrt{t/n}\Big)\leq C_1e^{-t/C_1}.
\end{align}
An entirely similar argument leads to a lower bound (which will be used later on): 
\begin{align}\label{ineq:dist_min_norm_4_1}
\Prob^\xi\Big(\min_{w \in \R^n} H_\eta(w)\leq \max_{\beta>0}\min_{\gamma >0}\overline{\mathsf{D}}_\eta(\beta,\gamma) -\sqrt{t/n}\Big)\leq C_1e^{-t/C_1}.
\end{align}

\noindent (\textbf{Step 2}). In this step, we will obtain a lower bound on $\min_{w \in D_{\eta;\epsilon}(\mathsf{g})} H_\eta(w)$ for the exceptional set $D_\epsilon(\mathsf{g})$ defined in (\ref{def:D_eps}), with a suitable choice of $\epsilon$. Let us take $C_2,C_2'>0$ to be the constants in Theorem \ref{thm:gordon_gap}, and let $\epsilon(t,n)\equiv C_2(t/n)^{1/4}$ for $C_2'\log(en)\leq t\leq n/C_2'$. To this end, using Theorem \ref{thm:CGMT}-(1) (that holds without convexity), for any $z \in \R$ and $L_v>0$
\begin{align*}
\Prob^\xi\Big(\min_{w \in B_n(L_0)\cap D_{\eta;\epsilon}(\mathsf{g})} H_\eta(w)\leq z\Big)&\leq \Prob^\xi\Big(\min_{w \in B_n(L_0)\cap D_{\eta;\epsilon}(\mathsf{g})} \max_{v \in B_m(L_v)} h_\eta(w,v)\leq z\Big)\\
&\leq 2 \Prob^\xi\Big(\min_{w \in B_n(L_0)\cap D_{\eta;\epsilon}(\mathsf{g})} \max_{v \in B_m(L_v)} \ell_\eta(w,v)\leq z\Big)\\
& = 2 \Prob^\xi\Big(\min_{w \in B_n(L_0)\cap D_{\eta;\epsilon}(\mathsf{g})} L_\eta(w;L_v)\leq z\Big).
\end{align*}
By choosing $L_v\asymp 1$ of constant order but large enough, $\epsilon\equiv \epsilon(t,n)$ and $
z\equiv \bar{z}(t) = \max_{\beta>0}\min_{\gamma >0}\overline{\mathsf{D}}_\eta(\beta,\gamma) +2\sqrt{t/n}$, we have for $C_2'\log(en)\leq t \leq n/C_2'$,
\begin{align}\label{ineq:dist_min_norm_5}
&\Prob^\xi\bigg(\min_{w \in B_n(L_0)\cap D_{\eta;\epsilon(t,n)}(\mathsf{g})} H_\eta(w)\leq \max_{\beta>0}\min_{\gamma >0}\overline{\mathsf{D}}_\eta(\beta,\gamma) +2\sqrt{\frac{t}{n}}\bigg)\leq 2C_2 e^{-t/C_2}. 
\end{align}
\noindent (\textbf{Step 3}).  Combining (\ref{ineq:dist_min_norm_5}) and the localization in (\ref{ineq:dist_min_norm_0}), there exist some $C_3,C_3'>0$ depending on $K$ such that for $C_3' \log (en)\leq t\leq n/C_3'$, on an event $E_3(t)$ with $\Prob^\xi(E_3(t))\geq 1-C_3e^{-t/C_3}$,
\begin{align*}
\min_{w \in B_n(L_0)\cap D_{\eta;\epsilon(t,n)}(\mathsf{g})} H_\eta(w)&\geq \max_{\beta>0}\min_{\gamma >0}\overline{\mathsf{D}}_\eta(\beta,\gamma) +2\sqrt{t/n}\\
&>\max_{\beta>0}\min_{\gamma >0}\overline{\mathsf{D}}_\eta(\beta,\gamma) +\sqrt{t/n}\geq \min_{w \in \R^n} H_\eta(w) = \min_{w \in B_n(L_0)} H_\eta(w).
\end{align*}
So on $E_3(t)$, $\hat{w}_\eta\notin D_{\eta;\epsilon(t,n)}(\mathsf{g})\cap B_n(L_0)$, i.e., for $C_3'\log (en)\leq t\leq n/C_3'$,
\begin{align*}
\Prob^\xi\Big(\abs{\mathsf{g}(\hat{w}_\eta)-\E \mathsf{g}(w_{\eta,\ast})}\geq C_3(t/n)^{1/4}\Big)\leq C_3 e^{-t/C_3}.
\end{align*}
Using a change of variable and suitably adjusting the constant $C_3$, for any $1$-Lipschitz function $\mathsf{g}_0:\R^n\to \R$, $\eta \in \Xi_K$ and $\epsilon \in (0,1/2]$,
\begin{align*}
\Prob^\xi\Big(\bigabs{\mathsf{g}_0(\hat{\mu}_\eta)-\E \mathsf{g}_0\big(\hat{\mu}_{(\Sigma,\mu_0)}^{\seq}(\gamma_{\eta,\ast};\tau_{\eta,\ast})\big) }\geq \epsilon\Big)\leq C_3 n e^{-n\epsilon^4/C_3}. 
\end{align*}
\noindent (\textbf{Step 4}). In this step we shall establish uniform guarantees. We write $\hat{\mu}_{(\Sigma,\mu_0)}^{\seq}(\gamma_{\eta,\ast};\tau_{\eta,\ast})=\hat{\mu}_{\eta;(\Sigma,\mu_0)}^{\seq,\ast}$ in this part of the proof. First, in the case $\phi^{-1}\geq 1+1/K$, using $
\hat{\mu}_\eta = n^{-1} X^\top \big(XX^\top/n +\eta I\big)^{-1}Y$, for $\eta_1,\eta_2 \in [0,K]$,
\begin{align}\label{ineq:dist_min_norm_6}
\pnorm{ \hat{\mu}_{\eta_1}-\hat{\mu}_{\eta_2} }{}&\lesssim n^{-1} \pnorm{G}{\op}(\pnorm{G}{\op}+\pnorm{\xi}{})\cdot  \pnorm{\big(XX^\top/n +\eta_1 I\big)^{-1}-\big(XX^\top/n +\eta_2 I\big)^{-1}   }{\op}\nonumber\\
&\lesssim \pnorm{\Sigma^{-1}}{\op}^2\cdot \Big(1+\frac{\pnorm{G}{\op}+\pnorm{\xi}{}}{\sqrt{n}}\Big)^2\cdot  \pnorm{(GG^\top/n)^{-1}}{\op}^2\cdot  \abs{\eta_1-\eta_2}. 
\end{align}
Here the last inequality follows by the fact that any p.s.d. matrix $A$, 
$\pnorm{(A+\eta_1 I)^{-1}-(A+\eta_2 I)^{-1}}{\op}\leq \lambda_{\min}^{-2}(A)\abs{\eta_1-\eta_2}$. As $\pnorm{\Sigma^{-1}}{\op}\lesssim n$ under $\mathcal{H}_\Sigma\leq K$, there exists $C_4=C_4(K)>0$ such that on an event $E_{4}$ with $\Prob^\xi(E_{4})\geq 1-C_4e^{-n/C_4}$, 
\begin{align}\label{ineq:dist_min_norm_7}
\pnorm{ \hat{\mu}_{\eta_1}-\hat{\mu}_{\eta_2} }{}\leq C_4 n^2 \abs{\eta_1-\eta_2}. 
\end{align}
On the other hand, note that for $\eta_1,\eta_2 \in [0,K]$, using Proposition \ref{prop:fpe_est}-(3),
\begin{align}\label{ineq:dist_min_norm_8_0}
&\pnorm{\hat{\mu}_{\eta_1;(\Sigma,\mu_0)}^{\seq,\ast}-\hat{\mu}_{\eta_2;(\Sigma,\mu_0)}^{\seq,\ast} }{}
\lesssim (1\vee e_g)\pnorm{\Sigma^{-1}}{\op}^2\abs{\eta_1-\eta_2}.
\end{align}
So we have
\begin{align}\label{ineq:dist_min_norm_8}
&\bigabs{\E \mathsf{g}_0\big(\hat{\mu}_{\eta_1;(\Sigma,\mu_0)}^{\seq,\ast}\big)- \E \mathsf{g}_0\big(\hat{\mu}_{\eta_2;(\Sigma,\mu_0)}^{\seq,\ast}\big)}\leq C_4 n^2\abs{\eta_1-\eta_2}. 
\end{align}
Now by taking an $\epsilon/(2C_4 n^2)$-net $\Lambda_\epsilon$ of $[0,K]$ and a union bound,
\begin{align}\label{ineq:dist_min_norm_9}
&\Prob^\xi\Big(\sup_{\eta \in [0,K]}\bigabs{\mathsf{g}_0(\hat{\mu}_\eta)-\E \mathsf{g}_0\big(\hat{\mu}_{\eta;(\Sigma,\mu_0)}^{\seq,\ast}\big) }\geq 2\epsilon\Big)\nonumber\\
&\leq \Prob^\xi\Big(\max_{\eta \in \Lambda_\epsilon }\bigabs{\mathsf{g}_0(\hat{\mu}_\eta)-\E \mathsf{g}_0\big(\hat{\mu}_{\eta;(\Sigma,\mu_0)}^{\seq,\ast}\big) }\geq \epsilon\Big)+\Prob(E_4^c)\nonumber\\
&\leq (1+2C_4K n^2/\epsilon )\cdot C_3 n e^{-n\epsilon^4/C_3}+C_4 e^{-n/C_4}\leq C \cdot \epsilon^{-1} n^3 e^{-n\epsilon^4/C}.
\end{align}
By adjusting constants, we may replace $n^3/\epsilon$ by $n$. We then conclude by further taking expectation with respect to $\xi$, and noting that $\Prob(\xi \in \mathscr{E}_{1,\xi}(\sqrt{t/n}))\geq 1- C e^{-t/C}$. 

Next, in the case $\phi^{-1}<1+1/K$, we work with $\eta \in [1/K,K]$ and use the standard form of $\hat{\mu}_\eta$ with $
\hat{\mu}_\eta = n^{-1} \big(X^\top X/n +\eta I\big)^{-1}X^\top Y$. As $\eta\geq 1/K$, the spectrum of the middle inverse matrix is bounded by $1/\eta\leq K$, so we may replicate the above calculations in (\ref{ineq:dist_min_norm_7}) and (\ref{ineq:dist_min_norm_8}) to reach a similar estimate as in (\ref{ineq:dist_min_norm_9}).\qed

\subsection{Proof of Theorem \ref{thm:min_norm_dist} for $\hat{r}_{\eta;G}$}

Recall the cost function $h_\eta=h_{\eta;G},\ell_\eta$ defined in (\ref{def:h_l}). It is easy to see that
\begin{align}\label{def:v_hat}
\hat{v}_\eta \equiv \argmax_{v \in \mathbb{R}^m}  \min_{w \in \mathbb{R}^n}h_\eta(w,v) = \frac{1}{\sqrt{n}\eta}(G\hat{w}_\eta - \xi ) = -\frac{\hat{r}_\eta}{\eta}.
\end{align}
We shall define the `population' version of $\hat{v}_\eta$ as
\begin{align}\label{def:v_eta_ast}
v_{\eta,\ast} \equiv \frac{1}{\phi\tau_{\eta,\ast}}\bigg(\sqrt{\phi\gamma_{\eta,\ast}^2 - \sigma_\xi^2 }\cdot \frac{h}{\sqrt{n}} - \frac{\xi}{\sqrt{n}}  \bigg)
\end{align}
in the Gordon problem.

\begin{proposition}\label{prop:locate v_eta_n}
	Suppose the following hold for some $K>0$.
	\begin{itemize}
		\item $1/K \leq \phi^{-1} ,\eta \leq K$, $\pnorm{\mu_0}{}\vee \pnorm{\Sigma}{}\vee \mathcal{H}_\Sigma\leq K$.
		\item Assumption \ref{assump:noise} holds with $\sigma_\xi^2 \in [1/K,K]$. 
	\end{itemize}
    There exist constants $C,C' > 0$ depending on $K$ such that for $C'\log (en)\le  t \le n/C'$, $\eta \in [1/K,K]$ and  $\xi \in \mathscr{E}_{1,\xi}(\sqrt{t/n})$,
	\begin{align*}
	&\Prob^\xi\bigg(\hbox{The map $v \mapsto \ell_\eta(w_{\eta,\ast};v)$ is $\eta$-strongly concave with unique maximizer $v_{\eta,n}$}\\
	&\qquad \hbox{satisfying $ \|v_{\eta,n} \| \leq C$ and $\|v_{\eta,n} - v_{\eta,\ast}\| \leq C\sqrt{t/n} $.}\\
	&\qquad \hbox{Furthermore, $\bigabs{\max_v \ell_{\eta} (w_{\eta,\ast},v) - \max_{\beta>0}\min_{\gamma >0} \overline{\mathsf{D}}_\eta(\beta,\gamma)} \leq C\sqrt{t/n}$. 
	}\bigg)\geq 1-C e^{-t/C}. 
	\end{align*}
\end{proposition}
We need the following before the proof of Proposition \ref{prop:locate v_eta_n}.
\begin{lemma}\label{lem:w_eta,ast concentr}
	Suppose $1/K\leq \phi^{-1},\sigma_\xi^2\leq K$, and $\pnorm{\mu_0}{}\vee \pnorm{\Sigma}{\op}\vee \mathcal{H}_\Sigma \leq K$ for some $K>0$. Recall $w_{\eta,\ast}$ defined in (\ref{def:w_ast}). Then there exist constants $C,C' > 0$ depending on $K$ such that for $C'\log (en)\le  t \le n/C'$, $\eta \in \Xi_K$ and $\xi \in \mathscr{E}_{1,\xi}(\sqrt{t/n})$, 
	\begin{align*}
	&\Prob^\xi\Big(\max\Big\{  \abs{ \big(\mathrm{id}-\E\big)  \iprod{g/\sqrt{n}}{w_{\eta,\ast}} }, \, \abs{ \big(\mathrm{id}-\E\big)\| w_{\eta,\ast}\|^2}, \, \abs{ \big(\mathrm{id}-\E\big)F(w_{\eta,\ast})},\\
	&\qquad\qquad  n^{-1}\abs{\big(\mathrm{id}-\E\big)\pnorm{\,\|w_{\eta,\ast}\| h-\xi }{}^2 }   \Big\} \geq  \sqrt{t/n }\Big)\leq Ce^{-t/C}.
	\end{align*}
\end{lemma}
\begin{proof}
	All the constants in $\lesssim,\gtrsim,\asymp$ below may depend on $K$. Recall 
	$
	w_{\eta,\ast} = (\Sigma + \tau_{\eta,\ast}I )^{-1} \Sigma^{1/2}(-\tau_{\eta,\ast}\mu_0 + \gamma_{\eta,\ast}\Sigma^{1/2} g/\sqrt{n} )
	$. Under the assumed conditions, $\gamma_{\eta,\ast},\tau_{\eta,\ast}\asymp 1$. We shall consider the four terms separately below.
	
	For the first term, we have
	\begin{align*}
	&n^{-1/2}\abs{ \iprod{g}{w_{\eta,\ast}}- \E \iprod{g}{w_{\eta,\ast}}  } \leq \tau_{\eta,\ast}\cdot n^{-1/2} \abs{\iprod{ (\Sigma + \tau_{\eta,\ast}I )^{-1} \Sigma^{1/2} \mu_0 }{ g }}\\
	&\qquad +\gamma_{\eta,\ast}\cdot  n^{-1} (\mathrm{id}-\E)\pnorm{(\Sigma+\tau_{\eta,\ast}I)^{-1/2}\Sigma^{1/2} g}{}^2 \equiv A_{1,1}+A_{1,2}.
	\end{align*}
	The concentration of the term $A_{1,1}$ can be handled using Gaussian tails and the fact that $\pnorm{(\Sigma + \tau_{\eta,\ast}I )^{-1} \Sigma^{1/2} \mu_0}{}^2\lesssim 1$. For the term $A_{1,2}$, with $H_1(g)\equiv \pnorm{(\Sigma+\tau_{\eta,\ast}I)^{-1/2}\Sigma^{1/2} g}{}^2$, it is easy to evaluate $
	\pnorm{\nabla H_1(g)}{}^2  = 4 \pnorm{ (\Sigma+\tau_{\eta,\ast}I)^{-1}\Sigma g }{ }^2\leq 4 H_1(g)$ and $\E H_1(g)\leq n$, so Proposition \ref{prop:conc_H_generic} applies to conclude the concentration of $A_{1,2}$. 
	
	For the second term, we may decompose 
	\begin{align*}
	&\big| \| w_{\eta,\ast}\|^2- \E  \| w_{\eta,\ast}\|^2 \big|\lesssim {\tau_{\eta,\ast}\gamma_{\eta,\ast}}\cdot n^{-1/2}\abs{\iprod{(\Sigma + \tau_{\eta,\ast}I )^{-2} \Sigma^{3/2}\mu_0}{g}}\\
	&\qquad + \gamma_{\eta,\ast}^2\cdot n^{-1} (\mathrm{id}-\E) \pnorm{(\Sigma+\tau_{\eta,\ast}I)^{-1}\Sigma g}{}^2.
	\end{align*}
	From here we may handle the concentration of the above two terms in a completely similar fashion to $A_{1,1}$ and $A_{1,2}$ above.
	
	For the third term, recall that $
	\hat{\mu}_{(\Sigma,\mu_0)}^{\seq}(\gamma;\tau)= (\Sigma+\tau I)^{-1}\Sigma^{1/2} \big(\Sigma^{1/2}\mu_0+\gamma g/\sqrt{n}\big)$, so 
	\begin{align*}
	&\bigabs{ F(w_{\eta,\ast})-\E F(w_{\eta,\ast})} = \frac{1}{2}\bigabs{\pnorm{\hat{\mu}_{(\Sigma,\mu_0)}^{\seq}(\gamma_{\eta,\ast};\tau_{\eta,\ast})}{}^2-\E \pnorm{\hat{\mu}_{(\Sigma,\mu_0)}^{\seq}(\gamma_{\eta,\ast};\tau_{\eta,\ast})}{}^2 }\\
	& \lesssim \gamma_{\eta,\ast}\cdot n^{-1/2}\abs{\iprod{(\Sigma + \tau_{\eta,\ast}I )^{-2} \Sigma^{3/2}\mu_0}{g}} + \gamma_{\eta,\ast}^2\cdot n^{-1}(\mathrm{id}-\E) \pnorm{(\Sigma+\tau_{\eta,\ast}I)^{-1}\Sigma^{1/2} g}{}^2.
	\end{align*}
	The concentration properties of the two terms on the right hand side above can be handled similarly to the case for the second term.

	For the last term, we have
	\begin{align*}
	&n^{-1}\bigabs{\pnorm{\,\|w_{\eta,\ast}\| h-\xi }{}^2 -  \E \pnorm{\,\|w_{\eta,\ast}\| h-\xi }{}^2 }\\
	&\lesssim n^{-1} \bigabs{ \|w_{\eta,\ast}\|^2 \pnorm{h}{}^2- \E \|w_{\eta,\ast}\|^2 \pnorm{h}{}^2 }+ n^{-1}\|w_{\eta,\ast}\|\abs{ \iprod{h}{\xi}}\equiv A_{4,1}+A_{4,2}. 
	\end{align*}
	On the other hand, on the event $\mathscr{E}_1(\sqrt{t/n})$, 
	\begin{align*}
	A_{4,1}&\lesssim ({\pnorm{h}{}^2}/{n}) \big| \| w_{\eta,\ast}\|^2- \E  \| w_{\eta,\ast}\|^2 \big|+ n^{-1} \E  \| w_{\eta,\ast}\|^2\cdot \abs{\pnorm{h}{}^2-m}\lesssim \sqrt{t/n},
	\end{align*}
	and $
	A_{4,2}\lesssim (1\vee e_g)\cdot n^{-1}\abs{\iprod{h}{\xi}}\lesssim  \sqrt{t/n}$. 
	Combining the above estimates concludes the concentration claim for the last term. 
\end{proof}

\begin{proof}[Proof of Proposition \ref{prop:locate v_eta_n}]
	Fix $\xi \in \mathscr{E}_{1,\xi}(\sqrt{t/n})$. All the constants in $\lesssim,\gtrsim,\asymp$ below may depend on $K$.
	
	\noindent (\textbf{Step 1}). In this step, we establish both the uniqueness and the apriori estimates for $v_{\eta,n}$. Using Lemma \ref{lem:w_eta,ast concentr}, we may choose a sufficiently large $C,C'>0$ depending on $K$ such that $C'\log (en)\le t\le n/C'$,
	\begin{align*}
	&\Prob^\xi \Big(E_0(t)\equiv \Big\{ \max\Big\{  \abs{  (\mathrm{id}-\E)\langle g/\sqrt{n},w_{\eta,\ast} \rangle},\, \abs{ (\mathrm{id}-\E) \| w_{\eta,\ast}\|^2},\, \abs{ (\mathrm{id}-\E)F(w_{\eta,\ast})}, \\
	& \qquad\qquad\qquad  n^{-1}\abs{(\mathrm{id}-\E)\pnorm{\,\|w_{\eta,\ast}\| h-\xi }{}^2 }   \Big\} \leq \sqrt{t/n }\Big\}   \Big) \geq 1 - Ce^{-t/C}.
	\end{align*}
	Therefore, on the event $E_0(t)$,
	\begin{align*}
	\langle g/\sqrt{n},w_{\eta,\ast} \rangle \ge \E \langle g/\sqrt{n},w_{\eta,\ast} \rangle - \sqrt{t/n} = \gamma_{\eta,\ast} \cdot n^{-1}\mathrm{tr} \big((\Sigma + \tau_{\eta,\ast}I)^{-1}\Sigma\big) - \sqrt{t/n}.
	\end{align*}
	Note that $
	n^{-1}\mathrm{tr} \big((\Sigma + \tau_{\eta,\ast}I)^{-1}\Sigma\big)\gtrsim  \mathcal{H}_\Sigma^{-1}\gtrsim 1$, by choosing sufficiently large $C$, we conclude $\langle g/\sqrt{n},w_{\eta,\ast} \rangle > 0$ on the event $E_0(t)$. 
	This implies that $v\mapsto \ell_{\eta}(w_{\eta,\ast},v)$ is $\eta$-strongly concave with respect to $\pnorm{\cdot}{}$, so $v_{\eta,n}$ exists uniquely on $E_0(t)$.
	
	Next we derive apriori estimates. We claim that on $E_0(t)$, $v_{\eta,n}=\argmax_{v \in \mathbb{R}^m} \ell_\eta(w_{\eta,\ast},v)$ takes the following form:
	\begin{align}\label{ineq:locate v_eta_n_1}
	v_{\eta,n}  = \frac{1}{\sqrt{n}\eta} \bigg(1 - \frac{\langle g,w_{\eta,\ast} \rangle}{\pnorm{\,\|w_{\eta,\ast}\| h-\xi}{}} \bigg)_+\cdot \big(\|w_{\eta,\ast}\| h - \xi\big).
	\end{align}
	To see this, using the definition
	\begin{align*}
	v_{\eta,n}&=\argmax_{v \in \R^m} \bigg\{\frac{1}{\sqrt{n}}\Big( -\pnorm{v}{}\iprod{g}{w_{\eta,\ast}}+\pnorm{w_{\eta,\ast}}{}\iprod{h}{v}-\iprod{v}{\xi}\Big)-\frac{\eta\pnorm{v}{}^2}{2}\bigg\}\\
	& = \argmax_{\alpha\geq 0} \bigg\{\frac{\alpha}{\sqrt{n}}\bigg(-\iprod{g}{w_{\eta,\ast} }+ \pnorm{\, \pnorm{w_{\eta,\ast}}{}h-\xi}{}\bigg)-\frac{\eta\alpha^2}{2}\bigg\}\cdot \frac{\|w_{\eta,\ast}\| h - \xi}{ \pnorm{\,\|w_{\eta,\ast}\| h - \xi}{}}\\
	& = \frac{1}{\sqrt{n}\eta} \bigg(-\iprod{g}{w_{\eta,\ast} }+ \pnorm{\, \pnorm{w_{\eta,\ast}}{}h-\xi}{}\bigg)_+\cdot \frac{\|w_{\eta,\ast}\| h - \xi}{ \pnorm{\,\|w_{\eta,\ast}\| h - \xi}{}}.
	\end{align*}
	Some simple algebra leads to the expression in (\ref{ineq:locate v_eta_n_1}). The boundedness of $\|v_{\eta,n} \|$ then follows from the boundedness of $\|w_{\eta,\ast} \|$.

	\noindent (\textbf{Step 2}). In this step, we establish the bound on $\|v_{\eta,n} - v_{\eta,\ast}\|$.  The key observation is that we may rewrite $v_{\eta,\ast}$ defined via (\ref{def:v_eta_ast}) into the following form
	\begin{align}\label{ineq:locate v_eta_n_2}
	v_{\eta,\ast} = \frac{1}{\sqrt{n}\eta} \bigg(1 - \frac{\E \langle g,w_{\eta,\ast} \rangle}{\E^{1/2} \pnorm{\,\|w_{\eta,\ast}\|\cdot h-\xi }{}^2} \bigg)\cdot \big(\E^{1/2}\| w_{\eta,\ast}\|^2 \cdot h - \xi \big). 
	\end{align}
	This can be seen by observing
	\begin{align}\label{ineq:locate v_eta_n_2_1}
	\begin{cases}
	\E  \| w_{\eta,\ast}\|^2  =\E\err_{(\Sigma,\mu_0)}(\gamma_{\eta,\ast};\tau_{\eta,\ast})= \phi\gamma_{\eta,\ast}^2-\sigma_\xi^2,\\
	\E\iprod{g}{w_{\eta,\ast}} = \frac{\sqrt{n}}{\gamma_{\eta,\ast}}\cdot \E\dof_{(\Sigma,\mu_0)}(\gamma_{\eta,\ast};\tau_{\eta,\ast}) = \sqrt{n}\gamma_{\eta,\ast}\cdot\big(\phi-\frac{\eta}{\tau_{\eta,\ast}}\big),\\
	\E^{1/2} \pnorm{\,\|w_{\eta,\ast}\|\cdot h-\xi }{}^2= \sqrt{m}\big(\E  \| w_{\eta,\ast}\|^2+\sigma_\xi^2\big)^{1/2} = \sqrt{m \phi } \gamma_{\eta,\ast},
	\end{cases}
	\end{align}
	and therefore $
	1 - \frac{\E \langle g,w_{\eta,\ast} \rangle}{\E^{1/2} \pnorm{\,\|w_{\eta,\ast}\|\cdot h-\xi }{}^2}   = \frac{\eta}{\phi\tau_{\eta,\ast}}$. 
	Now with (\ref{ineq:locate v_eta_n_1})-(\ref{ineq:locate v_eta_n_2}), we may use Lemma \ref{lem:w_eta,ast concentr} to estimate 
	\begin{align}\label{ineq:locate v_eta_n_3}
	\pnorm{v_{\eta,n} - v_{\eta,\ast}}{}&\leq \frac{1}{\sqrt{n}\eta} \bigabs{\|w_{\eta,\ast}\|-\E^{1/2}\| w_{\eta,\ast}\|^2}\cdot \pnorm{h}{}\nonumber\\
	&\quad +\frac{1}{\sqrt{n}\eta}\biggabs{\frac{\langle g,w_{\eta,\ast} \rangle}{\pnorm{\,\|w_{\eta,\ast}\| h-\xi}{}}- \frac{\E \langle g,w_{\eta,\ast} \rangle}{\E^{1/2} \pnorm{\,\|w_{\eta,\ast}\|\cdot h-\xi }{}^2}  } \pnorm{\E^{1/2}\| w_{\eta,\ast}\|^2 \cdot h - \xi}{} \nonumber\\
	&\equiv V_1+V_2.
	\end{align}
	We first handle the term $V_1$. As $\E \pnorm{w_{\eta,\ast}}{}^2\geq \gamma_{\eta,\ast}^2 \mathrm{tr} \big(\Sigma^2 (\Sigma + \tau_{\eta,\ast})^{-2}\big)/n\gtrsim 1$, on the event $E_0(t)\cap \mathscr{E}_{1,0}(\sqrt{t/n})$, 
	\begin{align}\label{ineq:locate v_eta_n_4}
	V_1\lesssim \frac{\pnorm{h}{}}{\sqrt{n}}\cdot \frac{ \abs{\,\|w_{\eta,\ast}\|^2-\E\| w_{\eta,\ast}\|^2 } }{ \E^{1/2}\| w_{\eta,\ast}\|^2 } \lesssim \sqrt{t/n}.
	\end{align}
	Next we handle $V_2$. On the event $E_0(t)\cap \mathscr{E}_{1,0}(\sqrt{t/n})$, 
	\begin{align}\label{ineq:locate v_eta_n_5}
	V_2 &\lesssim \pnorm{\,\|w_{\eta,\ast}\| h-\xi}{}^{-1}\cdot \bigabs{ \iprod{g}{w_{\eta,\ast}}-\E \iprod{g}{w_{\eta,\ast}} }\nonumber\\
	&\qquad+ \E \iprod{g}{w_{\eta,\ast}}\cdot \bigabs{\,\pnorm{\,\|w_{\eta,\ast}\| h-\xi}{}^{-1}- \E^{-1/2} \pnorm{\,\|w_{\eta,\ast}\|\cdot h-\xi }{}^2  }\nonumber\\
	&\lesssim n^{-1/2} \bigabs{ \iprod{g}{w_{\eta,\ast}}-\E \iprod{g}{w_{\eta,\ast}} }+ n^{-1/2} \bigabs{\,\pnorm{\,\|w_{\eta,\ast}\| h-\xi }{} -  \E^{1/2} \pnorm{\,\|w_{\eta,\ast}\| h-\xi }{}^2 }\nonumber\\
	&\lesssim n^{-1/2} \bigabs{ \iprod{g}{w_{\eta,\ast}}-\E \iprod{g}{w_{\eta,\ast}} }+ n^{-1} \bigabs{\,\pnorm{\,\|w_{\eta,\ast}\| h-\xi }{}^2 -  \E \pnorm{\,\|w_{\eta,\ast}\| h-\xi }{}^2 }\nonumber\\
	&\lesssim \sqrt{t/n}.
	\end{align}
	The desired estimate for $\|v_{\eta,n} - v_{\eta,\ast}\|$ follows from (\ref{ineq:locate v_eta_n_3})-(\ref{ineq:locate v_eta_n_5}).
	
	\noindent (\textbf{Step 3}). In this step, we prove the claimed bound on $|\max_v \ell_\eta (w_{\eta,\ast},v) - \overline{\mathsf{D}}_{\eta}(\beta_{\eta,\ast},\gamma_{\eta,\ast}) |$. First note that
	\begin{align}\label{ineq:locate v_eta_n_6}
	&\max_{v \in \mathbb{R}^m}\ell_\eta(w_{\eta,\ast},v)\nonumber\\
	&\equiv \max_{v \in \mathbb{R}^m}\bigg\{\frac{1}{\sqrt{n}}\Big( -\pnorm{v}{}\iprod{g}{w_{\eta,\ast}}+\pnorm{w_{\eta,\ast}}{}\iprod{h}{v}-\iprod{v}{\xi}\Big)+F(w_{\eta,\ast})-\frac{\eta\pnorm{v}{}^2}{2}\bigg\} \nonumber\\
	&= \frac{1}{2n\eta}\big(\|\, \|w_{\eta,\ast} \| h - \xi \| - \iprod{g}{w_{\eta,\ast}} \big)_+^2 + F(w_{\eta,\ast}).
	\end{align}
	On the other hand, with $\#_{\eta;(\Sigma,\mu_0)}^\ast\equiv\#_{(\Sigma,\mu_0)}(\gamma_{\eta,\ast};\tau_{\eta,\ast})$, $\# \in \{\err,\dof\}$,
	\begin{align*}
	\E \mathsf{e}_F\bigg( \frac{\gamma_{\eta,\ast}}{\sqrt{n}}g;\frac{\gamma_{\eta,\ast}}{\beta_{\eta,\ast}} \bigg)& = \frac{\beta_{\eta,\ast} }{2\gamma_{\eta,\ast}}\big(\E\err_{\eta;(\Sigma,\mu_0)}^\ast -2\E\dof_{\eta;(\Sigma,\mu_0)}^\ast+ \gamma_{\eta,\ast}^2\big)+ \E F(w_{\eta,\ast}),
	\end{align*}
	so we may rewrite $\max_{\beta>0}\min_{\gamma >0} \overline{\mathsf{D}}_\eta(\beta,\gamma)= \overline{\mathsf{D}}_{\eta}(\beta_{\eta,\ast},\gamma_{\eta,\ast}) $ as follows:
	\begin{align*}
	\overline{\mathsf{D}}_{\eta}(\beta_{\eta,\ast},\gamma_{\eta,\ast}) &= \frac{\beta_{\eta,\ast}}{2}\bigg( \gamma_{\eta,\ast}\big(\phi - 1\big) + \frac{\sigma_\xi^2}{\gamma_{\eta,\ast}} \bigg) - \frac{\eta\beta_{\eta,\ast}^2}{2} + \E \mathsf{e}_F\bigg( \frac{\gamma_{\eta,\ast}}{\sqrt{n}}g;\frac{\gamma_{\eta,\ast}}{\beta_{\eta,\ast}} \bigg)\\
	& =  \frac{\beta_{\eta,\ast} }{2\gamma_{\eta,\ast}}\Big( \phi\gamma_{\eta,\ast}^2 + \sigma_\xi^2 + \E\err_{\eta;(\Sigma,\mu_0)}^\ast-2\E\dof_{\eta;(\Sigma,\mu_0)}^\ast\Big)- \frac{\eta \beta_{\eta,\ast}^2}{2} + \E F(w_{\eta,\ast})\\
	& = \frac{\beta_{\eta,\ast} }{\gamma_{\eta,\ast}}\Big( \phi\gamma_{\eta,\ast}^2 -\E\dof_{\eta;(\Sigma,\mu_0)}^\ast\Big)- \frac{\eta \beta_{\eta,\ast}^2}{2} + \E F(w_{\eta,\ast}).
	\end{align*}
	Further using the second and third equations in (\ref{ineq:locate v_eta_n_2_1}), it now follows that 
	\begin{align}\label{ineq:locate v_eta_n_7}
	\max_{\beta>0}\min_{\gamma >0} \overline{\mathsf{D}}_\eta(\beta,\gamma) 	&=  \frac{\beta_{\eta,\ast}}{\sqrt{n}} \Big( \E^{1/2}  \| \,\|w_{\eta,\ast} \| h - \xi \|^2  - \E \iprod{g}{w_{\eta,\ast}}  \Big) - \frac{\eta \beta_{\eta,\ast}^2}{2} + \E F(w_{\eta,\ast}) \nonumber\\
	&= \frac{1}{2n\eta}\Big( \E^{1/2}  \|\, \|w_{\eta,\ast} \| h - \xi \|^2  - \E \iprod{g}{w_{\eta,\ast}} \Big)^2 + \E F(w_{\eta,\ast}).
	\end{align}
	Now combining (\ref{ineq:locate v_eta_n_6}) and (\ref{ineq:locate v_eta_n_7}), on the event $E_0(t)\cap \mathscr{E}_{1,0}(\sqrt{t/n})$, we may estimate 
	\begin{align*}
	&\bigabs{\max_{v \in \R^m} \ell_\eta (w_{\eta,\ast},v) - \max_{\beta>0}\min_{\gamma >0} \overline{\mathsf{D}}_\eta(\beta,\gamma) }\\
	&\lesssim n^{-1/2} \bigabs{ \iprod{g}{w_{\eta,\ast}}-\E \iprod{g}{w_{\eta,\ast}} }+ n^{-1/2} \bigabs{\pnorm{\,\|w_{\eta,\ast}\| h-\xi }{} -  \E^{1/2} \pnorm{\,\|w_{\eta,\ast}\| h-\xi }{}^2 } \\
	&\qquad\qquad + \abs{ F(w_{\eta,\ast})-\E F(w_{\eta,\ast})} \lesssim \sqrt{t/n},
	\end{align*}
	completing the proof. 
\end{proof}

	\begin{proof}[Proof of Theorem \ref{thm:min_norm_dist} for $\hat{r}_\eta$]  	Fix $\xi \in \mathscr{E}_{1,\xi}(\sqrt{t/n})$. All the constants in $\lesssim,\gtrsim,\asymp$ below may depend on $K$. We sometimes write $\overline{\mathscr{D}}_\eta\equiv \max_{\beta>0}\min_{\gamma >0} \overline{\mathsf{D}}_\eta(\beta,\gamma)$.

		As $\hat{r}_\eta = -\eta \hat{v}_{\eta}$, we only need to study $\hat{v}_{\eta}$. Fix $\epsilon > 0$, and any 
		$\mathsf{h}: \mathbb{R}^m \to \mathbb{R}$, let
		\begin{align*}
		D_{\eta;\epsilon} (\mathsf{h}) \equiv \big\{ v\in \mathbb{R}^m: |\mathsf{h}(v) - \E^\xi \mathsf{h}({v}_{\eta,\ast})| \ge \epsilon  \big\}.
		\end{align*}

		\noindent (\textbf{Step 1}). In this step we establish the Gordon cost cap: there exist constants $C_1,C_1'>0$ depending on $K$  such that for $C_1'\log (en)\leq t\leq n/C_1'$, 
	   \begin{align}\label{ineq:dist_res_gaussian_1}
	   \Prob^\xi\Big(E_1(t)^c\equiv \Big\{\max_{v \in D_{\eta;C_1(t/n)^{1/4}}(\mathsf{h}) }\ell_\eta(w_{\eta,\ast},v) \geq \overline{\mathscr{D}}_\eta - C_1^{-1}\sqrt{t/n}\Big\}\Big)\leq C_1e^{-t/C_1}.
	   \end{align}
		To this end, first note that by the Lipschitz property of $\mathsf{h}$, the Gaussian concentration and Proposition \ref{prop:locate v_eta_n}, there exist some $C_0,C_0'>0$ depending on $K$ such that for $C_0' \log (en)\leq t \leq n/C_0'$, on an event $E_{1,0}(t)$ with probability at least $1-C_0 e^{-t/C_0}$, we have uniformly in $v \in \mathsf{D}_{\eta;\epsilon}(\mathsf{h})$,
		\begin{align*}
		\epsilon \leq & |\mathsf{h}(v) - \E^\xi\mathsf{h}({v}_{\eta,\ast})| \leq |\mathsf{h}(v) - \mathsf{h}({v}_{\eta,\ast})| + |\mathsf{h}(v_{\eta,\ast}) - \E^\xi \mathsf{h}({v}_{\eta,\ast})| \\
		\leq& \|v - {v}_{\eta,n} \| + \|v_{\eta,\ast} - v_{\eta,n}  \| + C\sqrt{t/n} \leq  \|v - {v}_{\eta,n} \| + C_0\sqrt{t/n},
		\end{align*}
		and all the properties in Proposition \ref{prop:locate v_eta_n} hold.
		In other word, on $E_{1,0}(t)$ with the prescribed range of $t$, 
		\begin{align*}
		\inf_{v \in D_{\eta;\epsilon}(\mathsf{h})}\pnorm{v - v_{\eta,n}}{}\geq  \big(\epsilon-C_0\sqrt{t/n}\big)_+.
		\end{align*}
		Using the $\eta$-strong concavity of $v\mapsto \ell_\eta(w_{\eta,\ast},v)$ on $E_{1,0}(t)$, we have 
		\begin{align*}
		\max_{v \in D_{\eta;\epsilon}(\mathsf{h}) }\ell_\eta(w_{\eta,\ast},v) &\leq  \ell_\eta(w_{\eta,\ast},v_{\eta,n}) -\frac{\eta}{2}\inf_{v \in \mathsf{D}_{\eta;\epsilon}(\mathsf{h})}\pnorm{v - v_{\eta,n}}{}^2\\
		&  \leq \max_{\beta>0}\min_{\gamma >0} \overline{\mathsf{D}}_\eta(\beta,\gamma) -\frac{\eta}{2}\big(\epsilon-C_0\sqrt{t/n}\big)_+^2 +C_1 \sqrt{t/n}. 
		\end{align*}
		By choosing $
		\epsilon\equiv \epsilon_{\eta;v}(t,n)\equiv  C_0\sqrt{t/n}+2\sqrt{C_1/\eta}\cdot (t/n)^{1/4}$,
		we have on $E_{1,0}(t)$, 
		\begin{align}\label{ineq:dist_res_gaussian_1_0}
		\max_{v \in D_{\eta;\epsilon_{\eta;v}(t,n)}(\mathsf{h}) }\ell_\eta(w_{\eta,\ast},v) \leq \max_{\beta>0}\min_{\gamma >0} \overline{\mathsf{D}}_\eta(\beta,\gamma)- C_1\sqrt{t/n}. 
		\end{align}
		Adjusting constants proves the claim in (\ref{ineq:dist_res_gaussian_1}).
		
		\noindent (\textbf{Step 2}). In this step, we provide an upper bound for the original cost over exceptional set. More concretely, we will prove that there exist constants $C_2,C_2'>0$ depending on $K$ such that for any $L_v>0$, and $C_2'\log (en)\leq t \leq n/C_2'$, 
		\begin{align}\label{ineq:dist_res_gaussian_2}
		&\Prob^\xi \Big(E_2(t)^c\equiv \Big\{\max_{v \in D_{\eta;C_2(t/n)^{1/4}}(\mathsf{h})\cap B_m(L_v)} \min_{w \in \R^n} h_\eta(w,v) \nonumber\\
		&\qquad\qquad \geq  \overline{\mathscr{D}}_\eta- C_2^{-1}\sqrt{t/n}\Big\}\Big)\leq C_2 e^{-t/C_2}.
		\end{align}
		To see this, first note by Proposition \ref{prop:L_minimizer}, there exists some $C_2=C_2(K)>0$ such that on an event $E_{2,0}$ with $\Prob^\xi (E_{2,0})\geq 1-C_2e^{-n/C_2}$, $\pnorm{w_{\eta,\ast}}{}\leq C_2$. So with $
		\bar{z}_{\eta;v}(t,n)\equiv   \max_{\beta>0}\min_{\gamma >0} \overline{\mathsf{D}}_\eta(\beta,\gamma)- C_1^{-1}\sqrt{t/n}$, 
		for any $L_v>0$, an application of Theorem \ref{thm:CGMT}-(1) yields that for $C_1'\log (en)\leq t\leq n/C_1'$,
		\begin{align*}
		&\Prob^\xi \Big(\max_{v \in D_{\eta;C_1(t/n)^{1/4}}(\mathsf{h})\cap B_m(L_v)} \min_{w \in \R^n} h_\eta(w,v) \geq  \bar{z}_{\eta;v}(t,n)\Big)\\
		& \leq \Prob^\xi \Big(\max_{v \in D_{\eta;C_1(t/n)^{1/4}}(\mathsf{h})\cap B_m(L_v)} \min_{w \in B_n(C_2)} h_\eta(w,v) \geq  \bar{z}_{\eta;v}(t,n)\Big) \\
		& \leq 2\Prob^\xi \Big(\max_{v \in D_{\eta;C_1(t/n)^{1/4}}(\mathsf{h})\cap B_m(L_v)} \min_{w \in B_n(C_2)} \ell_\eta(w,v) \geq  \bar{z}_{\eta;v}(t,n)\Big) \\
		& \leq 2\Prob^\xi \Big(\max_{v \in D_{\eta;C_1(t/n)^{1/4}}(\mathsf{h})\cap B_m(L_v)} \ell_\eta(w_{\eta,\ast},v) \geq \bar{z}_{\eta;v}(t,n)\Big)  + 2\Prob^\xi (E_{2,0}^c)\\
		&\leq 2\Prob^\xi \Big(\max_{v \in D_{\eta;C_1(t/n)^{1/4}}(\mathsf{h})} \ell_\eta(w_{\eta,\ast},v) \geq \bar{z}_{\eta;v}(t,n)\Big)  + 2\Prob^\xi (E_{2,0}^c)\leq C e^{-t/C}, 
		\end{align*}
		proving the claim (\ref{ineq:dist_res_gaussian_1}) by possibly adjusting constants. 
	
	    \noindent (\textbf{Step 3}).  In this step, we recall a lower bound for the original cost optimum, essentially established in the Step 1 in the proof of Theorem \ref{thm:min_norm_dist}. In particular, using (\ref{ineq:dist_min_norm_0}), (\ref{ineq:dist_min_norm_00}) and (\ref{ineq:dist_min_norm_4_1}),  there exist $C_3,C_3',C_3''>0$ depending on $K$, such that for $C_3'\log (en)\leq t\leq n/C_3'$, 
	    \begin{align}\label{ineq:dist_res_gaussian_3}
	    \Prob^\xi \Big(E_{3,0}(t)^c\equiv\Big\{\max_{v \in B_m(C_3'')} \min_{w \in \R^n} h_\eta(w,v)  \leq  \overline{\mathscr{D}}_\eta- C_3^{-1}\sqrt{t/n}\Big\}\Big)\leq C_3 e^{-t/C_3},
	    \end{align}
	    and
	    \begin{align}\label{ineq:dist_res_gaussian_4}
	    \Prob^\xi \Big(E_{3,1}^c\equiv \Big\{\max_{v \in B_m(C_3'')} \min_{w \in \R^n} h_\eta(w,v) =\max_{v \in \R^m} \min_{w \in \R^n} h_\eta(w,v)\Big\}\Big)\leq C_3 e^{-n/C_3}.
	    \end{align}
	    
	    \noindent (\textbf{Step 4}). By choosing without loss of generality $C_3>C_2$, on the event $E_2(t)\cap E_{3,0}(t)\cap E_{3,1}$, (\ref{ineq:dist_res_gaussian_2})-(\ref{ineq:dist_res_gaussian_4}) yield that for any $C'\log (en)\leq t\leq n/C'$, 
	    \begin{align*}
	    &\max_{v \in \mathsf{D}_{\eta;C_2(t/n)^{1/4}}(\mathsf{h})\cap B_m(C_3'')} \min_{w \in \R^n} h_\eta(w,v)\leq \max_{\beta>0}\min_{\gamma >0} \overline{\mathsf{D}}_\eta(\beta,\gamma)- C_2^{-1}\sqrt{t/n}\\
	    &< \max_{\beta>0}\min_{\gamma >0} \overline{\mathsf{D}}_\eta(\beta,\gamma)- C_3^{-1}\sqrt{t/n}\leq \max_{v \in B_m(C_3'')} \min_{w \in \R^n} h_\eta(w,v) = \max_{v \in \R^m} \min_{w \in \R^n} h_\eta(w,v). 
	    \end{align*}
	    This means on the event $E_2(t)\cap E_{3,0}(t)\cap E_{3,1}$, $\hat{v}_\eta \notin \mathsf{D}_{\eta;C_2(t/n)^{1/4}}(\mathsf{h})$, i.e., there exist some $C_4,C_4'>0$ depending on $K$ such that for $C_4'\log (en)\leq t\leq n/C_4'$ and $1/K\leq \eta\leq K$, 
	    \begin{align}\label{ineq:dist_res_gaussian_5}
	    \Prob^\xi \Big(\abs{\mathsf{h}(\hat{v}_\eta) - \E^\xi \mathsf{h}(v_{\eta,\ast}) }\geq C_4 (t/n)^{1/4}\Big)\leq C_4 e^{-t/C_4}.
	    \end{align} 
	   \noindent (\textbf{Step 5}). In this final step, we shall prove uniform version of the estimate (\ref{ineq:dist_res_gaussian_5}). For $\eta_1,\eta_2 \in [1/K,K]$, using the definition of $\hat{v}_\eta$ in (\ref{def:v_hat}),
	   \begin{align*}
	   &\abs{\mathsf{h}(\hat{v}_{\eta_1})-  \mathsf{h}(\hat{v}_{\eta_2}) }\leq \pnorm{ \hat{v}_{\eta_1}-\hat{v}_{\eta_2}}{} \\
	   &\leq n^{-1/2}\bigpnorm{\eta_1^{-1}G \hat{w}_{\eta_1}- \eta_2^{-1}G \hat{w}_{\eta_2} }{}+({\pnorm{\xi}{}}/{\sqrt{n}})\cdot \abs{\eta_1^{-1}-\eta_2^{-1}}\\
	   &\leq  \frac{ \pnorm{G\hat{w}_{\eta_1}}{}+\pnorm{\xi}{}}{\sqrt{n}}\cdot \abs{\eta_1^{-1}-\eta_2^{-1}}+ \frac{1}{\sqrt{n}\eta_2} \bigpnorm{G(\hat{w}_{\eta_1}-\hat{w}_{\eta_2})}{}\\
	   &\lesssim \Big(1+\pnorm{\hat{\mu}_{\eta_1}}{}\frac{\pnorm{G}{\op}}{\sqrt{n}}\Big)\cdot \abs{\eta_1-\eta_2}+ \frac{\pnorm{G}{\op}}{\sqrt{n}}\cdot  \pnorm{\hat{\mu}_{\eta_1}-\hat{\mu}_{\eta_2}}{}.
	   \end{align*}
	   Using that $
	   \pnorm{\hat{\mu}_{\eta}}{} = \pnorm{n^{-1}\big(X^\top X/n+\eta I\big)^{-1}X^\top Y}{}\leq  \pnorm{X^\top Y}{}/(n\eta) \lesssim \big(1+{\pnorm{G}{\op}}/{\sqrt{n}}\big)^2$, 
	   we have
	   \begin{align}\label{ineq:dist_res_gaussian_6}
	   \abs{\mathsf{h}(\hat{v}_{\eta_1})-  \mathsf{h}(\hat{v}_{\eta_2}) }\lesssim\big(1+{\pnorm{G}{\op}}/{\sqrt{n}}\big)^3\cdot \big(\abs{\eta_1-\eta_2}\vee\pnorm{\hat{\mu}_{\eta_1}-\hat{\mu}_{\eta_2}}{} \big).
	   \end{align}
	   In view of (\ref{ineq:dist_min_norm_7}), there exists some $C_5>0$ depending on $K$, such that on an event $E_{5,1}$ with $\Prob^\xi(E_{5,1})\geq 1-C_5 e^{-n/C_5}$, 
	   \begin{align}\label{ineq:dist_res_gaussian_7}
	   \abs{\mathsf{h}(\hat{v}_{\eta_1})-  \mathsf{h}(\hat{v}_{\eta_2}) }\leq C_5 n^2 \abs{\eta_1-\eta_2}. 
	   \end{align}
	   On the other hand, using the definition of $v_{\eta,\ast}$ in (\ref{def:v_eta_ast}), Proposition \ref{prop:fpe_est}-(3) and the fact that $\phi\gamma_{\eta,\ast}^2-\sigma_\xi^2 = \E\err_{(\Sigma,\mu_0)}(\gamma_{\eta,\ast};\tau_{\eta,\ast})\geq \tr\big((\Sigma+\tau_{\eta,\ast} I)^{-2}\Sigma^2\big)\gtrsim 1$, we have
	   \begin{align}\label{ineq:dist_res_gaussian_8}
	   &\bigabs{\E^\xi \mathsf{h}(v_{\eta_1,\ast})-\E^\xi \mathsf{h}(v_{\eta_2,\ast})}\leq \E^{1/2, \xi} \pnorm{ v_{\eta_1,\ast}-v_{\eta_2,\ast} }{}^2\nonumber\\
	   &\lesssim \bigabs{{\tau_{\eta_1,\ast}^{-1}} \sqrt{\phi\gamma_{\eta_1,\ast}^2-\sigma_\xi^2 } -  {\tau_{\eta_2,\ast}^{-1}} \sqrt{\phi\gamma_{\eta_2,\ast}^2-\sigma_\xi^2 }}+ \bigabs{\tau_{\eta_1,\ast}^{-1}-\tau_{\eta_2,\ast}^{-1}}\nonumber\\
	   &\lesssim \bigabs{\gamma_{\eta_1,\ast}^2-\gamma_{\eta_2,\ast}^2  }+ \bigabs{\tau_{\eta_1,\ast}^{-1}-\tau_{\eta_2,\ast}^{-1}}\leq C_5 \abs{\eta_1-\eta_2}.
	   \end{align}
	   Now we may mimic the proof in (\ref{ineq:dist_min_norm_9}) to conclude that, by possibly enlarging $C_5>0$, for any $\epsilon \in (0,1/2]$ and $\xi \in \mathscr{E}_{1,\xi}(\epsilon^2/C_5)$, 
	   \begin{align*}
	   \Prob^\xi \Big(\sup_{\eta \in [1/K,K]}\abs{\mathsf{h}(\hat{v}_\eta) - \E^\xi \mathsf{h}(v_{\eta,\ast}) }\geq \epsilon \Big)\leq C_5 n e^{-n\epsilon^4/C_5},
	   \end{align*}
	   as desired.
	\end{proof}

\section{Universality: Proof of Theorem \ref{thm:universality_min_norm}}\label{section:proof_general_design}

\subsection{Comparison inequalities}

For $\mathsf{f}:\R^n\to \R$, let
\begin{align*}
\mathcal{H}_{\mathsf{f}}(w,A)\equiv \frac{1}{2n}\pnorm{Aw-\xi}{}^2+\mathsf{f}(w). 
\end{align*}

The following theorem is proved in \cite[Theorem 2.3]{han2023universality}. 

\begin{theorem}\label{thm:universality_smooth}
	Suppose $1/K\leq \phi^{-1}\leq K$ for some $K>1$. Let $A_0,B_0 \in \R^{m\times n}$ be two random matrices with independent components, such that $\E A_{0;ij}=\E B_{0;ij}=0$ and $\E A_{0;ij}^2 = \E B_{0;ij}^2$ for all $i \in [m], j \in [n]$. Further assume that 
	\begin{align*}
	M\equiv  \max_{i\in[m],j\in[n]}\big(\E \abs{A_{0;ij}}^{6}+\E \abs{B_{0;ij}}^{6}\big)<\infty.
	\end{align*}
	Let $A\equiv A_0/\sqrt{n}$ and $B\equiv B_0/\sqrt{n}$. Then there exists some $C_0=C_0(K,M)>0$ such that the following hold: For any $\mathcal{S}_n \subset [-L_n,L_n]^n$ with $L_n\geq 1$, and any $\mathsf{T} \in C^3(\R)$, we have
	\begin{align*}
	&\Big|\E \mathsf{T}\Big(\min_{w \in \mathcal{S}_n } \mathcal{H}_{\mathsf{f}}(w,A)\Big) - \E  \mathsf{T}\Big(\min_{w \in \mathcal{S}_n } \mathcal{H}_{\mathsf{f}}(w,B)\Big)  \Big|\leq C_0\cdot K_{\mathsf{T}}\cdot \mathsf{r}_{\mathsf{f}}(L_n).
	\end{align*}
	Here $K_{\mathsf{T}}\equiv 1+\max_{\ell \in [0:3]} \pnorm{\mathsf{T}^{(\ell)}}{\infty}$, and $\mathsf{r}_{\mathsf{f}}(L_n)$ is defined by
	\begin{align*}
	\mathsf{r}_{\mathsf{f}}(L_n)&\equiv \inf_{\delta \in (0,n^{-5/2})}\bigg\{ \mathscr{N}_{\mathsf{f}}(L_n,\delta) +\bigg(1+\frac{1}{m}\sum_{i=1}^m \E\abs{\xi_i}^3\bigg)^{1/3}\cdot \frac{L_n^{2}  \log_+^{2/3}(L_n/\delta)}{n^{1/6} }\bigg\},
	\end{align*}
	where $
	\mathscr{N}_{\mathsf{f}}(L_n,\delta)\equiv \sup\, \abs{\mathsf{f}(w)-\mathsf{f}(w') }$
	with the supremum taken over all $w,w' \in [-L_n,L_n]^n$ such that $\pnorm{w-w'}{\infty}\leq \delta$. Consequently, for any $z\in \R,\epsilon>0$,
	\begin{align*}
	&\Prob\Big(\min_{w \in \mathcal{S}_n } \mathcal{H}_{\mathsf{f}}(w,A)>z+3\epsilon \Big)\leq \Prob\Big(\min_{w \in \mathcal{S}_n } \mathcal{H}_{\mathsf{f}}(w,B)>z+\epsilon \Big)+ C_1(1\vee \epsilon^{-3}) \mathsf{r}_{\mathsf{f}}(L_n).
	\end{align*}
	Here $C_1>0$ is an absolute multiple of $C_0$. 
\end{theorem}

Let for $u \in \R^m, w \in \R^n, A \in \R^{m\times n}$ and a measurable function $Q: \R^m\times \R^n \to \R$
\begin{align}
X(u,w;A)\equiv  u^\top A w + Q(u,w).
\end{align}
The following theorem is proved in \cite[Theorem 2.5]{han2023universality}. 
\begin{theorem}\label{thm:min_max_universality}
	Let $A,B \in \R^{m\times n}$ be two random matrices with independent entries and matching first two moments, i.e., $\E A_{ij}^\ell = \E B_{ij}^\ell$ for all $i \in [m], j\in [n],\ell=1,2$. There exists a universal constant $C_0>0$ such that the following hold. For any measurable subsets $\mathcal{S}_u \subset [-L_u,L_u]^m$, $\mathcal{S}_w \subset [-L_w,L_w]^n$ with $L_u,L_w\geq 1$, and any $\mathsf{T} \in C^3(\R)$, we have
	\begin{align*}
	&\Big|\E \mathsf{T}\Big(\max_{u \in \mathcal{S}_u} \min_{w \in \mathcal{S}_w} X(u,w;A)\Big)-\E \mathsf{T}\Big(\max_{u \in \mathcal{S}_u} \min_{w \in \mathcal{S}_w} X(u,w;B)\Big)\Big|\\
	&\leq C_0\cdot K_{\mathsf{T}}\cdot \inf_{\delta \in (0,1)} \Big\{ M_1 L\delta+ \mathscr{N}_Q(L,\delta)+ \log_+^{2/3}(L/\delta)\cdot (m+n)^{2/3} M_3^{1/3}L^2\Big\}.
	\end{align*}
	Here $K_{\mathsf{T}}\equiv 1+\max_{\ell \in [0:3]} \pnorm{\mathsf{T}^{(\ell)}}{\infty}$, $L\equiv L_u+L_w$, $M_\ell\equiv \sum_{i \in [m], j \in [n]} \big(\E \abs{A_{ij}}^\ell+\E\abs{B_{ij}}^\ell\big)$, and $
	\mathscr{N}_Q(L,\delta)\equiv \sup\,\abs{Q(u,w)-Q(u',w')}$
	with the supremum taken over all $u,u' \in [-L,L]^m, w,w' \in [-L,L]^n$ such that $\pnorm{u-u'}{\infty}\vee \pnorm{w-w'}{\infty}\leq \delta$.
	The conclusion continues to hold when max-min is flipped to min-max. 
\end{theorem}

\subsection{Delocalization}

Recall that $\hat{\mu}_\eta$ defined in (\ref{def:minimum_norm_est}) can be rewritten as 
\begin{align*}
\hat{\mu}_\eta=\argmin_{\mu \in \R^n} \max_{v \in \R^m}\bigg\{\frac{1}{2}\pnorm{\mu}{}^2+\frac{1}{\sqrt{n}}\iprod{v}{X\mu-Y}-\frac{\eta}{2}\pnorm{v}{}^2\bigg\}.
\end{align*}
For any $\eta>0$, we have the following closed form for $\hat{\mu}_\eta$:
\begin{align}\label{eqn:u_v_form}
\hat{\mu}_\eta = n^{-1}\big({X^\top X}/{n}+\eta I_n\big)^{-1}X^\top Y,\quad \hat{v}_\eta = -(\sqrt{n}\eta)^{-1}(Y-X\hat{\mu}_\eta). 
\end{align}
The above formula does not include the interpolating case $\eta=0$ when $n>m$. To give an alternative expression, note that the first-order condition for the above minimax optimization is $
\hat{\mu}_\eta = X^\top \hat{v}_\eta/\sqrt{n}$, $Y - X\hat{\mu}_\eta=-\sqrt{n}\eta \hat{v}_\eta$, or equivalently,
\begin{align}\label{eqn:u_v_alter_form}
\hat{\mu}_\eta=n^{-1}X^\top \big({XX^\top}/{n}+\eta I_m\big)^{-1} Y,\quad \hat{v}_\eta=-n^{-1/2}\big({XX^\top}/{n}+\eta I_m\big)^{-1}Y. 
\end{align}
The following proposition proves delocalization for $\hat{w}_\eta\equiv \Sigma^{1/2}(\hat{\mu}_\eta-\mu_0)$ and $\hat{v}_\eta$.

\begin{proposition}\label{prop:delocal_u_v}
Suppose Assumption \ref{assump:design} holds and the following hold for some $K>0$.
\begin{itemize}
	\item $1/K\leq \phi^{-1}\leq K$, $\pnorm{\Sigma^{-1}}{\op}\vee \pnorm{\Sigma}{\op} \leq K$.
	\item Assumption \ref{assump:noise} holds with $\sigma_\xi^2 \in [1/K,K]$. 
\end{itemize}
Fix $\vartheta \in (0,1/2]$. Then there exist some constant $C=C(K,\vartheta)>0$, two measurable sets $\mathcal{U}_\vartheta\subset B_n(1),\mathcal{E}_\vartheta\subset \R^m$ with $\min\{\mathrm{vol}(\mathcal{U}_\vartheta)/\mathrm{vol}(B_n(1)),\Prob(\xi \in \mathcal{E}_\vartheta)\}\geq 1-Ce^{-n^{2\vartheta}/C}$, such that
\begin{align*}
\sup_{\mu_0 \in \mathcal{U}_\vartheta,\xi \in \mathcal{E}_\vartheta} \Prob^\xi\Big(\sup_{\eta \in \Xi_K}\Big\{\pnorm{\hat{w}_\eta}{\infty}\vee \pnorm{\hat{v}_\eta}{\infty}\Big\}\geq C n^{-1/2+\vartheta}\Big)\leq C n^{-100}. 
\end{align*}
The sets $\mathcal{U}_\vartheta, \mathcal{E}_\vartheta$ can be taken as 
\begin{align*}
\mathcal{U}_\vartheta&\equiv \Big\{\mu_0 \in B_n(1): \sup_{\eta \in \Xi_K}\bigpnorm{\Sigma^{1/2}\big(\E\hat{\mu}_{(\Sigma,\mu_0)}^{\seq}(\gamma_{\eta,\ast};\tau_{\eta,\ast})-\mu_0\big)}{\infty}\leq C_0 n^{-1/2+\vartheta}\Big\},\\
\mathcal{E}_\vartheta&\equiv \Big\{\xi \in \R^m: \pnorm{\xi}{\infty}\leq C_0 n^\vartheta, \bigabs{\pnorm{\xi}{}^2/m-\sigma_\xi^2}\leq C_0 n^{-1/2+\vartheta} \Big\}
\end{align*}
for some large enough $C_0=C_0(K)>0$.
\end{proposition}

\begin{remark}
Proposition \ref{prop:delocal_u_v} formalizes the delocalization required by our comparison argument: uniformly over $\eta\in\Xi_K$, 
$\hat w_\eta=\Sigma^{1/2}(\hat\mu_\eta-\mu_0)$ (and likewise $\hat v_\eta$) is small in $\ell_\infty$. The set $\mathcal U_\vartheta$ encodes this coordinatewise control via the sequence model proxy, thereby ruling out highly localized signals. The event $\mathcal E_\vartheta$ imposes mild noise concentration
in sup-norm and empirical variance; for i.i.d. sub-gaussian coordinates, this event holds with overwhelming probability.
\end{remark}

\begin{remark}
Delocalization in the same sense of the above proposition holds for $\pnorm{\mathsf{P}\hat{\mu}_\eta+\mathsf{q}}{\infty}$ with any deterministic matrix $\mathsf{P}\in \R^{n\times n}$ and vector $\mathsf{q} \in \R^n$ satisfying $\pnorm{\mathsf{P}}{\op}\vee \pnorm{\mathsf{q}}{}\leq 1$, with a (slightly) different construction of $\mathcal{U}_\vartheta$. 
\end{remark}

\begin{proof}[Proof of Proposition \ref{prop:delocal_u_v}]
All the constants in $\lesssim,\gtrsim,\asymp$ below may depend on $K$.
	
\noindent (1). Let us consider delocalization for $\hat{w}_\eta$. Using (\ref{eqn:u_v_alter_form}), for any $s \in [n]$,
\begin{align}\label{ineq:delocal_u_v_1}
\iprod{e_s}{\hat{w}_\eta}& = n^{-1}\iprod{\Sigma^{1/2}e_s}{X^\top (\phi \check{\Sigma}+\eta I_m)^{-1} X\mu_0}-\iprod{\Sigma^{1/2}e_s}{\mu_0}\nonumber\\
&\qquad + n^{-1}\iprod{\Sigma^{1/2}e_s}{X^\top (\phi \check{\Sigma}+\eta I_m)^{-1} \xi}\equiv A_{1;s}+A_{2;s}. 
\end{align}
We first handle $A_{1;s}$. Let $\rho$ be the asymptotic eigenvalue density of $\check{\Sigma}=XX^\top /m$ and fix $c>0$. By \cite[Theorem 3.16-(i), Remark 3.17 and Lemma 4.4-(i)]{knowles2017anisotropic}, for any small $\vartheta>0$ and large $D>0$,
\begin{align*}
&\Prob^\xi\Big(\bigabs{m^{-1}\iprod{\Sigma^{1/2} e_s}{X^\top (\check{\Sigma}-z I_m)^{-1} X\mu_0}\\
&\qquad - \iprod{\Sigma^{1/2}e_s}{\mathfrak{m}(z)\Sigma(I_n+\mathfrak{m}(z)\Sigma)^{-1} \mu_0 }}\geq n^{-1/2+\vartheta}\sqrt{\Im\mathfrak{m}(z)/\Im z}\Big)\leq C n^{-D}
\end{align*}
holds for all $z \in [-1/c,1/c]\times (0,1/c]$. With $\kappa \equiv \kappa(z) \equiv \mathrm{dist}(\Re z, \mathrm{supp}\;\rho) \ge n^{-2/3+c}$, by further using the simple relation $\Im m(z)/ \Im z=\int \frac{\rho(\d x)}{(\Re z-x)^2+\Im^2 z}\leq \kappa^{-2}$, the error bound $n^{-1/2+\vartheta}\sqrt{\Im\mathfrak{m}(z)/\Im z}$ in the above display can be replaced by $\kappa^{-1} n^{-1/2+\vartheta}$.

When $\phi^{-1} \ge 1 + 1/K$, according to \cite[Theorem 6.3-(2)]{bai2010spectral}, $\mathrm{supp}\;\rho \in (C_0^{-1},C_0)$ for some constant $C_0 > 1$. Therefore, for $z\equiv z(b) \equiv -\eta/\phi + \sqrt{-1}b $ with a small enough $b > 0$ to be chosen later, it is easy to see that $\kappa\geq \kappa_0\equiv (\eta/\phi)\vee C_0^{-1}\bm{1}_{\phi^{-1}\geq 1+1/K}$. Therefore, on an event $E_{1,0;s}(b)$ with $\Prob^\xi(E_{1,0;s}(b))\geq 1- C n^{-D}$, 
\begin{align}\label{ineq:delocal_u_v_2}
&\bigabs{m^{-1}\iprod{\Sigma^{1/2} e_s}{X^\top (\check{\Sigma}-z(0) I_m)^{-1} X\mu_0} \nonumber\\
	&\qquad - \iprod{\Sigma^{1/2}e_s}{\mathfrak{m}(z(0))\Sigma(I_n+\mathfrak{m}(z(0))\Sigma)^{-1} \mu_0 }}\leq (I)+(II)+\kappa_0^{-1} n^{-1/2+\vartheta},
\end{align}
where
\begin{itemize}
	\item $(I)= \abs{ m^{-1}\iprod{\Sigma^{1/2} e_s}{X^\top (\check{\Sigma}-z(b) I_m)^{-1} X\mu_0}-m^{-1}\iprod{\Sigma^{1/2} e_s}{X^\top (\check{\Sigma}-z(0) I_m)^{-1} X\mu_0} }$,
	\item $(II)= \abs{\iprod{\Sigma^{1/2}e_s}{\mathfrak{m}(z(b))\Sigma(I_n+\mathfrak{m}(z(b))\Sigma)^{-1} \mu_0 }-\iprod{\Sigma^{1/2}e_s}{\mathfrak{m}(z(0))\Sigma(I_n+\mathfrak{m}(z(0))\Sigma)^{-1} \mu_0 } }$.
\end{itemize}
By a derivative calculation, it is easy to derive
\begin{align*}
(I)&\lesssim \big(\pnorm{Z}{\op}/\sqrt{n}\big)^2\cdot  \big(\pnorm{(ZZ^\top/n)^{-1}}{\op}\bm{1}_{\phi^{-1}\geq 1+1/K}\wedge \eta^{-1}\big)^2\cdot b.
\end{align*}
Now by using the concentration result in \cite[Theorem 1.1]{rudelson2009smallest}, on an event $E_{1,1;s}$ with $\Prob^\xi(E_{1,1;s})\geq 1- e^{-n/C}$, we have $(I)\leq C b$.

For $(II)$, using the boundedness of $\mathfrak{m}(z(b))$ around $0$ for $\phi^{-1}\geq 1+1/K$, we may estimate
\begin{align*}
(II)&\lesssim \big(\bm{1}_{\phi^{-1}\geq 1+1/K}\wedge \eta^{-1}\big)\cdot \abs{\mathfrak{m}(z(b))-\mathfrak{m}(z(0))}\\
&\leq  \big(\bm{1}_{\phi^{-1}\geq 1+1/K}\wedge \eta^{-1}\big)\cdot \int_{C_0^{-1}\mathbf{1}_{\phi^{-1} \geq 1 + 1/K}}^\infty \frac{b}{|x - z(b)||x-z(0)|}\,  \rho(\mathrm{d}x) \\
&  \leq  \Big\{C_0^2\mathbf{1}_{\phi^{-1} \geq 1 + 1/K}^{-1} \wedge  \eta^{-3}\Big\}\cdot  b.
\end{align*}
Combining the above estimates, for $b$ chosen small enough, say, $b=n^{-100}$, on the event $E_{1,0;s}(n^{-100})\cap E_{1,1;s}$, 
\begin{align*}
\abs{A_{1;s}-  \iprod{\Sigma^{1/2}e_s}{\mathfrak{m}(-\eta/\phi)\Sigma(I_n+\mathfrak{m}(-\eta/\phi)\Sigma)^{-1} \mu_0 -\mu_0}  }\lesssim n^{-1/2+\vartheta}.
\end{align*}
Using $\tau_{\eta,\ast}^{-1} = \mathfrak{m}(-\eta/\phi)$ and the definition of $\hat{\mu}_{(\Sigma,\mu_0)}^{\seq}(\gamma_{\eta,\ast};\tau_{\eta,\ast})$, recall $w_{\eta,\ast}=\Sigma^{1/2}\big(\hat{\mu}_{(\Sigma,\mu_0)}^{\seq}(\gamma_{\eta,\ast};\tau_{\eta,\ast})-\mu_0\big)$ defined in (\ref{def:w_ast}), we then have
\begin{align}\label{ineq:delocal_u_v_3}
\sup_{\mu_0 \in B_n(1)}\Prob^\xi\Big(\max_{s \in [n]}\abs{A_{1;s}-\bigiprod{e_s}{\E w_{\eta,\ast}}} \geq C n^{-1/2+\vartheta}\Big)\leq Cn^{-D}.
\end{align}
The term $A_{2;s}$ can be handled similarly, now reading off the $(1,2)$ element in \cite[Eqn. (3.10)]{knowles2017anisotropic}, which shows that for any $\xi \in \R^m$, 
\begin{align}\label{ineq:delocal_u_v_4}
\Prob^\xi\Big(\max_{s \in [n]}\abs{A_{2;s}}\geq  C(\pnorm{\xi}{}/\sqrt{m})\cdot n^{-1/2+\vartheta}\Big)\leq C n^{-D}.
\end{align}
Combining (\ref{ineq:delocal_u_v_1}), (\ref{ineq:delocal_u_v_3}) and (\ref{ineq:delocal_u_v_4}), we have
\begin{align}\label{ineq:delocal_u_v_5}
\sup_{\mu_0 \in B_n(1),\xi \in \mathcal{E}_\vartheta}\Prob^\xi\Big(\pnorm{\hat{w}_\eta}{\infty}\geq \pnorm{\E w_{\eta,\ast}}{\infty}+C n^{-1/2+\vartheta}\Big)\leq C n^{-D}. 
\end{align}
Now we will construct $\mathcal{U}_\vartheta\subset B_n(1)$ with the desired volume estimate, and $\sup_{\mu_0 \in \mathcal{U}_\vartheta}\sup_{\eta \in \Xi_K}\pnorm{\E w_{\eta,\ast}}{\infty}\leq Cn^{-1/2+\vartheta}$. To this end, we place a uniform prior on $\mu_0\sim U_0g_0/\pnorm{g_0}{}$, where $U_0\sim \mathrm{Unif}[0,1]$ and $g_0\sim \mathcal{N}(0,I_n)$ are independent of all other random variables. Then $\sup_{\eta \in \Xi_K} \pnorm{\E w_{\eta,\ast}}{\infty} \leq \sup_{\eta \in \Xi_K} \tau_{\eta,\ast}\pnorm{(\Sigma+\tau_{\eta,\ast} I_n)^{-1}\Sigma^{1/2} g_0}{\infty}/\pnorm{g_0}{}$. Using Proposition \ref{prop:fpe_est}-(3) and a standard Gaussian tail bound,  $\Prob_{\mu_0}\big(\mathcal{U}_{\vartheta}\equiv\big\{\sup_{\eta \in \Xi_K}\pnorm{\E w_{\eta,\ast}}{\infty}\geq C_1 n^{-1/2+\vartheta}\big\}\big)\leq C e^{-n^{2\vartheta}/C}$. Moreover, $\Prob(\xi \notin \mathcal{E}_\vartheta)\leq e^{-n^{2\vartheta}/C}$. The pointwise-in-$\eta$ delocalization claim on $\hat{w}_\eta$ follows. As $\eta\mapsto \pnorm{\hat{w}_\eta}{\infty}$ is $C$-Lipschitz with exponentially high probability, the uniform version follows by a standard discretization and union bound argument.

\noindent (2). Let us consider delocalization for $\hat{v}_\eta$. Using again (\ref{eqn:u_v_alter_form}), for any $t \in [m]$,
\begin{align*}
-\iprod{e_t}{\hat{v}_\eta}& = n^{-1/2}\iprod{e_t}{(\phi \check{\Sigma}+\eta I_m)^{-1}X\mu_0}+ n^{-1/2}\iprod{e_t}{(\phi \check{\Sigma}+\eta I_m)^{-1}\xi}\\
& \equiv B_{1;t}+B_{2;t}. 
\end{align*}
The term $B_{1;t}$ can be handled, by reading off the $(2,1)$ element in \cite[Eqn. (3.10)]{knowles2017anisotropic}, which shows that 
\begin{align}\label{ineq:delocal_u_v_6}
\sup_{\mu_0 \in B_n(1)}\Prob^\xi\Big(\max_{t \in [m]}\abs{B_{1;t}}\geq  C n^{-1/2+\vartheta}\Big)\leq C n^{-D}.
\end{align}
The term $B_{2;t}$ relies on the local law described by the $(2,2)$ element in \cite[Eqn. (3.10)]{knowles2017anisotropic}: for any $\xi \in \R^m$, 
\begin{align}\label{ineq:delocal_u_v_7}
\Prob^\xi\Big(\max_{t \in [m]}\abs{B_{2;t}-\phi^{-1}\mathfrak{m}(-\eta/\phi)\xi_t}\geq  C (\pnorm{\xi}{}/\sqrt{m})\cdot n^{-1/2+\vartheta}\Big)\leq C n^{-D}.
\end{align}
Consequently, combining (\ref{ineq:delocal_u_v_6})-(\ref{ineq:delocal_u_v_7}), we have
\begin{align*}
\sup_{\mu_0 \in B_n(1),\xi \in \mathcal{E}_\vartheta}\Prob^\xi\Big(\pnorm{\hat{v}_\eta}{\infty}\geq C n^{-1/2+\vartheta}\Big)\leq C n^{-D}. 
\end{align*}
The claim follows.
\end{proof}

\subsection{Universality of the global cost optimum}

\begin{theorem}\label{thm:universality_global_cost}
Suppose Assumption \ref{assump:design} holds and the following hold for some $K>0$.
\begin{itemize}
	\item $1/K\leq \phi^{-1}\leq K$, $\pnorm{\Sigma}{\op}\vee \pnorm{\Sigma^{-1}}{\op} \leq K$.
	\item Assumption \ref{assump:noise} holds with $\sigma_\xi^2 \in [1/K,K]$. 
\end{itemize}
Fix $\vartheta \in (0,1/18)$. There exists some $C=C(K,\vartheta)>0$ such that for $\rho_0 \leq 1/C$, $\eta \in \Xi_K$ and $\xi \in \mathcal{E}_\vartheta$, 
\begin{align*}
\sup_{\mu_0 \in \mathcal{U}_\vartheta}\Prob^\xi\Big(\bigabs{\min_{w \in \R^n} H_{\eta;Z}(w)- \max_{\beta>0}\min_{\gamma >0}\overline{\mathsf{D}}_\eta(\beta,\gamma)} \geq \rho_0\Big)\leq C \rho_0^{-3}\cdot  n^{-1/6+3\vartheta}.
\end{align*}
Here $\mathcal{U}_\vartheta$ is specified as in Proposition \ref{prop:delocal_u_v}.
\end{theorem}
\begin{proof}
Fix $\vartheta>0$, $\mu_0 \in \mathcal{U}_\vartheta$ and $\xi \in \mathcal{E}_\vartheta$ as specified in Proposition \ref{prop:delocal_u_v}. Let $L_n\equiv C_0n^\vartheta$. By the same proposition, with $\Prob^\xi$-probability at least $1-C_0n^{-100}$, 
\begin{align}\label{ineq:gordon_cost_universality_1}
&\min_{w \in \R^n} H_{\eta;Z}(w)  = \min_{\pnorm{w}{\infty}\leq L_n/\sqrt{n}}\max_{\pnorm{v}{\infty}\leq L_n/\sqrt{n}}\bigg\{\frac{1}{\sqrt{n}}\iprod{v}{Zw}-\frac{1}{\sqrt{n}}\iprod{v}{\xi}-\frac{\eta}{2}\pnorm{v}{}^2+F(w)\bigg\}\nonumber\\
& = \min_{\pnorm{\tilde{w}}{\infty}\leq L_n}\max_{\pnorm{\tilde{v}}{\infty}\leq L_n}\bigg\{\frac{1}{n^{3/2}}\iprod{\tilde{v}}{Z\tilde{w}}-\frac{1}{n}\iprod{\tilde{v}}{\xi}-\frac{\eta}{2n}\pnorm{\tilde{v}}{}^2+F(\tilde{w}/\sqrt{n})\bigg\},
\end{align}
and
\begin{align}\label{ineq:gordon_cost_universality_2}
&\min_{w \in \R^n} H_{\eta;G}(w) =\min_{\pnorm{\tilde{w}}{\infty}\leq L_n}\max_{\pnorm{\tilde{v}}{\infty}\leq L_n}\bigg\{\frac{1}{n^{3/2}}\iprod{\tilde{v}}{G\tilde{w}}-\frac{1}{n}\iprod{\tilde{v}}{\xi}-\frac{\eta}{2n}\pnorm{\tilde{v}}{}^2+F(\tilde{w}/\sqrt{n})\bigg\}.
\end{align}
By writing $Q(\tilde{v},\tilde{w})\equiv -\frac{1}{n}\iprod{\tilde{v}}{\xi}-\frac{\eta}{2n}\pnorm{\tilde{v}}{}^2+F(\tilde{w}/\sqrt{n})$, we have
\begin{align*}
\mathscr{N}_Q(L,\delta)&\equiv \sup_{\substack{\pnorm{\tilde{v} }{\infty}\vee \pnorm{\tilde{v}'}{\infty}\leq L, \pnorm{\tilde{v}-\tilde{v}'}{\infty}\leq \delta,\nonumber\\
\pnorm{\tilde{w}}{\infty}\vee \pnorm{\tilde{w}'}{\infty}\leq L, \pnorm{\tilde{w}-\tilde{w}'}{\infty}\leq \delta }} \bigabs{Q(\tilde{v},\tilde{w})-Q(\tilde{v}',\tilde{w}')}\lesssim_K  (1\vee L) \delta\cdot \bigg(1+ \frac{\pnorm{\xi}{1}}{n}\bigg). 
\end{align*}
Now with $X_Q(\tilde{v},\tilde{w};Z)\equiv n^{-3/2}\iprod{\tilde{v}}{Z\tilde{w}}+Q(\tilde{v},\tilde{w})$, for $\xi \in \mathcal{E}_\vartheta$, by applying Theorem \ref{thm:min_max_universality}, we have for any $\mathsf{T} \in C^3(\R)$,
\begin{align}\label{ineq:gordon_cost_universality_3}
&\Big| \E^\xi \mathsf{T}\Big(\min_{\pnorm{\tilde{w}}{\infty}\leq L_n}\max_{\pnorm{\tilde{v}}{\infty}\leq L_n}X_Q(\tilde{v},\tilde{w};Z)\Big)  -  \E^\xi \mathsf{T}\Big(\min_{\pnorm{\tilde{w}}{\infty}\leq L_n}\max_{\pnorm{\tilde{v}}{\infty}\leq L_n}X_Q(\tilde{v},\tilde{w};G)\Big)  \Big|\nonumber\\
&\lesssim_{K} K_{\mathsf{T}}\cdot \inf_{\delta \in (0,1)}\Big\{\sqrt{n}L_n\delta+L_n\delta+\log_+^{2/3}(L_n/\delta)\cdot  n^{-1/6} L_n^2\Big\} \leq C_1\cdot K_{\mathsf{T}}\cdot  n^{-1/6+3\vartheta}. 
\end{align}
Replicating the last paragraph of proof of \cite[Theorem 2.3]{han2023universality} (right above Section 4.3 therein), for any $z>0,\rho_0>0$,
\begin{align*}
&\Prob^\xi\Big(\min_{\pnorm{\tilde{w}}{\infty}\leq L_n}\max_{\pnorm{\tilde{v}}{\infty}\leq L_n}X_Q(\tilde{v},\tilde{w};Z)> z+3\rho_0\Big)\\
&\leq \Prob^\xi\Big(\min_{\pnorm{\tilde{w}}{\infty}\leq L_n}\max_{\pnorm{\tilde{v}}{\infty}\leq L_n}X_Q(\tilde{v},\tilde{w};G)>z+\rho_0\Big)+ C \rho_0^{-3}  n^{-1/6+3\vartheta}.
\end{align*}
Combined with (\ref{ineq:gordon_cost_universality_1})-(\ref{ineq:gordon_cost_universality_2}), we have 
\begin{align*}
\Prob^\xi\Big(\min_{w \in \R^n} H_{\eta;Z}(w) >z+3\rho_0\Big)&\leq \Prob^\xi\Big(\min_{w \in \R^n} H_{\eta;G}(w) >z+\rho_0\Big)+C_2 \rho_0^{-3}  n^{-1/6+3\vartheta}.
\end{align*}
In view of (\ref{ineq:dist_min_norm_4}) (in Step 1 of the final proof of Theorem \ref{thm:min_norm_dist}), for $\rho_0 \in (C_3n^{-1/2+\vartheta}, 1/C_3)$, we take $
z\equiv z_\eta\equiv \max_{\beta>0}\min_{\gamma >0}\overline{\mathsf{D}}_\eta(\beta,\gamma)$ and $t\equiv \rho_0^2 n/C_3$ therein, so that for $\xi \in \mathcal{E}_\vartheta \subset \mathscr{E}_{1,\xi}(\rho_0/C_3^{1/2})$,
\begin{align*}
\Prob^\xi\Big(\min_{w \in \R^n} H_{\eta;G}(w) >z_\eta+\rho_0\Big)\leq C_3 e^{-\rho_0^2n/C_3}.
\end{align*}
Combining the estimates, for $\xi \in \mathcal{E}_\vartheta$, $\rho_0 \in (C_3n^{-1/2+\vartheta},1/C_3)$,
\begin{align*}
\Prob^\xi\Big(\min_{w \in \R^n} H_{\eta;Z}(w) >z_\eta+3\rho_0\Big)&\leq C_4\big\{e^{-\rho_0^2n/C_4}+\rho_0^{-3}  n^{-1/6+3\vartheta}\big\}.
\end{align*}
The first term above can be assimilated into the second one, and $\rho_0\geq C_3n^{-1/2+\vartheta}$ can be dropped. The lower bound follow similarly by utilizing (\ref{ineq:dist_min_norm_4_1}).
\end{proof}

\subsection{Universality of the cost over exceptional sets}
\begin{theorem}\label{thm:universality_exceptional_set}
Suppose Assumption \ref{assump:design} holds and the following hold for some $K>0$.
\begin{itemize}
	\item $1/K\leq \phi^{-1}\leq K$, $\pnorm{\Sigma}{\op}\vee \pnorm{\Sigma^{-1}}{\op} \leq K$.
	\item Assumption \ref{assump:noise} with variance $\sigma_\xi^2 \in [1/K,K]$. 
\end{itemize}
Fix $\vartheta \in (0,1/18)$. Then there exists some $C=C(K,\vartheta)>0$ such that for $\mathsf{g}:\R^n\to \R$ being $1$-Lipschitz with respect to $\pnorm{\cdot}{\Sigma^{-1}}$, $\rho_0 \leq 1/C$, $\eta \in \Xi_K$ and $\xi \in \mathcal{E}_\vartheta$,
\begin{align*}
\sup_{\mu_0 \in \mathcal{U}_\vartheta}\Prob^\xi\Big(\min_{w \in D_{\eta;C\rho_0^{1/2}}(\mathsf{g})\cap B_{(2,\infty)}(C,\frac{L_n}{\sqrt{n}}) }H_{\eta;Z}(w)\leq \max_{\beta>0}\min_{\gamma >0}\overline{\mathsf{D}}_\eta(\beta,\gamma)+\rho_0\Big)\leq C\rho_0^{-6}\cdot n^{-1/6+3\vartheta}.
\end{align*}
Here $B_{(2,\infty)}(C,L_n/\sqrt{n})\equiv B_n(C)\cap L_\infty(L_n/\sqrt{n})$ with $L_n\equiv C n^{\vartheta}$, and $\mathcal{U}_\vartheta$ is specified as in Proposition \ref{prop:delocal_u_v}.
\end{theorem}

\begin{proof}
Fix $\epsilon,\vartheta>0$, $\mu_0 \in \mathcal{U}_\vartheta$ and $\xi \in \mathcal{E}_\vartheta$ as specified in Proposition \ref{prop:delocal_u_v}. We define a renormalized version of $D_{\epsilon;\eta}(\mathsf{g})$ as
\begin{align*}
\tilde{D}_{\epsilon;\eta}(\mathsf{g})\equiv \big\{\tilde{w}\in \R^n: \abs{\mathsf{g}(\tilde{w}/\sqrt{n})-\E \mathsf{g}(\tilde{w}_{\eta,\ast}/\sqrt{n})}\geq \epsilon\big\},
\end{align*}
where $\tilde{w}_{\eta,\ast}=\sqrt{n} {w}_{\eta,\ast}$.

\noindent (\textbf{Step 1}). Let $L_n\equiv C_0n^{\vartheta}$. For any $z \in \R$ and $\rho_0>0$, with $Z_n\equiv Z/\sqrt{n}$,
\begin{align}\label{ineq:gordon_cost_except_universality_1}
&\Prob^\xi\Big(\min_{w \in D_{\eta;\epsilon}(\mathsf{g})\cap B_{(2,\infty)}(C_0,\frac{L_n}{\sqrt{n}}) }H_{\eta;Z}(w)\leq z+\rho_0\Big)\\
& = \Prob^\xi\Big(\min_{w \in D_{\eta;\epsilon}(\mathsf{g})\cap B_{(2,\infty)}(C_0,\frac{L_n}{\sqrt{n}}) } \Big\{F(w)+\frac{1}{2n\eta}\pnorm{Zw-\xi}{}^2\Big\}\leq z+\rho_0\Big)\nonumber\\
& = \Prob^\xi\Big(\min_{\tilde{w} \in \tilde{D}_{\eta;\epsilon}(\mathsf{g})\cap  B_{(2,\infty)}(\sqrt{n}C_0,L_n)} \Big\{\eta F(\tilde{w}/\sqrt{n})+\frac{1}{2n}\pnorm{Z_n\tilde{w}-\xi}{}^2\Big\}\leq \eta(z+\rho_0)\Big).\nonumber
\end{align}
Now we may apply Theorem \ref{thm:universality_smooth}. To do so, let us write $\mathsf{f}(\tilde{w})\equiv \eta F(\tilde{w}/\sqrt{n})$ to match the notation. Then a simple calculation leads to
\begin{align*}
\mathscr{N}_{\mathsf{f}}(L,\delta) &\equiv \sup_{\pnorm{\tilde{w}}{\infty}\vee \pnorm{\tilde{w}'}{\infty}\leq L, \pnorm{\tilde{w}-\tilde{w}'}{\infty}\leq \delta }\abs{\mathsf{f}(\tilde{w})-\mathsf{f}(\tilde{w}')}\lesssim_K (1\vee L)\delta,
\end{align*}
Consequently, an application of Theorem \ref{thm:universality_smooth} leads to
\begin{align*}
&\hbox{RHS of (\ref{ineq:gordon_cost_except_universality_1})}-C_1\big(1\vee (\eta \rho_0)^{-3}\big) L_n^2 n^{-1/6}\log^{2/3}(L_n n)\\
& \leq \Prob^\xi\Big(\min_{\tilde{w} \in \tilde{D}_{\eta;\epsilon}(\mathsf{g})\cap  B_{(2,\infty)}(\sqrt{n}C_0,L_n) } \Big\{\eta F(\tilde{w}/\sqrt{n}) +\frac{1}{2n}\pnorm{G_n\tilde{w}-\xi}{}^2\Big\} \leq \eta(z+3\rho_0)\Big)\\
&\leq \Prob^\xi\Big(\min_{w \in D_{\eta;\epsilon}(\mathsf{g})\cap  B_{(2,\infty)}(C_0,\frac{L_n}{\sqrt{n}}) }H_{\eta;G}(w)\leq z+3 \rho_0\Big)\\
&\leq \Prob^\xi\Big(\min_{w \in D_{\eta;\epsilon}(\mathsf{g})\cap B_n(C_0)}H_{\eta;G}(w)\leq z+3 \rho_0\Big).
\end{align*}
Here in the last inequality we simply drop the $L_\infty$ constraint. Now for $C_2n^{-1/2+\vartheta}\leq \rho_0\leq 1/C_2$, by choosing $
z\equiv z_\eta \equiv \max_{\beta>0}\min_{\gamma >0}\overline{\mathsf{D}}_\eta(\beta,\gamma)$ and $t\equiv 2\rho_0^2 n/C_3$ in Theorem \ref{thm:gordon_gap}, where $C_3$ is the constant therein, we have 
\begin{align}\label{ineq:gordon_cost_except_universality_2}
&\Prob^\xi\Big(\min_{w \in D_{\eta;C_4\rho_0^{1/2}}(\mathsf{g}) \cap B_{(2,\infty)}(C,\frac{L_n}{\sqrt{n}})   }H_{\eta;Z}(w)\leq \max_{\beta>0}\min_{\gamma >0}\overline{\mathsf{D}}_\eta(\beta,\gamma)+\rho_0\Big)\nonumber\\
&\leq C\Big\{ e^{-\rho_0^2n/C_3}+ (\eta \rho_0)^{-3}\cdot  n^{-1/6+3\vartheta}\Big\}\leq C_4\cdot (\eta \rho_0)^{-3}\cdot  n^{-1/6+3\vartheta}.
\end{align}
The constraints $\rho_0\geq C_2n^{-1/2+\vartheta}$ can be removed by enlarging $C_4$ if necessary. 

\noindent (\textbf{Step 2}). In this step we shall trade the dependence of the above bound with respect to $\eta>0$ with a possible worsened dependence on $\rho_0$, primarily in the regime $\phi^{-1}\geq 1+1/K$. Fix $\eta_0\in \Xi_K$. Let $\eta>0$ be chosen later and $\eta_1\equiv \eta_0+\eta$. Without loss of generality we assume $\eta_0,\eta_1 \in \Xi_K$, so by (\ref{ineq:cont_D_eta}) in Proposition \ref{prop:local_barD}, $\abs{z_{\eta_1}-z_{\eta_0}}\leq C_5\eta$. By enlarging $C_5$ if necessary we assume that $C_5$ exceeds the constant in Lemma \ref{lem:D_0_eta}. Using Lemma \ref{lem:D_0_eta}, for $\epsilon=2C_4\rho_0^{1/2}$, with the choice $\eta= C_4\rho_0/C_5\leq C_4\rho_0^{1/2}/C_5$ (we assume without loss of generality $\rho_0\leq 1$), 
\begin{align*}
&\Prob^\xi\Big(\min_{w \in D_{\eta_0;\epsilon}(\mathsf{g})\cap B_{(2,\infty)}(C_0,\frac{L_n}{\sqrt{n}})  }H_{\eta_0;Z}(w)\leq z_{\eta_0}+\rho_0\Big)\\
&\leq \Prob^\xi\Big(\min_{w \in D_{\eta_0;\epsilon}(\mathsf{g})\cap B_{(2,\infty)}(C_0,\frac{L_n}{\sqrt{n}})  }H_{\eta_1;Z}(w)\leq z_{\eta_0}+\rho_0\Big)\quad \hbox{(since $H_{\eta_1;Z}\leq H_{\eta_0;Z}$)} \\
&\leq \Prob^\xi\Big(\min_{w \in D_{\eta_1;(\epsilon-C_5\eta)_+}(\mathsf{g}) \cap B_{(2,\infty)}(C_0,\frac{L_n}{\sqrt{n}})  }H_{\eta_1;Z}(w)\leq z_{\eta_0}+\rho_0\Big)\quad \hbox{(by Lemma \ref{lem:D_0_eta})} \\
&\leq \Prob^\xi\Big(\min_{w \in D_{\eta_1;(\epsilon-C_5\eta)_+}(\mathsf{g}) \cap B_{(2,\infty)}(C_0,\frac{L_n}{\sqrt{n}}) }H_{\eta_1;Z}(w)\leq z_{\eta_1}+C_5\eta+ \rho_0\Big)\\
& \leq \Prob^\xi\Big(\min_{w \in D_{\eta;C_4\rho_0^{1/2}}(\mathsf{g})\cap B_{(2,\infty)}(C_0,\frac{L_n}{\sqrt{n}}) }H_{\eta_1;Z}(w)\leq \max_{\beta>0}\min_{\gamma >0}\overline{\mathsf{D}}_{\eta_1}(\beta,\gamma)+C\rho_0\Big)\\
&\leq C\cdot  (\eta_0+\rho_0)^{-3} \rho_0^{-3}\cdot  n^{-1/6+3\vartheta} \leq C\cdot  \rho_0^{-6} n^{-1/6+3\vartheta}.
\end{align*}
The proof is complete by adjusting  constants. 
\end{proof}

\begin{lemma}\label{lem:D_0_eta}
Suppose $\pnorm{\mu_0}{}\vee \pnorm{\Sigma}{\op}\vee \pnorm{\Sigma^{-1}}{\op}\leq K$. Let $\mathsf{g}:\R^n\to \R$ be $1$-Lipschitz with respect to $\pnorm{\cdot}{\Sigma^{-1}}$. Then there exists some constant $C=C(K)>0$ such that for any $\epsilon>0,\eta_0,\eta_1 \in \Xi_K$ with $\eta_1\geq \eta_0$, we have $D_{\eta_0;\epsilon}(\mathsf{g})\subset D_{\eta_1;(\epsilon-C(\eta_1-\eta_0))_+}(\mathsf{g})$. 
\end{lemma}
\begin{proof}
Using the definition of $w_{\eta,\ast}$ in (\ref{def:w_ast}), we have
\begin{align*}
&\abs{\E \mathsf{g}(w_{\eta_1,\ast})-\E \mathsf{g}(w_{\eta_0,\ast})}\leq \E\pnorm{ w_{\eta_1,\ast}-w_{\eta_0,\ast}}{\Sigma^{-1}}\\
&=\E \bigpnorm{\hat{\mu}_{(\Sigma,\mu_0)}^{\seq}(\gamma_{\eta_1,\ast};\tau_{\eta_1,\ast})-\hat{\mu}_{(\Sigma,\mu_0)}^{\seq}(\gamma_{\eta_0,\ast};\tau_{\eta_0,\ast})}{}\leq C\cdot (\eta_1-\eta_0).
\end{align*}
Here the last inequality follows from the calculations in (\ref{ineq:dist_min_norm_8_0}). So for any $w \in D_{\eta_0;\epsilon}(\mathsf{g})$,
\begin{align*}
\epsilon\leq \abs{\mathsf{g}(w)-\E \mathsf{g}(w_{\eta_0,\ast})}\leq \abs{\mathsf{g}(w)-\E \mathsf{g}(w_{\eta_1,\ast})}+C(\eta_1-\eta_0). 
\end{align*}
This means $w \in D_{\eta_1;(\epsilon-C(\eta_1-\eta_0))_+}(\mathsf{g})$, as desired. 
\end{proof}

\subsection{Proof of the universality Theorem \ref{thm:universality_min_norm} for $\hat{\mu}_{\eta;Z}$}

Fix $\vartheta>0$, $\mu_0 \in \mathcal{U}_\vartheta$ and $\xi \in \mathcal{E}_\vartheta$. Let $L_n\equiv C_0 n^{\vartheta}$, and $E_0\equiv \{\hat{w}_{n;Z} \in B_{(2,\infty)}(C_0,L_n/\sqrt{n})=B_n(C_0)\cap L_\infty(L_n/\sqrt{n})\}$. We assume that $C_0$ exceeds the constants in Proposition \ref{prop:delocal_u_v} and Theorem \ref{thm:universality_exceptional_set}. By Proposition \ref{prop:delocal_u_v} and a simple $\ell_2$ estimate, $\Prob^\xi(E_0^c)\leq C_0 n^{-100}$. We further let $z_\eta \equiv \max_{\beta>0}\min_{\gamma >0}\overline{\mathsf{D}}_\eta(\beta,\gamma)$ for $\eta\geq 0$. 

Let $\mathsf{g}:\R^n\to \R$ be $1$-Lipschitz with respect to $\pnorm{\cdot}{\Sigma^{-1}}$. Then for $\rho_0 \leq 1/C_0$ and $\eta \in \Xi_K$, we have
\begin{align*}
&\Prob^\xi\big(\hat{w}_{\eta;Z} \in D_{\eta;C_0\rho_0^{1/2}}(\mathsf{g})\big)\\
&\leq \Prob^\xi\big(\hat{w}_{\eta;Z} \in D_{\eta;C_0\rho_0^{1/2}}(\mathsf{g})\cap B_{(2,\infty)}(C_0,L_n/\sqrt{n})\big)+\Prob^\xi(E_0^c)\\
&\leq \Prob^\xi\Big(\min_{w \in B_{(2,\infty)}(C_0,\frac{L_n}{\sqrt{n}})} H_{\eta;Z}(w)\geq z_\eta+\rho_0\Big)\\
&\qquad + \Prob^\xi\Big(\min_{w \in D_{\eta;C_0\rho_0^{1/2}}(\mathsf{g})\cap B_{(2,\infty)}(C_0,\frac{L_n}{\sqrt{n}})} H_{\eta;Z}(w)\leq z_\eta+2\rho_0\Big)+ C_0n^{-100}. 
\end{align*}
Here in the last inequality we used the simple fact that 
\begin{align*}
&\Big\{\min_{w \in B_{(2,\infty)}(C_0,\frac{L_n}{\sqrt{n}})} H_{\eta;Z}(w)<z_\eta+\rho_0\Big\}\cap \Big\{\min_{w \in D_{\eta;C_0\rho_0^{1/2}}(\mathsf{g})\cap B_{(2,\infty)}(C_0,\frac{L_n}{\sqrt{n}})} H_{\eta;Z}(w)>z_\eta+2\rho_0\Big\}\\
&\subset \big\{\hat{w}_{\eta;Z} \notin D_{\eta;C_0\rho_0^{1/2}}(\mathsf{g})\cap B_{(2,\infty)}(C_0,{L_n}/{\sqrt{n}})\big\}.
\end{align*}
Invoking Theorems \ref{thm:universality_global_cost} and \ref{thm:universality_exceptional_set}, by enlarging $C_0$ if necessary, we have for $\rho_0 \leq 1/C_0$ and $\eta \in \Xi_K$,
\begin{align*}
\Prob^\xi\Big(\abs{\mathsf{g}(\hat{w}_{\eta;Z})-\E \mathsf{g}(w_{\eta,\ast})}\geq \rho_0^{1/2}\Big)\leq C_0\cdot  \rho_0^{-6} n^{-1/6+3\vartheta},
\end{align*}
or equivalently, for $\mathsf{g}_0:\R^n\to \R$ being $1$-Lipschitz with respect to $\pnorm{\cdot}{}$, 
\begin{align*}
\Prob^\xi\Big(\bigabs{\mathsf{g}_0(\hat{\mu}_{\eta;Z})-\E \mathsf{g}_0\big(\hat{\mu}_{(\Sigma,\mu_0)}^{\seq}(\gamma_{\eta,\ast};\tau_{\eta,\ast})\big) }\geq \rho_0\Big)\leq C_0\cdot  \rho_0^{-12} n^{-1/6+3\vartheta}. 
\end{align*}
Now we may follow Step 4 in the proof of Theorem \ref{thm:min_norm_dist} to strengthen the above statement to a uniform one in $\eta$; we only sketch the differences below. Using (\ref{ineq:dist_min_norm_6}) with $G$ therein replaced by $Z$, and the assumption $\pnorm{\Sigma^{-1}}{\op}\leq K$, we arrive at a modified form of (\ref{ineq:dist_min_norm_7}): on an event $E_{1}$ with $\Prob^\xi(E_{1})\geq 1-C_1 e^{-n/C_1}$, for any $\eta_1,\eta_2 \in \Xi_K$,
\begin{align}\label{ineq:universality_w_1}
\pnorm{\hat{\mu}_{\eta_1;Z}-\hat{\mu}_{\eta_2;Z}}{}\leq C_1\abs{\eta_1-\eta_2}.
\end{align}
Using (\ref{ineq:dist_min_norm_8_0}) with $\pnorm{\Sigma^{-1}}{\op}\leq K$, we arrive at a modified form of (\ref{ineq:dist_min_norm_8}): for any $\eta_1,\eta_2 \in \Xi_K$,
\begin{align*}
\bigabs{\E \mathsf{g}_0\big(\hat{\mu}_{(\Sigma,\mu_0)}^{\seq}(\gamma_{\eta_1,\ast};\tau_{\eta_1,\ast})\big)- \E \mathsf{g}_0\big(\hat{\mu}_{(\Sigma,\mu_0)}^{\seq}(\gamma_{\eta_2,\ast};\tau_{\eta_2,\ast})\big)}\leq C_1\abs{\eta_1-\eta_2}. 
\end{align*}
Now using a standard discretization and a union bound, we have
\begin{align*}
&\Prob^\xi\Big(\sup_{\eta \in \Xi_K }\bigabs{\mathsf{g}_0(\hat{\mu}_{\eta;Z})-\E \mathsf{g}_0\big(\hat{\mu}_{(\Sigma,\mu_0)}^{\seq}(\gamma_{\eta,\ast};\tau_{\eta,\ast})\big) }\geq \rho_0\Big)
\leq C_2\cdot \rho_0^{-13} n^{-1/6+3\vartheta}.
\end{align*}
The proof is complete by taking expectation with respect to $\xi$ and note that $\Prob(\xi\in \mathcal{E}_\vartheta)\geq 1- C e^{-n^{2\vartheta}/C}$ as in Proposition \ref{prop:delocal_u_v}. \qed

\subsection{Proof of the universality Theorem \ref{thm:universality_min_norm} for $\hat{r}_{\eta;Z}$}

\begin{proposition}\label{prop:universality_exceptional_set_residual}
Suppose Assumption \ref{assump:design} holds and the following hold for some $K>0$.
\begin{itemize}
	\item $1/K\leq \phi^{-1},\eta \leq K$, $\pnorm{\Sigma}{\op}\vee \pnorm{\Sigma^{-1}}{\op} \leq K$.
	\item Assumption \ref{assump:noise} holds with $\sigma_\xi^2 \in [1/K,K]$. 
\end{itemize}
Fix $\vartheta \in (0,1/18)$. Then there exists some $C=C(K,\vartheta)>0$ such that for any $1$-Lipschitz function $\mathsf{h}:\R^m\to \R$, $\rho_0 \leq 1/C$, $\eta \in \Xi_K$ and $\xi \in \mathcal{E}_\vartheta$, 
\begin{align*}
&\sup_{\mu_0 \in \mathcal{U}_\vartheta}\Prob^\xi\Big(\max_{v \in D_{\eta;C\rho_0^{1/2}}(\mathsf{h})\cap L_\infty(\frac{L_n}{\sqrt{n}})} \min_{w \in \R^n} h_{\eta;Z}(w,v) \geq \max_{\beta>0}\min_{\gamma >0} \overline{\mathsf{D}}_\eta(\beta,\gamma)-\rho_0\Big)\leq  C \rho_0^{-3} n^{-1/6+3\vartheta}.
\end{align*}
Here $L_n\equiv C n^{\vartheta}$, and $\mathcal{U}_\vartheta$ is specified as in Proposition \ref{prop:delocal_u_v}.
\end{proposition}
\begin{proof}
Fix $\epsilon,\vartheta>0$, $\mu_0 \in \mathcal{U}_\vartheta$ and $\xi \in \mathcal{E}_\vartheta$ as specified in Proposition \ref{prop:delocal_u_v}. We define a renormalized version of $D_{\epsilon;\eta}(\mathsf{h})$ as
\begin{align*}
\tilde{D}_{\epsilon;\eta}(\mathsf{h})\equiv \big\{\tilde{r}\in \R^m: \abs{\mathsf{h}(\tilde{r}/\sqrt{n})-\E^\xi \mathsf{h}(\tilde{r}_{\eta,\ast}/\sqrt{n})}\geq \epsilon\big\},
\end{align*}
where $\tilde{r}_{\eta,\ast}=\sqrt{n} {r}_{\eta,\ast}$. Let $L_n=C_0n^{\vartheta}$ and $Q(\tilde{v},\tilde{w})$ be defined as in the proof of Theorem \ref{thm:universality_global_cost}. Then we have,
\begin{align*}
& \max_{v \in D_{\eta;\epsilon}(\mathsf{h})\cap L_\infty(L_n/\sqrt{n})} \min_{w \in L_\infty(L_n/\sqrt{n})} h_{\eta;Z}(w,v)\\
& = \max_{\tilde{v} \in \tilde{D}_{\eta;\epsilon}(\mathsf{h})\cap L_\infty(L_n)} \min_{\tilde{w} \in L_\infty(L_n)} h_{\eta;Z}(\tilde{w}/\sqrt{n},\tilde{v}/\sqrt{n})\\
& = \max_{\tilde{v} \in \tilde{D}_{\eta;\epsilon}(\mathsf{h})\cap L_\infty(L_n)} \min_{\tilde{w} \in L_\infty(L_n)}\bigg\{\frac{1}{n^{3/2}}\iprod{\tilde{v}}{Z \tilde{w}}-\frac{1}{n}\iprod{\tilde{v}}{\xi}+F(\tilde{w}/\sqrt{n})-\frac{\eta\pnorm{\tilde{v}}{}^2}{2n}\bigg\}\\
&= \max_{\tilde{v} \in \tilde{D}_{\eta;\epsilon}(\mathsf{h})\cap L_\infty(L_n)} \min_{\tilde{w} \in L_\infty(L_n)}\bigg\{\frac{1}{n^{3/2}}\iprod{\tilde{v}}{Z \tilde{w}}+Q(\tilde{v},\tilde{w})\bigg\}. 
\end{align*}
Using the comparison inequality in Theorem \ref{thm:min_max_universality} and a similar calculation as in (\ref{ineq:gordon_cost_universality_3}), with $s_{n}(\rho_0)\equiv  \rho_0^{-3} n^{-1/6+3\vartheta}$, 
\begin{align*}
&\Prob^\xi\Big(\max_{v \in D_{\eta;\epsilon}(\mathsf{h})\cap L_\infty(L_n/\sqrt{n})} \min_{w \in L_\infty(L_n/\sqrt{n})} h_{\eta;Z}(w,v)\geq z-\rho_0\Big)\\
& = \Prob^\xi\Big( \max_{\tilde{v} \in \tilde{D}_{\eta;\epsilon}(\mathsf{h})\cap L_\infty(L_n)} \min_{\tilde{w} \in L_\infty(L_n)}\bigg\{\frac{1}{n^{3/2}}\iprod{\tilde{v}}{Z \tilde{w}}+Q(\tilde{v},\tilde{w})\bigg\}\geq z-\rho_0\Big)\\
&\leq \Prob^\xi\Big( \max_{\tilde{v} \in \tilde{D}_{\eta;\epsilon}(\mathsf{h})\cap L_\infty(L_n)} \min_{\tilde{w} \in L_\infty(L_n)}\bigg\{\frac{1}{n^{3/2}}\iprod{\tilde{v}}{G \tilde{w}}+Q(\tilde{v},\tilde{w})\bigg\}\geq z-3\rho_0\Big) +Cs_{n}(\rho_0)\\
&= \Prob^\xi\Big(\max_{v \in D_{\eta;\epsilon}(\mathsf{h})\cap L_\infty(L_n/\sqrt{n})} \min_{w \in L_\infty(L_n/\sqrt{n})} h_{\eta;G}(w,v)\geq z-3\rho_0\Big)+ C_1s_{n}(\rho_0).
\end{align*}
Using the convex Gaussian min-max theorem (cf. Theorem \ref{thm:CGMT}), 
\begin{align}\label{ineq:universality_residual_1}
& \Prob^\xi\Big(\max_{v \in D_{\eta;\epsilon}(\mathsf{h})\cap L_\infty(L_n/\sqrt{n})} \min_{w \in L_\infty(L_n/\sqrt{n})} h_{\eta;Z}(w,v)\geq z-\rho_0\Big)\nonumber\\
&\leq 2 \Prob\Big(\max_{v \in D_{\eta;\epsilon}(\mathsf{h})\cap L_\infty(L_n/\sqrt{n})} \min_{w \in L_\infty(L_n/\sqrt{n})} \ell_{\eta}(w,v)\geq z-3\rho_0\Big) + C_1s_{n}(\rho_0).
\end{align}
On the other hand, using the definition of $w_{\eta,\ast}$ in (\ref{def:w_ast}), and the fact that for any $\mu_0 \in \mathcal{U}_\vartheta$, $\pnorm{\E w_{\eta,\ast}}{\infty}\leq L_n/\sqrt{n}$, we have $
\Prob\big(\pnorm{w_{\eta,\ast}}{\infty} \geq L_n/\sqrt{n}\big)\leq C e^{-n^{2\vartheta}/C}$. Combined with (\ref{ineq:universality_residual_1}), we have
\begin{align*}
& \Prob^\xi\Big(\max_{v \in D_{\eta;\epsilon}(\mathsf{h})\cap L_\infty(L_n/\sqrt{n})} \min_{w \in L_\infty(L_n/\sqrt{n})} h_{\eta;Z}(w,v)\geq z-\rho_0\Big)\\
&\leq 2 \Prob\Big(\max_{v \in D_{\eta;\epsilon}(\mathsf{h})\cap L_\infty(L_n/\sqrt{n})} \ell_{\eta}(w_{\eta,\ast},v)\geq z-3\rho_0\Big)+C_2s_{n}(\rho_0).
\end{align*}
In view of (\ref{ineq:dist_res_gaussian_1_0}), now by choosing $z\equiv z_\eta\equiv \max_{\beta>0}\min_{\gamma >0} \overline{\mathsf{D}}_\eta(\beta,\gamma)$ and $\epsilon\equiv C_3 \rho^{1/2}_0$, for $\rho_0\geq C_4 n^{-1/2+\vartheta}$, $\xi \in \mathcal{E}_\vartheta\subset  \mathscr{E}_{1,\xi}(\rho_0/C)$, it follows that
\begin{align*}
& \Prob^\xi\Big(\max_{v \in D_{\eta;C_3\rho_0^{1/2}}(\mathsf{h})\cap L_\infty(L_n/\sqrt{n})} \min_{w \in \R^n} h_{\eta;Z}(w,v)\geq z_\eta-\rho_0\Big)\\
& \leq \Prob^\xi\Big(\max_{v \in D_{\eta;C_3\rho_0^{1/2}}(\mathsf{h})\cap L_\infty(L_n/\sqrt{n})} \min_{w \in L_\infty(L_n/\sqrt{n})} h_{\eta;Z}(w,v)\geq z_\eta-\rho_0\Big)\leq  C_4s_{n}(\rho_0).
\end{align*}
The claim follows by adjusting constants. 
\end{proof}

\begin{proof}[Proof of Theorem \ref{thm:universality_min_norm} for $\hat{r}_{\eta;Z}$]
Fix $\vartheta>0$, $\mu_0 \in \mathcal{U}_\vartheta$ and $\xi \in \mathcal{E}_\vartheta$ as specified in Proposition \ref{prop:delocal_u_v}. We continue writing $z_\eta\equiv \max_{\beta>0}\min_{\gamma >0} \overline{\mathsf{D}}_\eta(\beta,\gamma)$ in the proof. Using the delocalization results in Proposition \ref{prop:delocal_u_v}, on an event $E_0$ with $\Prob^\xi(E_0)\geq 1-C_0n^{-100}$, we have $\pnorm{\hat{w}_{\eta;Z}}{\infty}\vee \pnorm{\hat{r}_{\eta;Z}}{\infty}\leq L_n/\sqrt{n}$ with $L_n=C_0n^{\vartheta}$. Using Theorem \ref{thm:universality_global_cost}, for $\rho_0 \leq 1/C$, and $\eta \in \Xi_K$, by possibly adjusting $C_0>0$,
\begin{align}\label{ineq:universality_residual_proof_1}
&\Prob^\xi\Big(\max_{v \in L_\infty(L_n/\sqrt{n})} \min_{w \in \R^m} h_{\eta;Z}(w,v)\leq z_\eta-\rho_0/2\Big)\nonumber\\
&\leq \Prob^\xi\Big( \min_{w \in \R^m} H_{\eta;Z}(w)\leq z_\eta-\rho_0/2\Big)+ \Prob^\xi(E_0^c)\leq C_0 \rho_0^{-3}\cdot n^{-1/6+3\vartheta}. 
\end{align}
Let us take $C_1>0$ to be the constant in Proposition \ref{prop:universality_exceptional_set_residual}. By noting that
\begin{align*}
&\Big\{\max_{v \in L_\infty(L_n/\sqrt{n})} \min_{w \in \R^m} h_{\eta;Z}(w,v)> z_\eta-\rho_0/2\Big\}\\
&\quad \cap \Big\{\max_{v \in D_{\eta;C_1\rho_0^{1/2}}(\mathsf{h})\cap L_\infty(L_n/\sqrt{n})} \min_{w \in \R^n} h_{\eta;Z}(w,v)< z_\eta-\rho_0\Big\}\\
& \subset \big\{\hat{v}_{\eta;Z}\notin D_{\eta;C_1\rho_0^{1/2}}(\mathsf{h})\cap L_\infty(L_n/\sqrt{n})\big\},
\end{align*}
it follows from (\ref{ineq:universality_residual_proof_1}) and Proposition \ref{prop:universality_exceptional_set_residual} that
\begin{align*}
&\Prob^\xi\Big(\hat{v}_{\eta;Z}\in D_{\eta;C_1\rho_0^{1/2}}(\mathsf{h})\Big)\\
&\leq \Prob^\xi \Big(\hat{v}_{\eta;Z}\in D_{\eta;C_1\rho_0^{1/2}}(\mathsf{h}) \cap L_\infty(L_n/\sqrt{n})\Big)+\Prob^\xi\Big(\hat{v}_{\eta;Z} \notin L_\infty(L_n/\sqrt{n}) \Big)\\
&\leq \Prob^\xi\Big(\max_{v \in L_\infty(L_n/\sqrt{n})} \min_{w \in \R^m} h_{\eta;Z}(w,v)\leq z_\eta-\rho_0/2\Big)\\
&\qquad + \Prob^\xi\Big(\max_{v \in D_{\eta;C_1\rho_0^{1/2}}(\mathsf{h})\cap L_\infty(L_n/\sqrt{n})} \min_{w \in \R^n} h_{\eta;Z}(w,v)\geq  z_\eta-\rho_0 \Big)+ \Prob^\xi(E_0^c)\\
&\leq C \rho_0^{-3}\cdot n^{-1/6+3\vartheta}. 
\end{align*}	
Finally we only need to extend the above display to a uniform control over $\eta \in [1/K,K]$ by continuity arguments similar to Step 5 of the proof of Theorem \ref{thm:min_norm_dist} for $\hat{r}_{\eta;G}$. By (\ref{ineq:dist_res_gaussian_6}) (where $G$ therein is replaced by $Z$) and (\ref{ineq:universality_w_1}), on an event $E_1$ with $\Prob^\xi(E_1)\geq 1-C e^{-n/C}$, for any $\eta_1,\eta_2 \in [1/K,K]$, 
\begin{align*}
\pnorm{\hat{r}_{\eta_1;Z}-\hat{r}_{\eta_2;Z}}{}\leq C \abs{\eta_1-\eta_2}.
\end{align*}
On the other hand, (\ref{ineq:dist_res_gaussian_8}) remains valid, so we may proceed with an $\epsilon$-net argument over  $[1/K,K]$ to conclude. 
\end{proof}

\section{Proof of Theorem \ref{thm:lq_risk}}\label{subsection:proof_lq_risk}

To keep notation simple, we work with $\mathsf{A}=I_n$ and write $\Gamma_{\eta;(\Sigma,\pnorm{\mu_0}{})}^{I_n}=\Gamma_{\eta;(\Sigma,\pnorm{\mu_0}{})}$. The general case follows from minor modifications.

\begin{lemma}\label{lem:lq_risk_lower_bound}
Suppose the conditions  in Theorem \ref{thm:lq_risk} hold for some $K>0$. Fix $q \in [1,\infty)$. There exists some constant $c=c(K,q)>0$ such that $n^{\frac{1}{2}-\frac{1}{q}}\E \pnorm{ \hat{\mu}_{(\Sigma,\mu_0)}^{\seq}(\gamma_{\eta,\ast};\tau_{\eta,\ast})-\mu_0 }{q}\geq c$ uniformly in $\eta \in \Xi_K$. 
\end{lemma}
\begin{proof}
We may write $\E \pnorm{ \hat{\mu}_{(\Sigma,\mu_0)}^{\seq}(\gamma_{\eta,\ast};\tau_{\eta,\ast})-\mu_0 }{q} = \E \big(\sum_{j=1}^n \abs{a_j+ b_j g_j}^q\big)^{1/q}$ for some $a_j, b_j \in \R$ with $b_j\asymp 1$, and $g_j \sim \mathcal{N}(0,1/n)$ not necessarily independent of each other. So for some $c_j \in \R$,
\begin{align*}
\E \pnorm{ \hat{\mu}_{(\Sigma,\mu_0)}^{\seq}(\gamma_{\eta,\ast};\tau_{\eta,\ast})-\mu_0 }{q}\gtrsim \E \Big(\sum_{j=1}^n \abs{c_j+ g_j}^q\Big)^{1/q}.
\end{align*}
If $\sum_{j=1}^n \abs{c_j}^q \geq C_0 \sum_{j=1}^n \E \abs{g_j}^q$ for a large enough $C_0>0$, the lower bound follows trivially. Otherwise, with $Z\equiv \sum_{j=1}^n \abs{c_j+g_j}^q$, we have $\E Z\geq \sum_{j=1}^n \inf_{c \in \R} \E \abs{c+g_j}^q \gtrsim n^{1-q/2}$ and $\E Z^2 \lesssim \E \big(\sum_{j=1}^n (\abs{g_j}^q+\E\abs{g_j}^q)\big)^2\lesssim (n^{1-q/2})^2$, so by Paley-Zygmund inequality, $\Prob(Z\geq \E Z/2)\geq (\E Z)^2/(4\E Z^2)\geq c_0$ for some $c_0>0$. In other words, on an event $E_0$ with $\Prob(E_0)\geq c_0$, $Z\geq c_0 n^{1-q/2}$. Using the above display, this means that $\E \pnorm{ \hat{\mu}_{(\Sigma,\mu_0)}^{\seq}(\gamma_{\eta,\ast};\tau_{\eta,\ast})-\mu_0 }{q}\gtrsim \E Z^{1/q}\geq \E Z^{1/q}\bm{1}_{E_0}\gtrsim n^{1/q-1/2}$. 
\end{proof}

\begin{lemma}\label{lem:lq_risk_gw}
Suppose the conditions  in Theorem \ref{thm:lq_risk} hold for some $K>0$. Fix $q \in [1,\infty)$. Then there exist constants $C=C(K,q)>1$, $\vartheta=\vartheta(q) \in (0,1/50)$, and a measurable set $\mathcal{U}_\vartheta\subset B_n(1)$ with $\mathrm{vol}(\mathcal{U}_\vartheta)/\mathrm{vol}(B_n(1))\geq 1-C e^{-n^{\vartheta}/C}$, such that 
\begin{align*}
\sup_{\mu_0 \in \mathcal{U}_{\vartheta} } n^{\frac{1}{2}-\frac{1}{q}}\bigabs{\E \pnorm{ \hat{\mu}_{(\Sigma,\mu_0)}^{\seq}(\gamma_{\eta,\ast};\tau_{\eta,\ast})-\mu_0 }{q}-n^{-\frac{1}{2}}\pnorm{\mathrm{diag}\big(\Gamma_{\eta;(\Sigma,\pnorm{\mu_0}{})}\big) }{q/2}^{\frac{1}{2}}M_q  }\leq C n^{-\vartheta}.
\end{align*}
Here $M_q=\E^{1/q}\abs{\mathcal{N}(0,1)}^q$. 
\end{lemma}
\begin{proof}
We write $\tau_{\eta,\ast}=\tau_\eta$, $\gamma_{\eta,\ast}=\gamma_\eta$ for notational simplicity in the proof. All the constants in $\lesssim,\gtrsim,\asymp$ below may depend on $K,q$. Recall the general fact $\pnorm{x}{q}\leq n^{-\frac{1}{2}+\frac{1}{q\wedge 2}} \pnorm{x}{}^{\frac{2}{q\vee 2}}\pnorm{x}{\infty}^{1-\frac{2}{q\vee 2}}$ for $x \in \R^n$ and $q\in (0,\infty)$.

By Proposition \ref{prop:tau_gamma_m} below, for any $\vartheta \in (0,1/2)$, there exists some $\mathcal{U}_{\vartheta}\subset B_n(1)$ with $\mathrm{vol}(\mathcal{U}_{\vartheta})/\mathrm{vol}(B_n(1))\geq 1-Ce^{-n^{1-2\vartheta}/C}$, such that $\sup_{\mu_0 \in \mathcal{U}_{\vartheta}}\sup_{\eta \in \Xi_K}\abs{\gamma_\eta^2-\tilde{\gamma}_\eta^2(\pnorm{\mu_0}{})}\leq n^{-\vartheta}$. Consequently, uniformly in $\mu_0 \in \mathcal{U}_{\vartheta}$ and $\eta \in \Xi_K$,
\begin{align}\label{ineq:lq_risk_gw_1}
&\bigabs{\E \pnorm{ \hat{\mu}_{(\Sigma,\mu_0)}^{\seq}(\gamma_{\eta};\tau_{\eta})-\mu_0 }{q}-\E \pnorm{ \hat{\mu}_{(\Sigma,\mu_0)}^{\seq}(\tilde{\gamma}_\eta(\pnorm{\mu_0}{});\tau_{\eta})-\mu_0 }{q} }\nonumber\\
&\lesssim \abs{\gamma_\eta-\tilde{\gamma}_\eta(\pnorm{\mu_0}{}) }\cdot n^{-1/2} \E \pnorm{(\Sigma+\tau_\eta I)^{-1}\Sigma^{1/2} g}{q}\nonumber\\
&\leq n^{-\frac{1}{2}-\vartheta}\cdot n^{-\frac{1}{2}+\frac{1}{q\wedge 2}}\cdot \E \Big\{\pnorm{(\Sigma+\tau_\eta I)^{-1}\Sigma^{1/2} g}{}^{\frac{2}{q\vee 2}}\cdot \pnorm{(\Sigma+\tau_\eta I)^{-1}\Sigma^{1/2} g}{\infty}^{1-\frac{2}{q\vee 2}}\Big\}\nonumber\\
&\lesssim n^{-\frac{1}{2}+\frac{1}{q}-\vartheta}\big(\log n\big)^{\frac{1}{2}-\frac{1}{q\vee 2}}.
\end{align}
For $g' \in \R^n$, let $\mathsf{f}(g') \equiv n^{-1/2} \pnorm{(\Sigma+\tau_\eta I)^{-1} g'}{q}$, and
\begin{align*}
\mathsf{F}_{\pnorm{\mu_0}{}}(g')&\equiv n^{-1/2}\bigpnorm{(\Sigma+\tau_\eta I)^{-1} \big( -\tau_\eta \pnorm{\mu_0}{}g'+\tilde{\gamma}_\eta(\pnorm{\mu_0}{})\Sigma^{1/2} g\big) }{q},\\
\mathsf{F}_{\pnorm{\mu_0}{},0}(g')& \equiv \biggpnorm{(\Sigma+\tau_\eta I)^{-1}\bigg( -\tau_\eta \pnorm{\mu_0}{}\frac{g'}{\pnorm{g'}{}}+\tilde{\gamma}_\eta(\pnorm{\mu_0}{})\Sigma^{1/2} \frac{g}{\sqrt{n}}\bigg)  }{q}.
\end{align*}
Then for $g_1',g_2' \in \R^n$, 
\begin{align*}
&\bigabs{\mathsf{F}_{\pnorm{\mu_0}{}}(g_1')-\mathsf{F}_{\pnorm{\mu_0}{}}(g_2')}\vee \bigabs{\mathsf{f}(g_1')-\mathsf{f}(g_2')} \lesssim n^{-1+\frac{1}{q\wedge 2}} \pnorm{g_1'-g_2'}{}.
\end{align*} 
By Gaussian concentration inequality, for any $\vartheta \in (0,1/2)$, we may find some $\mathcal{G}_{\vartheta,\pnorm{\mu_0}{}}\subset \R^n$ with $\Prob(g_0 \in \mathcal{G}_{\vartheta,\pnorm{\mu_0}{}})\geq 1- Ce^{-n^{2\vartheta}/C}$, $g_0 \sim \mathcal{N}(0,I_n)$, such that uniformly in $g' \in \mathcal{G}_{\vartheta,\pnorm{\mu_0}{}}$,
\begin{align}\label{ineq:lq_risk_gw_2}
&\max\Big\{\bigabs{\pnorm{g'}{}-\sqrt{n} }, n^{1-\frac{1}{q\wedge 2}}   \bigabs{\mathsf{F}_{\pnorm{\mu_0}{}}(g')-\E_{g_0} \mathsf{F}_{\pnorm{\mu_0}{}}(g_0)}, \nonumber\\
&\qquad\qquad n^{1-\frac{1}{q\wedge 2}} \bigabs{ \mathsf{f}(g')- \E \mathsf{f}(g_0)} \Big\} \leq n^{\vartheta}. 
\end{align}
As $\E \mathsf{f}(g_0)=n^{-1/2}\E \pnorm{(\Sigma+\tau_\eta I) g}{q}\lesssim n^{-\frac{1}{2}+\frac{1}{q}}(\log n)^{\frac{1}{2}-\frac{1}{q\vee 2}}$, for $\vartheta$ small enough, uniformly in $g' \in \mathcal{G}_{\vartheta,\pnorm{\mu_0}{}}$, 
\begin{align}\label{ineq:lq_risk_gw_3}
&\bigabs{ \mathsf{F}_{\pnorm{\mu_0}{}}(g')-\mathsf{F}_{\pnorm{\mu_0}{},0}(g') }\lesssim \abs{\mathsf{f}(g')}\cdot \abs{1-\sqrt{n}/\pnorm{g'}{}}\nonumber\\
&\lesssim \big(\E \mathsf{f}(g_0)+ n^{-1+\frac{1}{q\wedge 2}+\vartheta} \big)\cdot n^{-\frac{1}{2}+\vartheta} \lesssim n^{-1+\frac{1}{q}+\vartheta}(\log n)^{\frac{1}{2}-\frac{1}{q\vee 2}}.
\end{align}
Combining (\ref{ineq:lq_risk_gw_2})-(\ref{ineq:lq_risk_gw_3}), for $\vartheta$ small enough,
\begin{align}\label{ineq:lq_risk_gw_4}
\sup_{g' \in \mathcal{G}_{\vartheta,\pnorm{\mu_0}{}}} n^{1-\frac{1}{q\wedge 2}} \bigabs{\mathsf{F}_{\pnorm{\mu_0}{},0}(g')-\E_{g_0} \mathsf{F}_{\pnorm{\mu_0}{}}(g_0)}\lesssim n^{\vartheta}. 
\end{align}
Now let $\partial \mathcal{G}_{\vartheta,\pnorm{\mu_0}{}}\equiv \{g'/\pnorm{g'}{}: g' \in \mathcal{G}_{\vartheta,\pnorm{\mu_0}{}}\}\subset \partial B_n(1)$. Using that $\{g_0 \in \mathcal{G}_{\vartheta,\pnorm{\mu_0}{}}\} \subset \big\{g_0/\pnorm{g_0}{} \in \partial \mathcal{G}_{\vartheta,\pnorm{\mu_0}{}} \}$, we have $\Prob\big(g_0/\pnorm{g_0}{} \in \partial \mathcal{G}_{\vartheta,\pnorm{\mu_0}{}} \big)\geq \Prob(g_0 \in \mathcal{G}_{\vartheta,\pnorm{\mu_0}{}})\geq 1- Ce^{-n^{2\vartheta}/C}$. So with 
\begin{align*}
\mathcal{V}_{\vartheta}\equiv \big\{\mu_0=U_0g': U_0 \in [0,1], g' \in \partial \mathcal{G}_{\vartheta,U_0} \big\}\subset B_n(1),
\end{align*}
we have $
\Prob\big(\mathrm{Unif}(B_n(1)) \in \mathcal{V}_{\vartheta}\big)=\E_{U_0}\Prob_{g_0}\big(g_0/\pnorm{g_0}{} \in \partial \mathcal{G}_{\vartheta,U_0} \big)\geq 1- Ce^{-n^{2\vartheta}/C}$. In other words, for this constructed set $\mathcal{V}_{\vartheta}$, we have the desired volume estimate $\mathrm{vol}(\mathcal{V}_{\vartheta})/\mathrm{vol}(B_n(1))\geq 1- Ce^{-n^{2\vartheta}/C}$, and by (\ref{ineq:lq_risk_gw_4}),
\begin{align}\label{ineq:lq_risk_gw_5}
n^{1-\frac{1}{q\wedge 2}} \sup_{\mu_0 \in \mathcal{V}_{\vartheta}}\bigabs{\E \pnorm{ \hat{\mu}_{(\Sigma,\mu_0)}^{\seq}(\tilde{\gamma}_\eta(\pnorm{\mu_0}{});\tau_{\eta})-\mu_0 }{q}- \E_{g_0} \mathsf{F}_{\pnorm{\mu_0}{}}(g_0)}\lesssim n^{\vartheta}.
\end{align}
On the other hand, using the definition of  $\Gamma_{\eta;(\Sigma,\pnorm{\mu_0}{})}$ in (\ref{def:Gamma_matrix}), we may compute 
\begin{align}\label{ineq:lq_risk_gw_6}
\E_{g_0} \mathsf{F}_{\pnorm{\mu_0}{}}(g_0) = \E \bigpnorm{\Gamma_{\eta;(\Sigma,\pnorm{\mu_0}{})}^{1/2} g/\sqrt{n}}{q}.
\end{align} 
Combining (\ref{ineq:lq_risk_gw_1}), (\ref{ineq:lq_risk_gw_5}) and (\ref{ineq:lq_risk_gw_6}), for $\vartheta$ chosen small enough, 
\begin{align*}
&\sup_{\mu_0 \in \mathcal{U}_{\vartheta}\cap \mathcal{V}_{\vartheta} } n^{1/2-1/q}\bigabs{\E \pnorm{ \hat{\mu}_{(\Sigma,\mu_0)}^{\seq}(\gamma_{\eta};\tau_{\eta})-\mu_0 }{q}-\E \bigpnorm{\Gamma_{\eta;(\Sigma,\pnorm{\mu_0}{})}^{1/2} g/\sqrt{n}}{q} }\\
&\lesssim n^{\frac{1}{2}-\frac{1}{q}}\cdot \Big(n^{-\frac{1}{2}+\frac{1}{q}-\vartheta}\big(\log n\big)^{\frac{1}{2}-\frac{1}{q\vee 2}}+n^{-1+\frac{1}{q\wedge 2}+\vartheta}\Big)\\
& = n^{-\vartheta}\big(\log n\big)^{\frac{1}{2}-\frac{1}{q\vee 2}}+ n^{-\frac{1}{2}-\frac{1}{q} +\frac{1}{q\wedge 2}+\vartheta }\lesssim n^{-\vartheta/2}.
\end{align*}
The claim follows from Lemma \ref{lem:gaussian_lq}.
\end{proof}

\begin{proof}[Proof of Theorem \ref{thm:lq_risk}]
	We write $\hat{\mu}_{(\Sigma,\mu_0)}^{\seq}(\gamma_{\eta,\ast};\tau_{\eta,\ast})=\hat{\mu}_{\eta;(\Sigma,\mu_0)}^{\seq,\ast}$ in the proof.
	
	First we consider $1\leq  q\leq 2$. This is the easy case, as $\mathsf{g}_q(x)\equiv \pnorm{x-\mu_0}{q}/n^{1/q-1/2}$ is $1$-Lipschitz with respect to $\pnorm{\cdot}{}$. So applying Theorems \ref{thm:min_norm_dist} and \ref{thm:universality_min_norm} verifies the existence of some small $\vartheta>0$ such that for some $\mathcal{U}_{\vartheta} \subset B_n(1)$ with $\mathrm{vol}(\mathcal{U}_{\vartheta})/\mathrm{vol}(B_n(1))\geq 1- Ce^{-n^{\vartheta}/C}$,
	\begin{align*}
	\sup_{\mu_0 \in \mathcal{U}_{\vartheta} }\Prob\Big(\sup_{\eta \in \Xi_K} n^{\frac{1}{2}-\frac{1}{q}}\bigabs{\pnorm{\hat{\mu}_\eta-\mu_0}{q}- \E \pnorm{ \hat{\mu}_{\eta;(\Sigma,\mu_0)}^{\seq,\ast}-\mu_0 }{q} }\geq n^{-\vartheta}\Big)\leq C n^{-1/7}.
	\end{align*}
	The ratio formulation follows from Lemmas \ref{lem:lq_risk_lower_bound} and \ref{lem:lq_risk_gw} by further intersecting $\mathcal{U}_{\vartheta}$ and the set therein. 
	
	Next we consider $q \in (2,\infty)$. Let $L_n\equiv n^{\vartheta_1}$ for some $\vartheta_1$ to be chosen later. Using Proposition \ref{prop:delocal_u_v} and its proofs below (\ref{ineq:delocal_u_v_5}), for $\vartheta_1>0$ chosen small enough, we may find some $\mathcal{U}_{\vartheta_1}\subset B_n(1)$ with the desired volume estimate, such that  $\sup_{\mu_0 \in \mathcal{U}_{\vartheta_1}}\sup_{\eta \in \Xi_K}\pnorm{\E \hat{\mu}_{\eta;(\Sigma,\mu_0)}^{\seq,\ast}-\mu_0}{\infty}\leq L_n/\sqrt{n}$, and 
	\begin{align}\label{ineq:lq_risk_1}
	\sup_{\mu_0 \in \mathcal{U}_{\vartheta_1}} \Prob\bigg(\sup_{\eta \in \Xi_K}\Big\{\pnorm{\hat{\mu}_\eta-\mu_0}{\infty}\vee \pnorm{\hat{\mu}_{\eta;(\Sigma,\mu_0)}^{\seq,\ast}-\mu_0}{\infty}\Big\}\geq \frac{L_n}{\sqrt{n}} \bigg)\leq C n^{-2D},
	\end{align}
	where we choose $D>0$ sufficiently large. Recall for $x\in \R^n$ and $q> 2$, $\pnorm{x}{q}\leq  \pnorm{x}{}^{2/q}\pnorm{x}{\infty}^{1-2/q}$. This motivates the choice 
	\begin{align*}
	\mathsf{g}_q(x)\equiv \bigg[ \bigg(\frac{L_n}{\sqrt{n}}\bigg)^{\frac{2}{q}-1}\biggpnorm{(x-\mu_0) \wedge \bigg(\frac{L_n}{\sqrt{n}}\bigg)^{\frac{2}{q}-1}\vee \bigg\{-\bigg(\frac{L_n}{\sqrt{n}}\bigg)^{\frac{2}{q}-1}\bigg\} }{q}\bigg]^{\frac{q}{2}},
	\end{align*}
	which verifies that $\mathsf{g}_q$ is $1$-Lipschitz with respect to $\pnorm{\cdot}{}$. Using (\ref{ineq:lq_risk_1}), 
	\begin{align}\label{ineq:lq_risk_2}
	&\inf_{\mu_0 \in \mathcal{U}_{\vartheta_1}} \Prob\Big( \mathsf{g}_q(\hat{\mu}_\eta)=n^{(1-\frac{q}{2})\vartheta_1}\cdot \big\{ n^{\frac{1}{2}-\frac{1}{q}}\pnorm{\hat{\mu}_\eta-\mu_0}{q} \big\}^{\frac{q}{2}},\forall \eta \in \Xi_K\Big) \geq 1-C n^{-D},
	\end{align}
	and with $E_{\mu_0}\equiv \big\{\sup_{\eta \in \Xi_K} \pnorm{\hat{\mu}_{\eta;(\Sigma,\mu_0)}^{\seq,\ast}-\mu_0}{\infty}\leq L_n/\sqrt{n}\big\}$,
	\begin{align}\label{ineq:lq_risk_3}
	&\sup_{\mu_0 \in \mathcal{U}_{\vartheta_1}}\sup_{\eta \in \Xi_K}\Big|\E \mathsf{g}_q\big( \hat{\mu}_{\eta;(\Sigma,\mu_0)}^{\seq,\ast} \big)- n^{(1-\frac{q}{2})\vartheta_1}\E \big\{ n^{\frac{1}{2}-\frac{1}{q}}\pnorm{\hat{\mu}_{\eta;(\Sigma,\mu_0)}^{\seq,\ast}-\mu_0}{q} \Big\}^{\frac{q}{2}} \Big| \nonumber\\
	& = n^{(1-\frac{q}{2})\vartheta_1}\sup_{\mu_0 \in \mathcal{U}_{\vartheta_1}}\sup_{\eta \in \Xi_K}\E \big\{ n^{\frac{1}{2}-\frac{1}{q}}\pnorm{\hat{\mu}_{\eta;(\Sigma,\mu_0)}^{\seq,\ast}-\mu_0}{q} \Big\}^{ \frac{q}{2} } \bm{1}_{E_{\mu_0}^c}\nonumber\\
	&\qquad +\sup_{\mu_0 \in \mathcal{U}_{\vartheta_1}}\sup_{\eta \in \Xi_K}\E \mathsf{g}_q\big( \hat{\mu}_{\eta;(\Sigma,\mu_0)}^{\seq,\ast} \big)\bm{1}_{E_{\mu_0}^c}\lesssim n^{-D}. 
	\end{align}
	As the map $g \mapsto \pnorm{\hat{\mu}_{\eta;(\Sigma,\mu_0)}^{\seq,\ast}-\mu_0}{q}= \pnorm{(\Sigma+\tau_{\eta,\ast} I)^{-1}\big(-\tau_{\eta,\ast} \mu_0+\gamma_{\eta,\ast} \Sigma^{1/2} g/\sqrt{n}\big)}{q}$ is $Cn^{-1/2}$-Lipschitz with respect to $\pnorm{\cdot}{}$, Gaussian concentration yields
	\begin{align*}
	\Prob\Big(n^{1/2} \bigabs{\pnorm{\hat{\mu}_{\eta;(\Sigma,\mu_0)}^{\seq,\ast}-\mu_0}{q}- \E \pnorm{\hat{\mu}_{\eta;(\Sigma,\mu_0)}^{\seq,\ast}-\mu_0}{q}}\geq n^{\vartheta_1}\Big)\leq C n^{-2D}.
	\end{align*}
	Using the Lipschitz property of the maps, we may strengthen the above inequality to a uniform control over $\eta \in \Xi_K$. This means uniformly in  $\mu_0 \in \mathcal{U}_{\vartheta_1}, \eta\in \Xi_K$,
	\begin{align}\label{ineq:lq_risk_4}
	\Big|\E \Big\{ n^{\frac{1}{2}-\frac{1}{q}}\pnorm{\hat{\mu}_{\eta;(\Sigma,\mu_0)}^{\seq,\ast}-\mu_0}{q} \Big\}^{\frac{q}{2}}- \Big\{ n^{\frac{1}{2}-\frac{1}{q}}\E\pnorm{\hat{\mu}_{\eta;(\Sigma,\mu_0)}^{\seq,\ast}-\mu_0}{q} \Big\}^{\frac{q}{2}}\Big|\lesssim n^{-\frac{1}{q}+\vartheta_1}. 
	\end{align}
	Combining (\ref{ineq:lq_risk_3})-(\ref{ineq:lq_risk_4}), we have uniformly in  $\mu_0 \in \mathcal{U}_{\vartheta_1}, \eta\in \Xi_K$,
	\begin{align}\label{ineq:lq_risk_5}
	&\Big|\E \mathsf{g}_q\big( \hat{\mu}_{\eta;(\Sigma,\mu_0)}^{\seq,\ast} \big) -n^{(1-\frac{q}{2})\vartheta_1}\big\{ n^{\frac{1}{2}-\frac{1}{q}}\E\pnorm{\hat{\mu}_{\eta;(\Sigma,\mu_0)}^{\seq,\ast}-\mu_0}{q} \Big\}^{q/2} \Big|\leq C n^{(2-\frac{q}{2})\vartheta_1-\frac{1}{q}}. 
	\end{align}
	Combining (\ref{ineq:lq_risk_2}) and (\ref{ineq:lq_risk_5}) proves the existence of some small $\vartheta_2$ and some $\mathcal{U}_{\vartheta_2}\subset B_n(1)$ with the desired volume estimate, such that 
	\begin{align*}
	\sup_{\mu_0 \in \mathcal{U}_{\vartheta_2} }\Prob\Big(n^{\frac{1}{2}-\frac{1}{q}}\sup_{\eta \in \Xi_K}\bigabs{\pnorm{\hat{\mu}_\eta-\mu_0}{q}- \E \pnorm{ \hat{\mu}_{\eta;(\Sigma,\mu_0)}^{\seq,\ast}-\mu_0 }{q} }\geq n^{-\vartheta_2}\Big)\leq C n^{-1/7}.
	\end{align*}
	The ratio formulation follows again from Lemmas \ref{lem:lq_risk_lower_bound} and \ref{lem:lq_risk_gw}.
\end{proof}

\section{Proofs for Section \ref{section:err_phase_trans}}\label{section:proof_error}

\subsection{A rigorous version of (\ref{eqn:errors_explicit}) and its proof}

    The following theorem follows easily from Theorems \ref{thm:min_norm_dist} and \ref{thm:universality_min_norm}.
    
    \begin{theorem}\label{thm:errors_explicit}
    	Suppose Assumption \ref{assump:design} holds and the following hold for some $K>0$.
    	\begin{itemize}
    		\item $1/K\leq \phi^{-1} \leq K$, $\pnorm{\Sigma^{-1}}{\op}\vee \pnorm{\Sigma}{\op}\leq K$.
    		\item Assumption \ref{assump:noise} holds with $\sigma_\xi^2 \in [1/K,K]$. 
    	\end{itemize}
    	Fix a small enough $\vartheta \in (0,1/50)$. Then there exist a constant $C=C(K,\vartheta)>1$, and a measurable set $\mathcal{U}_{\vartheta}\subset B_n(1)$ with $\mathrm{vol}(\mathcal{U}_{\vartheta})/\mathrm{vol}(B_n(1))\geq 1-C e^{-n^{\vartheta}/C}$, such that for any $\epsilon \in (0,1/2]$, and $\# \in \{\pred,\est,\ins,\res\}$,
    	\begin{align*}
    	\sup_{\mu_0 \in \mathcal{U}_{\vartheta}}\Prob\Big(\sup_{\eta \in \Xi^{\#}}\abs{R^{\#}_{(\Sigma,\mu_0)}(\eta)-\bar{R}^{\#}_{(\Sigma,\mu_0)}(\eta)}\geq \epsilon \Big) \leq C\cdot
    	\begin{cases}
    	n e^{-n\epsilon^4/C}, & Z=G;\\
    	\epsilon^{-c_0}n^{-1/6.5}, & \hbox{otherwise}.
    	\end{cases}
    	\end{align*}
    	Here $\Xi^{\#}=\Xi_K$ for $\# \in \{\pred,\est\}$ and $\Xi^{\#}=[1/K,K]$ for  $\# \in \{\ins,\res\}$, and $c_0>0$ is universal. Moreover, when $Z=G$, the supremum in the above display extends to $\mu_0 \in B_n(1)$, and the constant $C>0$ does not depend on $\vartheta$.
    \end{theorem}
    
    \begin{remark}
    \begin{enumerate}
    	\item For $\# \in \{\ins,\res\}$, we may take $\Xi^{\#}=\Xi_K$ at the cost of an worsened probability estimate $C(ne^{-n\epsilon^{c_0}/C}+\epsilon^{-c_0}n^{-1/6.5}\bm{1}_{Z\neq G})$, cf. Lemma \ref{lem:university_squared_risk}.
    	\item    The closest non-asymptotic results on exact risk characterizations  related to our Theorem \ref{thm:errors_explicit}, appear to be those presented in (i) \cite[Theorems 2 and 5]{hastie2022surprises}, which proved non-asymptotic additive approximations $R^{\pred}_{(\Sigma,\mu_0)}(\eta)= \bar{R}^{\pred}_{(\Sigma,\mu_0)}(\eta)+\smallop(1)$, and (ii) \cite[Theorems 1 and 2]{cheng2024dimension}, which provided substantially refined, multiplicative approximations $R^{\pred}_{(\Sigma,\mu_0)}(\eta)/\bar{R}^{\pred}_{(\Sigma,\mu_0)}(\eta)=1+\smallop(1)$ that hold beyond the proportional regime. Both works \cite{hastie2022surprises,cheng2024dimension} leverage the closed form of the Ridge(less) estimator $\hat{\mu}_\eta$ to analyze the bias and variance terms in $R^{\pred}_{(\Sigma,\mu_0)}(\eta)$, by means of calculus for the resolvent of the sample covariance. Their analysis works under $\eta\gg n^{-c_0}$ for some suitable $c_0>0$. For the case $\#=\pred$, Theorem \ref{thm:errors_explicit} above complements the results in \cite{hastie2022surprises,cheng2024dimension} by providing uniform control in $\eta$ when $\phi^{-1}>1$ (under a set of different conditions). 
    \end{enumerate}
  
    \end{remark}

   \begin{proof}[Proof of Theorem \ref{thm:errors_explicit}]
    For $\vartheta$ chosen small enough, we fix $\mu_0 \in \mathcal{U}_\vartheta$, where $\mathcal{U}_\vartheta$ is specified in Theorem \ref{thm:universality_min_norm}. We omit the subscripts in $R^{\#}_{(\Sigma,\mu_0)}(\eta)=R^{\#}(\eta)$, $\bar{R}^{\#}_{(\Sigma,\mu_0)}(\eta)=\bar{R}^{\#}(\eta)$, and write $\hat{\mu}_{(\Sigma,\mu_0)}^{\seq}(\gamma_{\eta,\ast};\tau_{\eta,\ast})=\hat{\mu}_{\eta;(\Sigma,\mu_0)}^{\seq,\ast}$ in the proof. All the constants in $\lesssim,\gtrsim,\asymp$ and $\bigo$ below may possibly depend on $K$.

	\noindent (1). Consider the case $\#=\pred$. We omit the superscript $\pred$ as well. Using Theorem \ref{thm:universality_min_norm}-(1) with $\mathsf{g}(x)=\pnorm{\Sigma^{1/2}(x-\mu_0)}{}$, on an event $E_0$ with $\Prob(E_0^c)\leq C_0 \epsilon^{-c_0}n^{-1/6.5}$, 
	\begin{align*}
	\sup_{\eta \in \Xi_K}\bigabs{\sqrt{R(\eta)}- \E \pnorm{\Sigma^{1/2}  \big(\hat{\mu}_{\eta;(\Sigma,\mu_0)}^{\seq,\ast}-\mu_0\big)  }{} }\leq \epsilon.
	\end{align*}
	By Gaussian-Poincar\'e inequality, $
	0\leq \bar{R}(\eta)- \big(\E \pnorm{\Sigma^{1/2}  \big(\hat{\mu}_{\eta;(\Sigma,\mu_0)}^{\seq,\ast}-\mu_0\big)  }{}\big)^2 = \var\big(  \pnorm{\Sigma^{1/2}  (\hat{\mu}_{\eta;(\Sigma,\mu_0)}^{\seq,\ast}-\mu_0)  }{} \big) \lesssim n^{-1}$. As $\bar{R}(\eta)\asymp 1$ uniformly in $\eta \in \Xi_K$, on $E_0$,
	\begin{align}\label{ineq:university_squared_risk_1}
	\sup_{\eta \in \Xi_K}\bigabs{R^{1/2}(\eta)- \bar{R}^{1/2}(\eta)}\leq \epsilon+C_0' n^{-1}.
	\end{align}
   On the other hand, using both the standard form $\hat{\mu}_\eta = n^{-1}\big(X^\top X/n+\eta I_n\big)^{-1}X^\top Y$ and the alternative form $\hat{\mu}_\eta= n^{-1}X^\top \big(XX^\top/n+\eta I_m\big)^{-1} Y$,  we have
	\begin{align}\label{ineq:university_squared_risk_2_0}
	\sup_{\eta \in \Xi_K} \pnorm{\hat{\mu}_\eta}{} \lesssim \Big(\pnorm{(ZZ^\top/n)^{-1}}{\op}\bm{1}_{\phi^{-1}\geq 1+1/K}^{-1}\wedge 1\Big)\cdot \Big(1+\frac{\pnorm{Z}{\op}+\pnorm{\xi}{}}{\sqrt{n}}\Big)^2. 
	\end{align}
	Consequently, on an event $E_1$ with $\Prob(E_1^c)\leq C_1 e^{-n/C_1}$, 
	\begin{align}\label{ineq:university_squared_risk_2}
	\sup_{\eta \in \Xi_K} \pnorm{\hat{\mu}_\eta}{}\leq C_1. 
	\end{align}
	Finally, using (\ref{ineq:university_squared_risk_1}) and (\ref{ineq:university_squared_risk_2}), on $E_0\cap E_1$, 
	\begin{align*}
	&\sup_{\eta \in \Xi_K}\abs{R(\eta)-\bar{R}(\eta)}\lesssim \sup_{\eta \in \Xi_K}\abs{R^{1/2}(\eta)- \bar{R}^{1/2}(\eta)} \Big(1+\sup_{\eta \in \Xi_K} \pnorm{\hat{\mu}_\eta}{}\Big)  \lesssim \epsilon+n^{-1}.
	\end{align*}
	The claim follows. The case $\#=\est$ follows from minor modifications so will be omitted.
	
	\noindent (2). Consider the case $\#=\res$. We omit the superscript $\res$ as well. Further fix $\xi \in \mathcal{E}_\vartheta$ as specified in Theorem \ref{thm:universality_min_norm} (the concrete form of $\mathcal{E}_\vartheta$ is given in Proposition \ref{prop:delocal_u_v}). Using the same Theorem \ref{thm:universality_min_norm}-(2) with $\mathsf{h}(x)=\pnorm{x}{}$, 
	\begin{align*}
	\Prob^\xi\Big(\sup_{\eta \in [1/K,K]}\bigabs{  \pnorm{\hat{r}_\eta}{}- \E^\xi \pnorm{r_{\eta,\ast}}{} }\geq \epsilon\Big)\leq C \epsilon^{-c_0}\cdot n^{-1/6.5}. 
	\end{align*}
	By Gaussian-Poincar\'e inequality, $
	0\leq \E^\xi \pnorm{r_{\eta,\ast}}{}^2-\big(\E^\xi \pnorm{r_{\eta,\ast}}{}\big)^2= \var^{\xi}\big(\pnorm{r_{\eta,\ast}}{}\big)\lesssim 1/n$. 
	Combined with the fact that $\E^\xi \pnorm{r_{\eta,\ast}}{}^2=(\eta \gamma_{\eta,\ast}/\tau_{\eta,\ast})^2+\bigo(\abs{\pnorm{\xi}{}^2/m-\sigma_\xi^2})$, for $\eta \in [1/K,K]$, using the stability estimate in Proposition \ref{prop:fpe_est}-(3), 
	\begin{align*}
	\bigabs{\E^\xi \pnorm{r_{\eta,\ast}}{}-  \eta \gamma_{\eta,\ast}/\tau_{\eta,\ast} }\lesssim  \bigabs{\big(\E^\xi \pnorm{r_{\eta,\ast}}{}\big)^2-(\eta \gamma_{\eta,\ast}/\tau_{\eta,\ast})^2} \lesssim n^{-1/2+\vartheta}. 
	\end{align*}
	So for $\epsilon \in (C n^{-1/2+\vartheta},1/C]$, 
	\begin{align*}
	\Prob^\xi\Big(\sup_{\eta \in [1/K,K]}\bigabs{  \pnorm{\hat{r}_\eta}{}- \eta \gamma_{\eta,\ast}/\tau_{\eta,\ast} }\geq \epsilon\Big)\leq C \epsilon^{-c_0}\cdot n^{-1/6.5}. 
	\end{align*}
	Now taking expectation over $\xi$, for the same range of $\epsilon$,
	\begin{align}\label{ineq:university_squared_risk_3}
	\Prob\Big(\sup_{\eta \in [1/K,K]}\bigabs{  \pnorm{\hat{r}_\eta}{}- \eta \gamma_{\eta,\ast}/\tau_{\eta,\ast} }\geq \epsilon\Big)\leq C \epsilon^{-c_0}\cdot n^{-1/6.5}. 
	\end{align}
	On the other hand, using (\ref{ineq:university_squared_risk_2_0}),
	\begin{align*}
	\sup_{\eta \in [1/K,K]}\pnorm{\hat{r}_\eta}{}&\lesssim \Big(\pnorm{(ZZ^\top/n)^{-1}}{\op}\bm{1}_{\phi^{-1}\geq 1+1/K}^{-1}\wedge 1\Big)\cdot \Big(1+\frac{\pnorm{Z}{\op}+\pnorm{\xi}{}}{\sqrt{n}}\Big)^3.
	\end{align*}
	Consequently, on an event $E_3$ with $\Prob(E_3^c)\leq C_3 e^{-n/C_3}$, $
	\sup_{\eta \in [1/K,K]} \pnorm{\hat{r}_\eta}{}\leq C_3$, and therefore
	\begin{align*}
	\sup_{\eta \in [1/K,K]}\bigabs{  \pnorm{\hat{r}_\eta}{}^2- \big(\eta \gamma_{\eta,\ast}/\tau_{\eta,\ast}\big)^2 }
	&\leq C_3\cdot \sup_{\eta \in [1/K,K]}\bigabs{  \pnorm{\hat{r}_\eta}{}- \eta \gamma_{\eta,\ast}/\tau_{\eta,\ast} }. 
	\end{align*}
	The claim follows. The case $\#=\ins$ proceeds similarly, but with the function now taken as $\mathsf{h}(x)=\pnorm{x-\xi/\sqrt{n}}{}$, and the claim follows by computing that
	\begin{align*}
	&\E^\xi  \pnorm{r_{\eta,\ast}-\xi/\sqrt{n}}{}^2= \phi\cdot\bigg\{\bigg(\frac{\eta}{\phi \tau_{\eta,\ast}}\bigg)^2\big(\phi \gamma_{\eta,\ast}^2-\sigma_\xi^2\big)+\frac{\pnorm{\xi}{}^2}{m}\cdot \bigg(\frac{\eta}{\phi \tau_{\eta,\ast}}-1\bigg)^2 \bigg\}\\
	& = \bigg(\frac{\eta \gamma_{\eta,\ast}}{\tau_{\eta,\ast}}\bigg)^2+\phi \sigma_\xi^2\cdot \bigg[\bigg(\frac{\eta}{\phi \tau_{\eta,\ast}}-1\bigg)^2-\bigg(\frac{\eta}{\phi \tau_{\eta,\ast}}\bigg)^2\bigg]+\bigo\big(\abs{\pnorm{\xi}{}^2/m-\sigma_\xi^2}\big).
	\end{align*} 
	The proof is complete. 
    \end{proof}

\subsection{Proof of Theorem \ref{thm:error_rmt}}

\begin{lemma}\label{lem:conc_g_quad}
	Suppose $1/K\leq \phi^{-1}\leq K$, and $\pnorm{\Sigma}{\op}\vee \mathcal{H}_\Sigma\leq K$ for some $K>0$. Then with $g\sim \mathcal{N}(0,I_n)$, there exists some $C=C(K)>0$ such that for $\epsilon \in (0,1)$, and $q \in \{0,1/2\}$,
	\begin{align*}
	\Prob\Big(\sup_{\eta \in \Xi_K}\bigabs{\pnorm{(\Sigma+\tau_{\eta,\ast}I)^{-1}\Sigma^{q}g/\pnorm{g}{}}{}^2-n^{-1}\tr\big((\Sigma+\tau_{\eta,\ast}I)^{-2}\Sigma^{2q}\big)}>\epsilon\Big)\leq C\epsilon^{-1} e^{-n\epsilon^2/C}.
	\end{align*}
\end{lemma}
\begin{proof}
	We only prove the case $q=1/2$. All the constants in $\lesssim,\gtrsim,\asymp$ below may depend on $K$. We write $A_\eta\equiv (\Sigma+\tau_{\eta,\ast}I)^{-2}\Sigma$ for notational simplicity. Note that
	\begin{align*}
	&\bigabs{\pnorm{(\Sigma+\tau_{\eta,\ast}I)^{-1}\Sigma^{1/2}g/\pnorm{g}{}}{}^2-n^{-1}\tr\big((\Sigma+\tau_{\eta,\ast}I)^{-2}\Sigma\big)}\\
	& = n^{-1}\bigabs{e_g^{-2}\pnorm{A_\eta^{1/2}g}{}^2 - \E \pnorm{A_\eta^{1/2}g}{}^2 }\\
	& \lesssim e_g^{-2}\cdot n^{-1}\bigabs{ \pnorm{A_\eta^{1/2}g}{}^2-\E \pnorm{A_\eta^{1/2}g}{}^2  }+ \abs{e_g^{-2}-1}.
	\end{align*}
	Here in the last inequality we used $\E \pnorm{A_\eta^{1/2}g}{}^2\lesssim n$. As 
	\begin{itemize}
		\item $\pnorm{A_\eta}{F}^2 = \tr\big((\Sigma+\tau_{\eta,\ast} I)^{-4}\Sigma^2\big)\lesssim n (1\wedge\tau_{\eta,\ast})^{-4}\asymp n$, and
		\item $\pnorm{A_\eta}{F}^2\gtrsim \tr(\Sigma^2)\cdot (1\vee \tau_{\eta,\ast})^{-4}\gtrsim n$,
	\end{itemize}
	we have uniformly in $\eta \in [0,K]$, $\pnorm{A_\eta}{F}\asymp \sqrt{n}$. It is easy to see that $\pnorm{A_\eta}{\op}\asymp 1$. So by Hanson-Wright inequality, there exists some constant $C_1=C_1(K)$ such that for $\epsilon \in (0,1)$,
	\begin{align*}
	&\Prob\Big(\bigabs{\pnorm{(\Sigma+\tau_{\eta,\ast}I)^{-1}\Sigma^{1/2}g/\pnorm{g}{}}{}^2-n^{-1}\tr\big((\Sigma+\tau_{\eta,\ast}I)^{-2}\Sigma\big)}>\epsilon\Big)\\
	&\leq \Prob\Big(\bigabs{n^{-1}\big(\pnorm{A_\eta^{1/2}g}{}^2-\E \pnorm{A_\eta^{1/2}g}{}^2\big)  }>\epsilon/4\Big)+ \Prob\big(\abs{e_g^{-2}-1}>\epsilon/2\big)+\Prob(e_g^2\leq 1/2)\\
	&\leq C_1 e^{-n\epsilon^2/C_1}.
	\end{align*}
	On the other hand, for any $\eta_1,\eta_2 \in \Xi_K$, using Proposition \ref{prop:fpe_est}-(3), 
	\begin{align*}
	\bigabs{\pnorm{(\Sigma+\tau_{\eta_1,\ast}I)^{-1}\Sigma^{1/2}g/\pnorm{g}{}}{}^2-\pnorm{(\Sigma+\tau_{\eta_2,\ast}I)^{-1}\Sigma^{1/2}g/\pnorm{g}{}}{}^2}&\lesssim \abs{\eta_1-\eta_2},\\
	n^{-1}\bigabs{\tr\big((\Sigma+\tau_{\eta_1,\ast}I)^{-2}\Sigma\big)-\tr\big((\Sigma+\tau_{\eta_2,\ast}I)^{-2}\Sigma\big)}&\lesssim \abs{\eta_1-\eta_2},
	\end{align*}
	so we may conclude by a standard discretization and union bound argument.
\end{proof}

\begin{proposition}\label{prop:tau_gamma_m}
	The following hold with $ \mathfrak{m}_\eta\equiv  \mathfrak{m}(-\eta/\phi)$, $ \mathfrak{m}_\eta'\equiv  \mathfrak{m}'(-\eta/\phi)$. 
	\begin{enumerate}
		\item $\tau_{\eta,\ast}=1/\mathfrak{m}_\eta$ and $\partial_\eta \tau_{\eta,\ast} =  \mathfrak{m}_\eta'/(\phi  \mathfrak{m}_\eta^2)$. 
		\item It holds that
		\begin{align*}
		\frac{1}{n}\tr \big((\Sigma+\tau_{\eta,\ast} I)^{-2} \Sigma\big)& =\frac{\phi \mathfrak{m}_\eta^2}{\mathfrak{m}_\eta'}\big(\mathfrak{m}_\eta-({\eta}/{\phi})\mathfrak{m}_\eta'\big),\\
		\frac{1}{n}\tr \big((\Sigma+\tau_{\eta,\ast} I)^{-2}\big)& =\frac{\phi \mathfrak{m}_\eta^2}{\mathfrak{m}_\eta'}\big((\phi^{-1}-1) \mathfrak{m}_\eta' +2({\eta}/{\phi})\cdot \mathfrak{m}_\eta\mathfrak{m}_\eta'- \mathfrak{m}_\eta^2\big).
		\end{align*}
		\item Suppose $1/K\leq \phi^{-1} \leq K$, and $\pnorm{\Sigma}{\op}\vee \mathcal{H}_\Sigma\leq K$ for some $K>0$. There exists some constant $C=C(K)>0$ such that the following hold. For any $\epsilon \in (0,1/2]$, for some $\mathcal{U}_\epsilon\subset B_n(1)$ with $\mathrm{vol}(\mathcal{U}_\epsilon)/\mathrm{vol}(B_n(1))\geq 1- C \epsilon^{-1}e^{-n\epsilon^2/C}$, 
		\begin{align*}
		\sup_{\mu_0 \in \mathcal{U}_\epsilon} \sup_{\eta \in \Xi_K} \biggabs{\gamma_{\eta,\ast}^2-  \frac{\sigma_\xi^2 \mathfrak{m}_\eta'+ \pnorm{\mu_0}{}^2  \big(\phi\mathfrak{m}_\eta-\eta\mathfrak{m}_\eta'\big) }{\phi \mathfrak{m}_\eta^2} } \leq \epsilon.
		\end{align*}
		When $\Sigma=I_n$, we may take $\mathcal{U}_\epsilon=B_n(1)$ and the above inequality holds with $\epsilon=0$.
	\end{enumerate}
\end{proposition}

\begin{proof}
	(1) follows from definition so we focus on (2)-(3).

	\noindent (2). Differentiating both sides of (\ref{eq:traceSigma_to_m}) with respect to $\eta$ yields that
	\begin{align*}
	-n^{-1}\tr \big((\Sigma+\tau_{\eta,\ast} I)^{-2} \Sigma\big)\cdot \partial_\eta \tau_{\eta,\ast} = -\big(\mathfrak{m}_\eta-({\eta}/{\phi})\mathfrak{m}_\eta'\big).
	\end{align*}
	Now using $\partial_\eta \tau_{\eta,\ast} =  \mathfrak{m}_\eta'/(\phi  \mathfrak{m}_\eta^2)$ to obtain the formula for $n^{-1}\tr \big((\Sigma+\tau_{\eta,\ast} I)^{-2} \Sigma\big)$. 
	
	Next, using that $
	\phi- \frac{\eta}{\tau_{\eta,\ast}} =n^{-1}\tr \big((\Sigma+\tau_{\eta,\ast} I)^{-1} \Sigma\big) = 1-\tau_{\eta,\ast}\cdot n^{-1}\tr \big((\Sigma+\tau_{\eta,\ast} I)^{-1} \big)$,
	we may solve
	\begin{align*}
	n^{-1}\tr \big((\Sigma+\tau_{\eta,\ast} I)^{-1} \big) = \mathfrak{m}_\eta\big(1-\phi+\eta \cdot\mathfrak{m}_\eta\big). 
	\end{align*}
	Differentiating with respect to $\eta$ on both sides of the above display, we obtain
	\begin{align*}
	-n^{-1}\tr \big((\Sigma+\tau_{\eta,\ast} I)^{-2} \big)\cdot \partial_\eta \tau_{\eta,\ast} &= -\phi^{-1} {\mathfrak{m}_\eta'}\big(1-\phi+\eta \cdot\mathfrak{m}_\eta\big)+ \mathfrak{m}_\eta\cdot\big(\mathfrak{m}_\eta-({\eta}/{\phi})\mathfrak{m}_\eta'\big)\\
	& = -(\phi^{-1}-1) \mathfrak{m}_\eta' -2({\eta}/{\phi})\cdot \mathfrak{m}_\eta\mathfrak{m}_\eta'+ \mathfrak{m}_\eta^2,
	\end{align*}
	proving the second identity.

	\noindent (3).	Let $\mu_0 \equiv U_0 g_0/\pnorm{g_0}{}$, where $U_0\sim \mathrm{Unif}[0,1]$ and $g_0\sim \mathcal{N}(0,I_n)$ are independent variables. Then $\mu_0$ is uniformly distributed on $B_n(1)$. 
	For some $\epsilon>0$ to be chosen later, let 
	\begin{align}\label{ineq:tau_gamma_m_1}
	\mathcal{G}_\epsilon&\equiv \Big\{g \in \R^n: \sup_{\eta \in \Xi_K} \Big|\bigpnorm{(\Sigma+\tau_{\eta,\ast}I)^{-1}\Sigma^{1/2}\frac{g}{\pnorm{g}{}} }{}^2 - \frac{1}{n}\tr\big((\Sigma+\tau_{\eta,\ast}I)^{-2}\Sigma\big) \Big|\leq \epsilon\Big\}.
	\end{align}
	Let $\mathcal{U}_\epsilon\equiv \{U g/\pnorm{g}{}: U \in [0,1], g \in \mathcal{G}_\epsilon\}\subset B_n(1)$. Using Lemma \ref{lem:conc_g_quad}, there exists some constant $C_0=C_0(K)>0$ such that $
	{\mathrm{vol}(\mathcal{U}_\epsilon)}/{\mathrm{vol}(B_n(1))}=\Prob_{\mu_0}(\mu_0 \in \mathcal{U}_\epsilon)\geq 1-C_0 \epsilon^{-1} e^{-n\epsilon^2/C_0}$, and moreover, 
	\begin{align*}
	\sup_{\mu_0 \in \mathcal{U}_\epsilon} \sup_{\eta \in \Xi_K} \bigabs{\pnorm{(\Sigma+\tau_{\eta,\ast}I)^{-1}\Sigma^{1/2}\mu_0}{}^2- \pnorm{\mu_0}{}^2\cdot n^{-1}\tr\big((\Sigma+\tau_{\eta,\ast}I)^{-2}\Sigma\big) }\leq \epsilon.
	\end{align*}
	Note that when $\Sigma=I_n$, the above estimate holds for all $\mu_0 \in B_n(1)$ with $\epsilon=0$.
	
	Combining the above display with the formula (\ref{ineq:fpe_property_2}) for $\gamma_{\eta,\ast}^2$, and the fact that the denominator therein is of order $1$ (depending on $K$), we have 
	\begin{align*}
	\sup_{\mu_0 \in \mathcal{U}_\epsilon} \sup_{\eta \in \Xi_K} \biggabs{\gamma_{\eta,\ast}^2- \frac{\sigma_\xi^2+\pnorm{\mu_0}{}^2\tau_{\eta,\ast}^2\cdot \frac{1}{n}\tr\big((\Sigma+\tau_{\eta,\ast}I)^{-2}\Sigma\big)}{\frac{\eta}{\tau_{\eta,\ast}}+\tau_{\eta,\ast}\cdot \frac{1}{n}\tr\big((\Sigma+\tau_{\eta,\ast}I)^{-2}\Sigma\big)  }  }\leq C_1\epsilon.
	\end{align*}
	Now using (2), the second term in the above display equals to  
	\begin{align*}
	&\frac{\sigma_\xi^2+\pnorm{\mu_0}{}^2\cdot \frac{\phi }{\mathfrak{m}_\eta'}\big(\mathfrak{m}_\eta-\frac{\eta}{\phi}\mathfrak{m}_\eta'\big)}{\eta \mathfrak{m}_\eta+ \frac{\phi \mathfrak{m}_\eta}{\mathfrak{m}_\eta'}\big(\mathfrak{m}_\eta-\frac{\eta}{\phi}\mathfrak{m}_\eta'\big) } = \frac{\sigma_\xi^2 \mathfrak{m}_\eta'+ \pnorm{\mu_0}{}^2  \big(\phi\mathfrak{m}_\eta-\eta\mathfrak{m}_\eta'\big) }{\phi \mathfrak{m}_\eta^2}.
	\end{align*}
	The claim follows by adjusting constants. 
\end{proof}

\begin{proof}[Proof of Theorem \ref{thm:error_rmt}]
	As $\bar{R}^{\pred}_{(\Sigma,\mu_0)}(\eta) = \phi \gamma_{\eta,\ast}^2-\sigma_\xi^2$, directly invoking Proposition \ref{prop:tau_gamma_m}-(3) yields the claim for $\bar{R}^{\pred}_{(\Sigma,\mu_0)}(\eta)$.
	
	Next we handle $\bar{R}^{\est}_{(\Sigma,\mu_0)}(\eta)$. Note that
	\begin{align}\label{ineq:error_rmt_1}
	\bar{R}^{\est}_{(\Sigma,\mu_0)}(\eta) = \tau_{\eta,\ast}^2 \pnorm{(\Sigma+\tau_{\eta,\ast}I)^{-1}\mu_0}{}^2+ \gamma_{\eta,\ast}^2 \cdot n^{-1} \tr\big((\Sigma+\tau_{\eta,\ast} I)^{-2}\Sigma\big). 
	\end{align}
	Using a similar construction as in the proof of Proposition \ref{prop:tau_gamma_m} via the help of Lemma \ref{lem:conc_g_quad}, this time with $q=0$ therein, we may find some $\mathcal{U}_\epsilon \subset B_n(1)$ with the desired volume estimate, such that both Proposition \ref{prop:tau_gamma_m}-(3) and
	\begin{align}\label{ineq:error_rmt_2}
	\sup_{\mu_0 \in \mathcal{U}_\epsilon} \sup_{\eta \in \Xi_K} \bigabs{\pnorm{(\Sigma+\tau_{\eta,\ast}I)^{-1}\mu_0}{}^2- \pnorm{\mu_0}{}^2\cdot n^{-1}\tr\big((\Sigma+\tau_{\eta,\ast}I)^{-2}\big) }\leq \epsilon
	\end{align}
	hold. Combining (\ref{ineq:error_rmt_1})-(\ref{ineq:error_rmt_2}), we may set
	\begin{align*}
	\mathscr{R}^{\est}_{(\Sigma,\mu_0)}(\eta)&\equiv \tau_{\eta,\ast}^2 \pnorm{\mu_0}{}^2\cdot  n^{-1}\tr\big((\Sigma+\tau_{\eta,\ast}I)^{-2}\big)\\
	&\qquad + (\phi \mathfrak{m}_\eta^2)^{-1}\Big(\sigma_\xi^2 \mathfrak{m}_\eta'+ \pnorm{\mu_0}{}^2   \big(\phi\mathfrak{m}_\eta-\eta\mathfrak{m}_\eta'\big) \Big)\cdot n^{-1} \tr\big((\Sigma+\tau_{\eta,\ast} I)^{-2}\Sigma\big)\\
	&\equiv R_{2,1}+R_{2,2}.
	\end{align*}
	By Proposition \ref{prop:tau_gamma_m}-(2), we may compute $R_{2,1},R_{2,2}$ separately:
	\begin{align*}
	R_{2,1}& =  \pnorm{\mu_0}{}^2 \cdot \frac{\phi }{\mathfrak{m}_\eta'}\cdot \Big((\phi^{-1}-1) \mathfrak{m}_\eta' +2({\eta}/{\phi})\cdot \mathfrak{m}_\eta\mathfrak{m}_\eta'- \mathfrak{m}_\eta^2\Big)\\
	& = \pnorm{\mu_0}{}^2(1-\phi)+ \Big\{2\pnorm{\mu_0}{}^2 \eta \mathfrak{m}_\eta- \pnorm{\mu_0}{}^2\phi\cdot  \frac{\mathfrak{m}_\eta^2}{\mathfrak{m}_\eta'}\Big\},\\
	R_{2,2}& = \frac{1}{\mathfrak{m}_\eta'} \Big(\sigma_\xi^2 \mathfrak{m}_\eta'+ \pnorm{\mu_0}{}^2  \big(\phi\mathfrak{m}_\eta-\eta\mathfrak{m}_\eta'\big)\Big)\cdot \big(\mathfrak{m}_\eta-({\eta}/{\phi})\mathfrak{m}_\eta'\big)\\
	& = \sigma_\xi^2 \big(\mathfrak{m}_\eta-({\eta}/{\phi})\mathfrak{m}_\eta'\big)+ \phi^{-1} \pnorm{\mu_0}{}^2 \eta^2 \mathfrak{m}_\eta' - \Big\{2\pnorm{\mu_0}{}^2 \eta \mathfrak{m}_\eta- \pnorm{\mu_0}{}^2\phi\cdot \frac{\mathfrak{m}_\eta^2}{\mathfrak{m}_\eta'}\Big\}. 
	\end{align*}
	Consequently,
	\begin{align*}
	\mathscr{R}^{\est}_{(\Sigma,\mu_0)}(\eta)& = \pnorm{\mu_0}{}^2(1-\phi)+ \sigma_\xi^2 \big(\mathfrak{m}_\eta-({\eta}/{\phi})\mathfrak{m}_\eta'\big)+ \phi^{-1} \pnorm{\mu_0}{}^2 \eta^2 \mathfrak{m}_\eta'\\
	& = \sigma_\xi^2 \cdot \Big\{\SNR_{\mu_0}(1-\phi)+  \mathfrak{m}_\eta + ({\eta}/{\phi})\big(\eta\cdot \SNR_{\mu_0}-1\big) \mathfrak{m}_\eta'\Big\}.
	\end{align*}
	The claims for $\mathscr{R}^{\ins}_{(\Sigma,\mu_0)}(\eta)$ and $\mathscr{R}^{\res}_{(\Sigma,\mu_0)}(\eta)$ follow from Proposition \ref{prop:tau_gamma_m}-(3).
\end{proof}

\subsection{Proof of Proposition \ref{prop:derivative_R_simplify}}

We will prove the following version of Proposition \ref{prop:derivative_R_simplify}, where $\mathfrak{M}^{\#}$ is represented via $\tau_{\eta,\ast}$ instead of $\mathfrak{m}$. In the proof below, we will also verify the representation of $\mathfrak{M}^{\#}$ via $\mathfrak{m}$ as stated in Proposition \ref{prop:derivative_R_simplify}.

\begin{proposition}\label{prop:derivative_R}
	Recall $\SNR_{\mu_0}=\pnorm{\mu_0}{}^2/\sigma_\xi^2$. Then for $\# \in \{\pred,\est,\ins\}$,
	\begin{align*}
	\partial_\eta \mathscr{R}^{\#}_{(\Sigma,\mu_0)}(\eta) & = \sigma_\xi^2 \cdot \mathfrak{M}^{\#}(\eta)\cdot  \big(\eta\cdot \SNR_{\mu_0}-1\big).
	\end{align*}
	Here with $T_{-p,q}(\eta)\equiv n^{-1} \tr\big((\Sigma+\tau_{\ast}(\eta)I)^{-p}\Sigma^q\big)$ for $p,q \in \mathbb{N}$,
	\begin{align*}
	\mathfrak{M}^{\#}(\eta)\equiv 
	\begin{cases}
	\phi \big(-\tau_{\ast}''(\eta)\big), & \# = \pred;\\
	2(\tau_\ast'(\eta))^2\big(T_{-3,1}(\eta)+\tau_\ast'(\eta) T_{-2,1}(\eta)T_{-3,2}(\eta)\big), & \# = \est;\\
	\frac{2(\tau_\ast'(\eta))^2}{\tau_\ast^2(\eta)}\Big(\eta^2 \tau_\ast'(\eta) T_{-3,2}(\eta)+\tau_\ast^3(\eta) T_{-2,1}^2(\eta)\Big), & \#=\ins.
	\end{cases}
	\end{align*}
	Suppose further $1/K\leq \phi^{-1}\leq K$ and $\pnorm{\Sigma}{\op} \vee \mathcal{H}_\Sigma\leq K$ for some $K>0$. Then there exists some $C=C(K)>0$ such that uniformly in $\eta \in \Xi_K$ and for all $\# \in \{\pred,\est,\ins\}$,
	\begin{enumerate}
		\item $1/C\leq \mathfrak{M}^{\#}(\eta)\leq C$, and
		\item if $\eta_\ast\equiv \SNR_{\mu_0}^{-1} \in \Xi_K$, then $
		1/C\leq {\abs{\mathscr{R}^{\#}_{(\Sigma,\mu_0)}(\eta)- \mathscr{R}^{\#}_{(\Sigma,\mu_0)}(\eta_\ast)}}\big/{ \pnorm{\mu_0}{}^2(\eta-\eta_\ast)^2}\leq C$. 
	\end{enumerate}
	 
\end{proposition}

\begin{proof}
	In the proof we write $\tau_{\eta,\ast}=\tau_\eta$. Recall the notation $\mathfrak{m}_\eta=\mathfrak{m}(-\eta/\phi)$, $\mathfrak{m}_\eta'= \mathfrak{m}'(-\eta/\phi)$, and we naturally write $\mathfrak{m}_\eta''\equiv \mathfrak{m}''(-\eta/\phi)$. By differentiating with respect to $\eta$ for both sides of $\mathfrak{m}_\eta=1/\tau_\eta$, with some calculations we have
	\begin{align}\label{ineq:derivative_R_1}
	\mathfrak{m}_\eta = \tau_{\eta}^{-1},\quad \mathfrak{m}_\eta' =  {\phi \tau_{\eta}'}/{\tau_\eta^2},\quad \mathfrak{m}_\eta'' = -{\phi^2}\big(\tau_{\eta}''\tau_{\eta}-2(\tau_{\eta}')^2\big)/\tau_{\eta}^3.
	\end{align}	
	Using $\rho$, we may also write $
	\mathfrak{m}_\eta^{(q)} =  q!\int \frac{\rho(\d x)}{(x+\eta/\phi)^{q+1}}$ for $q \in \mathbb{N}$. Here by convention $0!=1$.

	\noindent (1). Using the formula for $\mathscr{R}^{\pred}_{(\Sigma,\mu_0)}$, 
	\begin{align*}
	\partial_\eta \mathscr{R}^{\pred}_{(\Sigma,\mu_0)}(\eta)& = \sigma_\xi^2\cdot \partial_\eta \Big\{ \mathfrak{m}_\eta^{-2}\big(\phi\cdot \SNR_{\mu_0} \mathfrak{m}_\eta-\big(\eta\cdot \SNR_{\mu_0}-1\big) \mathfrak{m}_\eta'\big)  \Big\}\\
	& =  \phi^{-1}\sigma_\xi^2\cdot \mathfrak{m}_\eta^{-3}\big(\mathfrak{m}_\eta\mathfrak{m}_\eta''-2(\mathfrak{m}_\eta')^2\big)\cdot \big(\eta\cdot \SNR_{\mu_0}-1\big).
	\end{align*}
	Some calculations show that
	\begin{align*}
	\mathfrak{m}_\eta\mathfrak{m}_\eta''-2(\mathfrak{m}_\eta')^2&=\tau_\eta^{-3}\phi^2(-\tau_\eta'')= 2\bigg\{\int \frac{\rho(\d x)}{(x+\eta/\phi)}\int \frac{\rho(\d x)}{(x+\eta/\phi)^3}-\bigg(\int \frac{\rho(\d x)}{(x+\eta/\phi)^2}\bigg)^2\bigg\},
	\end{align*}
	 so the identity follows.

	\noindent (2). Using the formula for $\mathscr{R}^{\est}_{(\Sigma,\mu_0)}$, 
	\begin{align}\label{ineq:derivative_R_2}
	\partial_\eta \mathscr{R}^{\est}_{(\Sigma,\mu_0)}(\eta) &=\sigma_\xi^2\cdot \partial_\eta \big(\SNR_{\mu_0}(1-\phi)+  \mathfrak{m}_\eta + ({\eta}/{\phi})(\eta\cdot \SNR_{\mu_0}-1) \mathfrak{m}_\eta'\big)\nonumber\\
	& = \phi^{-1}\sigma_\xi^2 \cdot \big(2\mathfrak{m}_\eta' -({\eta}/{\phi}) \mathfrak{m}_\eta'' \big)\cdot (\eta\cdot \SNR_{\mu_0}-1).
	\end{align}
	To compute the second term in the above display, recall the identity for $\tau_\eta',\tau_\eta''$ in (\ref{ineq:fpe_property_5_0})-(\ref{ineq:fpe_property_5_1}). Also recall $G_0(\eta)=\eta+\tau_\eta^2 T_{-2,1}(\eta)=\tau_\eta/\tau_\eta'$ defined in (\ref{ineq:fpe_property_5_0}). Then
	\begin{align*}
	&2\mathfrak{m}_\eta' -\frac{\eta}{\phi} \mathfrak{m}_\eta'' = \frac{\phi}{\tau_{\eta}}\bigg\{\frac{2\tau_\eta'}{\tau_\eta}\bigg(1-\frac{\eta\tau_\eta'}{\tau_\eta}\bigg)+\frac{\eta\tau_\eta''}{\tau_\eta}\bigg\} = \frac{\phi}{\tau_{\eta}}\bigg\{\frac{2\tau_\eta'}{\tau_\eta}\frac{\tau_\eta^2 T_{-2,1}(\eta)}{G_0(\eta)}-\frac{2\eta \tau_\eta\tau_\eta'}{G_0^2(\eta)}T_{-3,2}(\eta)\bigg\}\\
	& = \frac{2\phi \tau_{\eta}'}{G_0(\eta)} \Big(T_{-2,1}(\eta)-\frac{\eta}{G_0(\eta)} T_{-3,2}(\eta)\Big)\stackrel{(\ast)}{=}\frac{2\phi \tau_{\eta}'}{G_0(\eta)}\bigg\{\tau_\eta T_{-3,1}(\eta)+\bigg(1-\frac{\eta}{G_0(\eta)}\bigg)T_{-3,2}(\eta) \bigg\}\\
	& = 2\phi (\tau_\eta')^2\big(T_{-3,1}(\eta)+\tau_\eta' T_{-2,1}(\eta)T_{-3,2}(\eta)\big).
	\end{align*}
	Here in $(\ast)$ we used $T_{-2,1}(\eta)-T_{-3,2}(\eta)=\tau_\eta T_{-3,1}(\eta)$. The claimed identity follows by combining the above display and (\ref{ineq:derivative_R_2}). Using $\rho$, we may write 
	\begin{align*}
	2\mathfrak{m}_\eta' -({\eta}/{\phi}) \mathfrak{m}_\eta''&= 2\int \frac{x}{(x+\eta/\phi)^3}\,\rho(\d{x}).
	\end{align*}

	\noindent (3). Using the formula for $\mathscr{R}^{\ins}_{(\Sigma,\mu_0)}$, 
	\begin{align}\label{ineq:derivative_R_3}
	\partial_\eta \mathscr{R}^{\ins}_{(\Sigma,\mu_0)}(\eta) &=\sigma_\xi^2\cdot \partial_\eta\Big\{ \phi^{-1}\eta^2 \big(\phi\cdot \SNR_{\mu_0} \mathfrak{m}_\eta-(\eta\cdot \SNR_{\mu_0}-1) \mathfrak{m}_\eta'\big)+  (\phi- 2\eta\mathfrak{m}_\eta)\Big\}\nonumber\\
	& =\sigma_\xi^2\cdot \big(2\mathfrak{m}_\eta-4({\eta}/{\phi})\mathfrak{m}_\eta'+\phi^{-2}\eta^2 \mathfrak{m}_\eta''\big)\cdot (\eta\cdot\SNR_{\mu_0}-1).
	\end{align}
	The second term in the above display requires some non-trivial calculations: 
	\begin{align*}
	2\mathfrak{m}_\eta-\frac{4\eta}{\phi}\mathfrak{m}_\eta'+\frac{\eta^2}{\phi^2} \mathfrak{m}_\eta''&=\frac{1}{\tau_\eta}\bigg\{2-4\eta \frac{\tau_\eta'}{\tau_\eta}-\eta^2 \frac{\tau_\eta''}{\tau_\eta}+2\eta^2 \bigg(\frac{\tau_\eta'}{\tau_\eta}\bigg)^2\bigg\}\\
	& = \frac{1}{\tau_\eta}\bigg\{2-\frac{4\eta}{G_0(\eta)}+\eta^2 \frac{2\tau_\eta\tau_\eta' T_{-3,2}(\eta)}{G_0^2(\eta)}+\frac{2\eta^2}{G_0^2(\eta)}\bigg\}\\
	&= \frac{2}{\tau_\eta G_0^2(\eta)}\big\{G_0^2(\eta)-2\eta G_0(\eta)+\eta^2 \tau_\eta\tau_\eta' T_{-3,2}(\eta)+\eta^2\big\}.
	\end{align*}
	Expanding the $G_0(\eta)$ terms in the bracket using $G_0(\eta)=\eta+\tau_\eta^2 T_{-2,1}(\eta)$, with some calculations we arrive at
	\begin{align*}
	2\mathfrak{m}_\eta-\frac{4\eta}{\phi}\mathfrak{m}_\eta'+\frac{\eta^2}{\phi^2} \mathfrak{m}_\eta''&=\frac{2}{ G_0^2(\eta)}\Big(\eta^2 \tau_\eta' T_{-3,2}(\eta)+\tau_\eta^3 T_{-2,1}^2(\eta)\Big). 
	\end{align*}
	The claimed identity follows by combining the above display and (\ref{ineq:derivative_R_3}). Using $\rho$,  
	\begin{align*}
	2\mathfrak{m}_\eta-\frac{4\eta}{\phi}\mathfrak{m}_\eta'+\frac{\eta^2}{\phi^2} \mathfrak{m}_\eta''= 2 \int \frac{x^2}{(x+\eta/\phi)^3}\,\rho(\d{x}).
	\end{align*}	
	Finally, the claimed first two-sided bound on $\mathfrak{M}^{\#}$ follows from Proposition \ref{prop:fpe_est}, and the second bound follows by using the fundamental theorem of calculus. 
\end{proof}

\subsection{A rigorous version of (\ref{eqn:opt_reg_l2}) and its proof} \label{sec:thm_statement_opt_reg_l2}

The theorem below presents a rigorous formulation of (\ref{eqn:opt_reg_l2}).
\begin{theorem}\label{thm:small_intep}
	Suppose Assumptions \ref{assump:design}-\ref{assump:noise} hold, and $\pnorm{\Sigma^{-1}}{\op}\vee \pnorm{\Sigma}{\op}\leq K$ for some $K>0$. Fix a small enough $\vartheta \in (0,1/50)$.  The following hold for all $\# \in \{\pred,\est,\ins\}$. 
	\begin{enumerate}
		\item (\textbf{Noisy case}). Suppose $1/K\leq \phi^{-1} \leq K$  and $1/K\leq \sigma_\xi^2\leq K$. Fix $\delta \in (0,1/2]$ and $L\geq K/\delta^2$. There exist a constant $C=C(K,L,\delta,\vartheta)>0$ and a measurable set $\mathcal{U}_{\delta,\vartheta}\subset B_n(1)\setminus B_n(\delta)$ with $\mathrm{vol}(\mathcal{U}_{\delta,\vartheta})/\mathrm{vol}(B_n(1)\setminus B_n(\delta))\geq 1- C e^{-n^{\vartheta}/C}$, such that
		\begin{align*}
		\sup_{\mu_0 \in \mathcal{U}_{\delta,\vartheta}}\Prob\bigg(\inf_{\eta' \in \Xi_L: \abs{\eta'-\SNR_{\mu_0}^{-1}}\geq \delta}\bigabs{R^{\#}_{(\Sigma,\mu_0)}(\eta')- \min_{\eta \in \Xi_{L}}R^{\#}_{(\Sigma,\mu_0)}(\eta)}<\frac{1}{C}\bigg)\leq Cn^{-1/7}.
		\end{align*}
		\item (\textbf{Noiseless case}). Suppose $1+1/K\leq \phi^{-1} \leq K$  and $\sigma_\xi^2=0$. There exist a constant $C=C(K,\vartheta)>0$ and a measurable set $\mathcal{U}_{\vartheta}\subset B_n(1)$ with $\mathrm{vol}(\mathcal{U}_{\vartheta})/\mathrm{vol}(B_n(1))\geq 1- C e^{-n^{\vartheta}/C}$, such that
		\begin{align*}
		\sup_{\mu_0 \in \mathcal{U}_{\vartheta}}\Prob\Big(R^{\#}_{(\Sigma,\mu_0)}(0)\geq \min_{\eta \in [0,K]}R^{\#}_{(\Sigma,\mu_0)}(\eta)+n^{-\vartheta}\Big)\leq Cn^{-1/7}.
		\end{align*}
	\end{enumerate}
\end{theorem}

We need a few lemmas to prove Theorem \ref{thm:small_intep}.

The following lemma gives a technical extension of Theorem \ref{thm:errors_explicit} for $\# \in \{\pred,\est,\ins\}$ under  $\sigma_\xi^2\approx 0$ when $\phi^{-1}>1$. For $\# = \ins$, the extension also allows uniform control over $\eta \approx 0$ under both the above small variance scenario with $\phi^{-1}>1$, and under the original conditions.

\begin{lemma}\label{lem:university_squared_risk}
	Suppose Assumption \ref{assump:design} holds and the following hold for some $K>0$.
	\begin{itemize}
		\item $1+1/K\leq \phi^{-1} \leq K$, $ \pnorm{\Sigma^{-1}}{\op}\vee \pnorm{\Sigma}{\op}\leq K$.
		\item Assumption \ref{assump:noise}  with $\sigma_\xi^2 \in [0,K]$. 
	\end{itemize} 
   Fix a small enough $\vartheta \in (0,1/50)$. Then there exist a constant $C=C(K,\vartheta)>1$, and a measurable set $\mathcal{U}_{\vartheta}\subset B_n(1)$ with $\mathrm{vol}(\mathcal{U}_{\vartheta})/\mathrm{vol}(B_n(1))\geq 1-C e^{-n^{\vartheta}/C}$, such that for any $\epsilon \in (0,1/2]$, and $\# \in \{\pred,\est,\ins,\res\}$,
	\begin{align*}
	&\sup_{\mu_0 \in \mathcal{U}_\vartheta}\Prob\bigg(\sup_{\eta \in \Xi_K}\abs{R^{\#}_{(\Sigma,\mu_0)}(\eta,\sigma_\xi)-\bar{R}^{\#}_{(\Sigma,\mu_0)}(\eta,\sigma_\xi)}\geq \epsilon \bigg) \leq C\cdot
	\begin{cases}
	 n e^{-n\epsilon^{c_0}/C}, & Z=G;\\
	\epsilon^{-c_0}n^{-1/6.5}, & \hbox{otherwise}.
	\end{cases}
	\end{align*}
\end{lemma}
\begin{proof}	
	All the constants in $\lesssim,\gtrsim,\asymp$ below may possibly depend on $K$.
	
	\noindent (\textbf{Part 1}). We shall first extend the claim of Theorem \ref{thm:errors_explicit} for $\#=\pred$ to $\sigma_\xi^2\in [0,K]$ in the case $\phi^{-1}\geq 1+1/K$. Note that uniformly in $\eta \in [0,K]$, for $\sigma_\xi,\sigma_\xi' \in [0,K]$,
	\begin{align}\label{ineq:university_squared_risk_4}
	\pnorm{\hat{\mu}_\eta(\sigma_\xi)-\hat{\mu}_\eta(\sigma_\xi')}{}
	&\lesssim \abs{\sigma_\xi-\sigma_\xi'}\cdot n^{-1}\pnorm{Z}{\op}\pnorm{\xi_0}{}\cdot \pnorm{(ZZ^\top/n)^{-1}}{\op}.
	\end{align}
	Using the estimate (\ref{ineq:university_squared_risk_2_0}), uniformly in $\eta \in [0,K]$, for all $\sigma_\xi,\sigma_\xi' \in [0,K]$,
	\begin{align*}
	\abs{R^{\pred}_{(\Sigma,\mu_0)}(\eta,\sigma_\xi)-R^{\pred}_{(\Sigma,\mu_0)}(\eta,\sigma_\xi')}& \lesssim \pnorm{\hat{\mu}_\eta(\sigma_\xi)-\hat{\mu}_\eta(\sigma_\xi')}{}\cdot \big(\pnorm{\hat{\mu}_\eta(\sigma_\xi)}{}+\pnorm{\hat{\mu}_\eta(\sigma_\xi')}{}+\pnorm{\mu_0}{}\big)\\
	&\lesssim \abs{\sigma_\xi-\sigma_\xi'}\cdot  \pnorm{(ZZ^\top/n)^{-1}}{\op}^2\cdot \Big(1+\frac{\pnorm{Z}{\op}+\pnorm{\xi_0}{}}{\sqrt{n}}\Big)^4.
	\end{align*}
	So on an event $E_1$ with $\Prob(E_1)\geq 1- C_1 e^{-n/C_1}$, for $\sigma_\xi,\sigma_\xi' \in [0,K]$,
	\begin{align*}
	\sup_{\eta \in [0,K]} \abs{R^{\pred}_{(\Sigma,\mu_0)}(\eta,\sigma_\xi)-R^{\pred}_{(\Sigma,\mu_0)}(\eta,\sigma_\xi')}\leq C_1\cdot \abs{\sigma_\xi-\sigma_\xi'}.
	\end{align*}
	On the other hand,  using Lemma \ref{lem:fpe_noiseless}-(2),
	\begin{align*}
	&\sup_{\eta \in [0,K]}\abs{\bar{R}^{\pred}_{(\Sigma,\mu_0)}(\eta,\sigma_\xi)-\bar{R}^{\pred}_{(\Sigma,\mu_0)}(\eta,\sigma_\xi')}
	\leq C_1\cdot \abs{\sigma_\xi-\sigma_\xi'}.
	\end{align*}
	Using the above two displays, for any $\epsilon>0$, by choosing $\sigma_\xi'\equiv \epsilon/(2C_1)$, we have for any $\sigma_\xi \leq \sigma_\xi'$,
	\begin{align}\label{ineq:university_squared_risk_5}
	&\Prob\Big(\sup_{\eta \in [0,K]}\abs{R^{\pred}_{(\Sigma,\mu_0)}(\eta,\sigma_\xi)-\bar{R}^{\pred}_{(\Sigma,\mu_0)}(\eta,\sigma_\xi)}\geq 2\epsilon\Big)\nonumber\\
	&\leq \Prob\Big(\sup_{\eta \in [0,K]}\abs{R^{\pred}_{(\Sigma,\mu_0)}(\eta,\sigma_\xi')-\bar{R}^{\pred}_{(\Sigma,\mu_0)}(\eta,\sigma_\xi')}\geq  \epsilon\Big)+C_1 e^{-n/C_1}. 
	\end{align}
	The first term on the right hand side of the above display can be handled by the proven claim in Theorem \ref{thm:errors_explicit}, upon noting that (i) the constant $C$ therein depends on $K$ polynomially, and here we choose $K$ to be larger than $2C_1/\epsilon$; (ii) $(n/\epsilon)^{C'} e^{-n\epsilon^{C'}}\wedge 1\leq n e^{-n\epsilon^{C''}}$ holds for $C''$ chosen much larger than $C'$.
	
	The extension of the claim of Theorem \ref{thm:errors_explicit} for $\#=\est$ to $\sigma_\xi^2\in [0,K]$ follows a similar proof with minor modifications, so we omit the details.
	
	\noindent (\textbf{Part 2}). Next we consider the case $\# = \ins$. We need to extend the corresponding claim of Theorem \ref{thm:errors_explicit} to both $\sigma_\xi^2\in [0,K]$ and $\eta \in [0,K]$. 
	
	We first verify the (high probability) Lipschitz continuity of the maps $\sigma_\xi \mapsto R^{\ins}_{(\Sigma,\mu_0)}(\eta,\sigma_\xi), \bar{R}^{\ins}_{(\Sigma,\mu_0)}(\eta,\sigma_\xi)$. Note that uniformly in $\eta \in [0,K]$, by virtue of (\ref{ineq:university_squared_risk_4}), for any $\sigma_\xi,\sigma_\xi' \in [0,K]$,
	\begin{align*}
	&\abs{R^{\ins}_{(\Sigma,\mu_0)}(\eta,\sigma_\xi)-R^{\ins}_{(\Sigma,\mu_0)}(\eta,\sigma_\xi')}\\
	&\lesssim \Big(1+\frac{\pnorm{Z}{\op}}{\sqrt{n}}\Big)^2\cdot \pnorm{\hat{\mu}_\eta(\sigma_\xi)-\hat{\mu}_\eta(\sigma_\xi')}{}\cdot \big(\pnorm{\hat{\mu}_\eta(\sigma_\xi)}{}+\pnorm{\hat{\mu}_\eta(\sigma_\xi')}{}+\pnorm{\mu_0}{}\big)\\
	&\lesssim \abs{\sigma_\xi-\sigma_\xi'}\cdot \pnorm{(ZZ^\top/n)^{-1}}{\op}^2\cdot \Big(1+\frac{\pnorm{Z}{\op}+\pnorm{\xi_0}{}}{\sqrt{n}}\Big)^6.
	\end{align*}
	This verifies the high probability Lipschitz property of $\sigma_\xi \mapsto R^{\ins}_{(\Sigma,\mu_0)}(\eta,\sigma_\xi)$. The Lipschitz property of $\sigma_\xi \mapsto \bar{R}^{\ins}_{(\Sigma,\mu_0)}(\eta,\sigma_\xi)$ is easily verified. From here we may use a similar argument to (\ref{ineq:university_squared_risk_5}) to conclude the extension of the claim of Theorem \ref{thm:errors_explicit} for $\#=\ins$ to $\sigma_\xi^2\in [0,K]$. 
	
	Finally we verify the (high probability) Lipschitz continuity of the maps $\eta \mapsto R^{\ins}_{(\Sigma,\mu_0)}(\eta,\sigma_\xi), \bar{R}^{\ins}_{(\Sigma,\mu_0)}(\eta,\sigma_\xi)$. Using the estimates (\ref{ineq:dist_min_norm_6}) (with $G$ replaced by $Z$) and (\ref{ineq:university_squared_risk_2_0}), uniformly in $\sigma_\xi \in [0,K]$ and $\eta_1,\eta_2 \in [0,K]$,
	\begin{align*}
	&\abs{R^{\ins}_{(\Sigma,\mu_0)}(\eta_1,\sigma_\xi)-R^{\ins}_{(\Sigma,\mu_0)}(\eta_2,\sigma_\xi)}\\
	&\lesssim \Big(1+\frac{\pnorm{Z}{\op}}{\sqrt{n}}\Big)^2\cdot \pnorm{\hat{\mu}_{\eta_1}(\sigma_\xi)-\hat{\mu}_{\eta_2}(\sigma_\xi)}{}\cdot \big(\pnorm{\hat{\mu}_{\eta_1}(\sigma_\xi)}{}+\pnorm{\hat{\mu}_{\eta_2}(\sigma_\xi)}{}+\pnorm{\mu_0}{}\big)\\
	&\lesssim \Big(1+\frac{\pnorm{Z}{\op}+\pnorm{\xi_0}{}}{\sqrt{n}}\Big)^6\cdot\pnorm{(ZZ^\top/n)^{-1}}{\op}^3\cdot \abs{\eta_1-\eta_2}.
	\end{align*}
	The Lipschitz property of $\sigma_\xi \mapsto \bar{R}^{\ins}_{(\Sigma,\mu_0)}(\eta,\sigma_\xi)$ is again easily verified. Again from here we may argue similarly to (\ref{ineq:university_squared_risk_5}) to extend the claim of Theorem \ref{thm:errors_explicit} for $\#=\ins$ to $\eta\in [0,K]$. The case for $\# = \res$ is similar so we omit repetitive details.
\end{proof}

\begin{lemma}\label{lem:fpe_noiseless}
	Suppose $\phi^{-1}>1$. The following hold.
	\begin{enumerate}
		\item The system of equations 
		\begin{align*}
		\begin{cases}
		\phi\gamma^2  = \E \err_{(\Sigma,\mu_0)}(\gamma;\tau),\\
		\phi-\frac{\eta}{\tau} = \gamma^{-2}\E\dof_{(\Sigma,\mu_0)}(\gamma;\tau) = \frac{1}{n}\tr\big((\Sigma+\tau I)^{-1}\Sigma\big)
		\end{cases}
		\end{align*}
		admit a unique solution $(\gamma_{\eta,\ast}(0),\tau_{\eta,\ast}(0)) \in [0,\infty)\times (0,\infty)$.
		\item It holds that $\tau_{\eta,\ast}(0)=\tau_{\eta,\ast}(\sigma_\xi)$. If furthermore $1+1/K\leq \phi^{-1}\leq K$ and $\pnorm{\Sigma}{\op}\vee \mathcal{H}_\Sigma\leq K$ for some $K>0$, then there exists some $C=C(K)>0$ such that $\abs{\gamma_{\eta,\ast}^2(\sigma_\xi)-\gamma_{\eta,\ast}^2(0)}\leq C \sigma_\xi^2$. 
	\end{enumerate}
\end{lemma}
\begin{proof}
	\noindent The claim (1) follows verbatim from the proof of Proposition \ref{prop:fpe_est}-(1) by setting $\sigma_\xi^2=0$ therein. The claim (2) follows by using the formula (\ref{ineq:fpe_property_2}).
\end{proof}

\begin{proof}[Proof of Theorem \ref{thm:small_intep}]
	Let $\mathcal{U}_\vartheta\subset B_n(1)$ be as specified in Theorem \ref{thm:errors_explicit} or \ref{thm:universality_min_norm}. In view of its explicit form given in Proposition \ref{prop:delocal_u_v}, with $\mathcal{U}_{\delta,\vartheta}\equiv \mathcal{U}_\vartheta\cap \big(B_n(1)\setminus B_n(\delta)\big)$, the volume estimates $\min\big\{\mathrm{vol}(\mathcal{U}_\vartheta)/\mathrm{vol}(B_n(1)), \mathrm{vol}(\mathcal{U}_{\delta,\vartheta})/\mathrm{vol}(B_n(1)\setminus B_n(\delta))\big\}\geq 1-C e^{-n^{\vartheta}/C}$ hold.
	
	On the other hand, using the construction around (\ref{ineq:tau_gamma_m_1}), we may find some $\mathcal{V}_\epsilon \subset B_n(1), \mathcal{V}_{\epsilon,\delta}\subset B_n(1)\setminus B_n(\delta)$ (for the latter, we take $U_0\sim \mathrm{Unif}(\delta,1)$ therein) with $\min\big\{\mathrm{vol}(\mathcal{V}_\epsilon)/\mathrm{vol}(B_n(1)), \mathrm{vol}(\mathcal{V}_{\epsilon,\delta})/\mathrm{vol}(B_n(1)\setminus B_n(\delta))\big\}\geq 1-C\epsilon^{-1}e^{-n\epsilon^2/C}$, such that for $\# \in \{\pred,\est,\ins\}$,
	\begin{align}\label{ineq:small_intep_1}
	\sup_{\mu_0 \in \{\mathcal{V}_\epsilon,\mathcal{V}_{\epsilon,\delta}\}} \sup_{\eta \in \Xi_L}\abs{\bar{R}^{\#}_{(\Sigma,\mu_0)}(\eta)- \mathscr{R}^{\#}_{(\Sigma,\mu_0)}(\eta)} \leq \epsilon.
	\end{align}
	Now let 
	\begin{align}\label{ineq:small_intep_1a}
	\mathcal{W}_{\epsilon,\vartheta}\equiv \mathcal{U}_{\vartheta}\cap \mathcal{V}_{\epsilon},\quad \mathcal{W}_{\epsilon,\delta,\vartheta}\equiv \mathcal{U}_{\delta,\vartheta}\cap \mathcal{V}_{\epsilon,\delta}.
	\end{align}
	Then we have the volume estimates $\min\big\{\mathrm{vol}(\mathcal{W}_{\epsilon,\vartheta})/\mathrm{vol}(B_n(1)), \mathrm{vol}(\mathcal{W}_{\epsilon,\delta,\vartheta})/\mathrm{vol}(B_n(1)\setminus B_n(\delta))\big\}\geq 1-C\epsilon^{-1}e^{-n\epsilon^2/C}-C e^{-n^{\vartheta}/C}$.
	
	Moreover, by Proposition \ref{prop:derivative_R}, provided $\eta_\ast\equiv \sigma_\xi^2/\pnorm{\mu_0}{}^2=\SNR_{\mu_0}^{-1}\in \Xi_L$, 
	\begin{align}\label{ineq:small_intep_2}
     \pnorm{\mu_0}{}^2/C_0\leq \frac{\abs{\mathscr{R}^{\#}_{(\Sigma,\mu_0)}(\eta)- \mathscr{R}^{\#}_{(\Sigma,\mu_0)}(\eta_\ast)}}{ (\eta-\eta_\ast)^2}\leq C_0\pnorm{\mu_0}{}^2
	\end{align}
	holds uniformly in $\eta \in \Xi_L$ for some $C_0>0$.
	
	\noindent (\textbf{Noisy case $\sigma_\xi^2\in [1/K,K]$}). Fix $\mu_0 \in \mathcal{W}_{\epsilon,\delta,\vartheta}$. Under the assumed conditions, $\eta_\ast \in \Xi_L$. So using the estimates (\ref{ineq:small_intep_1}) and (\ref{ineq:small_intep_2}), for any $\eta'\geq \eta_\ast$,
	\begin{align*}
	\bar{R}^{\#}_{(\Sigma,\mu_0)}(\eta')- \inf_{\eta \in \Xi_L}\bar{R}^{\#}_{(\Sigma,\mu_0)}(\eta)&\geq \mathscr{R}^{\#}_{(\Sigma,\mu_0)}(\eta')-\inf_{\eta \in \Xi_L} \mathscr{R}^{\#}_{(\Sigma,\mu_0)}(\eta)-2\epsilon\geq  \frac{\delta^2(\eta'-\eta_\ast)^2}{C_0}-2\epsilon. 
	\end{align*}
	Combined with a similar inequality for $\eta'\leq \eta_\ast$, we conclude that for any $\mu_0 \in \mathcal{W}_{\epsilon,\delta,\vartheta}$ and $\eta' \in \Xi_L$,
	\begin{align*}
	\bigabs{\bar{R}^{\#}_{(\Sigma,\mu_0)}(\eta')- \inf_{\eta \in \Xi_L}\bar{R}^{\#}_{(\Sigma,\mu_0)}(\eta)}\geq \frac{\delta^2(\eta'-\eta_\ast)^2}{C_0}-2\epsilon.
	\end{align*}
	Now for $\abs{\eta'-\eta_\ast}\geq \Delta$, choosing $\epsilon \equiv \epsilon_0\equiv \delta^2\Delta^2/(4C_0)$, we have
	\begin{align*}
	\inf_{\mu_0 \in \mathcal{W}_{\epsilon,\delta,\vartheta}} \inf_{\eta' \in \Xi_L: \abs{\eta'-\eta_\ast}\geq \Delta}\bigabs{\bar{R}^{\#}_{(\Sigma,\mu_0)}(\eta')- \inf_{\eta \in \Xi_L}\bar{R}^{\#}_{(\Sigma,\mu_0)}(\eta)}\geq \frac{\delta^2\Delta^2}{2C_0}.
	\end{align*}
	From here the claim follows from Theorem \ref{thm:errors_explicit}.

	\noindent (\textbf{Noiseless case $\sigma_\xi^2=0$}). In this case, (\ref{ineq:small_intep_2}) implies that the map $\eta \mapsto  \mathscr{R}^{\#}_{(\Sigma,\mu_0)}(\eta)$ attains global minimum at $\eta=0$. So together with (\ref{ineq:small_intep_1}), it implies that uniformly in $\mu_0 \in \mathcal{W}_{\epsilon,\vartheta}$,
	\begin{align*}
	\bigabs{\min_{\eta \in [0,K]}\bar{R}^{\#}_{(\Sigma,\mu_0)}(\eta)- \bar{R}^{\#}_{(\Sigma,\mu_0)}(0)} \leq \epsilon.
	\end{align*}
	From here the claim follows from  Lemma \ref{lem:university_squared_risk} that holds for $\sigma_\xi=0$.
\end{proof}

\section{Proofs for Section \ref{section:application}}\label{section:proof_application}
\subsection{A rigorous version of (\ref{eqn:noi_reg_est}) and its proof}

\begin{theorem}\label{thm:trace_app}
	Suppose Assumption \ref{assump:design} holds, and $1/K\leq \phi^{-1} \leq K$, $\pnorm{\Sigma^{-1}}{\op}\vee \pnorm{\Sigma}{\op}\leq K$ hold for some $K>0$.
	\begin{enumerate}
		\item For any small  $\epsilon > 0$, there exists some $C_1=C_1(K,\epsilon)>0$ such that
		\begin{align*}
		\Prob \Big(\sup_{\eta \in \Xi_K}  \abs{\hat{\tau}_\eta-\tau_{\eta,\ast}}  \geq n^{-1/2+\epsilon} \Big)
		\leq C_1 n^{-100}.
		\end{align*}
		\item Suppose further Assumption \ref{assump:noise} holds with either (i) $\sigma_\xi^2 \in [1/K,K]$ or (ii) $\sigma_\xi^2 \in [0,K]$ with $1+1/K\leq \phi^{-1}\leq K$. Fix a small enough constant $\vartheta \in (0,1/50)$. Then there exist a constant $C_2=C_2(K,\vartheta)>1$, and a measurable set $\mathcal{U}_{\vartheta}\subset B_n(1)$ with $\mathrm{vol}(\mathcal{U}_{\vartheta})/\mathrm{vol}(B_n(1))\geq 1-C e^{-n^{\vartheta}/C}$, such that
		\begin{align*}
		\sup_{\mu_0 \in \mathcal{U}_{\vartheta}} \Prob \Big(\sup_{\eta \in \Xi_K}  \abs{\hat{\gamma}_\eta-\gamma_{\eta,\ast}}  \geq n^{-\vartheta} \Big)\leq C_2 n^{-1/7}.
		\end{align*}
	\end{enumerate}
\end{theorem}

\begin{proof}[Proof of Theorem \ref{thm:trace_app} for $\hat{\tau}_\eta$] All the constants in $\lesssim,\gtrsim,\asymp$ may depend on $K$. 

     Let $\kappa_0$ be defined in the same way as in the proof of Proposition \ref{prop:delocal_u_v}. Using a similar local law and continuity argument as in the proof of that proposition, on an event $E_0$ with $\Prob(E_0)\geq 1- Cn^{-D}$, 
	\begin{align*}
	\sup_{\eta \in \Xi_K}\bigabs{m^{-1}\tr \big(\check{\Sigma}+ (\eta/\phi) I\big)^{-1}-\mathfrak{m}\big(-\eta/\phi\big) } \lesssim \kappa_0^{-1}n^{-1/2+\epsilon}.
	\end{align*}
	So on $E_0\cap \mathscr{E}(C_1)$, where $\mathscr{E}(C_1) \equiv \{ \pnorm{Z}{\op}/\sqrt{n} \leq  C_1 \}$ with $\Prob(\mathscr{E}(C_1))\geq 1- Ce^{-n/C}$, uniformly in $\eta \in \Xi_K$,
	\begin{align*}
	\abs{\hat{\tau}_\eta-\tau_{\eta,\ast}}&\leq \frac{\bigabs{\frac{1}{m}\tr \big(\check{\Sigma}+ \frac{\eta}{\phi} I\big)^{-1}-\mathfrak{m}\big(-\frac{\eta}{\phi}\big) }}{\frac{1}{m}\tr \big(\check{\Sigma}+ \frac{\eta}{\phi} I\big)^{-1}\cdot \mathfrak{m}\big(-\frac{\eta}{\phi}\big) }\lesssim \Big\{C_1^2\mathbf{1}_{\phi^{-1} \geq 1 + 1/K}^{-1} \wedge  \eta^{-1}\Big\} \cdot  \kappa_0^{-1} n^{-1/2+\epsilon}.
	\end{align*}
	Here in the last inequality, we use the following estimate for $\mathfrak{m}(z)$: As $\mathfrak{m}$ is the Stieltjes transform of $\rho$ (cf. \cite[Lemma 2.2]{knowles2017anisotropic}), $\mathfrak{m}(z)\geq 0$ for $z\leq 0$, and 
	\begin{align*}
	\frac{1}{\mathfrak{m}(z)} = (-z)+ \frac{1}{m} \tr \Big(\big( I + \Sigma \mathfrak{m}(z) \big)^{-1}\Sigma\Big)\lesssim 1+\abs{z}. 
	\end{align*}
	The claim follows.
\end{proof}

\begin{proof}[Proof of Theorem \ref{thm:trace_app} for $\hat{\gamma}_\eta$]
All the constants in $\lesssim,\gtrsim,\asymp$ may depend on $K$. 

Using Theorem \ref{thm:errors_explicit}, the stability of $\tau_{\eta,\ast}$ in Proposition \ref{prop:fpe_est}, and the proven fact in (1) on $\hat{\tau}_\eta$, it holds for $\epsilon \in (0,1/2]$ that
\begin{align}\label{ineq:trace_app_gamma_1}
\Prob\Big(\sup_{\eta \in [1/K,K]}\bigabs{ \eta^{-1}{\hat{\tau}_\eta} \pnorm{\hat{r}_\eta(\sigma_\xi)}{}- \gamma_{\eta,\ast}(\sigma_\xi) }\geq \epsilon\Big)\leq C_1 \epsilon^{-c_0}n^{-1/6.5}. 
\end{align} 
Next we consider extension to $\eta \in [0,K]$ in the regime $\phi^{-1}\geq 1+1/K$. By KKT condition, we have $n^{-1}X^\top (Y-X\hat{\mu}_\eta) = \eta \hat{\mu}_\eta$, so a.s. $\hat{r}_\eta/\eta=(Y-X\mu_\eta)/(\sqrt{n}\eta) = \sqrt{n}(XX^\top)^{-1} X\hat{\mu}_\eta$ for any $\eta>0$. So we only need to verify the high probability Lipschitz continuity for $\eta \mapsto \sqrt{n}\hat{\tau}_\eta(XX^\top)^{-1} X\hat{\mu}_\eta$: for any $\eta_1,\eta_2 \in [0,K]$, using the estimate (\ref{ineq:dist_min_norm_6}) (with $G$ replaced by $Z$) we obtain, for some universal $c_0>1$,
\begin{align*}
&\bigabs{\sqrt{n} \hat{\tau}_{\eta_1} \bigpnorm{ (XX^\top)^{-1} X\hat{\mu}_{\eta_1}}{ }- \sqrt{n} \hat{\tau}_{\eta_2} \bigpnorm{ (XX^\top)^{-1} X\hat{\mu}_{\eta_2}}{ } }\\
&\lesssim \pnorm{(ZZ^\top/n)^{-1}}{\op} \cdot ({\pnorm{Z}{\op}}/{\sqrt{n}})\cdot\Big(\abs{ \hat{\tau}_{\eta_1}-\hat{\tau}_{\eta_2} } \cdot \pnorm{\hat{\mu}_{\eta_1}}{}+\abs{\hat{\tau}_{\eta_2}}\cdot \pnorm{\hat{\mu}_{\eta_1}-\hat{\mu}_{\eta_2}}{} \Big)\\
&\lesssim \Big(1+\frac{\pnorm{Z}{\op}+\pnorm{\xi_0}{}}{\sqrt{n}}+\pnorm{(ZZ^\top /n)^{-1}}{\op}\Big)^{c_0}\cdot \abs{\eta_2-\eta_1}. 
\end{align*}
Finally we consider extension to $\sigma_\xi^2 \in [0,K]$ in the same regime $\phi^{-1}\geq 1+1/K$ by verifying a similar high probability uniform-in-$\eta$ Lipschitz continuity property for $\sigma_\xi \mapsto \sqrt{n}(XX^\top)^{-1} X\hat{\mu}_\eta(\sigma_\xi)$: for any $\sigma_\xi,\sigma_\xi' \in [0,K]$, using the estimate (\ref{ineq:university_squared_risk_4}),
\begin{align*}
&\sup_{\eta \in [0,K]}\bigabs{\sqrt{n} \hat{\tau}_{\eta} \bigpnorm{ (XX^\top)^{-1} X\hat{\mu}_{\eta}(\sigma_\xi) }{ }- \sqrt{n} \hat{\tau}_{\eta} \bigpnorm{ (XX^\top)^{-1} X\hat{\mu}_{\eta}(\sigma_\xi')}{ } }\\
&\lesssim \pnorm{(ZZ^\top /n)^{-1}}{\op}^2\cdot ({\pnorm{Z}{\op}}/{\sqrt{n}})\cdot \sup_{\eta \in [0,K]} \pnorm{\hat{\mu}_{\eta}(\sigma_\xi)-\hat{\mu}_{\eta}(\sigma_\xi')}{}\\
&\lesssim  \Big(1+\frac{\pnorm{Z}{\op}+\pnorm{\xi_0}{}}{\sqrt{n}}+\pnorm{(ZZ^\top /n)^{-1}}{\op}\Big)^{c_0} \cdot \abs{\sigma_\xi-\sigma_\xi'}. 
\end{align*}
The claimed bound follows.
\end{proof}

\subsection{Proof of Theorem \ref{thm:tuning_res}}

	Recall we have $\gamma_{\eta,\ast}^2 = \phi^{-1}\big(\sigma_\xi^2+ \bar{R}^{\pred}_{(\Sigma,\mu_0)}(\eta)\big)$. For both the case $\sigma_\xi^2 \in [1/K, K]$ and $\sigma_\xi^2 \in [0,K]$ with $\phi^{-1}\geq 1+1/K$, we take $\mathcal{W}_{\epsilon,\delta,\vartheta}\subset B_n(1)\setminus B_n(\delta)$ as constructed in (\ref{ineq:small_intep_1a}) in the proof of Theorem \ref{thm:small_intep}, with $\epsilon\equiv \epsilon_n \equiv n^{-\vartheta}$. Fix $\mu_0 \in \mathcal{W}_{\epsilon,\delta,\vartheta}$, then $\eta_\ast = \SNR_{\mu_0}^{-1} \in \Xi_L$. Using Theorems \ref{thm:error_rmt} and \ref{thm:trace_app}, on an event $E_0$ with $\Prob(E_0^c)\leq C n^{-1/7}$,  
	\begin{align}\label{ineq:tuning_res_0}
	\sup_{\eta \in \Xi_L}\bigabs{\hat{\gamma}_\eta^2- \phi^{-1}\big(\sigma_\xi^2+ \mathscr{R}^{\pred}_{(\Sigma,\mu_0)}(\eta)\big)  }\leq \epsilon. 
	\end{align}
	This in particular implies that on $E_0$, both the following inequalities hold:
	\begin{align}\label{ineq:tuning_res_1}
	\phi\hat{\gamma}_{\hat{\eta}^{\GCV} }^2-\sigma_\xi^2-\phi\epsilon & \leq \mathscr{R}^{\pred}_{(\Sigma,\mu_0)}(\hat{\eta}^{\GCV})\leq \phi\hat{\gamma}_{\hat{\eta}^{\GCV} }^2-\sigma_\xi^2+\phi\epsilon,\nonumber\\
	\phi\min_{\eta \in \Xi_L}\hat{\gamma}_{\eta }^2-\sigma_\xi^2-\phi\epsilon & \leq \min_{\eta \in \Xi_L} \mathscr{R}^{\pred}_{(\Sigma,\mu_0)}(\eta) \leq \phi \min_{\eta \in \Xi_L}\hat{\gamma}_{\eta }^2-\sigma_\xi^2+\phi\epsilon.
	\end{align}
	Using the definition of $\hat{\eta}^{\GCV}$ which gives $\hat{\gamma}_{\hat{\eta}^{\GCV} }^2= \min_{\eta \in \Xi_L}\hat{\gamma}_{\eta }^2$, the above two displays can be used to relate $\mathscr{R}^{\pred}_{(\Sigma,\mu_0)}(\hat{\eta}^{\GCV})$ and $\min_{\eta \in \Xi_L} \mathscr{R}^{\pred}_{(\Sigma,\mu_0)}(\eta)$: on the event $E_0$,
	\begin{align}\label{ineq:tuning_res_2}
	\bigabs{\mathscr{R}^{\pred}_{(\Sigma,\mu_0)}(\hat{\eta}^{\GCV})-\min_{\eta \in \Xi_L} \mathscr{R}^{\pred}_{(\Sigma,\mu_0)}(\eta)}\leq 2\phi\epsilon.
	\end{align}
	As $\eta_\ast \in \Xi_L$, $\min_{\eta \in \Xi_L} \mathscr{R}^{\#}_{(\Sigma,\mu_0)}(\eta)= \mathscr{R}^{\#}_{(\Sigma,\mu_0)}(\eta_\ast)$ for $\# \in \{\pred,\est,\ins\}$. Consequently, by the second inequality in Proposition \ref{prop:derivative_R}, we have on the event $E_0$,
	\begin{align}\label{ineq:tuning_res_3}
	 \abs{\hat{\eta}^{\GCV}-\eta_\ast}\leq \frac{C}{\pnorm{\mu_0}{}} \bigabs{\mathscr{R}^{\pred}_{(\Sigma,\mu_0)}(\hat{\eta}^{\GCV})-\mathscr{R}^{\pred}_{(\Sigma,\mu_0)}(\eta_\ast)}^{1/2}\leq C_1\epsilon^{1/2}.
	\end{align}
	This means on $E_0$, for both $\# \in \{\est,\ins\}$,
	\begin{align*}
	\bigabs{\mathscr{R}^{\#}_{(\Sigma,\mu_0)}(\hat{\eta}^{\GCV})-\min_{\eta \in \Xi_L} \mathscr{R}^{\#}_{(\Sigma,\mu_0)}(\eta)}\leq C_2 \epsilon.
	\end{align*}
	We may conclude from here by virtues of Theorems \ref{thm:errors_explicit} and \ref{thm:error_rmt}, together with Lemma \ref{lem:university_squared_risk}.\qed

\subsection{Proof of Theorem \ref{thm:tuning_cv}}

\begin{lemma}\label{lem:fpe_cv_stability}
	Consider the following version of (\ref{eqn:fpe}) with sample size $m-m_\ell$:
	\begin{align}\label{eqn:fpe_ell}
	\begin{cases}
	\frac{m-m_\ell}{n}\cdot \gamma^2  = \sigma_\xi^2+ \E\err_{(\Sigma,\mu_0)}(\gamma;\tau),\\
	\big(\frac{m-m_\ell}{n}-\frac{\eta}{\tau}\big)\cdot \gamma^2 = \E\dof_{(\Sigma,\mu_0)}(\gamma;\tau).
	\end{cases}
	\end{align}
	\begin{enumerate}
		\item The fixed point equation (\ref{eqn:fpe_ell}) admits a unique solution $(\gamma_{\eta,\ast}^{(\ell)},\tau_{\eta,\ast}^{(\ell)}) \in (0,\infty)^2$, for all $(m,n)\in \N^2$ when $\eta>0$ and $m<n$ when $\eta=0$. 
		\item Further suppose $1/K\leq \phi^{-1}, \sigma_\xi^2 \leq K$, $m_\ell/n\leq 1/(2K)$ and $\pnorm{\Sigma^{-1}}{\op}\vee \pnorm{\Sigma}{\op}\leq K$ for some $K>10$. Then there exists some $C=C(K)>1$ such that uniformly in $\eta \in \Xi_K$, $1/C\leq \gamma_{\eta,\ast}^{(\ell)},\tau_{\eta,\ast}^{(\ell)} \leq C$. Moreover,
		\begin{align*}
		\abs{\gamma_{\eta,\ast}^{(\ell)}-\gamma_{\eta,\ast}}\vee \abs{\tau_{\eta,\ast}^{(\ell)}-\tau_{\eta,\ast}}\leq \frac{C m_\ell}{n}.
		\end{align*}
	\end{enumerate}
\end{lemma}
\begin{proof}
	All the constants in $\lesssim,\gtrsim,\asymp$ below may depend on $K$. We only need to prove (2). The method of proof is similar to that of Proposition \ref{prop:fpe_est}-(3). Instead of considering (\ref{eqn:fpe_ell}), we shall consider the system of equations
	\begin{align}\label{ineq:fpe_cv_stability_1}
	\begin{cases}
	\phi-\alpha  = \frac{1}{\gamma^2}\big(\sigma_\xi^2+\tau^2 \pnorm{(\Sigma+\tau I)^{-1}\Sigma^{1/2}\mu_0}{}^2\big)+ \frac{1}{n} \tr\big( (\Sigma+\tau I)^{-2}\Sigma^2\big),\\
	\phi-\alpha = \frac{1}{n}\tr\big((\Sigma+\tau I)^{-1}\Sigma\big)+ \frac{\eta}{\tau},
	\end{cases}
	\end{align}
	indexed by $\alpha\geq 0$. For $\alpha \in [0,1/(2K)]$, the solution $(\gamma_{\eta,\ast}(\alpha),\tau_{\eta,\ast}(\alpha))$ exists uniquely for $\eta>0$ and also for $\eta=0$ if additionally $m<n$. Moreover, using the apriori estimate in Proposition \ref{prop:fpe_est}-(2), we have uniformly in $\eta \in \Xi_K$ and $\alpha \in [0,1/(2K)]$, $\gamma_{\eta,\ast}(\alpha),\tau_{\eta,\ast}(\alpha)\asymp 1$. Now differentiating on both sides of the second equation in (\ref{ineq:fpe_cv_stability_1}) with respect to $\alpha$, we obtain
	\begin{align*}
	1 = \Big(n^{-1}\tr\big((\Sigma+\tau_{\eta,\ast}(\alpha)I)^{-2}\Sigma\big)+{\eta}{\tau_{\eta,\ast}^{-2}(\alpha)}\Big)\cdot \tau_{\eta,\ast}'(\alpha).
	\end{align*}
	This means uniformly in $\eta \in \Xi_K$ and $\alpha \in [0,1/(2K)]$, $\tau_{\eta,\ast}'(\alpha)\asymp 1$. Next, using the first equation in (\ref{ineq:fpe_cv_stability_1}), we obtain 
	\begin{align*}
	\gamma_{\eta,\ast}^2(\alpha) = \frac{\sigma_\xi^2+\tau_{\eta,\ast}^2(\alpha) \pnorm{(\Sigma+\tau_{\eta,\ast}(\alpha) I)^{-1}\Sigma^{1/2}\mu_0}{}^2}{\phi-\alpha- \frac{1}{n}\tr\big( (\Sigma+\tau_{\eta,\ast}(\alpha) I)^{-2}\Sigma^2\big) }\equiv \frac{G_{1,\eta}(\alpha)}{G_{2,\eta}(\alpha)}.
	\end{align*}
	Using similar calculations as in (\ref{ineq:fpe_property_7})-(\ref{ineq:fpe_property_8}), we have uniformly in $\eta \in \Xi_K$ and $\alpha \in [0,1/(2K)]$, $G_{1,\eta}(\alpha),G_{2,\eta}(\alpha)\asymp 1$, and $
	\abs{ G_{1,\eta}'(\alpha) }\vee \abs{ G_{2,\eta}'(\alpha) }  \lesssim 1$. 
	This concludes the claim. 
\end{proof}

\begin{proof}[Proof of Theorem \ref{thm:tuning_cv}]	
	All the constants in $\lesssim,\gtrsim,\asymp$ below may  depend on $K,L$.
	
	As $\pnorm{Y^{(\ell)}-X^{(\ell)} \hat{\mu}^{(\ell)}_{\eta} }{}^2=\pnorm{Z^{(\ell)}\Sigma^{1/2}(\mu_0-\hat{\mu}^{(\ell)}_{\eta})+\xi^{(\ell)}}{}^2$ and $\hat{\mu}^{(\ell)}_{\eta}$ is independent of $(Z^{(\ell)}, \xi^{(\ell)})$, by using Lemma \ref{lem:conc_quad} first conditionally on $(Z^{(-\ell)},\xi^{(-\ell)})$ and then further taking expectation over $(Z^{(-\ell)},\xi^{(-\ell)})$, we have for $0<\varrho\leq 1$,
	\begin{align*}
	&\Prob\Big(E_{0,\ell}^c(\eta)\equiv \Big\{\bigabs{m_\ell^{-1}\pnorm{Y^{(\ell)}-X^{(\ell)} \hat{\mu}^{(\ell)}_{\eta} }{}^2-\big(\pnorm{\Sigma^{1/2}(\hat{\mu}^{(\ell)}_{\eta}-\mu_0)}{}^2+\sigma_\xi^2\big)}\\
	&\qquad\qquad \geq C_0\big(\sigma_\xi^2\vee \pnorm{\Sigma^{1/2}(\hat{\mu}^{(\ell)}_{\eta}-\mu_0)}{}^2\big) m_\ell^{-(1-\varrho)/2}\Big\}\Big)\leq C_0 e^{-m_\ell^\varrho/C_0}.
	\end{align*}
	Here $C_0>0$ is a universal constant. Using similar arguments as in (\ref{ineq:university_squared_risk_2}) (by noting that the normalization in $\hat{\mu}_\eta^{(\ell)}$ is still $n$), there exists some constant $C_1>0$ such that for any $\ell \in [k]$, on an event $E_{1,\ell}$ with $\Prob(E_{1,\ell}^c)\leq C_1 e^{-m_\ell/C_1}$, $\sup_{\eta \in \Xi_L} \pnorm{\hat{\mu}^{(\ell)}_{\eta}}{}\leq C_1$. This means that for any $\eta \in \Xi_L$, on the event $\cap_{\ell \in [k]}(E_{0,\ell}(\eta)\cap E_{1,\ell})$, 
	\begin{align}\label{ineq:tuning_cv_1}
	\max_{\ell \in [k]} m_\ell^{(1-\varrho)/2}\cdot \bigabs{m_\ell^{-1}\pnorm{Y^{(\ell)}-X^{(\ell)} \hat{\mu}^{(\ell)}_{\eta} }{}^2-\big(\pnorm{\Sigma^{1/2}(\hat{\mu}^{(\ell)}_{\eta}-\mu_0)}{}^2+\sigma_\xi^2\big) }\leq C_1'. 
	\end{align}
	On the other hand, using Theorem \ref{thm:errors_explicit}, we may find some $\mathcal{U}_{\vartheta;\ell}\subset B_n(1)$ with $\mathrm{vol}(\mathcal{U}_{\vartheta;\ell})/\mathrm{vol}(B_n(1))\geq 1- C_2 e^{-n^{\vartheta}/C_2}$, such that for $\epsilon \in (0, 1/2]$, on an event $E_{2,\ell}(\epsilon)$ with $\Prob(E_{2,\ell}^c(\epsilon))\leq C_2 (n e^{-n\epsilon^4/C_2}+\epsilon^{-c_0} n^{-1/6.5}\bm{1}_{Z\neq G})$, for $\mu_0 \in \mathcal{U}_{\vartheta;\ell}$,
	\begin{align*}
	\sup_{\eta \in \Xi_L}\biggabs{ \pnorm{\Sigma^{1/2}(\hat{\mu}^{(\ell)}_{\eta}-\mu_0)}{}^2- \bigg\{\frac{m-m_\ell}{n}(\gamma_{\eta,\ast}^{(\ell)})^2-\sigma_\xi^2\bigg\} }\leq \epsilon. 
	\end{align*}
	Here $\gamma_{\eta,\ast}^{(\ell)}$ is taken from Lemma \ref{lem:fpe_cv_stability}, and we extend the definition to $\ell=0$ with $\hat{\mu}_\eta^{(0)}\equiv \hat{\mu}_\eta$ and $\gamma_{\eta,\ast}^{(0)}\equiv \gamma_{\eta,\ast}$.  Using the statement (2) of the same Lemma \ref{lem:fpe_cv_stability}, on the event $E_{2,\ell}(\epsilon)$, we then have
	\begin{align*}
	\sup_{\eta \in \Xi_L}\bigabs{ \pnorm{\Sigma^{1/2}(\hat{\mu}^{(\ell)}_{\eta}-\mu_0)}{}^2- \big\{\phi\gamma_{\eta,\ast}^2-\sigma_\xi^2\big\} }\leq \epsilon+\frac{C_2 m_\ell}{n}.
	\end{align*}
	Replacing $\phi\gamma_{\eta,\ast}^2-\sigma_\xi^2$ by $R^{\pred}_{(\Sigma,\mu_0)}(\eta)=\pnorm{\Sigma^{1/2}(\hat{\mu}_{\eta}-\mu_0)}{}^2$ yields that, on $\cap_{\ell \in [0:k]} E_{2,\ell}(\epsilon)$,
	\begin{align}\label{ineq:tuning_cv_2}
	\sup_{\eta \in \Xi_L}\bigabs{ \pnorm{\Sigma^{1/2}(\hat{\mu}^{(\ell)}_{\eta}-\mu_0)}{}^2- R^{\pred}_{(\Sigma,\mu_0)}(\eta)}\leq 2\epsilon+\frac{C_2 m_\ell}{n}.
	\end{align}
	Combining (\ref{ineq:tuning_cv_1})-(\ref{ineq:tuning_cv_2}), for $\mu_0\in \mathcal{U}_{\vartheta}\equiv \cap_{\ell \in [0:k]} \mathcal{U}_{\vartheta;\ell}$, $\epsilon \in (0, 1/2]$ and $\eta \in \Xi_L$, 
	\begin{align}\label{ineq:tuning_cv_3}
	&\Prob\bigg(\bigabs{R^{\CV,k}_{(\Sigma,\mu_0)}(\eta)-\big(R^{\pred}_{(\Sigma,\mu_0)}(\eta)+\sigma_\xi^2\big)} \geq C_2'\cdot \bigg\{\frac{1}{k}\sum_{\ell\in [k]}\frac{1}{m_\ell^{(1-\varrho)/2}}+\frac{1}{k}+\epsilon\bigg\} \bigg)\nonumber\\
	&\leq C_2' \cdot 
	\begin{cases}
	\sum_{\ell \in [k]}e^{-m_\ell^\varrho/C_0}+ kne^{-n\epsilon^4/C_2}, & Z=G;\\
	\sum_{\ell \in [k]}e^{-m_\ell^\varrho/C_0}+ \epsilon^{-c_0}\cdot kn^{-1/6.5}, & \hbox{otherwise}.
	\end{cases}
	\end{align}
	Now we strengthen the estimate (\ref{ineq:tuning_cv_3}) into a uniform version. It is easy to verify that on an event $E_{3,\ell}$ with $\Prob(E_{3,\ell}^c)\leq C_3 e^{-m_\ell/C_3}$, $\pnorm{Z^{(\ell)}}{\op}\leq C_3(\sqrt{m_\ell}+\sqrt{n})$, $\pnorm{\xi^{(\ell)}}{}\leq C_3 \sqrt{m_\ell}$, and for $\eta_1,\eta_2 \in \Xi_L$, $
	\pnorm{ \hat{\mu}_{\eta_1}^{(\ell)}- \hat{\mu}_{\eta_2}^{(\ell)}}{}\leq C_3 \abs{\eta_1-\eta_2}$. 
	So on $\cap_{\ell \in [k]} (E_{1,\ell}\cap E_{3,\ell})$, for $\eta_1,\eta_2 \in \Xi_L$,
	\begin{align*}
	\bigabs{R^{\CV,k}_{(\Sigma,\mu_0)}(\eta_1)-R^{\CV,k}_{(\Sigma,\mu_0)}(\eta_2)}&\lesssim \frac{1}{k}\sum_{\ell \in [k]} \frac{1}{m_\ell} \bigabs{ \pnorm{Z^{(\ell)}}{\op}\pnorm{ \hat{\mu}_{\eta_1}^{(\ell)}- \hat{\mu}_{\eta_2}^{(\ell)}}{}\cdot \big(\pnorm{Z^{(\ell)}}{\op}+\pnorm{\xi^{(\ell)}}{}\big)  }\\
	&\lesssim \frac{1}{k}\sum_{\ell \in [k]} \frac{m_\ell+n}{m_\ell}\cdot \abs{\eta_1-\eta_2} \leq C_3'\cdot  \frac{1}{k}\sum_{\ell \in [k]} \frac{n}{m_\ell}\cdot \abs{\eta_1-\eta_2},
	\end{align*}
	and
	\begin{align*}
	\bigabs{R^{\pred}_{(\Sigma,\mu_0)}(\eta_1)-R^{\pred}_{(\Sigma,\mu_0)}(\eta_2)}\leq C_3' \abs{\eta_1-\eta_2}.
	\end{align*}
	From here, using (i) (\ref{ineq:tuning_cv_3}) along with a discretization and union bound that strengthens (\ref{ineq:tuning_cv_3}) to a uniform control, and (ii) Theorem \ref{thm:errors_explicit} which replaces $R^{\pred}_{(\Sigma,\mu_0)}(\eta)$ by $\bar{R}^{\pred}_{(\Sigma,\mu_0)}(\eta)$, we obtain for $\mu_0\in \mathcal{U}_{\vartheta}$ and $\epsilon \in (0,1/2]$,
	\begin{align*}
	&\Prob\bigg(\sup_{\eta \in \Xi_L}\bigabs{R^{\CV,k}_{(\Sigma,\mu_0)}(\eta)-\big(\bar{R}^{\pred}_{(\Sigma,\mu_0)}(\eta)+\sigma_\xi^2\big)} \geq C_3''\cdot \bigg\{\frac{1}{k}\sum_{\ell\in [k]}\frac{1}{m_\ell^{(1-\varrho)/2}}+\frac{1}{k}+\epsilon\bigg\}\equiv \epsilon_{\{m_\ell\}} \bigg)\nonumber\\
	&\leq \mathfrak{p}_0\equiv \frac{C_3''}{\epsilon k}\sum_{\ell \in [k]} \frac{n}{m_\ell} \cdot
	\begin{cases}
	\sum_{\ell \in [k]}e^{-m_\ell^\varrho/C_0}+ kne^{-n\epsilon^4/C_2}, & Z=G;\\
	\sum_{\ell \in [k]}e^{-m_\ell^\varrho/C_0}+ \epsilon^{-c_0}\cdot kn^{-1/6.5}, & \hbox{otherwise}.
	\end{cases}
	\end{align*}
	Now with the same $\mathcal{W}_{\epsilon,\delta,\vartheta}$ as in  (\ref{ineq:small_intep_1a}) using $\epsilon\equiv \epsilon_n\equiv n^{-\vartheta}$, for any $\mu_0 \in \mathcal{U}_\vartheta\cap \mathcal{W}_{\epsilon,\delta,\vartheta}$, we may further replace $\bar{R}^{\pred}_{(\Sigma,\mu_0)}(\eta)$ by $\mathscr{R}^{\pred}_{(\Sigma,\mu_0)}(\eta)$ in the above display (with a possibly slightly larger $\epsilon_{\{m_\ell\}}$, but for notational simplicity we abuse this notation). In summary, for any $\mu_0 \in \mathcal{U}_\vartheta\cap\mathcal{W}_{\epsilon,\delta,\vartheta}$, on an event $E_4$ with $\Prob(E_4^c)\leq \mathfrak{p}_0$,
	\begin{align*}
    \sup_{\eta \in \Xi_L} \bigabs{ R^{\CV,k}_{(\Sigma,\mu_0)}(\eta)-\big(\mathscr{R}^{\pred}_{(\Sigma,\mu_0)}(\eta)+\sigma_\xi^2\big)}\leq \epsilon_{\{m_\ell\}}.
	\end{align*}
    From here, using similar arguments as in (\ref{ineq:tuning_res_1})-(\ref{ineq:tuning_res_2}), on the event $E_4$, 
    \begin{align*}
    \bigabs{\mathscr{R}^{\pred}_{(\Sigma,\mu_0)}(\hat{\eta}^{\CV})-\min_{\eta \in \Xi_L}\mathscr{R}^{\pred}_{(\Sigma,\mu_0)}(\eta) }\leq 2 \epsilon_{\{m_\ell\}}.
    \end{align*}
    Similar to (\ref{ineq:tuning_res_3}), on the event $E_4$, we have
    \begin{align}\label{ineq:tuning_cv_4}
    \abs{\hat{\eta}^{\CV}-\eta_\ast}\leq  C_4\cdot \epsilon_{\{m_\ell\}}^{1/2}.
    \end{align} 
    From here we may argue along the same lines as those following (\ref{ineq:tuning_res_3}) in the proof of Theorem \ref{thm:tuning_res} to conclude with probability estimated at $\mathfrak{p}_0$, by further noting that $\mathcal{U}_\vartheta\cap \mathcal{W}_{\epsilon,\delta,\vartheta}$ satisfies the desired volume estimate. Under the further condition $\min_{\ell \in [k]} m_\ell \geq\log^{2/\delta} m$, by taking $\varrho=\delta$, $\mathfrak{p}_0$ simplifies as indicated in the statement of the theorem for $n$ large. 
\end{proof}

\subsection{Proof of Theorem \ref{thm:CI_cv}}

We only prove the case for $\#=\GCV$; the other case is similar. All constants in $\lesssim,\gtrsim,\asymp$ and  $\mathcal{O}$ may possibly depend on $K,L$. Let $\mathcal{W}_{\epsilon,\delta,\vartheta}\subset B_n(1)\setminus B_n(\delta)$ be as constructed in (\ref{ineq:small_intep_1a}) with $\epsilon\equiv \epsilon_n\equiv n^{-\vartheta}$. 

\noindent (1). We first prove the statement for the length of the CI. Note that $|\mathrm{CI}_j(\eta)|=2\hat{\gamma}_\eta (\Sigma^{-1})_{jj}^{1/2}z_{\alpha/2}/\sqrt{n}$. By Theorem \ref{thm:trace_app}-(2), on an event $E_0$ with the probability indicated therein, 
\begin{align*}
&\max_{j \in [n]}\sup_{\eta \in \Xi_L}\biggabs{|\mathrm{CI}_j(\eta)|- 2\gamma_{\eta,\ast} (\Sigma^{-1})_{jj}^{1/2} \frac{z_{\alpha/2}}{\sqrt{n}}}\leq \frac{2\pnorm{\Sigma^{-1}}{\op}^{1/2} z_{\alpha/2}}{\sqrt{n}} \sup_{\eta \in \Xi_L}\abs{\hat{\gamma}_\eta-\gamma_{\eta,\ast}}\lesssim \frac{z_{\alpha/2}}{\sqrt{n}}\cdot \epsilon. 
\end{align*}
Consequently, on the event $E_0$, for any $\mu_0 \in \mathcal{W}_{\epsilon,\delta,\vartheta}$,
\begin{align*}
&\sqrt{n} z_{\alpha/2}^{-1}\cdot \max_{j \in [n]} \bigabs{ |\mathrm{CI}_j(\hat{\eta}^{\GCV})|-\min_{\eta \in \Xi_L} |\mathrm{CI}_j(\eta)| }\\
&\lesssim \abs{\gamma_{\hat{\eta}^{\GCV},\ast}-\min_{\eta \in \Xi_L}\gamma_{\eta,\ast}}+\epsilon \\
&\lesssim \abs{\gamma_{\hat{\eta}^{\GCV},\ast}^2-\min_{\eta \in \Xi_L}\gamma_{\eta,\ast}^2}+\epsilon \quad  (\hbox{using Proposition \ref{prop:fpe_est}-(3)})\\
&\asymp \bigabs{\bar{R}^{\pred}_{(\Sigma,\mu_0)}(\hat{\eta}^{\GCV})-\min_{\eta \in \Xi_L} \bar{R}^{\pred}_{(\Sigma,\mu_0)}(\eta) }+\epsilon \quad (\hbox{using definition of $\gamma_{\eta,\ast}^2$})\\
& \lesssim \bigabs{\mathscr{R}^{\pred}_{(\Sigma,\mu_0)}(\hat{\eta}^{\GCV})-\min_{\eta \in \Xi_L} \mathscr{R}^{\pred}_{(\Sigma,\mu_0)}(\eta) }+\epsilon\quad (\hbox{using Theorem \ref{thm:error_rmt}}).
\end{align*}
As in the proof of Theorem \ref{thm:tuning_res}, for $\sigma_\xi^2\leq K$, $\eta_\ast=\SNR_{\mu_0}^{-1} \in \Xi_L$, so by using Proposition \ref{prop:derivative_R}-(2), on the event $E_0$, for any $\mu_0 \in \mathcal{W}_{\epsilon,\delta,\vartheta}$,
 \begin{align*}
 &\sqrt{n} z_{\alpha/2}^{-1}\cdot \max_{j \in [n]} \bigabs{ |\mathrm{CI}_j(\hat{\eta}^{\GCV})|-\min_{\eta \in \Xi_L} |\mathrm{CI}_j(\eta)| }\\
 &\lesssim \abs{\mathscr{R}^{\pred}_{(\Sigma,\mu_0)}(\hat{\eta}^{\GCV})- \mathscr{R}^{\pred}_{(\Sigma,\mu_0)}(\eta_\ast) }+\epsilon \lesssim \abs{ \hat{\eta}^{\GCV}-\eta_\ast}^2+\epsilon. 
 \end{align*}
 The above reasoning also proves that on the same event $E_0$, for any $\mu_0 \in \mathcal{W}_{\epsilon,\delta,\vartheta}$,
 \begin{align*}
 \abs{\gamma_{\hat{\eta}^{\GCV},\ast}-\gamma_{\eta_\ast,\ast} }\lesssim \abs{ \hat{\eta}^{\GCV}-\eta_\ast}^2+\epsilon.
 \end{align*}
 From here, in view of (\ref{ineq:tuning_res_3}), by adjusting constants, on an event $E_1$ with $\Prob(E_1^c)\leq C_1 n^{-1/7}$, it holds that
 \begin{align}\label{ineq:CI_cv_1}
 &\sqrt{n} z_{\alpha/2}^{-1}\cdot \max_{j \in [n]} \bigabs{ |\mathrm{CI}_j(\hat{\eta}^{\GCV})|-\min_{\eta \in \Xi_L} |\mathrm{CI}_j(\eta)| } \vee \abs{\gamma_{\hat{\eta}^{\GCV},\ast}-\gamma_{\eta_\ast,\ast} }\vee \abs{ \hat{\eta}^{\GCV}-\eta_\ast}^2  \leq \epsilon.
 \end{align}
 This proves the claim for the length of the CI.
 
 \noindent (2). Next we prove the statement for the coverage. We note that a similar Lipschitz continuity argument as in the proof of Lemma \ref{lem:university_squared_risk} shows that for any $1$-Lipschitz $\mathsf{g}:\R^n \to \R$, on an event $E_2(\mathsf{g})$ with $\Prob(E_2(\mathsf{g})^c)\leq Cn^{-1/7}$,
 \begin{align}\label{ineq:CI_cv_2}
 \sup_{\eta \in \Xi_L}\bigabs{\mathsf{g}(\hat{\mu}^{\dR}_{\eta})-\E \mathsf{g}\big(\mu_0+ \gamma_{\eta,\ast} \Sigma^{-1/2}g/\sqrt{n}\big) }\leq \epsilon.
 \end{align}
 On the other hand, using the Lipschitz continuity of $\eta\mapsto \tau_{\eta,\ast}$ in Proposition \ref{prop:fpe_est}-(3), 
 \begin{align*}
 \bigabs{\mathsf{g}(\hat{\mu}^{\dR}_{\hat{\eta}^{\GCV}})-\mathsf{g}(\hat{\mu}^{\dR}_{\eta_\ast})}&\leq \pnorm{ \hat{\mu}^{\dR}_{\hat{\eta}^{\GCV}}- \hat{\mu}^{\dR}_{\eta_\ast}}{}\lesssim \abs{\tau_{ \hat{\eta}^{\GCV},\ast}-\tau_{\eta_\ast,\ast}}\sup_{\eta \in \Xi_L} \pnorm{\hat{\mu}_\eta}{}+\pnorm{ \hat{\mu}_{\hat{\eta}^{\GCV}}- \hat{\mu}_{\eta_\ast}}{}\\
 &\lesssim \pnorm{ \hat{\eta}^{\GCV}-\eta_\ast}{}\sup_{\eta \in \Xi_L} \pnorm{\hat{\mu}_\eta}{}+\pnorm{ \hat{\mu}_{\hat{\eta}^{\GCV}}- \hat{\mu}_{\eta_\ast}}{}.
 \end{align*}
 So by enlarging $C_1$ if necessary, we may assume without loss of generality that on $E_1\cap E_2(\mathsf{g})$,
 \begin{align}\label{ineq:CI_cv_3}
 \bigabs{\mathsf{g}(\hat{\mu}^{\dR}_{\hat{\eta}^{\GCV}})-\mathsf{g}(\hat{\mu}^{\dR}_{\eta_\ast})}\leq C_1 \epsilon^{1/2}.
 \end{align} 
 Now we shall make a good choice of $\mathsf{g}$ in (\ref{ineq:CI_cv_1}).  Let $\Delta \in (0,1)$ and $\mathsf{g}_{0,\Delta}:\R\to [0,1]$ be a function such that $\mathsf{g}_{0,\Delta}=1$ on $[-1,1]$, $\mathsf{g}_{0,\Delta}=0$ on $\R\setminus (-1-\Delta,1+\Delta)$, and linearly interpolated in $(-1-\Delta,-1)\cup (1,1+\Delta)$. Let
 \begin{align}\label{ineq:CI_cv_4}
 \mathsf{g}(u)\equiv \frac{\Delta}{n}\sum_{j=1}^n \mathsf{g}_{0,\Delta}\bigg(\frac{u_j-\mu_{0,j}}{(\gamma_{\eta_\ast,\ast}+\epsilon)(\Sigma^{-1})_{jj}^{1/2}z_{\alpha/2}/\sqrt{n} }\bigg).
 \end{align}
 It is easy to verify the Lipschitz property of $\mathsf{g}$: for any $u_1,u_2 \in \R^n$, $
 \abs{\mathsf{g}(u_1)-\mathsf{g}(u_2)}\lesssim n^{-1/2} \Delta\pnorm{\mathsf{g}_{0,\Delta}}{\mathrm{Lip}}\sum_{j=1}^n \abs{u_{1,j}-u_{2,j}}\lesssim \pnorm{u_1-u_2}{}$. 
 Consequently, we may apply (\ref{ineq:CI_cv_2}) with $\mathsf{g}$ defined in (\ref{ineq:CI_cv_4}) to obtain that on the event $E_1\cap E_2(\mathsf{g})$,
 \begin{align}\label{ineq:CI_cv_5}
 \mathscr{C}^{\dR}(\hat{\eta}^{\GCV})& = \frac{1}{n}\sum_{j=1}^n \bm{1}\Big(\hat{\mu}^{\dR}_{\hat{\eta}^{\GCV},j} \in \Big[\mu_{0,j}\pm \hat{\gamma}_{\hat{\eta}^{\GCV} }(\Sigma^{-1})^{1/2}_{jj} \frac{z_{\alpha/2}}{\sqrt{n} }\Big]\Big)\nonumber\\
 &\leq \frac{1}{n}\sum_{j=1}^n \bm{1}\Big(\hat{\mu}^{\dR}_{\hat{\eta}^{\GCV},j} \in \Big[\mu_{0,j}\pm (\gamma_{\eta_\ast,\ast}+\epsilon) (\Sigma^{-1})^{1/2}_{jj} \frac{z_{\alpha/2}}{\sqrt{n} }\Big]\Big)\nonumber\\
 &\leq \Delta^{-1} \cdot \mathsf{g}\big(\hat{\mu}^{\dR}_{ \hat{\eta}^{\GCV} }\big)\quad \hbox{(using $\bm{1}_{[-1,1]}\leq \mathsf{g}_{0,\Delta}$)}\nonumber\\
 & \leq \Delta^{-1}\cdot  \mathsf{g}\big(\hat{\mu}^{\dR}_{ \eta_\ast }\big) + \bigo(\epsilon^{1/2}/\Delta) \quad \hbox{(by (\ref{ineq:CI_cv_3}))}\nonumber\\
 &\leq \Delta^{-1}\cdot \E \mathsf{g}\big(\mu_0+ \gamma_{\eta_\ast,\ast} \Sigma^{-1/2}g/\sqrt{n}\big)+ \bigo(\epsilon^{1/2}/\Delta). 
 \end{align}
 Now using $\mathsf{g}_{0,\Delta}\leq \bm{1}_{[-1-\Delta,1+\Delta]}$ and the anti-concentration of the standard normal random variable, we may compute
 \begin{align}\label{ineq:CI_cv_6}
 &\Delta^{-1}\cdot \E \mathsf{g}\big(\mu_0+ \gamma_{\eta_\ast,\ast} \Sigma^{-1/2}g/\sqrt{n}\big) = \E \mathsf{g}_{0,\Delta}\bigg(\frac{\gamma_{\eta_\ast,\ast}}{\gamma_{\eta_\ast,\ast}+\epsilon}\cdot \frac{g}{z_{\alpha/2}}\bigg)\nonumber\\
 &\leq \Prob\Big(\mathcal{N}(0,1) \in \Big[\pm z_{\alpha/2}\cdot \big(1+{\epsilon}/{ \gamma_{\eta_\ast,\ast}}\big)\cdot(1+\Delta)\Big]\Big)\leq 1-\alpha +\bigo(\epsilon+\Delta). 
 \end{align}
 Combining the above two displays (\ref{ineq:CI_cv_5})-(\ref{ineq:CI_cv_6}), on the event $E_1\cap E_2(\mathsf{g})$,
 \begin{align*}
 \mathscr{C}^{\dR}(\hat{\eta}^{\GCV})\leq 1-\alpha+ \bigo(\epsilon+\Delta+\epsilon^{1/2}/\Delta). 
 \end{align*}
 Finally choosing $\Delta=\epsilon^{1/4}$ to conclude the upper control. The lower control can be proved similarly so we omit the details. \qed

\section{Auxiliary results}

\begin{proposition}\label{prop:conc_H_generic}
	Let $H:\R^n \to \R_{\geq 0}$ be a non-negative, differentiable function. Suppose there exists some deterministic $\Gamma>0$ such that $\pnorm{\nabla H(g)}{}^2\leq \Gamma^2 H(g)$ almost surely for $g\sim \mathcal{N}(0,I_n)$. Then there exists some universal constant $C>0$ such that for all $t\geq 0$,
	\begin{align*}
	\Prob\Big(\abs{H(g)-\E H(g)}/C\geq \Gamma \E^{1/2} H(g)\cdot \sqrt{t}+\Gamma^2\cdot t \Big)\leq Ce^{-t/C}. 
	\end{align*}
\end{proposition}
\begin{proof}
	The method of proof via the Gaussian log-Sobolev inequality and the Herbst's argument is well known. We give some details for the convenience of the reader. Let $Z\equiv H(g)-\E H(g)$ be the centered version of $H$, and $G(g)\equiv \lambda Z = \lambda(H(g)-\E H(g))$. Then $
	\pnorm{\nabla G(g)}{}^2 = \lambda^2 \pnorm{\nabla H(g)}{}^2\leq \lambda^2 \Gamma^2 \cdot H(g)=\lambda^2\Gamma^2\cdot \big(Z+\E H(g)\big)$. 
	By the Gaussian log-Sobolev inequality (see e.g., \cite[Theorem 5.4]{boucheron2013concentration}, or \cite[Theorem 2.5.6]{gine2015mathematical}), 
	\begin{align*}
	\mathrm{Ent}(e^{\lambda Z})=\E[\lambda Z e^{\lambda Z}]-\E e^{\lambda Z}\log \E e^{\lambda Z}\leq \frac{1}{2} \E \big[\lambda^2\Gamma^2 \big(Z+\E H(g)\big) e^{\lambda Z} \big].
	\end{align*}
	With $m_Z(\lambda)\equiv \E e^{\lambda Z}$ denoting the moment generation function of $Z$, the above inequality is equivalent to 
	\begin{align*}
	\lambda m_Z'(\lambda)-m_Z(\lambda) \log m_Z(\lambda)\leq \frac{\Gamma^2 \lambda^2 }{2}\Big(m_Z'(\lambda)+\E H(g)\cdot m_Z(\lambda)\Big).
	\end{align*}
	Now dividing $\lambda^2 m_\lambda(Z)$ on both sides of the above display, we have $
	\big({\log m_Z(\lambda)}/{\lambda}\big)' \leq \frac{\Gamma^2}{2}\big(\log m_Z(\lambda)+\lambda \E H(g)\big)'$. 
	Integrating both sides with the condition $\lim_{\lambda\downarrow 0}(\log m_Z(\lambda)/\lambda)=0$ and $\log m_Z(\lambda)=0$, we arrive at $
	\log m_Z(\lambda)\leq \frac{\Gamma^2}{2}\big(\lambda \log m_Z(\lambda)+\lambda^2 \E H(g)\big)$. 
	Solving for $\log m_Z(\lambda)$ and using the standard method to convert to tail bound yield the claimed inequality. 
\end{proof}

\begin{lemma}\label{lem:gaussian_lq}
	Let $\Sigma \in \R^{n\times n}$ be an invertible covariance matrix with $\pnorm{\Sigma}{\op}\vee \pnorm{\Sigma^{-1}}{\op}\leq K$ for some $K>0$. Then for any $q \in [1,\infty)$, there exists some $C=C(K,q)>0$ such that 
	\begin{align*}
	\biggabs{ \frac{\E \pnorm{\mathcal{N}(0,\Sigma)}{q}}{\pnorm{\mathrm{diag}(\Sigma)}{q/2}^{1/2} M_q}-1}\leq C n^{-\frac{1}{q\vee 2}}\sqrt{\log n}.
	\end{align*}
	where $M_q \equiv  \E^{1/q}\abs{\mathcal{N}(0,1)}^q=2^{1/2}\big\{\Gamma\big((q+1)/2\big)/\sqrt{\pi}\big\}^{1/q}$.
\end{lemma}
\begin{proof}
	Let $g\sim \mathcal{N}(0,I_n)$. We first prove that for some $C_0>1$,
	\begin{align}\label{ineq:gaussian_lq_1}
	n^{\frac{1}{q\vee 2}}/C_0\leq \E \pnorm{\Sigma^{1/2}g}{q}\leq C_0 n^{\frac{1}{q}}.
	\end{align}
	The upper bound in the above display is trivial. For the lower bound, using $\pnorm{x}{}\leq n^{\frac{1}{2}-\frac{1}{q\vee 2}}\pnorm{x}{q}$, we find $\E \pnorm{\Sigma^{1/2}g}{q}\geq n^{-\frac{1}{2}+\frac{1}{q\vee 2}}\E\pnorm{\Sigma^{1/2} g}{}\gtrsim n^{\frac{1}{q\vee 2}}$. This proves (\ref{ineq:gaussian_lq_1}).
	
	As $\pnorm{x}{q}\leq n^{-\frac{1}{2}+\frac{1}{q\wedge 2}} \pnorm{x}{}$,  the map $g\mapsto \pnorm{\Sigma^{1/2} g}{q}$ is $\pnorm{\Sigma}{\op}^{1/2} n^{-\frac{1}{2}+\frac{1}{q\wedge 2}}$-Lipschitz with respect to $\pnorm{\cdot}{}$. So by Gaussian concentration, for any $t\geq 0$,
	\begin{align*}
	\Prob\Big( E(t)^c\equiv \Big\{n^{ \frac{1}{2}-\frac{1}{q\wedge 2}}\bigabs{\pnorm{\Sigma^{1/2} g}{q}-\E \pnorm{\Sigma^{1/2} g}{q}}\geq C\sqrt{t}\Big\} \Big)\leq Ce^{-t/C}.
	\end{align*}
	Consequently, using the above concentration and (\ref{ineq:gaussian_lq_1}),
	\begin{align*}
	\E \pnorm{\Sigma^{1/2} g}{q}^q &\leq \E \pnorm{\Sigma^{1/2} g}{q}^q\bm{1}_{E(t)}+ \E^{1/2} \pnorm{\Sigma^{1/2}g}{q}^{2q}\cdot \Prob^{1/2}(E(t)^c)\\
	&\leq \big(\E \pnorm{\Sigma^{1/2}g}{q}+C\sqrt{t}\big)^q+C\cdot n^{1/q} \Prob^{1/2}(E(t)^c)\\
	&\leq \big(\E \pnorm{\Sigma^{1/2}g}{q}\big)^q\cdot\big\{\big(1+C n^{-\frac{1}{q\vee 2}}\sqrt{t}\big)^q+ C\cdot n^{\frac{1}{q}-\frac{1}{q\vee 2}} \Prob^{1/2}(E(t)^c)\big\}.
	\end{align*}
	By choosing $t=C_1\log n$ for some sufficiently large $C_1>0$, we have
	\begin{align*}
	\frac{\E \pnorm{\mathcal{N}(0,\Sigma)}{q}}{\pnorm{\mathrm{diag}(\Sigma)}{q/2}^{1/2} M_q} = \frac{ \E \pnorm{\Sigma^{1/2}g}{q} }{  \E^{1/q} \pnorm{\Sigma^{1/2}g}{q}^q }\geq \big(1-C n^{-\frac{1}{q\vee 2}}\sqrt{\log n}\big)_+.
	\end{align*}
	The upper bound follows similarly. 
\end{proof}

\begin{lemma}\label{lem:conc_quad}
	Let $Z \in \R^{m\times n}$ be a random matrix with independent, mean-zero, unit variance, uniformly sub-gaussian components. Suppose the coordinates of $\xi$ are i.i.d. mean zero and uniformly subgaussian with variance $\sigma_\xi^2>0$, and are independent of $Z$. Then there exists some universal constant $C>0$ such that for any $b \in \R^n$ and $0<\varrho\leq 1$, with probability at least $1-Ce^{-m^\varrho/C}$, 
	\begin{align*}
	\bigabs{m^{-1}\pnorm{Z b +\xi}{}^2- \big(\pnorm{b}{}^2+\sigma_\xi^2\big)} \leq C\cdot (\sigma_\xi^2\vee \pnorm{b}{}^2)\cdot m^{-(1-\varrho)/2}.
	\end{align*}
\end{lemma}
\begin{proof}
	Let $Z_1,\ldots,Z_m \in \R^n$ be the rows of $Z$. Then
	\begin{align*}
	\frac{1}{m}\pnorm{Z b +\xi}{}^2=\pnorm{b}{}^2 \frac{1}{m}\sum_{i=1}^m \biggiprod{Z_i}{\frac{b}{\pnorm{b}{}} }^2+\frac{2\sigma_\xi \pnorm{b}{}}{m}\sum_{i=1}^n\frac{\xi_i}{\sigma_\xi}\biggiprod{Z_i}{\frac{b}{\pnorm{b}{}} }+\sigma_\xi^2\frac{ \pnorm{\xi/\sigma_\xi}{}^2}{m}.
	\end{align*}
	Using standard concentration estimates, with probability at least $1-Ce^{-m^\varrho/C}$,
	\begin{itemize}
		\item $\bigabs{\pnorm{b}{}^2 \frac{1}{m}\sum_{i=1}^m \bigiprod{Z_i}{\frac{b}{\pnorm{b}{}} }^2 - \pnorm{b}{}^2}\leq  {C\pnorm{b}{}^2}\cdot {m^{-(1-\varrho)/2}}$,
		\item $\bigabs{\frac{2\sigma_\xi \pnorm{b}{}}{m}\sum_{i=1}^n\frac{\xi_i}{\sigma_\xi}\bigiprod{Z_i}{\frac{b}{\pnorm{b}{}} }}\leq  {C\sigma_\xi \pnorm{b}{}}\cdot{m^{-(1-\varrho)/2}}$,
		\item $\bigabs{\sigma_\xi^2\frac{ \pnorm{\xi/\sigma_\xi}{}^2}{m}-\sigma_\xi^2}\leq  {C \sigma_\xi^2 }\cdot{m^{-(1-\varrho)/2}}$. 
	\end{itemize}
	Collecting the bounds to conclude. 
\end{proof}

\section{Simulation details for Figure \ref{fig:2} and additional simulations}\label{section:simulation}

\subsection{Common numerical settings} We set $\Sigma = 1.99 \cdot I_n + 0.01 \cdot \mathsf{1}_n\mathsf{1}_n^{\top}$, with $\mathsf{1}_n$ representing an $n$-dimensional all one vector. The random design matrix $Z$ and the error $\xi$ are both generated by $t$-distribution with $10$ degrees of freedom, scaled by $\sqrt{0.8}$. This scaling choice ensures that $Z_{ij}$ and $\xi_i$ have mean zero and variance one. The concrete choice of the signal dimension $n$, the sample size $m$, and $\mu_0$ will be specified later.

\subsection{Simulation details for Figure \ref{fig:2}}
We investigate the efficacy of two cross validation schemes in Section \ref{section:application}, namely $\hat{\eta}^{\GCV}$ in (\ref{def:res_tuning}) and $\hat{\eta}^{\CV}$ in (\ref{def:CV_tuning}). We keep the sample size fixed at $m = 500$, and allow the signal dimension $n$ to vary so that the aspect ratio $\phi=m/n$ ranges from $[0.5,1.5]$. To facilitate the tuning process, we employ $31$ equidistant $\eta$'s within the range of $[0, 1.5]$. Moreover, the $k$-fold cross validation scheme $\hat{\eta}^{\CV}$ is carried out with the default choice $k=5$. 

To empirically verify Theorem \ref{thm:tuning_res} and \ref{thm:tuning_cv}, we report in the left panel of Figure \ref{fig:2} the empirical risks $R^{\#}_{(\Sigma,\mu_0)}(\hat{\eta}^{\GCV}), R^{\#}_{(\Sigma,\mu_0)}(\hat{\eta}^{\CV})$ for all $\# \in \{\pred,\est,\ins\}$. All the empirical risk curves are found to concentrate around their theoretical optimal counterparts $\mathscr{R}^{\#}_{(\Sigma,\mu_0)}(\eta_\ast)$. We note again that as $\hat{\eta}^{\GCV}$ and $\hat{\eta}^{\CV}$ are designed to tune the prediction risk, it is not surprising that $R^{\pred}_{(\Sigma,\mu_0)}(\hat{\eta}^{\GCV}), R^{\pred}_{(\Sigma,\mu_0)}(\hat{\eta}^{\CV})$ concentrate around $\mathscr{R}^{\pred}_{(\Sigma,\mu_0)}(\eta_\ast)$. The major surprise appears to be that $\hat{\eta}^{\GCV}$ and $\hat{\eta}^{\CV}$ also provide optimal tuning for estimation and in-sample risks, both theoretically validated in our Theorems \ref{thm:tuning_res} and \ref{thm:tuning_cv} and empirically confirmed here.

To empirically verify Theorem \ref{thm:CI_cv}, we report in the middle and right panels of Figure \ref{fig:2} the averaged coverage and length for the $95\%$-debiased Ridge CI's with cross-validation, namely $\{\mathrm{CI}_j(\hat{\eta}^{\#})\}$ for $\# \in \{\GCV,\CV\}$, and with oracle tuning $\eta_\ast=\SNR_{\mu_0}^{-1}$. For the middle panel, we observe that adaptive tuning via $\hat{\eta}^{\GCV}$ and $\hat{\eta}^{\CV}$ both provide approximate nominal coverage for a moderate sample size $m$ and signal dimension $n$. For the right panel, as the lengths of $\{\mathrm{CI}_j(\hat{\eta}^{\#})\}$ are solely determined by $\hat{\gamma}_{\hat{\eta}^{\#}}$, we report here only the length of $\mathrm{CI}_1(\hat{\eta}^{\#})$. We observe that the CI length for both $\mathrm{CI}_1(\hat{\eta}^{\GCV}),\mathrm{CI}_1(\hat{\eta}^{\CV})$ are also in excellent agreement to the oracle length across different aspect ratios.

\subsection{Validation of (\ref{eqn:opt_reg_l2})}
\begin{figure}[t]
	\begin{minipage}[t]{0.3\textwidth}
		\includegraphics[width=\textwidth]{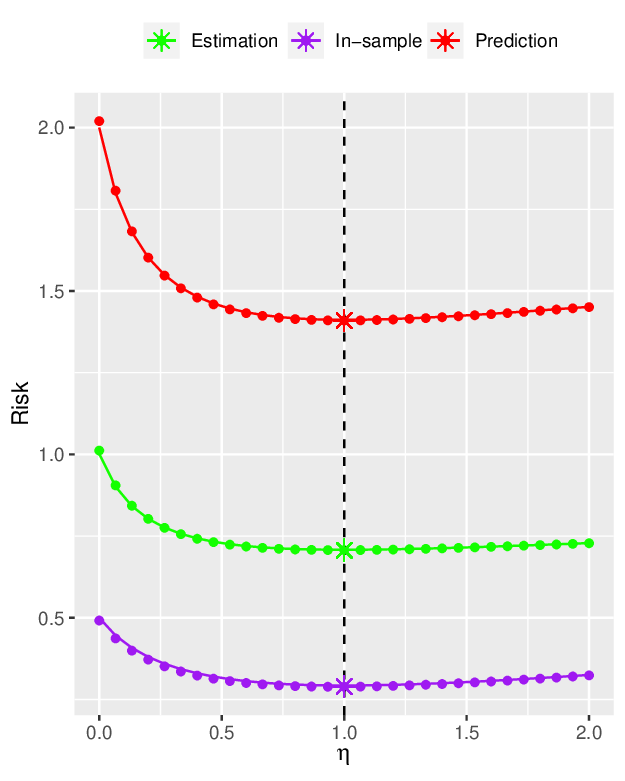}
	\end{minipage}
	\begin{minipage}[t]{0.3\textwidth}
		\includegraphics[width=\textwidth]{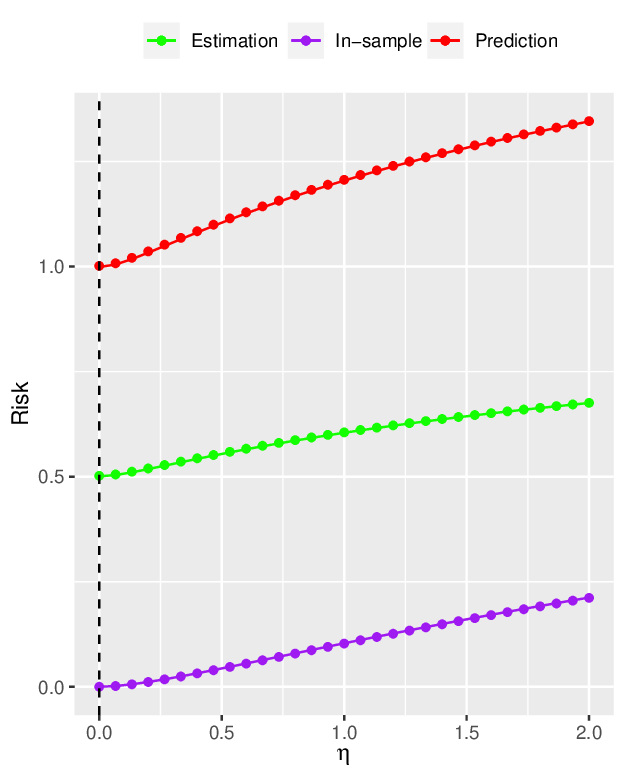}
	\end{minipage}
	\begin{minipage}[t]{0.3\textwidth}
		\includegraphics[width=\textwidth]{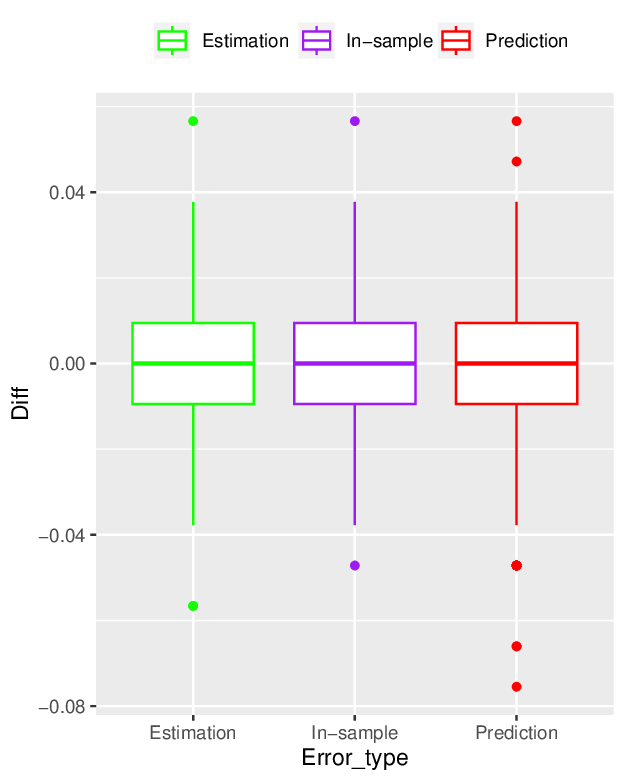}
	\end{minipage}
	\caption{Validation of (\ref{eqn:opt_reg_l2}) (see also Theorem~\ref{thm:small_intep} for a rigorous formulation). The theoretical risks $\bar{R}^{\#}_{(\Sigma,\mu_0)}(\eta)$ are computed by solving (\ref{eqn:fpe}), and the empirical risks $R^{\#}_{(\Sigma,\mu_0)}(\eta)$ are computed via Monte Carlo simulation over 200 repetitions. \emph{Left panel}: noisy case with minimal empirical risks attained at $\eta_\ast=\SNR_{\mu_0}^{-1}=1$ (marked with $\ast$). \emph{Middle panel}: noiseless case with all risks minimized at the interpolation regime $\eta_\ast =\SNR_{\mu_0}^{-1}=0$. \emph{Right panel}: differences between the global minimizer of the empirical risk curves and the oracle $\eta_\ast$ are concentrated around $0$ over 500 different $\mu_0$’s.}
	\label{fig:1}
\end{figure}
We next verify the optimal oracle regularization rule in (\ref{eqn:opt_reg_l2}) (see Theorem \ref{thm:small_intep} for a rigorous formulation) by simulation. 
We use $m = 100$, $n = 200$, and a unit vector $\mu_0$ chosen randomly (and then fixed) from the sphere $\partial B_n(1)$. For this setting, we plot both the theoretical risk curve $\eta \mapsto \bar{R}^{\#}_{(\Sigma,\mu_0)}(\eta)$ and the empirical risk curve $\eta \mapsto R^{\#}_{(\Sigma,\mu_0)}(\eta)$ for all $\# \in \{\pred,\est,\ins\}$. The left panel of Figure~\ref{fig:1} reports the noisy case with noise level $\sigma_\xi^2=1$ and $\SNR_{\mu_0}^{-1}=1$, while the middle panel reports the noiseless case $\sigma_\xi^2=0$ with $\SNR_{\mu_0}^{-1}=0$. These plots show excellent agreement with (\ref{eqn:opt_reg_l2}) in that the global minimum of both the theoretical and empirical risk curves is attained roughly at $\eta_\ast=\SNR_{\mu_0}^{-1}$.

In order to demonstrate the validity of the above phenomenon for `most' $\mu_0$'s, as claimed in Theorem~\ref{thm:small_intep}, we uniformly generate $500$ different $\mu_0$’s over $\partial B_n(1)$. For each $\mu_0$, we discretize $\eta \in [0,1.5]$ into $160$ grid points and select the empirical optimal value $\eta^{\#}$ by minimizing the empirical prediction, estimation, and in-sample risks. The difference between the empirical optimal $\eta^{\#}$ and the theoretically optimal tuning $\eta_\ast$ is depicted in the right panel of Figure~\ref{fig:1} through a boxplot of $\eta^{\#}- \eta_\ast$. It is easily seen that, for all three risks, these differences are highly concentrated around $0$.

We finally explain how the theoretical risk curves $\eta \mapsto \bar{R}^{\#}_{(\Sigma,\mu_0)}(\eta)$ are computed in practice. For each fixed $\eta$, we solve the fixed-point system (\ref{eqn:fpe}) for $(\gamma_{\eta,\ast},\tau_{\eta,\ast})$ as follows. The second equation in (\ref{eqn:fpe}) involves only the scalar variable $\tau$; under our assumptions, its right-hand side is monotone in $\tau$, so the solution $\tau_{\eta,\ast}$ is unique. We therefore solve this one-dimensional fixed-point equation for $\tau_{\eta,\ast}$ by a standard bisection method on a prescribed interval, up to a given numerical tolerance. Once $\tau_{\eta,\ast}$ is obtained, we plug it into the first equation in (\ref{eqn:fpe}) to compute $\gamma_{\eta,\ast}$. The resulting pair $(\gamma_{\eta,\ast},\tau_{\eta,\ast})$ is then substituted into the closed-form expressions for $\bar{R}^{\#}_{(\Sigma,\mu_0)}(\eta)$.

\section*{Acknowledgments}
The research of Q. Han is partially supported by NSF grant DMS-2143468. Both authors would like to thank the referees for their helpful comments and suggestions that significantly improved the quality of the paper.

\bibliographystyle{alpha}
\bibliography{mybib}

\end{document}